%% file: main.tex
\documentclass[reqno,11pt]{amsart}
\usepackage{presets}
\begin{document} 
\title[A CAT(0)-approach to the MLS rigidity of Sinai billiards]{A CAT(0)-approach to the marked length spectral rigidity of Sinai billiards}

\author{Douglas Finamore}
    \address{Centro de Matem\'atica, Computa\c{c}\~ao e Cogni\c{c}\~ao - UFABC\\
    	Alameda da Universidade | 09606-405 S\~ao Bernardo do Campo, SP, Brazil.}
    \email{douglas.finamore@pm.me}	

\author{Martin Leguil}
    \address{Centre de Math\'ematiques Laurent Schwartz - \'Ecole Polytechnique\\
    	Route de Saclay | 91128 Palaiseau Cedex, France.}
    \email{martin.leguil@polytechnique.edu}
 

\begin{abstract}
    We study the (marked) length spectral rigidity problem for Sinai billiards with finite horizon on the two-torus: to what extent do the lengths of (marked) periodic trajectories
    determine the geometry of the table?
    Since the billiard flow is neither smooth nor structurally stable, the natural marking is provided by the surface \(K\) obtained by gluing two copies of the table along their boundary.
    We introduce the~\emph{enriched marked length spectrum} \(\mathbf{EL}\), which assigns to each free homotopy class of \(K\) the length of its shortest representative for the intrinsic metric of the table; it records the billiard periodic orbits together with the ``gliding'' periodic trajectories that slide along the scatterers.
    Our main result is that two finite-horizon Sinai billiards with the same enriched marked length spectrum are isometric.
    Following Kourganoff, the billiard flow is approximated by the Anosov geodesic flows of a family of Riemannian metrics on \(K\) degenerating along the gluing locus. The limit is a \(\mathrm{CAT}(0)\) space
    on which we adapt Otal's argument for negatively curved surfaces, 
    extending marked length spectral rigidity beyond the Riemannian setting.
\end{abstract}
 
\maketitle

\tableofcontents

\section{Introduction}
    \label{sec:intro}
    \input{sections/intro.tex}

\section{Preliminaries on Sinai billiards}
    \label{sec:prelimin}
    \input{sections/prelimin.tex}
\section{Geodesic approximations}
    \label{sec:geodesic}
    \input{sections/geodesic.tex}
\section{Spectral Rigidity}
    \label{sec:spectral.rigidity}
    \input{sections/spectral.rigidity.tex}

\appendix
    \section{Generic properties of periodic orbits of Sinai billiards}
        \label{sec:periodic.orbits}
        \input{sections/periodic.orbits.tex}

    \section{Billiard cycles and the enriched length functional}
        \label{sec:billiards.cycles}
        \input{sections/billiard.cycles.tex}

\subsection*{Acknowledgements}
    We are grateful to the LESET/MATH-AMSUD project, coordinated by J\'er\^{o}me Buzzi, Pablo D. Carrasco, and Radu Saghin, for making this collaboration possible.
    We thank the Centre de Math\'ematiques Laurent Schwartz (CMLS) at \'Ecole polytechnique for its hospitality during the development of this project.
    We also thank Andrey Gogolev for his valuable insights and discussions related to this work, and Micka\"el Kourganoff for permission to reproduce the figures from~\cite{kourganoff_anosov_2016}.
    We are very grateful to David Lautenschlager and Peter Albers for their careful reading of a previous version of this work, as well as for their insightful comments regarding inaccuracies in some statements. In the preparation of the revised version, the authors used Claude Opus (Anthropic) to improve the exposition and to draft the proofs of some technical lemmas in Section~\ref{sec:spectral.rigidity}. The authors reviewed and take full responsibility for all content.

    D.F. was supported by CAPES Grant No. 88887.898617/2023-00 as part of the LESET/MATH-AMSUD project, and would like to acknowledge support from the Max Planck Society as well, through the award of a Max Planck Partner Group for Geometry and Probability in Dynamical Systems.
    M.L. was supported by the ANR AAPG 2021 PRC CoSyDy (Conformally symplectic dynamics, beyond symplectic dynamics; Grant No.
    ANR-CE40-0014) and by the ANR JCJC PADAWAN (Parabolic dynamics, bifurcations, and wandering domains; Grant No.
    ANR-21-CE40-0012).

\bibliographystyle{plain}
\bibliography{bibs/references}
\end{document}

%% file: sections/intro.tex
\par We consider dynamical systems with hyperbolic behaviour arising from geometry, such as geodesic flows on negatively curved Riemannian manifolds or dispersing billiards.
A central question is how to parametrise the moduli space of such systems.
Specifically, given a Riemannian manifold equipped with a negatively curved metric, one seeks to describe the set of all such metrics up to the action of the diffeomorphism group.
Similarly, for billiards, the goal is to understand the set of isometries within a fixed diffeomorphism class of dispersing planar billiards.
These systems typically exhibit a dense set of periodic orbits (restricted to the non-wandering set in the non-compact case), making periodic data a natural choice for moduli.
The information recorded for periodic orbits can be categorised as follows:
\begin{enumerate}
    \item[$\bullet$] {\bf Topological or symbolic data}. 
    	On negatively curved manifolds, and more generally, for metrics with Anosov geodesic flow, it is well-known (see e.g.~\cite{klingenberg_riemannian_1974}) that there is a unique closed geodesic in each free homotopy class.
	For dispersing billiards, such as open dispersing billiards or Sinai billiards, the symbolic data naturally assigned to each periodic orbit is the sequence of scatterers encountered by the trajectory.
    \item[$\bullet$] {\bf Length data}.
    	The \emph{length spectrum} is the collection of all lengths (or periods) of closed orbits.
    \item[$\bullet$] {\bf Lyapunov exponents}.
    	Each periodic orbit in a hyperbolic system is associated with exponents measuring the exponential divergence rate of nearby orbits.
	For \(3\)-dimensional volume preserving hyperbolic flows, this reduces to a single exponent.
\end{enumerate}
A natural question that arises is: to what extent do periodic orbits characterise the system?
A dynamical analogue of Mark Kac's famous question, ``Can one hear the shape of a drum?'', is the following inverse problem: does the length spectrum encode the system's geometry?
Unfortunately, this problem is typically intractable due to the lack of structure in the length spectrum.
Even if two systems share the same length spectrum, one cannot \emph{a priori} rule out the possibility that a given length corresponds to periodic orbits with very different behaviour in each system; we refer e.g.\ to~\cite{vigneras_varietes_1980}, where examples of isospectral yet non-isometric manifolds are constructed.
Matching data between such systems is therefore generally a hard task.
To impose more structure on the spectrum, one usually augments it with symbolic data, yielding a \emph{marked length spectrum}.
Informally, this can be viewed as a (set-valued) function
\[
    \text{symbolic data } \omega \quad \mapsto\quad \{\text{lengths of all periodic orbits with symbol }\omega\}.
\]
In the hyperbolic setting, the expansiveness of the dynamics typically reduces the set on the right-hand side to a single positive number.
For certain systems, such as geodesic flows on negatively curved Riemannian manifolds or open dispersing billiards satisfying the no-eclipse condition (see e.g.~\cite{morita_symbolic_1991}), the dynamics is structurally stable.
Consequently, the symbolic data does not depend on the choice of negatively curved metric for a given Riemannian manifold, nor on the specific geometry of the open dispersing billiard.
This defines a natural marked length spectrum, allowing us to rephrase the previous question as follows:
does the marked length spectrum fully determine the system up to isometry?
\par The question of marked length spectral determination for negatively curved metrics on surfaces was answered affirmatively and simultaneously by Otal~\cite{otal_spectre_1990} and Croke~\cite{croke1990rigidity} in their celebrated work.
Significant progress has recently been made in generalising these results in several directions.
In higher dimensions, Guillarmou and Lefeuvre~\cite{guillarmou_marked_2019} proved local marked length spectral determination for metrics of non-positive curvature with Anosov geodesic flow.
In the same work, they also provided new stability estimates, quantifying the extent to which the marked length spectrum controls the distance between metrics.
Even more recently, for surfaces of genus at least \(2\), Guillarmou, Lefeuvre, and Paternain~\cite{guillarmou2025marked} were able to drop the negative curvature assumption and establish global marked length spectral determination.
More precisely, they showed that two metrics with Anosov geodesic flows on a given surface that share the same marked length spectrum are isometric via an isometry isotopic to the identity.
\par A related framework is that of flat metrics, more precisely, non-positively curved Euclidean cone metrics, where curvature is concentrated at conical singularities.
In this context, Bankovic and Leininger~\cite{bankovic_marked-length-spectral_2017} showed that two flat metrics on a closed, oriented surface that assign the same lengths to all closed curves differ by an isometry isotopic to the identity.
It is also worth noting that the question of marked length spectral rigidity has been investigated in different contexts.
For instance, recent work by Abbondandolo and Mazzucchelli~\cite{abbondandolo_length_2026} establishes that \(\mathbb{S}^1\)-symmetric metrics on the \(2\)-sphere with a single equator and the same marked length spectrum have conjugate geodesic flows, and coincide under an additional \(\mathbb{Z}_2\)-symmetry; here the dynamics is integrable, which sets it apart from the hyperbolic geodesic flows discussed earlier.
\par Spectral rigidity of billiards has been studied for various classes of domains.
De Simoi, Kaloshin and Wei~\cite{desimoikal} showed that convex domains with \(\mathbb{Z}_2\)-symmetry and that are sufficiently close to a circle exhibit dynamical (deformational) spectral rigidity.
In~\cite{de2023marked}, it is shown that open dispersing billiards with analytic boundaries satisfying the no-eclipse condition and with \(\mathbb{Z}_2 \times \mathbb{Z}_2\)-symmetry can be recovered from their marked length spectral data.
The dynamics of such billiards is hyperbolic, aligning them more closely with Anosov geodesic flows.
These billiards are structurally stable and admit natural symbolic dynamics~\cite{morita_symbolic_1991}.
However, the non-wandering dynamics occurs only on a Cantor set, and analyticity is required to globally determine the domain.
\par In this paper, we investigate the case of Sinai billiards, obtained by removing from the \(2\)-torus \(\mathbb{T}^2\) a finite number of strongly convex scatterers.
While the dynamics of Sinai billiards shares many features with Anosov geodesic flows on surfaces (such as hyperbolicity, ergodicity, and exponential mixing, see e.g.~\cite{chernov_chaotic_2006, baladi2018exponential}), it also introduces several additional challenges:
\begin{enumerate}
	\item[$\bullet$] the billiard map has a dense set of singularities, and the billiard flow is not Anosov; 
	\item[$\bullet$] the system is not structurally stable: perturbing the geometry of the boundary can create or destroy periodic orbits.
\end{enumerate}
A symbolic coding for Sinai billiards does arise naturally: by considering a \(\mathbb{Z}^2\)-periodic lift of the billiard table to \(\mathbb{R}^2\) and labelling each scatterer, each orbit can be assigned a symbolic sequence consisting of the labels of the obstacles it successively encounters.
However, due to the lack of structural stability, the question of marked length spectral determination for Sinai billiards may seem less canonical than for Anosov metrics on surfaces.
The marking is not~\emph{a priori} stable under perturbation, and the symbolic data itself must encode substantial geometric information.
In particular, if we discard the length data and collect all symbolic data corresponding to sequences realised by (periodic) billiard orbits (the realisable sequences of Section~\ref{sec:prelimin}) into a \emph{symbolic spectrum}, we can even pose the following stronger question:
\begin{qst}
    Does the symbolic spectrum determine a Sinai billiard up to isometry?
\end{qst}
While we are unable to answer this question for Sinai billiards, we note that several works address similar questions for flat billiards (see e.g.~\cite{duchin2021you, calderon2018hear}).
In the present work, rather than focusing on the marked length spectrum corresponding to periodic orbits of a Sinai billiard, we take a different approach.
Specifically, we build on a construction suggested by Birkhoff and Arnold and made explicit in the work of Kourganoff~\cite{kourganoff_anosov_2016}; it demonstrates that the billiard flow on a Sinai billiard can, in a certain sense, be approximated by a one-parameter family of Anosov geodesic flows on a surface lifting the table.
This allows us to define an \emph{enriched length spectrum}, obtained as the ``limit'' of the marked length spectra of the approximating geodesic flows.
Although this construction introduces periodic data that do not correspond to actual billiard orbits, it makes the marking more natural, since it uses the symbolic data shared by the whole family of approximating geodesic flows.
\par Let us quickly recall the construction in~\cite{kourganoff_anosov_2016}.
Fix a finite-horizon Sinai billiard \(\mathcal{D}\subset \mathbb{T}^2\).
Let \( E= \mathbb{T}^2\times \mathbb{R} \), and denote by \( p\colon E \to \mathbb{T}^2 \) the canonical projection onto the first two coordinates.
We consider a smooth closed immersed surface \( K \) in \( E \) such that \( p(K) = \mathcal{D} \), and such that  \(K^\circ \coloneqq K \setminus p^{-1}(\partial\mathcal{D})\) has two connected components each of which is diffeomorphically mapped to \( \mathrm{int}(\mathcal{D}) \) under \( p \).
We interpret \( K \) as a surface obtained by gluing two different copies of \( \mathcal{D} \) along their  boundary components.
Let us denote by \( \mathcal{C}\coloneqq \mathrm{Conj}(\pi_1(K)) \) the set of free homotopy classes on \( K \).
Note that \(\mathcal{C}\) is independent of the choice of surface \(K\) lifting \(\mathcal{D}\) that satisfies our requirements.

\begin{figure}[t]
    \centering
    \includegraphics[width=.8\textwidth]{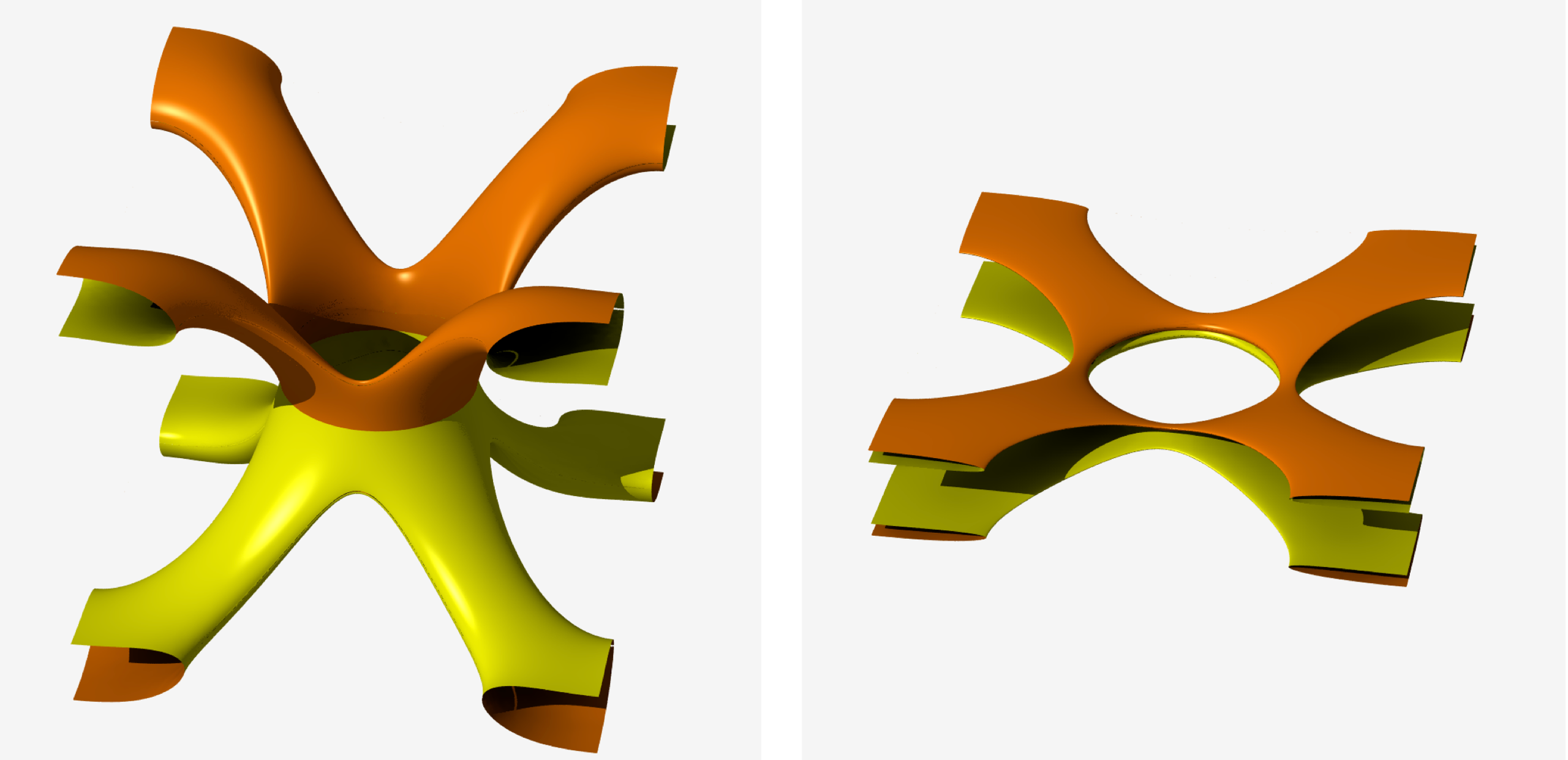}
    \caption{A surface \(K\) lifting the table \(\mathcal{D}\) and a deformation \(K_\epsilon\) (Credits: Micka\"el Kourganoff, Jos Leys).\label{liftsurface_png}}
\end{figure}

\par For \( \epsilon \in (0,1] \), we let \( \alpha_\epsilon \) be the contraction given by \( (x,y,z) \mapsto (x, y, \epsilon z) \) and write \( g_\epsilon \) to denote the metric induced on \( K_\epsilon \coloneqq \alpha_\epsilon(K) \) by the Euclidean metric.
Pulling back via \( \alpha_\epsilon \), we will consider \( \{g_\epsilon\} \) as a family of metrics on the surface \( K \).
We denote by \( \mathbf{g}_\epsilon \) the geodesic flow of the surface \( (K, g_\epsilon) \).
By Kourganoff~\cite{kourganoff_anosov_2016}, for \(\epsilon_0>0\) small enough, the geodesic flows \(\{\mathbf{g}_\epsilon\}_{0<\epsilon<\epsilon_0}\) associated to the family of metrics \(\{g_\epsilon\}_{0<\epsilon<\epsilon_0}\) are all Anosov.
Therefore, for \(0<\epsilon<\epsilon_0\), there exists a unique closed geodesic \( \gamma_\epsilon^\alpha\) of \( (K, g_\epsilon) \) in each free homotopy class \( \alpha \in \mathcal{C} \).
We will see in Lemma~\ref{lem:unique_minimising_geodesic} that the family \(\{p(\gamma_\epsilon^\alpha)\}_{0<\epsilon<\epsilon_0}\) converges in the Hausdorff topology to some closed curve \( \gamma^\alpha_0 \) in \(\mathcal{D}\) as \(\epsilon\to 0\), and their respective lengths \(\{l_\epsilon(\gamma_\epsilon^\alpha)\}_{0<\epsilon<\epsilon_0}\) also converge to a limiting length \(l_0(\gamma_0^\alpha)\).
We can then define the~\emph{enriched marked length spectrum} of \( \mathcal{D} \) as the function
\begin{align*}
    \mathbf{EL}_{\mathcal{D}}\colon \mathcal{C} & \to \mathbb{R}                \\
    \alpha                                      & \mapsto l_0(\gamma_0^\alpha).
\end{align*}
\begin{rmk}
    Every periodic billiard orbit of \(\mathcal{D}\) is the projection of \(\gamma_0^\alpha\) for some \(\alpha \in \mathcal{C}\): its lift to \(K\), which changes sheet at each collision, is a closed geodesic of \(K_0\) (after two periods if the number of collisions is odd), hence the unique closed geodesic of its class (Lemma~\ref{lem:pi1_hyperbolic}); orbits with only regular collisions give rise to exactly two classes, exchanged by the sheet involution (Remark~\ref{rem:sheet_involution}).
    Consequently, \(\mathbf{EL}_\mathcal{D}\) can be interpreted as a natural completion of the marked length spectrum of \(\mathcal{D}\), which is the restriction of \(\mathbf{EL}_\mathcal{D}\) to those \(\alpha \in \mathcal{C}\) associated with genuine billiard orbits.
\end{rmk}
\begin{rmk}
    As we shall see in Section~\ref{sec:geodesic}, there are several ways to define rigorously the ``limit'' of the family of Anosov geodesic flows \(\{\mathbf{g}_\epsilon\}_{0 < \epsilon < \epsilon_0}\) as \(\epsilon \to 0\).
    The limit \(d_0\) of the metrics \(d_\epsilon\) is the metric gluing of two copies of the table (Theorem~\ref{thm:K_0_as_gluing_metric_space}), and limits of \(g_\epsilon\)-geodesics are geodesics of \(d_0\) (Theorem~\ref{thm:d_0_is_a_metric}).
    Although the limit object is no longer an Anosov flow, it keeps several features inherited from the limiting CAT(0) spaces.
    Moreover, the enriched spectrum \(\mathbf{EL}_\mathcal{D}\) does not depend on the choice of the surface \(K\) lifting \(\mathcal{D}\), hence it is intrinsic to the billiard table \(\mathcal{D}\) (see e.g.\ Appendix~\ref{sec:billiards.cycles}).
\end{rmk}
Let \(\mathcal{D},\mathcal{D}'\subset \mathbb{T}^2\) be two Sinai billiards with finite horizon and \(C^r\) boundaries, \(r \geq 3\) (see Section~\ref{sec:prelimin}), that are diffeomorphic to each other.
We can choose surfaces \(K,K'\) lifting \(\mathcal{D},\mathcal{D}'\) which are diffeomorphic to each other; a diffeomorphism \(K \to K'\) identifies the sets of free homotopy classes of \(K\) and \(K'\) (in other words, it is a~\emph{marking}), and the equality of spectra in the statement below is understood with respect to such an identification, fixed once and for all (see Section~\ref{subsec:angle_dissipation}).
Our main result is:
\begin{theoalph}\label{thm:main}
    Let \(\mathcal{D},\mathcal{D}'\subset \mathbb{T}^2\) be two Sinai billiards with finite horizon and \(C^r\) boundaries, \(r \geq 3\), that are diffeomorphic to each other.
    If the billiards \(\mathcal{D},\mathcal{D}'\) share the same enriched marked length spectrum, i.e., \(\mathbf{EL}_{\mathcal{D}}=\mathbf{EL}_{\mathcal{D}'}\), then \(\mathcal{D}\) and \(\mathcal{D}'\) are actually isometric. In other words, there is an isometry \(\phi\) of \(\mathbb{T}^2\) with \(\phi(\mathcal{D}) = \mathcal{D}'\).
\end{theoalph}
\par\medskip
\noindent\textbf{Outline of the proof.} The strategy is to reduce the marked length spectral rigidity of the billiard to a boundary rigidity statement for a limiting \(\mathrm{CAT}(0)\) space, where an argument in the spirit of Otal~\cite{otal_spectre_1990} can be applied.
It proceeds in four main steps.

\emph{From billiards to a degenerating family of Anosov surfaces.} Following Kourganoff~\cite{kourganoff_anosov_2016}, each Sinai billiard \(\mathcal{D}\) is lifted to a surface \(K\) whose geodesic flow, after a vertical contraction by a factor \(\epsilon\), is Anosov and shadows in a certain sense the billiard flow (Section~\ref{sec:geodesic}, Theorems~\ref{thm:kourganoff_approximation} and~\ref{thm:kourganoff_on_steroids}).
As \(\epsilon \to 0\), the Riemannian metrics \(g_\epsilon\) degenerate along the gluing locus \(B(K)\), which lifts the boundary of the billiard \(\mathcal{D}\).
The main technical difficulty, which we deal with throughout Section~\ref{sec:geodesic}, is that the curvature of \(g_\epsilon\) concentrates near \(B(K)\), so the hyperbolicity estimates for the flows \(\mathbf{g}_\epsilon\) are not uniform in \(\epsilon\) there.
We control the degeneration through a boundary-layer normal form (Lemma~\ref{lem:normal_form}) and a uniform Gromov-hyperbolicity statement for the family of geodesic flows (Lemma~\ref{lem:uniform_hyperbolicity}).

\emph{The limiting \(\mathrm{CAT}(0)\) space.} The distances \(d_\epsilon\) decrease monotonically to a limit length metric \(d_0\) (Theorem~\ref{thm:d_0_is_a_metric}), which realises \(K_0\) as two copies of the table \((\mathcal{D}, d_\mathcal{L})\) glued along their boundaries (Theorem~\ref{thm:K_0_as_gluing_metric_space}).
The universal cover \(\widetilde{K}_0\) is a Hadamard space; the dispersing billiard dynamics forbids flat strips (Lemma~\ref{lem:no_flat_plane}), so \(\widetilde{K}_0\) is Gromov-hyperbolic (Corollary~\ref{cor:hyperbolic}) and its geodesics are determined by their endpoints at infinity (Lemma~\ref{lem:wK_is_visibility}).
By the comparison of lengths (Theorem~\ref{thm:length_comparison}), the \(g_\epsilon\)-geodesics are quasi-geodesics of \(\widetilde{K}_1\) with constants independent of \(\epsilon\).
This allows us to carry the coarse geometry over to the limit (Section~\ref{section:coarse_geometry}).

\emph{The Liouville current and the correspondence of geodesics.} On \(\widetilde{K}_0\) there is no Liouville form; instead, \(\widetilde{K}_0\) carries an intrinsic Liouville geodesic current \(\widehat{m}_0\), defined by the kinematic measure of its flat pieces (Definition~\ref{defi:kinematic_current}), whose support is exactly the set of admissible geodesics (Proposition~\ref{pf:prop_support}) and whose values on rectangles of endpoint pairs are the logarithms of the cross-ratios \(\operatorname{cr}_0\) at infinity (Proposition~\ref{prop:rectangle_formula}).
We show that the enriched marked length spectrum determines all cross-ratios (Proposition~\ref{pf:cor_spectrum_crossratio}), hence the current.
If \(\mathbf{EL}_{\mathcal{D}} = \mathbf{EL}_{\mathcal{D}'}\), the identification of the boundaries at infinity induced by the marking then gives a homeomorphism \(h\) between the spaces of (unparametrised) geodesics of \(\widetilde{K}_0\) and \(\widetilde{K}_0'\) which preserves the currents (Lemma~\ref{pf:lem_h_unparam}).
At this stage, \(h\) is a dynamical object, in the sense that it acts on geodesics and not yet on points.

\emph{From angle preservation to isometry.}
An averaging argument on the angle distortion of \(h\), in the spirit of Otal~\cite{otal_spectre_1990} and Wilkinson~\cite{wilkinson_anosov}, shows that \(h\) is angle-preserving (Theorem~\ref{lem:angle_preservation}).
In this argument, the product structure of the Liouville measure on pairs of crossing geodesics replaces the use of ergodicity.
Since the curvature of \(\widetilde{K}_0\) is concentrated on the gluing locus \(B(K)\), we use angle preservation along the walls.
It implies that the images of the geodesics through a wall point are concurrent, which defines an isometric map between the gluing loci.
The current then controls the crossing positions along every regular geodesic, from where it follows that \(h\) is induced by an isometry \(\widetilde{K}_0 \to \widetilde{K}_0'\) (Theorem~\ref{thm:enriched_rigidity}).
This isometry descends to the tables and extends to an isometry of \(\mathbb{T}^2\).\\

Natural questions then arise.
\begin{qst}
    Does partial information from the enriched marked length spectrum suffice to recover this conclusion? In particular, to what extent does the set of periodic billiard orbits encode the geometry of the billiard \(\mathcal{D}\)? Given our main result Theorem~\ref{thm:main}, this amounts to asking whether the billiard marked length spectrum encodes the enriched spectrum \(\mathbf{EL}_\mathcal{D}\).
\end{qst}

\begin{qst}\label{rem:analytic_spectra}
    It is natural to ask whether the enriched spectrum \(\mathbf{EL}_\mathcal{D}\) can be detected by an analytic spectrum.
    The relevant operator is naturally attached to the singular surface \(K_0\) obtained by gluing two copies of \(\overline{\mathcal{D}}\) along the walls:
    the Laplace spectrum of the singular surface \(K_0\) is the union of the Dirichlet and Neumann spectra of the table \(\mathcal{D}\).
    For the table itself, the Poisson relation of Andersson--Melrose~\cite{andersson_melrose_1977} (see also~\cite{petkov_stoyanov_2017}) locates the singularities of the wave trace inside the length spectrum of generalised geodesics in the sense of Melrose--Sj\"ostrand.
    Since dispersing walls are everywhere diffractive, the generalised geodesics of a Sinai table contain no gliding arcs, and generically (no grazing periodic rays, simple non-degenerate lengths~\cite{petkov_stoyanov_2017}, see also Proposition~\ref{prop_good_tables}) the wave trace of \(\mathcal{D}\) detects exactly the  (unmarked) billiard length spectrum.
    On the other hand, classes whose limit representatives \(\gamma_0^\alpha\) contain gliding arcs exist for every finite-horizon table (see Corollary~\ref{cor:dense_periodic_liftings} and Lemma~\ref{lem:topology.geodesics}).
    These classes are not seen by the wave trace of the billiard.
    In other words, the enrichment is the part of \(\mathbf{EL}_\mathcal{D}\) that the classical Poisson relation can not see.
    We expect the situation to be different on \(K_0\), since its seam is a union of \(m\) closed geodesics carrying concentrated negative curvature, and the gliding representatives are closed geodesics of \(K_0\) in their own right.
    By analogy with the diffractive trace formulas at conical singularities \cite{hillairet_contribution_2005}, it seems possible that the wave trace of \(K_0\) is singular at the gliding lengths as well, but we do not pursue it here.
    Combined with Theorem~\ref{thm:main} (and with the local constancy of the marking along deformations), such a trace formula would yield spectral deformation-rigidity of generic Sinai billiards for the pair of Dirichlet and Neumann boundary conditions.
\end{qst}

%% file: sections/prelimin.tex
\par We begin by introducing some established facts and notation on Sinai billiards.
Most of the standard facts recalled in this section can be found in~\cite{chernov_chaotic_2006}. 
Let \( \mathcal{D}\subset \mathbb{T}^2 \) be a Sinai billiard table with finite horizon whose boundary \( \partial \mathcal{D} \) is of class \( C^r \),  \( r \in \mathbb{N}_{\geq 3}\).
More precisely, we have
\[
    \partial\mathcal{D} = \Gamma_1 \cup \Gamma_2 \cup \dots \cup \Gamma_m,\quad m \geq 2,
\]
where each \( \Gamma_i \) is a \(C^r\) curve bounding a strongly convex open set \( \mathcal{O}_i \subset \mathbb{T}^2 \), called an obstacle, and \( \mathcal{D} \coloneqq \mathbb{T}^2\setminus \cup_{i=1}^m \mathcal{O}_i \).

\par For each \( i \in \{1,\dots,m\} \), we denote by \( \ell_i>0 \) the perimeter of \( \mathcal{O}_i \), we let \( \mathbb{T}_i\coloneqq\mathbb{R}/\ell_i \mathbb{Z} \), and let \( \gamma_i\colon \mathbb{T}_i\to \mathbb{T}^2 \) be an anticlockwise parametrisation of \( \Gamma_i \) by arc-length.
We set \( \mathbb{T}_\mathcal{D}\coloneqq\cup_{i=1}^m \mathbb{T}_i \), and let \( \gamma\colon \mathbb{T}_\mathcal{D}\to  \mathbb{T}^2 \), \( \mathbb{T}_i \ni s\mapsto \gamma_i(s) \).
For each \( s \in \mathbb{T}_\mathcal{D} \), we denote by \( n_s \) the unit outward normal vector at the point \( \gamma(s)\in \partial \mathcal{D} \), and by \( \mathcal{K}(s) \) the curvature at this point; following the convention in \cite{chernov_chaotic_2006}, we assume that \( \mathcal{K}(s)> 0 \).
\par We let \( \mathcal{M}\coloneqq\mathbb{T}_\mathcal{D} \times [-\frac{\pi}{2},\frac{\pi}{2}] \), and denote by \( f\colon\mathcal{M} \to\mathcal{M} \), \( (s,\varphi)\mapsto (s',\varphi') \) the billiard map within \( \mathcal{D} \), where \( s \) corresponds to the position on the boundary, and \( \varphi \) is the oriented angle from the inward normal vector \( -n_s \) to the velocity vector at \( \gamma(s) \).
\par Moreover, \( \mathcal{S}_0 \coloneqq \partial \mathcal{M} = \{(s,\varphi); |\varphi| = \pi/2\} \) is the set of~\emph{grazing} collisions of \( f \); a collision which is not grazing is called \emph{regular}.
Set
\[
    \mathcal{S}_{n+1} \coloneqq \mathcal{S}_n \cup f^{-1}(\mathcal{S}_n) \text{ and }  \mathcal{S}_{-(n+1)} \coloneqq \mathcal{S}_{-n} \cup f(\mathcal{S}_{-n}).
\]
Then, for every \( n \in \mathbb{N} \), \( \mathcal{S}_{n+1}\), resp. \( \mathcal{S}_{-(n+1)}\) is the singularity set of the map \( f^{n+1} \), resp. \( f^{-(n+1)} \), while these maps are of class \( C^{r-1} \) at \( \mathcal{M}\setminus(\mathcal{S}_{n+1} \cup \mathcal{S}_{-(n+1)}) \).
The points in \( \mathcal{S}_\mathcal{M} \coloneqq \cup_{i \in \mathbb{Z}} \mathcal{S}_i \) are called the \emph{singular points} of the billiard map, and \( \mathcal{M}^\mathrm{reg} \coloneqq \mathcal{M}\setminus\mathcal{S}_\mathcal{M} \) is the \emph{regular set} of \( f \).
The singular set has empty interior, but is dense in \(\mathcal{M}\), while the regular set has full Lebesgue measure, but is badly disconnected.
A periodic orbit of \( f \) is called \emph{singular} if it has a grazing collision, i.e., if it meets \( \mathcal{S}_0 \) (equivalently, \( \mathcal{S}_\mathcal{M} \)), and \emph{regular} otherwise.
\par Finally, for each \( (s, \varphi) \in \mathcal{M} \), its free-flight time is the function
\[
    \tau(s, \varphi) = \|\gamma(\mathrm{Pr}_1(f(s, \varphi))) - \gamma(s)\|,
\]
where we denote by \( \| \cdot \| \) the Euclidean distance in \( \mathbb{T}^2 \) and by \(\mathrm{Pr}_1\colon \mathcal{M} \to \partial\mathcal{D}\) the projection onto the first coordinate \(s\simeq \gamma(s)\).
The \emph{finite horizon} hypothesis means that no trajectory of \( \mathcal{D} \) avoids the obstacles forever; equivalently, every line of \( \mathbb{R}^2 \) meets a lifted obstacle. By compactness, there is then a uniform upper bound \( \tau_\mathrm{max} \) on the time it takes for any trajectory, regardless of its starting point, to experience a collision. Since the obstacles are pairwise disjoint and closed, there is also a uniform lower bound \( \tau_\mathrm{min} > 0 \) on the free flight between two consecutive collisions, so that \( \tau_\mathrm{min} \leq \tau(x) \leq \tau_\mathrm{max} \) for every \( x \in \mathcal{M} \).
The suspension flow of \( f \) with roof function \( \tau \) is called the billiard flow \( \Psi \) of \( \mathcal{D} \).
Its phase space is denoted by \( \Omega \), and the usual coordinates are \((x, y, \omega)\), where \((x,y)\) are coordinates on \(\mathcal{D}\) and \(\omega\) is the angle between the velocity of the trajectory and the \(x\)-axis.
The saturations of \(\mathcal{S}_\mathcal{M}\) and \(\mathcal{M}^\mathrm{reg}\) under the flow form the singular and regular set of \(\Omega\), denoted by \(\mathcal{S}_\Omega\) and \(\Omega^\mathrm{reg}\).
It is known that the billiard flow preserves the Lebesgue measure \(dx\,dy\,d\omega\) and the contact form \(\lambda = \cos\omega\,dx + \sin\omega\,dy\)~\cite{chernov_chaotic_2006}.
Indeed, the flow is generated by the Reeb field of such form.
\par Neither the coordinates nor the hyperbolic structure recalled below are used in the sequel; we record them for the reader's convenience. The Jacobi coordinates \((\eta, \xi, \omega)\), introduced in~\cite{wojtkowski_two_1994}, are defined as
\begin{align*}
    \eta & = x\cos\omega + y\sin\omega, \\
    \xi  & = x\sin\omega - y\cos\omega.
\end{align*}
In particular, while in the usual coordinates the action of the flow between collisions is \(\Psi^t(x, y, \omega) = (x + t\cos\omega, y+t\sin\omega, \omega)\), in Jacobi coordinates this action is simply translation on the first coordinate:
\[
    \Psi^t(\eta, \xi, \omega) = (\eta + t, \xi, \omega).
\]
Inverting the change of coordinates gives \(x = \eta\cos\omega  + \xi\sin\omega \) and \( y = \eta\sin\omega  - \xi\cos\omega \).
Hence the forms
\[
    d\eta + \xi d\omega = \cos\omega\, dx + \sin\omega\, dy = \lambda
\]
and
\[
    \lambda \wedge d\lambda =  d\eta\wedge d\xi\wedge d\omega = - dx\wedge dy\wedge d\omega
\]
are both preserved by the flow.
In Jacobi coordinates, the Reeb field of \(\lambda\) is exactly the coordinate field \(\partial_\eta\), which corresponds to the displacement of a particle in the direction of the flow.
\par The map \(f\) and the flow \(\Psi\) are (uniformly) hyperbolic: at every regular point there is an invariant splitting into stable and unstable directions (together with the flow direction \(\partial_\eta\) in the case of the flow), with uniform contraction and expansion rates; the associated Lyapunov exponents are defined almost everywhere by Oseledets' theorem.
This gives a decomposition
\[
    T_x\Omega = T^c_x\Omega \oplus T^\perp_x\Omega = \mathbb{R}\partial_\eta \oplus E^s(x) \oplus E^u(x),
\]
of the tangent space of every regular point, where \(E^s(x), E^u(x)\) are the stable and unstable bundles, respectively.
As usual, \(E^s(x)\), resp. \(E^u(x)\), contracts, resp. expands exponentially under the flow as \(t \to \infty\), and is locally integrable to \(C^{r-1}\) curves, called the strong stable (resp. unstable) local manifolds \(\mathcal{W}_\mathrm{loc}^s(x)\) (resp. \(\mathcal{W}_\mathrm{loc}^u(x)\)).
The end points of the manifolds \(\mathcal{W}_{\mathrm{loc}}^*(x)\) (not included in the manifolds) are always singular.
These manifolds are not uniform in length, and get smaller as one approaches the singular set.
In particular, every open set of \(\Omega\) contains points whose unstable manifolds are arbitrarily small, and therefore points in \(\Omega\) do not have product neighbourhoods.
The saturations of \(\mathcal{W}_\mathrm{loc}^s(x)\) and \(\mathcal{W}_\mathrm{loc}^u(x)\) with respect to the flow are the weak, or central, stable and unstable manifolds, respectively, denoted by \(\mathcal{W}_\mathrm{loc}^{cs}(x)\) and \(\mathcal{W}_\mathrm{loc}^{cu}(x)\).

\subsection{Liftings and the symbolic space}
We want to associate with each orbit of \(f\) a sequence of symbols that represents the scatterers it encounters along the way.
However, this approach is problematic on the torus because an orbit could hit the same scatterer several times in succession.
To avoid this issue, let \( \widetilde{\mathcal{D}}\subset \mathbb{R}^2 \) be the \( \mathbb{Z}^2 \)-periodic billiard table such that \( \pi(\widetilde{\mathcal{D}})=\mathcal{D} \), where \( \pi\colon [0,1)^2\to \mathbb{T}^2 \) denotes the canonical projection.
Abusing notation, we denote the billiard map and flow on \( \widetilde{\mathcal{D}} \) by \( f \) and \( \Psi \) as well.
\par Let \( \{B_1, \dots, B_k\} \) be the set of connected components of the lift \( \pi^{-1}(\partial\mathcal{D}) \subset [0,1)^2. \) We index all the scatterers in the lifted table on \( \mathbb{R}^2 \) using the alphabet \( \mathcal{A} \coloneqq \mathbb{Z}^2\times \{1, \dots, k\} \), as follows:
\[
    B_{(i,j;l)} \coloneqq B_l + (i,j),
\]
i.e., \( B_{(i,j;l)} \) is the unique scatterer obtained by translating \( B_l \) by \( (i,j) \in \mathbb{Z}^2 \).
Let \( \Sigma \) denote the space \( \mathcal{A}^\mathbb{Z} \) of two-sided sequences of elements in \( \mathcal{A} \).
This is a complete metric space with the usual distance
\[
    d_\Sigma(\alpha, \beta) = \sum_{i \in \mathbb{Z}} 2^{-|i|}(1 - \delta_{a_ib_i})
\]
between two sequences \( \alpha = \{a_i\}_{i \in \mathbb{Z}} \) and \( \beta = \{b_i\}_{i \in \mathbb{Z}} \).
In what follows, we will work with the equivalent metric \( \rho_\Sigma \) given by
\[
    \rho_\Sigma(\alpha, \beta) \coloneqq 2^{-\max\{n \in \mathbb{N}; a_i = b_i \text{ for all } |i| \leq n\}}.
\]
We associate to the points in \( \mathcal{M} \) the sequence of indices of scatterers where its lifted trajectory bounces.
\begin{dfn}
    The map \( \imath \colon \mathcal{M} \to \Sigma \) is given by
    \[
        \imath(s,\varphi) = \{a_i\}_{i \in \mathbb{Z}},
    \]
    where \(B_{a_0}\subset \pi^{-1}(\partial \mathcal{D})\), and for all \(i \in \mathbb{Z}\), \( a_i \) is the unique element of \( \mathcal{A} \) such that \( \mathrm{Pr}_1(f^i(\pi^{-1}(s), \varphi)) \in B_{a_i} \).
    We call the set \( \Sigma_\mathcal{M} \coloneqq \imath(\mathcal{M}) \) the \emph{space of realisable sequences}. 
\end{dfn}
The space \(\Sigma\) admits a natural \(\mathbb{Z}^2\)-action, defined by \(\{(i_p,j_p;l_p)\}_{p \in \mathbb{Z}}\mapsto \{(i_p+i,j_p+j;l_p)\}_{p \in \mathbb{Z}}\), for all \((i,j)\in \mathbb{Z}^2\).
We denote by \(\hat \Sigma\) the quotient of \(\Sigma\) under this action, and by \(\hat \imath\colon \mathcal{M} \to \hat \Sigma\)  the associated map, identifying \(\hat {\Sigma}_\mathcal{M}\coloneqq\hat \imath(\mathcal{M})\) with \( \Sigma_\mathcal{M} \).
\begin{figure}[t]
    \centering
    \includegraphics[width=.8\textwidth]{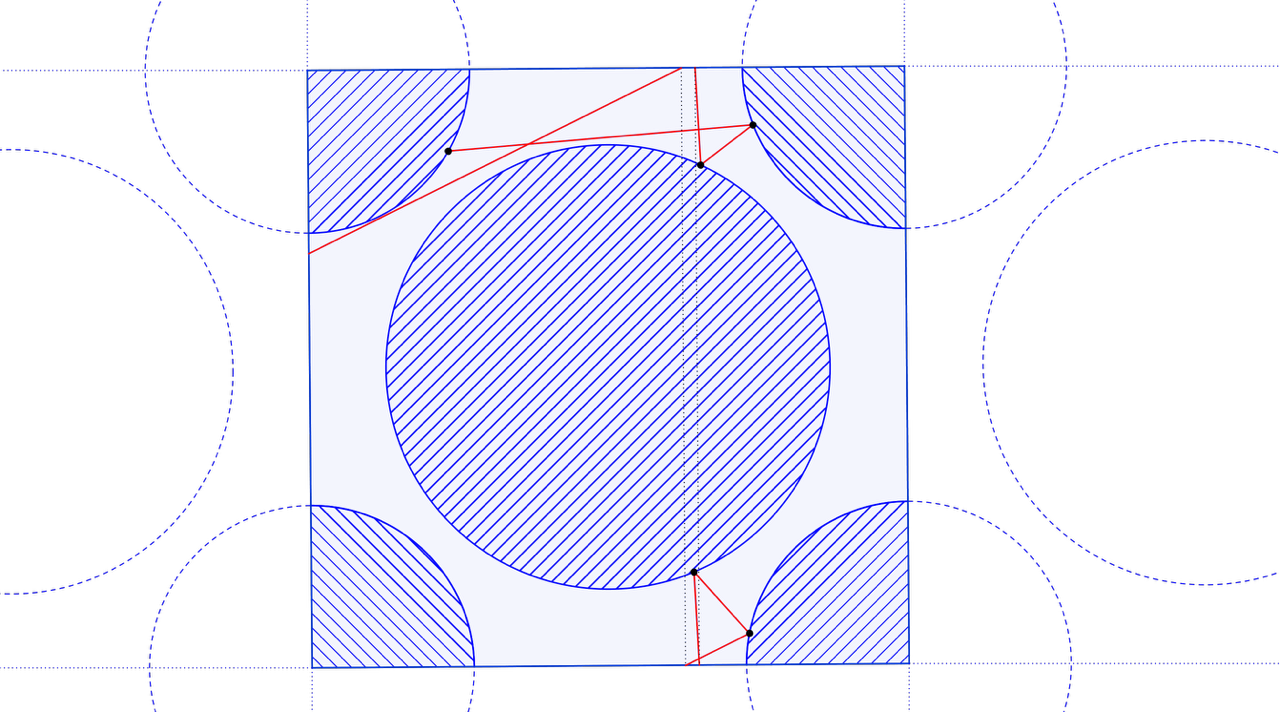}
    \caption{Lifted table.}
    \label{fig:lifted_table}
\end{figure}
\begin{rmk}
    The left shift \( \sigma\colon \Sigma\to \Sigma\), \(\{a_i\}_{i \in \mathbb{Z}} \mapsto \{a_{i+1}\}_{i \in \mathbb{Z}} \)  induces a map \(\hat{\sigma}\colon \hat \Sigma\to \hat \Sigma\).
    Moreover, by definition, the space \( \hat{\Sigma}_\mathcal{M} \) is invariant under \(\hat{\sigma}\), and \( \hat\imath \circ f = \hat{\sigma} \circ \hat\imath \).
\end{rmk}
\begin{pps}\label{pps:representation_map_continuity}
    The mapping \( \imath\colon \mathcal{M} \to \Sigma_\mathcal{M} \) is a bijection satisfying the following properties:
    \begin{enumerate}
        \item[(i)] the restriction of \( \imath \) to the regular set
              \(\mathcal{M}^\mathrm{reg}\) is continuous;
        \item[(ii)] the inverse mapping \( (\imath\rvert_{\mathcal{M}^\mathrm{reg}})^{-1} \) is Hölder continuous.
              In particular,
              \[
                  \imath\rvert_{\mathcal{M}^\mathrm{reg}}\colon \mathcal{M}^\mathrm{reg} \to \Sigma_{\mathcal{M}^\mathrm{reg}}\coloneqq \imath(\mathcal{M}^\mathrm{reg})
              \]
              is a homeomorphism;
        \item[(iii)] the discontinuity set of \( \imath \) is precisely the singular set \( \mathcal{S}_\mathcal{M} \).
    \end{enumerate}
\end{pps}

\begin{proof}
    For \( \alpha \in \Sigma \) and \( n \in \mathbb{N} \), consider the preimage \( C(\alpha; n) \coloneqq \imath^{-1}(B(\alpha, 2^{-n})) \), where \(B(\alpha, 2^{-n})\) is the open ball centred at \(\alpha\) with radius \(2^{-n}\) for the metric \( \rho_\Sigma \).
    It consists of points \( x \in \mathcal{M} \) whose first \( n+1 \) iterates under \( f \), both in the past and future, are encoded by the \(n\)-cylinder obtained by truncating \( \alpha \), i.e.,
    \[
        C(\alpha; n) = \{x \in \mathcal{M} \,\rvert\, f^{i}(x) \in B_{a_i} \text{ for all } |i| \leq n+1\}.
    \]
    The set \(H^s_n\) of points of \(\mathcal{M}\) sharing the forward itinerary \((a_0, \dots, a_{n+1})\) is a closed set each of whose components is a ``stable strip'', bounded by singularity curves of \(f^{n+1}\), on which \(f^{n+1}\) is continuous.
    Unstable curves of \(\mathcal{M}\) have length bounded by a constant \(D\) depending only on the table (they are graphs of functions with bounded slope over the compact set \(\mathcal{M}\)), while unstable curves inside \(H^s_n\) are expanded by \(f^{n+1}\) at a uniform exponential rate: there are constants \(K, \mathcal{B} > 0\), depending only on the table (on the curvature function \(\mathcal{K}\) and the free-flight bounds \(\tau_\mathrm{min}, \tau_\mathrm{max}\)), such that unstable vectors are expanded by at least \(K^{-1}e^{\mathcal{B}n}\) (see the growth estimates for dispersing billiards in~\cite[Chapters 4--5]{chernov_chaotic_2006}).
    Here an obstacle \( B \) that is not in the sequence \( B_{a_1}, \dots, B_{a_n} \) may be eclipsed by the dispersing wave fronts generated by the unstable curves of \(H^s_n\), and can then safely be ignored, see Figure~\ref{fig:eclipsed_obstacle}.
    Hence every unstable curve inside \(H^s_n\) has length at most \(KDe^{-\mathcal{B}n}\), i.e., the unstable width of \(H^s_n\) is at most \(KDe^{-\mathcal{B}n}\).
    Symmetrically, the set \(H^u_n\) of points sharing the backward itinerary \((a_{-(n+1)}, \dots, a_0)\) is a strip of stable width at most \(KDe^{-\mathcal{B}n}\).
    Since the stable and unstable cone fields are uniformly transverse, the cylinder \(C(\alpha; n) \subset H^s_n \cap H^u_n\) satisfies
    \begin{equation}\label{eq:lips_estimate}
        \operatorname{diam} C(\alpha;n)
        \leq c(\mathcal{D})\,e^{-\mathcal{B} n}
    \end{equation}
    for a constant \(c(\mathcal{D}) > 0\) depending only on the table.
    In particular, if \( \imath(x) = \imath(y) \), then \(x, y \in C(\alpha;n)\) for every \(n\), hence \(x = y\), and \(\imath\) is indeed a bijection onto its image.
    \begin{figure}[t]
        \centering
        \includegraphics[width=.8\textwidth]{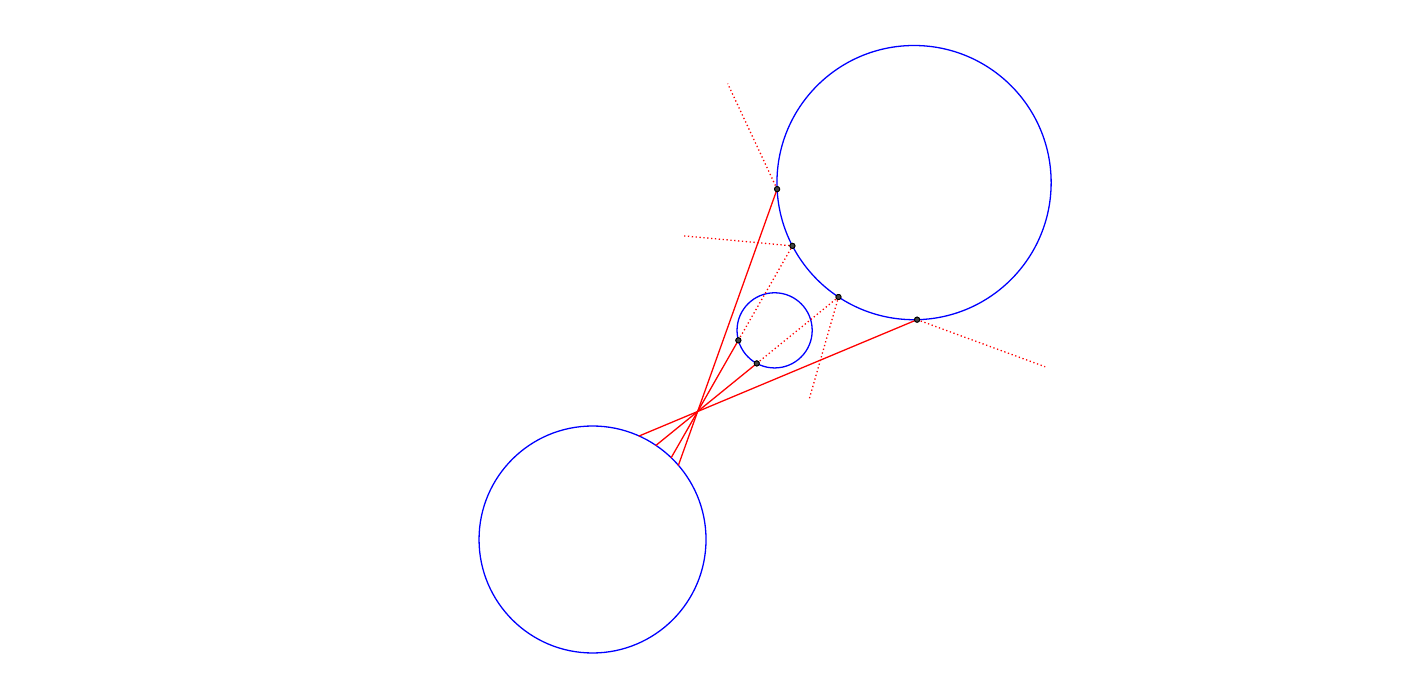}
        \caption{A smaller obstacle eclipsing a section of the interior of a wave front.}
        \label{fig:eclipsed_obstacle}
    \end{figure}
    \par Now, let us denote by \( \jmath \coloneqq\imath\rvert_{\mathcal{M}^\mathrm{reg}} \) the restriction of \( \imath \) to the regular set.
    We claim that \( \jmath^{-1}(B(\alpha, 2^{-n})) = C(\alpha; n)\cap\mathcal{M}^\mathrm{reg} \) is open.
    Indeed, since each \( f^i \) is continuous on \( \mathcal{M}^\mathrm{reg} \), there is a neighbourhood of any \( x \) in such preimage where every point bounces in exactly the same obstacles as \( x \) for \( |i| < n+1 \).
    Hence every point in \( C(\alpha; n)\cap\mathcal{M}^\mathrm{reg} \) is interior and the map \( \jmath \) is continuous, proving (i).
    \par As for (ii), note that if \( \rho_\Sigma(\alpha, \beta) = 2^{-n} \), then \( \imath^{-1}(\alpha), \imath^{-1}(\beta) \in C(\alpha; n-1) \), and it follows from the inequality in~\eqref{eq:lips_estimate} (applied with \(n-1\) in place of \(n\)) that,
    \[
        \|\imath^{-1}(\alpha)-\imath^{-1}(\beta)\|  \leq \operatorname{diam} C(\alpha;n-1)\leq c(\mathcal{D})e^{-\mathcal{B} (n-1)} \leq c'(\mathcal{D})\rho_\Sigma(\alpha, \beta)^{\mathcal{B} (\log 2)^{-1}},
    \]
    hence \( \imath^{-1} \) is \( \mathcal{B} (\log 2)^{-1}\)-H\"older continuous.
    \par Finally, we prove (iii).
    Property (i) amounts to saying that the discontinuity set of \( \imath \) is contained in the singular set \( \mathcal{S}_\mathcal{M} \), so all that remains to show is the converse inclusion.
    Let \( x_0 \in \mathcal{S}_\mathcal{M} \); replacing \( x_0 \) by an iterate, and \( \imath \) by its composition with a power of the shift, which does not affect continuity, we assume that \( f(x_0) \in \mathcal{S}_0 \).
    Let \( \mathrm{Pr}_1(x_0) \in \Gamma_i \) and \( \mathrm{Pr}_1(f(x_0)) \in \Gamma_j \).
    Then, due to the discontinuity of \( f \) at \( x_0 \), we can find points \( x \) arbitrarily close to \( x_0 \) and such that \( \mathrm{Pr}_1(f(x)) \notin \Gamma_j \), as well as points arbitrarily close and such that \( \mathrm{Pr}_1(f(x)) \in \Gamma_j \).
    Thus, there are sequences \( x_n \to x_0 \) in \( \mathcal{M} \) such that \( \rho_\Sigma(\imath(x_n), \imath(x_{n+1})) = 1 \) for every \( n\in \mathbb{N} \), and \( x_0 \) is therefore a discontinuity point for \( \imath \).
\end{proof}

\subsection{Periodic orbits.}\label{ssec:periodic_orbits}
\par Although it is not needed for the main part of the argument, we collect in this section general properties of periodic orbits and their lengths.
By Proposition~\ref{pps:representation_map_continuity}, we have:
\begin{cor}\label{cor:fin_card_peri}
    For any \( T>0 \), the number of periodic orbits of period at most \( T \) is finite.
\end{cor}
\begin{proof}
    Fix \( T>0 \).
    Fix a periodic point \( (s,\varphi) \) whose orbit under the billiard flow has length at most \( T \).
    Let \( q>0 \) be its period under the billiard map, \( f^q(s,\varphi)=(s,\varphi) \).
    In particular, we have \( q \leq q_0 \), with \( q_0 \coloneqq \bigl[ \frac{T}{\tau_\mathrm{min}} \bigr] + 1 \).
    Let \( \imath(s,\varphi) \eqqcolon (\rho_k)_k \in \Sigma \).
    Since \( f^q(s,\varphi)=(s,\varphi) \), we have \( \rho_q = \rho_0 + (\tilde i,\tilde j;0) \), for some \( (\tilde i,\tilde j) \in \mathbb{Z}^2 \), and
    \[
        \rho_p = \rho_{p \text{ mod } q} + \bigg\lfloor \frac{p}{q} \bigg\rfloor (\tilde i,\tilde j;0), \quad \forall\, p \in \mathbb{Z}.
    \]
    Moreover, by the finite horizon assumption, and the \( \mathbb{Z}^2 \)-periodicity of \(\widetilde{\mathcal{D}}\), there exists a uniform upper bound \( K \geq 1 \) on the difference \( |i_{p+1}-i_p| + |j_{p+1}-j_p| \leq K \) between the labels \( (i_p,j_p) \) and \( (i_{p+1},j_{p+1}) \) of the cells at steps \( p \) and \( p+1 \), for each  \( p \in \mathbb{Z} \).
    In particular, we have \( |\tilde i| + |\tilde j| \leq K q \).
    We conclude that periodic points of period at most \( T \) are completely encoded by finite words \( (\rho_0,\dots,\rho_q) \), \( q \leq q_0 \).
    Moreover, given  \( (\rho_0,\dots,\rho_\ell) \), \(1\leq \ell\leq q-1\), the label \( (i_{\ell+1},j_{\ell+1}) \) satisfies \( |i_{\ell+1}-i_\ell|+|j_{\ell+1}-j_\ell| \leq K \), so there are at most \( (2K^2+2K+1)k \) choices for the symbol \( \rho_{\ell+1} \).
    Therefore, the number of (infinite) words \( \imath(s,\varphi) \) associated to periodic points with period at most \( T \) is finite.
    By the injectivity of \( \imath \) shown in Proposition~\ref{pps:representation_map_continuity}, we deduce that the set of periodic orbits with period less than \(T\) for the billiard flow is finite.
\end{proof}

Fix an integer \( r \geq 3 \).
Given an integer \( m \geq 2 \), we denote by \( \mathfrak{D}^m \) the space of all Sinai billiard tables in \( \mathbb{T}^2 \) with \( m \) scatterers and \( C^r \) boundary, that have finite horizon.
Each table \( \mathcal{D} \) is defined by its boundary map \( \gamma\colon \mathbb{T}_{\mathcal{D}} \to \mathbb{T}^2 \), allowing us to identify \( \mathfrak{D}^m \) with an open subset of the space of embeddings \( C^r_{\mathrm{emb}}(\mathbb{T}_{\mathcal{D}}, \mathbb{T}^2) \), endowed with the \( C^r \) topology (strict convexity of the components and finite horizon are open conditions).
In particular, it is a Baire space.
\par We want to show that for a generic choice of table in \( \mathfrak{D}^m \), the set of periodic orbits has simple length spectrum and contains no singular orbits.

\begin{dfn}
    For each integer \( n \geq 1 \), we define \( \mathcal{R}_n \) as the subset of \( \mathfrak{D}^m \) consisting of tables without any singular periodic orbit of period less than or equal to \( n \).
\end{dfn}

\begin{dfn}
    For each integer \( n \geq 1 \), we define \( \mathfrak{S}_n \) as the subset of \( \mathfrak{D}^m \) consisting of tables such that no distinct periodic orbits of period less than or equal to \( n \) have the same length.
\end{dfn}

\begin{pps}\label{prop_good_tables}
    For each integer \( n \geq 1 \), the set \( \mathcal{R}_n \cap \mathfrak{S}_n \) is open and dense in \( \mathfrak{D}^m \).
    Consequently, the set
    \[
        \mathfrak{G} \coloneqq \bigcap_n \bigl(\mathcal{R}_n \cap \mathfrak{S}_n\bigr)
    \]
    of tables with no grazing periodic orbit and simple length spectrum is residual in the \( C^r \) topology.
\end{pps}

This result was already obtained by Petkov--Stoyanov, see e.g.~\cite{petkov_stoyanov_2017}; for completeness, we give a proof of Proposition~\ref{prop_good_tables} in Appendix~\ref{sec:periodic.orbits}, where  further properties of periodic orbits are investigated.

%% file: sections/geodesic.tex
\subsection{Kourganoff surfaces}
Recall that \( \Psi \) is the billiard flow of the table \( \mathcal{D} \), and let \( \Omega \) denote its phase space.
It is the suspension space of the billiard map \( f \) with roof function \( \tau \).
We consider \( E= \mathbb{T}^2\times \mathbb{R} \) with coordinates \( (x,y,z) \) and a canonical global frame \( (\partial_x, \partial_y, \partial_z) \).
We denote by \( p\colon E \to \mathbb{T}^2 \) the canonical projection on the first two coordinates.
\par Consider a smooth closed immersed surface \( K \) in \( E \) such that \( p(K) = \mathcal{D} \).
In what follows, we abuse the notation slightly and write
\begin{align*}
    K^\circ & \coloneqq K \cap p^{-1}(\mathrm{int}(\mathcal{D})), \\
    B(K)    & \coloneqq K \cap p^{-1}(\partial\mathcal{D}).
\end{align*}
Moreover, we assume \( K \) satisfies
\begin{enumerate}
    \item[(K0)] \(K^\circ\) has exactly two connected components, each mapped diffeomorphically onto \(\mathrm{int}(\mathcal{D})\) by \(p\). 
    \item[(K1)]\label{cond_k_un} The vector \( \partial_z \) is not tangent to \( K \) at any point \( x \in K^\circ \), i.e. \( \partial_z \notin T_xK \).
    \item[(K2)] For every point \( x \in B(K) \), the curve defined by the intersection of \( K \) and the affine plane \( \{x\} + (\mathbb{R}\partial_z \oplus T^\perp_xK) \) has non-vanishing curvature around \( x \).
\end{enumerate}
\par By~(K1), the restriction \(p|_{K^\circ}\) is a local diffeomorphism.
By~(K0), \(K^\circ\) consists of two sheets, each a copy of \(\mathrm{int}(\mathcal{D})\), so that every point of \(\mathrm{int}(\mathcal{D})\) has exactly two preimages. Near a point of \(B(K)\), condition~(K2) says that \(p|_K\) is a fold, along which two local sheets meet; since interior points have only two preimages, there is exactly one fold curve above each scatterer \(\Gamma_i \subset \partial\mathcal{D}\): each connected component of \(B(K)\) is projected diffeomorphically onto a single \(\Gamma_i\), and every point of \(\partial\mathcal{D}\) has a unique preimage under \(p\).
In particular, \(p|_K\) fails to be a local diffeomorphism precisely along \(B(K)\).
Since the kernel of \(dp\) is the vertical direction \(\mathbb R\partial_z\), we obtain
\[
    \ker\!\left(dp_x|_{T_xK}\right)
    =T_xK\cap\mathbb R\partial_z
    =\mathbb R\partial_z,
    \qquad x\in B(K),
\]
so the vertical direction is tangent to \(K\) exactly at the boundary.

We thus think of \(K\) as being obtained by gluing two copies of \(\mathcal D\) along their corresponding boundary components.
We will refine this statement in Lemma \ref{lem:sheet_isometry} below.
We refer the reader to Figure~\ref{liftsurface_png} for a picture of the construction.

\par For \( \epsilon \in (0,1] \), we let \( \alpha_\epsilon \) be the contraction given by \( (x,y,z) \mapsto (x, y, \epsilon z) \) and denote by \( g_\epsilon \) the metric induced on \( K_\epsilon \coloneqq \alpha_\epsilon(K) \) by the Euclidean metric.
Pulling back via \( \alpha_\epsilon \), we obtain a family of metrics  \( \{g_\epsilon\} \) on the surface \( K \).
We denote by \( \mathbf{g}_\epsilon \) the geodesic flow of the surface \( (K, g_\epsilon) \).
Then \( p \) induces a smooth mapping
\begin{equation}\label{eq:projection_flow_equivalence}
    \begin{aligned}
        \pi_K\colon T^1K^\circ & \to\Omega                                             \\
        (x,v)                  & \mapsto \bigg(p(x), \frac{dp_x(v)}{\|dp_x(v)\|}\bigg)
    \end{aligned}
\end{equation}
(note that~(K1) ensures we are not dividing by \( 0 \) in the definition).
Finally, we denote by \( A_0 \) the subset of \( \mathbb{R} \times T^1K^\circ \) consisting of points \( (t; x,v) \) such that the orbit segment \( \{\Psi^s(\pi_K(x,v)); s \in [0,t]\} \) has endpoints in \( \mathrm{int}(\mathcal{D}) \) and does not contain grazing collisions.
The set \(A_0\) is clearly open; moreover, it contains
\[\{(t;x,v) \in \mathbb{R} \times \pi_K^{-1}(\Omega^\mathrm{reg});\ \Psi^0(\pi_K(x,v)), \Psi^t(\pi_K(x,v)) \notin \partial\mathcal{D}\},
\]
which is dense in \(\mathbb{R} \times T^1K^\circ\).
We define mappings \( F_\epsilon\colon A_0 \to \Omega \), \( \epsilon \in [0,1] \) by
\begin{align*}
    F_0\colon (t; x,v)        & \mapsto \Psi^t(\pi_K(x,v)),                                           \\
    F_\epsilon\colon (t; x,v) & \mapsto \pi_K(\mathbf{g}^t_\epsilon(x,v)), \text{ for } \epsilon > 0.
\end{align*}
\begin{thm}[Theorem 5 in \cite{kourganoff_anosov_2016}]\label{thm:kourganoff_approximation}
    The family of maps \( F_\epsilon \) converges uniformly to \( F_0 \), as \( \epsilon \to 0 \), on every compact subset of \( A_0 \).
    Moreover, there is \( \epsilon_0 > 0 \) such that \( \mathbf{g}_\epsilon \) is an Anosov flow for every \( 0<\epsilon < \epsilon_0 \).
\end{thm}
In particular, up to rescaling, without loss of generality, we can assume that
\begin{enumerate}
    \item[(K3)] The family \( \{\mathbf{g}_\epsilon\}_{\epsilon \in (0,1]} \) is a one-parameter family of Anosov flows in \( K \).
\end{enumerate}
We will denote by \( \mathcal{K}(\mathcal{D}) \) the collection of all such surfaces:
\[
    \mathcal{K}(\mathcal{D}) \coloneqq \{K \subset E; p(K) = \mathcal{D} \text{ and } K \text{ satisfies (K0), (K1), (K2) and (K3)} \}.
\]

\begin{figure}[t]
    \centering
    \includegraphics[width=.8\textwidth]{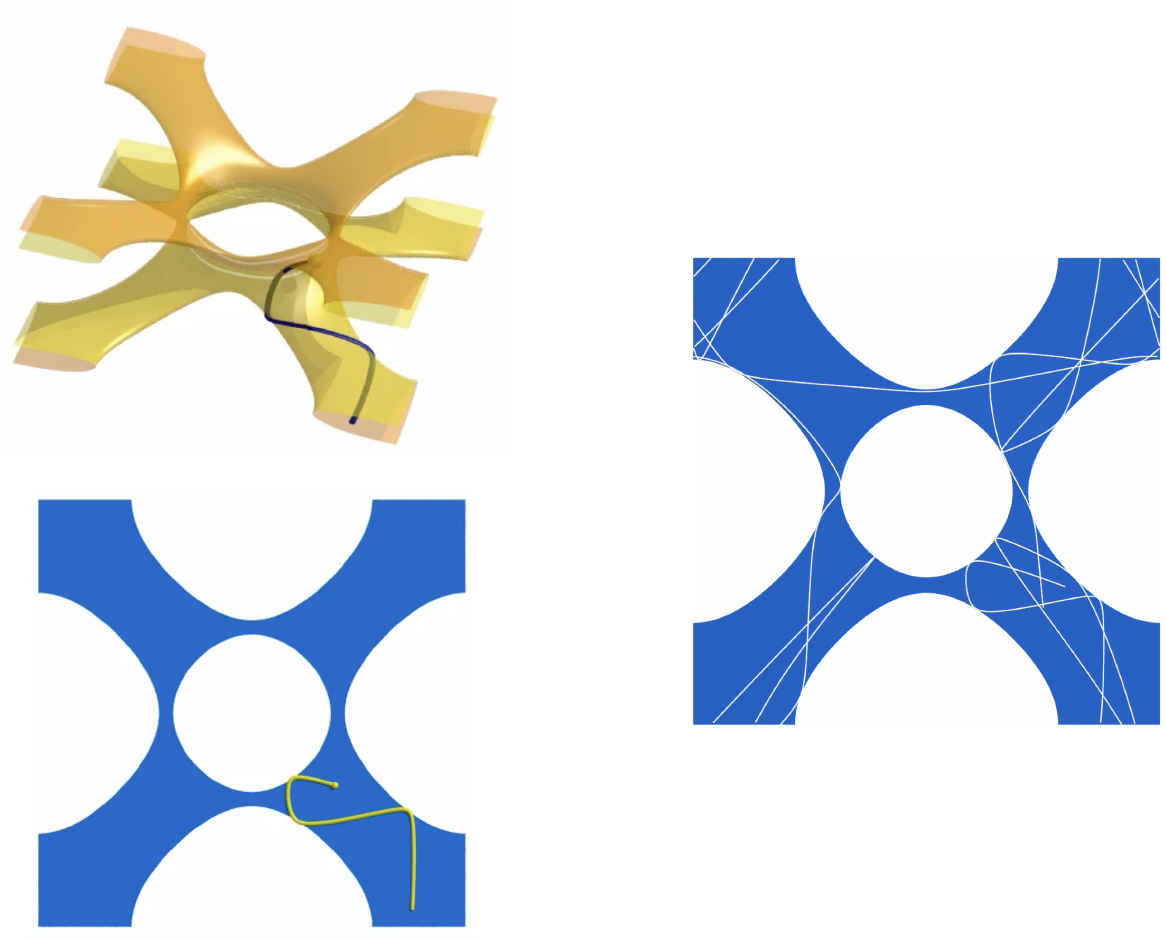}
    \caption{
        \textit{Left:} A geodesic arc on some surface \(K_\epsilon\) and its projection on the billiard table \(\mathcal{D}\).
        \textit{Right:} for \(\epsilon>0\) small and \((t;x,v)\in A_0\), the projection \(\{F_\epsilon(s;x,v)\}_{s\in [0,t]}\) closely shadows the billiard orbit segment \(\{F_0(s;x,v)\}_{s\in [0,t]}\) (Credits for pictures: Micka\"el Kourganoff, Jos Leys).}
\end{figure}

\subsubsection{Metric convergence of Kourganoff surfaces.}
Fix a Kourganoff surface \( K \in \mathcal{K}(\mathcal{D}) \).
Given a patch \(  \phi(u,v) = (x(u,v), y(u,v), z(u,v))  \) for \(  K  \), we consider the associated patch \(  \alpha_\epsilon \circ \phi = (x(u,v), y(u,v), \epsilon z(u,v))  \) for \(  K_\epsilon  \).
Comparing the First Fundamental Forms \( I = E \, du^2 + 2F \, du \, dv + G \, dv^2 \) and \( I_\epsilon = E_\epsilon \, du^2 + 2F_\epsilon \, du \, dv + G_\epsilon \, dv^2 \) for \( K \) and \( K_\epsilon \) in these coordinates, we obtain
\begin{equation}\label{eq:first_fund_forms_comparison}
    \begin{aligned}
        E_\epsilon & = E + (\epsilon^2 - 1) (\partial_u z)^2,         \\
        F_\epsilon & = F + (\epsilon^2 - 1) \partial_u z\partial_v z, \\
        G_\epsilon & = G + (\epsilon^2 - 1) (\partial_v z)^2.
    \end{aligned}
\end{equation}
Hence, for any \( \epsilon, \delta \in (0,1] \), the difference between the coefficients of the metrics \(g_\epsilon\) and \(g_\delta\), as well as the differences between their derivatives, are of order \( |\epsilon^2 - \delta^2| \).
Thus, the \(C^r\) norm of their difference, as a tensor field on \( K \), can be bounded by
\begin{equation}\label{eq:bound_geom}
    \|g_\epsilon - g_{\delta}\|_{C^r(K)}
    \leq  C|\epsilon^2 - \delta^2| \leq C,
\end{equation}
where \( C \) is a constant depending only on \( K \) (controlled by the \(C^r\) norm of the height function \(z\) on the compact surface \(K\)).
Hence, given a curve \( \gamma \), denoting by \(l_\epsilon(\gamma)\) its length with respect to the metric \(g_\epsilon\), the order of length variation of \(\gamma\) for two different metrics \( g_\epsilon \) and \( g_{\delta} \) is
\begin{equation}\label{eq:length_estimate}
    \big| l_\epsilon(\gamma) - l_{\delta}(\gamma) \big| \leq \int_{\gamma} \Big| \sqrt{I_\epsilon} - \sqrt{I_{\delta}} \Big|
    \leq C_K|\epsilon -\delta|l_1(\gamma).
\end{equation}

A first consequence of the equation above is that a curve \( \gamma \)\footnote{
    Please remark that through this section we frequently abuse the terminology and refer to both the maps \( [a,b] \to K \) and their images as curves.
} is rectifiable (i.e., has finite length) with respect to \( g_\epsilon \) if and only if it is rectifiable with respect to the reference metric \( g_1 \), thus the class of rectifiable curves is the same for all the metrics.
Given \( x,y \in K \), we will denote by \(  [x,y]_\epsilon  \) the minimising geodesic for \(  g_\epsilon  \) joining them.
Note that \( \{[x,y]_\epsilon\} \) is a collection of curves in \( K \) with endpoints \( x,y \), all of which are rectifiable with respect to every metric in the family \( \{g_\epsilon\} \).

Moreover, by comparing the fundamental forms in~\eqref{eq:first_fund_forms_comparison}, for \(v \in TK\), we have \(g_\epsilon(v,v) = |dp(v)|^2 + \epsilon^2 |dz(v)|^2\), which is non-decreasing in \(\epsilon\).
Hence, for every fixed curve \(\gamma\), the length \(\epsilon \mapsto l_\epsilon(\gamma)\) is non-decreasing, and therefore so is \(\epsilon \mapsto d_\epsilon(x,y)\); being bounded below by \(0\), it converges as \(\epsilon \to 0\).
We may thus define
\begin{equation}\label{eq:unif_conver_dist}
    d_0(x,y) \coloneqq \lim_{\epsilon \to 0} d_\epsilon(x,y) = \lim_{\epsilon \to 0} l_\epsilon([x,y]_\epsilon).
\end{equation}
A priori \(d_0\) is only a pseudo-distance on \(K\); Theorem~\ref{thm:d_0_is_a_metric} below shows that it is a distance, that \(d_\epsilon \to d_0\) uniformly on \(K \times K\), and describes the resulting length space \(K_0 \coloneqq (K, d_0)\).
\par Away from the seam \(B(K)\), Equation~\eqref{eq:bound_geom} controls the geometry.
No such control holds near \(B(K)\), where \(\det g_\epsilon \to 0\) and \(\|g_\epsilon^{-1}\| \to \infty\), and the sprays are not uniformly Lipschitz as \(\epsilon \to 0\).
Taking the non-degenerate limit spray \(f_0\) as the fixed reference field and by the continuous dependence of ODEs on the field (see e.g.~\cite[Chapter 3.12]{walter2013ordinary}), we obtain the following statement on the local uniform convergence of sprays. 
\begin{lem}\label{lem:uniform_conv_geodesic_flows}
    Let \(Q \subset K^\circ\) be relatively compact and \(J \subset \mathbb{R}\) compact, and suppose there are a compact \(Q' \Subset K^\circ\) and \(\epsilon_* > 0\) with \(\mathbf{g}_\epsilon^t(X) \in T^1Q'\) for all \(X \in T^1Q\), all \(t \in J\), and all \(0 < \epsilon \leq \epsilon_*\).
    Here \(T^1Q\), \(T^1Q'\) denote the \(g_0\)-unit tangent bundles, the flows \(\mathbf{g}_\epsilon\) being compared on \(TK\) through their sprays,
    and \(d_G\) is a fixed Riemannian distance on \(TK|_{Q'}\).
    Writing \(\mathbf{g}_0\) for the geodesic flow of the non-degenerate metric \(g_0\rvert_{Q'}\), we have
    \[
        \sup_{X \in T^1Q}\ \sup_{t \in J}\ d_G\bigl(\mathbf{g}_\epsilon^t X, \mathbf{g}_0^t X\bigr) \xrightarrow{\ \epsilon \to 0\ } 0,
    \]
    and the convergence is \(C^1\) in \(X\).
\end{lem}
\begin{proof}
    On \(Q' \Subset K^\circ\) we have \(\det g_0 \geq c(Q') > 0\), so \(g_\epsilon \to g_0\) in \(C^r(Q')\) gives \(g_\epsilon^{-1} \to g_0^{-1}\) in \(C^r(Q')\); hence the Christoffel symbols and sprays converge, \(f_\epsilon \to f_0\) in \(C^{r-1}(T^1Q')\), and the sectional curvatures of \(g_\epsilon\) are uniformly bounded on \(Q'\) for \(\epsilon \leq \epsilon_*\).
    Take \(f_0\) as the fixed reference field: it is Lipschitz on \(T^1Q'\) with a constant \(L = L(Q')\) independent of \(\epsilon\).
    By Gr\"onwall's lemma applied to the orbit segments (which remain in \(T^1Q'\)) with reference \(f_0\), perturbation \(f_\epsilon\), \(\delta_\epsilon \coloneqq \|f_\epsilon - f_0\|_{C^0(T^1Q')} \to 0\), and equal initial data, we have
    \[
        \sup_{t \in J} d_G(\mathbf{g}_\epsilon^t X, \mathbf{g}_0^t X) \leq \delta_\epsilon\,|J|\,e^{L|J|} \xrightarrow{\ \epsilon \to 0\ } 0,
    \]
    uniformly in \(X \in T^1Q\).
    By differentiating the geodesic equation and repeating the estimate for the Jacobi equation, whose coefficients also converge in \(C^0(T^1Q')\), we upgrade this to \(C^1\) convergence in \(X\).
\end{proof}

\begin{lem}\label{lem:uniform_lower_bound_inj_radius}
    There is a uniform lower bound \( R_0 > 0 \) on the injectivity radius \( R_\epsilon \) of the metrics \( g_\epsilon \).
\end{lem}
\begin{proof}
    As the geodesic flows \( \mathbf{g}_\epsilon \) are all Anosov, none of the surfaces \( K_\epsilon \) have conjugated points (see e.g.~\cite{klingenberg_riemannian_1974}).
    Therefore, according to Klingenberg's lemma (see \cite[Lemma 1.8]{AbreschMeyer1997}) the injectivity radius of \( K_\epsilon \) is
    \[
        \mathrm{inj}(g_\epsilon) = \frac{1}{2}l_\epsilon(\sigma_\epsilon)
    \]
    where \( \sigma_\epsilon \) is the closed geodesic of minimal \( g_\epsilon \)-length in \( K \).
    We claim that
    \[
        \inf_{0 < \epsilon \leq 1} l_\epsilon(\sigma_\epsilon) \geq s_0 \coloneqq \min\bigl\{\operatorname{sys}(\mathbb{T}^2),\ \min_i l_i,\ 2\tau_\mathrm{min}\bigr\} > 0,
    \]
    where \(\operatorname{sys}(\mathbb{T}^2)\) is the length of the shortest closed flat geodesic of the torus, \(l_i\) the length of the scatterer \(\Gamma_i \subset \mathcal{D}\), and \(\tau_\mathrm{min}\) the minimal distance between two distinct lifted obstacles in \(\mathbb{R}^2\) (lattice translates of the same obstacle included), i.e., the minimal flight time of the table.
    \par Since the projection \(p\colon E \to \mathbb{T}^2\) is \(1\)-Lipschitz for every \(g_\epsilon\), we have \(l_\epsilon(\sigma_\epsilon) \geq l_E(\gamma_\epsilon)\) for the projected loop \(\gamma_\epsilon \coloneqq p(\sigma_\epsilon)\), which is a loop in the closed table \(\overline{\mathcal{D}}\) (it avoids the open obstacles); therefore, it suffices to bound \(l_E(\gamma_\epsilon)\) below by \(s_0\) for the projection of any loop \(\sigma\) which is non-contractible in \(K\).
    Note that non-contractibility of \(\sigma\) in \(K\) does not imply non-contractibility of \(\gamma = p(\sigma)\) in \(\mathcal{D}\): the kernel of \(p_*\) contains for instance all the classes of billiard trajectories bouncing back and forth between two scatterers, whose projections are segments traversed twice.
    We argue according to the crossings of \(\sigma\) with the fold locus \(B(K)\), using that each transversal crossing switches the sheet, so that the number of fold crossings is even.
    \par \emph{Case 1: \(\sigma\) stays in one closed sheet.} In that case, \(p\) restricts to a homeomorphism of the sheet onto \(\overline{\mathcal{D}}\), so \(\gamma\) is non-contractible in \(\overline{\mathcal{D}}\).
    By minimising the Euclidean length in the free homotopy class of \(\gamma\) over the compact \(\overline{\mathcal{D}}\), we obtain a closed Euclidean geodesic of \(\overline{\mathcal{D}}\): a concatenation of straight chords, with endpoints on the walls, and of wall arcs.
    If this minimiser is essential in \(\mathbb{T}^2\), its length is at least \(\operatorname{sys}(\mathbb{T}^2)\).
    If it is trivial in \(\mathbb{T}^2\), we lift it to a  closed loop in \(\mathbb{R}^2\): a chord leaving a strictly convex lifted obstacle never returns to the same lift, so
    \begin{enumerate} 
    \item[$\bullet$] either the lifted loop touches a single obstacle lift and contains no chord; in that case, the class is a non-trivial power of a single scatterer class \(\Gamma_i\), and any representative winding \(m \neq 0\) times around the convex obstacle \(O_i\) has length at least \(|m|\, l_i \geq l_i\);
    \item[$\bullet$] or it contains at least two chords, each joining two distinct obstacle lifts, hence each of length \(\geq \tau_\mathrm{min}\), for a total \(\geq 2\tau_\mathrm{min}\).
    \end{enumerate}
    \par \emph{Case 2: \(\sigma\) crosses the fold locus.} The crossings decompose \(\sigma\) into an even number \(\geq 2\) of arcs alternating between the two sheets, with endpoints on the wall circles.
    We write the class of \(\sigma\) as the corresponding alternating word in the amalgamated product \(\pi_1(K) = \pi_1(\overline{\mathcal{D}}) *_{\pi_1(B(K))} \pi_1(\overline{\mathcal{D}})\).
    If some projected arc is homotopic (relative to its endpoints and within \(\overline{\mathcal{D}}\)) into the wall containing them, the word reduces; indeed, we can push the arc into the wall and merge it with its neighbours. After reduction, either the class lies in a vertex group (a one-sheet class, handled by Case 1, with the same length bound; indeed, replacing an arc by the wall arc it is homotopic to does not increase the Euclidean length)
    or there remain at least two irreducible arcs.
    Let us lift an irreducible arc to \(\mathbb{R}^2\), with endpoints on lifted walls.
    If the endpoints lie on two distinct obstacle lifts, then the arc has length at least \(\tau_\mathrm{min}\).
    If they lie on the same lift \(\widetilde{\Gamma}\), irreducibility means that the arc is not homotopic into \(\widetilde{\Gamma}\) relative to its endpoints within the complement of the open obstacles; then the Jordan region bounded by the arc and a boundary arc of \(\widetilde{\Gamma}\) contains another obstacle lift \(\Omega\), and the arc passes through a point \(z\) on the far side of \(\Omega\): moving from any \(x \in \Omega\) along the ray opposite to the nearest-point direction to \(\widetilde{\Gamma}\), the distance to the convex set \(\widetilde{\Gamma}\) is convex and increasing, so \(d(z, \widetilde{\Gamma}) \geq d(x, \widetilde{\Gamma}) \geq \tau_\mathrm{min}\), and the arc, whose endpoints lie on \(\widetilde{\Gamma}\), has length \(\geq 2\, d(z, \widetilde{\Gamma}) \geq 2\tau_\mathrm{min}\).
    In all cases the total Euclidean length is at least \(2\tau_\mathrm{min}\) when it remains at least two irreducible arcs (see Figure~\ref{fig:projected_limiting_curve}).
    \begin{figure}[t]
        \centering
        \includegraphics[width=.8\textwidth]{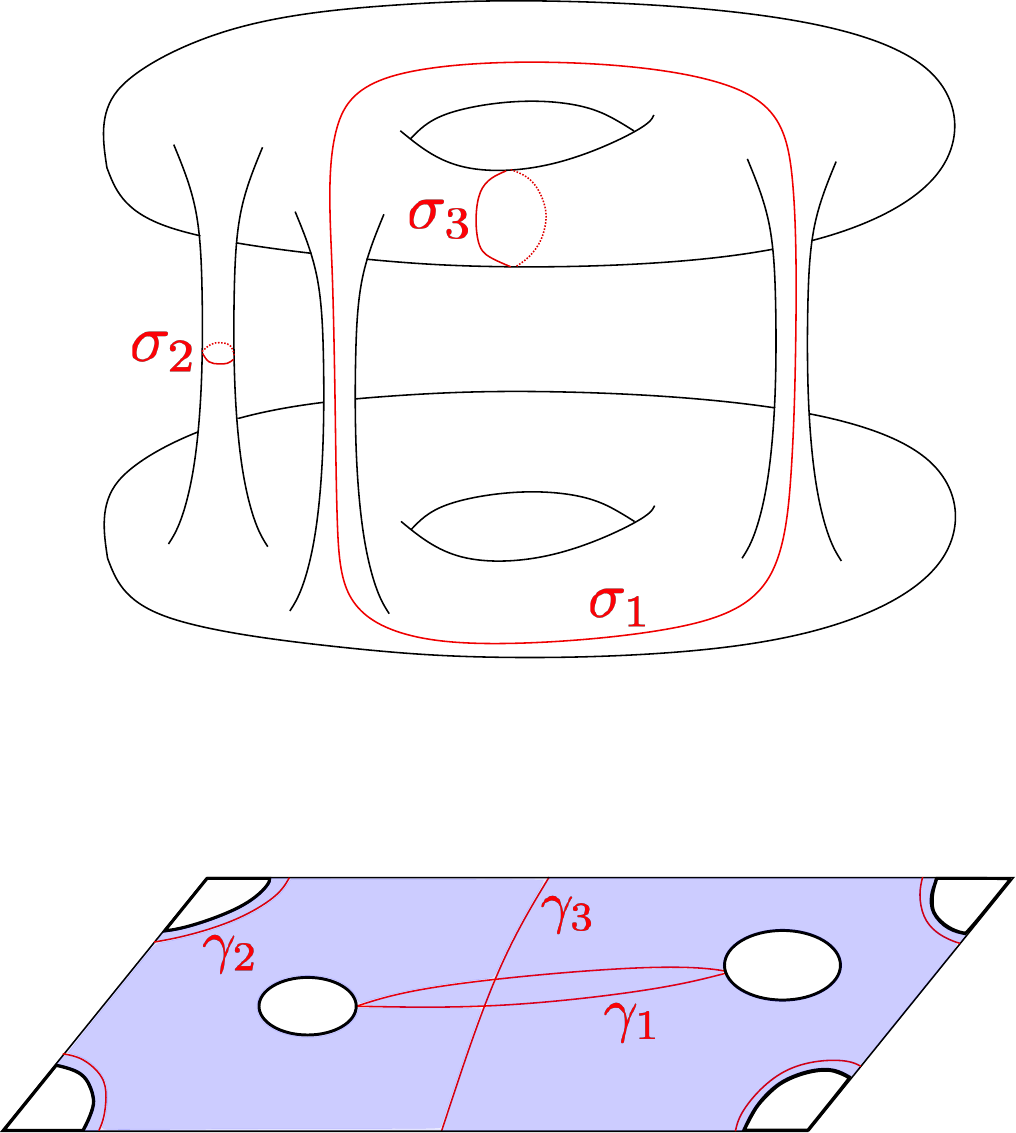}
        \caption{Some loops \(\{\sigma_i\}_{i=1}^3\) on the surface \(K\), and their projections \(\{\gamma_i=p(\sigma_i)\}_{i=1}^3\) on the billiard table \(\mathcal{D}\).}
        \label{fig:projected_limiting_curve}
    \end{figure}
    \par We conclude that \(l_\epsilon(\sigma_\epsilon) \geq l_E(\gamma_\epsilon) \geq s_0\), and \(R_0 \coloneqq \tfrac12 s_0\) is the desired uniform lower bound.
\end{proof}
\begin{lem}\label{lem:convergence_geodesics}
    Let \( x \in K^\circ \), and let \( U \) be the connected component containing \( x \) of \( p^{-1}(B_\mathcal{D}(p(x), \delta)) \), where \( \delta > 0 \) is small enough such that \( B_\mathcal{D}(p(x), \delta) \subset \mathrm{int}(\mathcal{D}) \) contains no collision point of any billiard trajectory, and \( \overline U \subset B_{g_1}(x, R_0) \), so that \( \overline U \) lies in a normal neighbourhood of \( x \) for every metric \( g_\epsilon \) (by Lemma~\ref{lem:uniform_lower_bound_inj_radius}, since \( d_\epsilon \leq d_1 \)). For every \( y \in U \), the geodesics \( [x,y]_\epsilon \) converge uniformly to a \( C^1 \) curve \( [x,y]_0 \subset U \), the lift of the Euclidean segment \([p(x), p(y)]\); in particular \( d_0(x,y) = |p(x) - p(y)| \).
\end{lem}
\begin{proof}
    Let \( \widetilde x \in K^\circ\), \(x = p(\widetilde x) \in \mathrm{int}(\mathcal{D}) \), and consider a neighbourhood \( U \) of \( \widetilde x \), which we choose sufficiently small so that \( \overline U \) is a normal neighbourhood for all of the metrics \( g_\epsilon \) (recall Lemma~\ref{lem:uniform_lower_bound_inj_radius}), and is moreover mapped under \( p \) into an open ball \( B \coloneqq B_{\mathcal{D}}(x, \delta) \) with sufficiently small radius \( \delta \), ensuring that the billiard trajectories starting at \( x \) in any  direction do not have any collision inside \( B \).
    \par Moreover, for a fixed \( \widetilde y \in U \), distinct from \( \widetilde x \), we write \( l_\epsilon = d_\epsilon(\widetilde x, \widetilde y) \) and let \( \widetilde{v}_\epsilon \) be the velocity of the curve \( [\widetilde x,\widetilde y]_\epsilon \) at the starting point \( \widetilde x \).
    Finally, we set \( y = p(\widetilde y) \), let \( T = |x-y|, v = T^{-1}(y-x) \), and consider the geodesics \( \gamma_\epsilon \) in \( K_\epsilon \) with initial condition \( (\widetilde x,\widetilde v) \), where \( \pi_K(\widetilde x, \widetilde v) = (x, v) \).
    \par By construction, the point \( (T; \widetilde x,\widetilde v) \) lies in \( A_0 \), and Theorem \ref{thm:kourganoff_approximation} guarantees that \( \pi_K(\mathbf{g}_\epsilon^t(\widetilde x,\widetilde v)) \) converges uniformly to \( \Psi^t(x, v) \) in the interval \( [0, T] \) as \( \epsilon \to 0 \).
    Thus, for \( \epsilon \) sufficiently small we have \( \gamma_\epsilon(T) \) arbitrarily close to \( \widetilde y \).
    Since the exponential map \(\exp^\epsilon_{\widetilde x}\) is a diffeomorphism onto a neighbourhood of \( \overline U \), and since \(\exp^\epsilon_{\widetilde x}(l_\epsilon \widetilde v_\epsilon) = \widetilde y\) while \(\exp^\epsilon_{\widetilde x}(T\widetilde v) = \gamma_\epsilon(T)\), we can estimate
    \[
        \|l_\epsilon\widetilde v_\epsilon - T\widetilde v\|_\epsilon \leq \Big\|d\big((\exp^\epsilon_{\widetilde x})^{-1}\big)\rvert_{\overline U}\Big\|_\epsilon d_\epsilon(\widetilde y, \gamma_\epsilon(T)).
    \]
    Around \( \widetilde x \) the norm of the derivative is close to \( 1 \).
    Moreover, as pointed out in \cite{kourganoff_anosov_2016}, as the surface \( K \) flattens out, the curvature away from \( B(K) \) tends to \( 0 \), while in a neighbourhood of \( B(K) \) it becomes strictly negative.
    Therefore, for \( \epsilon \) sufficiently small, the sectional curvatures of \( K_\epsilon \) on \( \overline U \) are bounded in absolute value by some \( 1/R^2 \), and we can estimate the derivative of \( \exp^\epsilon_{\widetilde x} \) in terms of the norms of Jacobi fields (using~\cite[Theorem 11.2]{lee2019introduction}, for instance), which are bounded by \(R\sinh(r/R)\) at radius \(r\), hence bounded on \(\overline{U}\) uniformly in \(\epsilon\).
    We may then find a constant \( L \), independent of \( \epsilon \), such that
    \[
        \|l_\epsilon\widetilde v_\epsilon - T\widetilde v\|_\epsilon \leq Ld_\epsilon(\widetilde y, \gamma_\epsilon(T)) \xrightarrow{\ \epsilon \to 0\ } 0.
    \]
    On the other hand, \(l_\epsilon = d_\epsilon(\widetilde x, \widetilde y) \geq |x - y| = T\) because \(p\) is \(1\)-Lipschitz, and \(l_\epsilon \leq l_\epsilon(\sigma)\), where \(\sigma \subset U\) is the lift of the segment \([x,y]\), whose \(g_\epsilon\)-length decreases to \(T\) as \(\epsilon \to 0\); hence \(l_\epsilon \to T\), and therefore \(\widetilde v_\epsilon \to \widetilde v\).
    Since \(\overline U \Subset K^\circ\) is collision-free and the geodesics \([\widetilde x, \widetilde y]_\epsilon\) stay in \(\overline U\) (being minimising in a normal ball), we may apply Lemma~\ref{lem:uniform_conv_geodesic_flows} with \(Q = U\) and \(Q' = \overline U\): the \(g_\epsilon\)-geodesics with initial data \((\widetilde x, \widetilde v_\epsilon)\) converge, uniformly on \([0,T]\) and in \(C^1\), to the \(g_0\)-geodesic with initial data \((\widetilde x, \widetilde v)\), which is the lift \(\sigma\) of the segment \([x,y]\). Hence \([\widetilde x, \widetilde y]_\epsilon \to [\widetilde x, \widetilde y]_0 \coloneqq \sigma\), and \(d_0(\widetilde x, \widetilde y) = \lim_\epsilon l_\epsilon = T = |x - y|\).
\end{proof}
Essentially the same argument shows that we may fix the tangent vectors \(v_\epsilon\) and let the endpoint \(y\) vary instead, which gives the following corollary.
\begin{cor}\label{unique.directions.away.from.seam.}
    Given \((x,v) \in T^1K^\circ\), let \(\gamma_{\epsilon;v}\) denote the unique geodesic of \((K, g_\epsilon)\) with initial condition \((x,v)\).
    If \(U\) is any connected neighbourhood of \(x\) not intersecting \(B(K)\), then there is a unique differentiable curve \(\gamma_{0;v}\), starting at \(x\), and whose tangent space at \(x\) is spanned by \(v\).
\end{cor}
\begin{thm}\label{thm:d_0_is_a_metric}
    The function \(d_0\) defined in~\eqref{eq:unif_conver_dist} is a distance on \(K\): \(d_0(x,y) > 0\) for all distinct \( x, y \in K \).
    Moreover, \( K_0 \coloneqq (K, d_0)  \) is a complete geodesic length space, whose topology is that of \(K\), and \(d_\epsilon \to d_0\) uniformly on \(K \times K\).
    Every \(g_1\)-rectifiable curve is \(d_0\)-rectifiable, and the length function is
    \begin{align*}
        l_0\colon \mathrm{Rec}(K, g_1) & \to [0, \infty)                  \\
        \gamma                       & \mapsto \lim_{\epsilon \to 0}l_\epsilon(\gamma) = l_E(p \circ \gamma),
    \end{align*}
    where \(l_E\) denotes the Euclidean length.
    Every locally uniform limit of geodesics of \(K_{\epsilon_n}\), \(\epsilon_n \to 0\), is a geodesic\footnote{\label{footnotegeodesic} We mean geodesic in the sense of length spaces, i.e., a locally minimising path with respect to \( d_0 \).
       In general, the tensors \( g_\epsilon \) degenerate as \( \epsilon \to 0 \) and the metric \( d_0 \) is not associated to any Riemannian structure.
    } of \( K_0 \); conversely, every billiard trajectory with only transversal collisions is the projection of such a limit.
    They are \(C^1\) on \(\gamma^{-1}(K \setminus B(K))\), with well-defined one-sided directions in \(T\mathcal{D}\) at each collision; the convergence is \(C^1\) away from \(B(K)\), but at a transversal crossing
    the limit may not be differentiable.
\end{thm}
\begin{proof} 
    Symmetry, the triangle inequality and \(d_0(x,x) = 0\) follow from the definition \eqref{eq:unif_conver_dist}. Moreover, since \(p\) is \(1\)-Lipschitz from \((K, g_\epsilon)\) to the Euclidean torus for every \(\epsilon\), we have \(d_0(x,y) \geq |p(x) - p(y)|\); hence \(d_0(x,y) > 0\) whenever \(p(x) \neq p(y)\).
    If \(p(x) = p(y)\) and \(x \neq y\), then \(p(x) \in \mathrm{int}(\mathcal{D})\) (boundary points have a unique preimage) and \(x, y\) lie on distinct sheets, so every rectifiable path \(\gamma\) from \(x\) to \(y\) crosses the seam \(B(K)\): its projection travels from \(p(x)\) to \(\partial\mathcal{D}\) and back, so that \(l_\epsilon(\gamma) \geq l_E(p \circ \gamma) \geq 2\,d_\mathcal{L}(p(x), \partial\mathcal{D})\) for every \(\epsilon\), and \(d_0(x,y) \geq 2\,d_\mathcal{L}(p(x), \partial\mathcal{D}) > 0\).
    Thus \(d_0\) is a metric.
    \par Since \(d_0 \leq d_1\), we have \(|d_0(x,y) - d_0(x',y')| \leq d_1(x,x') + d_1(y,y')\), so \(d_0\) is continuous on \(K \times K\) for the manifold topology. Hence \(d_\epsilon \to d_0\) uniformly (by Dini's theorem), and the identity \((K, d_1) \to (K, d_0)\)
     is a homeomorphism.
    In particular \(K_0\) is compact, hence complete, and it is the Gromov--Hausdorff limit of the length spaces \(K_\epsilon\); therefore it is a length space~\cite[Theorem 7.5.1]{burago2022course}, and, being compact, a geodesic space.
    \par Let \(\gamma\) be \(g_1\)-rectifiable. By monotone convergence,
    \[
        \lim_{\epsilon \to 0} l_\epsilon(\gamma) = \lim_{\epsilon \to 0}\int \sqrt{|dp(\dot\gamma)|^2 + \epsilon^2|dz(\dot\gamma)|^2} = \int|dp(\dot\gamma)| = l_E(p \circ \gamma).
    \]
    The length \(L_0(\gamma)\) induced by the metric \(d_0\) satisfies \(L_0(\gamma) \leq \lim_\epsilon l_\epsilon(\gamma)\), because \(d_0 \leq l_\epsilon\) on every sub-arc, and \(L_0(\gamma) \geq l_E(p \circ \gamma)\), because \(d_0 \geq |p(\cdot) - p(\cdot)|\).
    Hence \(l_0 \coloneqq L_0 = l_E \circ p = \lim_\epsilon l_\epsilon\) on \(g_1\)-rectifiable curves 
    (by the gluing description of Theorem~\ref{thm:K_0_as_gluing_metric_space} below, we will justify why every \(d_0\)-rectifiable curve has rectifiable projection, with the same formula). 
    \par Let us now study limits of geodesics. Let \(\gamma_n\) be unit-speed geodesics of \(K_{\epsilon_n}\), \(\epsilon_n \to 0\), converging locally uniformly to a curve \(\gamma\) (the \(\gamma_n\) are \(1\)-Lipschitz for \(d_0 \leq d_{\epsilon_n}\), so such limits exist along subsequences by Arzel\`a--Ascoli).
    Fix \(a_0 < b_0\) with \(b_0 - a_0 < R_0\), the uniform injectivity radius of Lemma~\ref{lem:uniform_lower_bound_inj_radius}; then \(\gamma_n\rvert_{[a_0, b_0]}\) is minimising, \(d_{\epsilon_n}(\gamma_n(a_0), \gamma_n(b_0)) = b_0 - a_0\), and by uniform convergence of \(d_{\epsilon_n}\) to \(d_0\) and continuity of \(d_0\) we get \(d_0(\gamma(a_0), \gamma(b_0)) = b_0 - a_0\).
    On the other hand,
    \[
        L_0(\gamma\rvert_{[a_0,b_0]}) = \sup \sum_i d_0(\gamma(t_i), \gamma(t_{i+1})) = \sup \sum_i \lim_n d_{\epsilon_n}(\gamma_n(t_i), \gamma_n(t_{i+1})) \leq b_0 - a_0,
    \]
    where the suprema is over partitions of \([a_0, b_0]\).
    Hence \(\gamma\) is locally minimising and parametrised by \(d_0\)-arc length, i.e., it is a unit-speed geodesic of \(K_0\).
    Conversely, if \((t; x, v) \in A_0\) for every \(t\), i.e., if the billiard trajectory of \(\pi_K(x,v)\) has only transversal collisions, Theorem~\ref{thm:kourganoff_approximation} shows that the \(g_\epsilon\)-geodesics with initial data \((x,v)\) converge locally uniformly to the lift of that trajectory starting at \(x\) (the lift being determined by the alternation of sheets at transversal crossings), which is a geodesic of \(K_0\) by what precedes.
    \par The regularity statements away from \(B(K)\) follow from Lemma~\ref{lem:convergence_geodesics}. We will investigate later (see the normal form of Lemma~\ref{lem:normal_form} below) the behavior of the height coordinate \(z \circ \gamma\) at a transversal crossing and see it has a square-root cusp, so the limit may not be differentiable there.
\end{proof}

\begin{rmk}\label{rmk:interior_points_have_tangent_directions}
    The results above, in particular Theorem~\ref{thm:d_0_is_a_metric}, do not guarantee uniqueness of the geodesic  when the initial condition lies in \(T^1B(K)\).
    In this case, if \(v\) is transverse to \(B(K)\) (i.e., if \(\pi_K(x,v)\) corresponds to a regular collision on \(\Omega\)), then Kourganoff's approximation theorem guarantees \(\gamma_{0;v}\) is unique.
    Otherwise, if \(v\) is tangent to \(B(K)\), then there are three geodesics in \(K_0\) starting at \(x\) with velocity \(v\): the component of \(B(K)\) containing \(x\), and the straight continuation of the line \(t \mapsto p(x) + tv\) in either of the two sheets (see e.g. Lemma~\ref{pf:lem_directions}(iv)).
\end{rmk}

\subsubsection{\(K_0\) as a metric gluing.}

There is yet another way to construct \(K_0\).
According to the proof of Lemma~\ref{lem:convergence_geodesics}, given a connected component \( S \) of \( K^\circ \), every \( x \in S \) has a neighbourhood \( U \) such that, for every \( y \in U \), the curve \( [x, y]_0 \) satisfies
\[
    l_0([x, y]_0) = \lim_{\epsilon \to 0}l_\epsilon([x,y]_\epsilon) = |p(x) - p(y)|.
\]
We can then restate the main ingredient of the proof as
\begin{lem}\label{lem:d_0_is_distance_reestated}
    Let \( S \) be a connected component of \( K^\circ \) and \( |\cdot| \) the Euclidean distance in \( \mathcal{D} \).
    Then \( p\rvert_S\colon (S, d_0) \to (\mathcal{D}, |\cdot|) \) is a local isometry.
\end{lem}
The argument is highly dependent on the approximation by geodesic flows, which fails if the points whose distance we are measuring are eclipsed by a scatterer, or after the billiard trajectory experiences a grazing collision.
To deal with that, we consider a different metric on \( \mathcal{D} \), making it into a length space.
\begin{dfn}[Length metric of \( \mathcal{D} \)]
    The table \( \mathcal{D} \) has a length structure \( \mathcal{L} \) consisting of all Euclidean rectifiable paths \( \gamma\colon I \to \mathbb{R}^2 \) whose image is completely contained in \( \mathcal{D} \) together with the usual Euclidean length \( l_E \) induced by \( |\cdot| \).
    We call the \emph{length metric} of \( \mathcal{D} \), denoted by \( d_\mathcal{L} \), the metric induced by \( \mathcal{L} \), i.e.,
    \[
        d_\mathcal{L}(x, y) = \inf\{l_E(\gamma); \gamma \in \mathcal{L}, \gamma(0) = x, \gamma(1) = y\}.
    \]
\end{dfn}
The geodesics of \( d_\mathcal{L} \) are straight lines in \( \mathcal{D} \), segments of boundary, and paths consisting of concatenations of these two types: such curves are straight lines that meet a scatterer at a tangential angle, follow along the boundary for a while, and leave into a straight trajectory again, with another grazing angle (see Figure \ref{fig:types_geodesics}).
\begin{figure}[t]
    \centering
    \includegraphics[width=.8\textwidth]{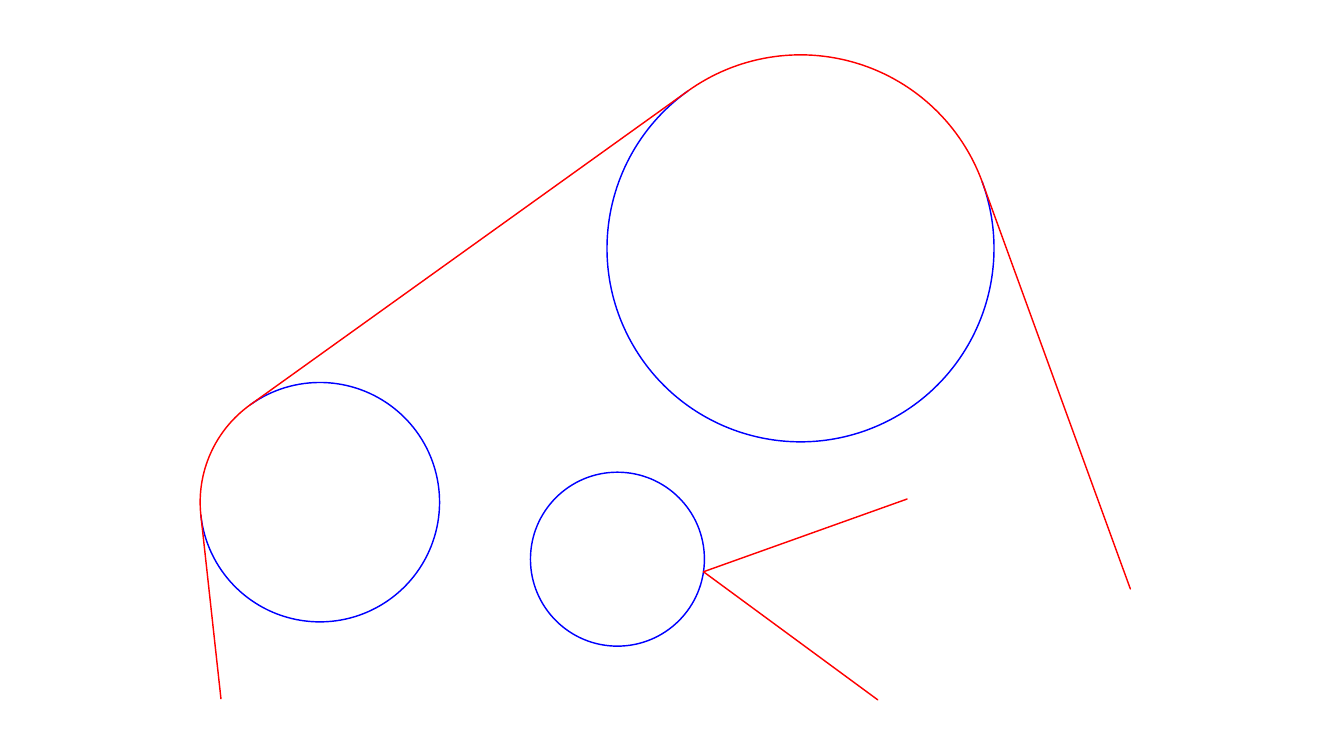}
    \caption{The longer curve following the obstacles is a \(d_\mathcal{L}\)-geodesic, while the billiard trajectory segment is merely piecewise geodesic.}
    \label{fig:types_geodesics}
\end{figure}
\begin{rmk}
    Similarly to \(K_0\), the space \((\mathcal{D}, d_\mathcal{L})\) is not uniquely geodesic.
    If \(x \in \Gamma_i \subset \partial\mathcal{D}\) is a boundary point, and \(v\) is the tangent vector to \(\Gamma_i\) at \(x\), then there are two geodesics starting at \(x\) in the direction of \(v\): the straight line \(x + tv\) into the interior of the table, and the curve \(\Gamma_i\) itself.
    Similarly, the points of \(B(K)\) are exactly those at which \(d_0\) fails to be uniquely geodesic.
\end{rmk}
\begin{lem}\label{lem:sheet_isometry}
    For any connected component \( S \subset K^\circ \), \( p\rvert_{\overline{S}}\colon (\overline S, d_0) \to (\mathcal{D}, d_\mathcal{L}) \) is an isometry.
\end{lem}
\begin{proof}
    The projection \( p \) is Lipschitz with respect to the metric \( g_1\) on \(K\) and the Euclidean norm \(|\cdot| \) on \(\mathbb{T}^2\), hence the image \( p(\gamma) \) of every rectifiable curve in \( K \) is rectifiable with respect to the Euclidean metric, and therefore \( p(\mathrm{Rec}(K)) \subset \mathcal{L} \).
    \par According to Lemma \ref{lem:d_0_is_distance_reestated} the map \( p \) is a local isometry on \( S \), hence it preserves the length of every rectifiable curve \( \gamma \) mapped into \( \mathrm{int}(\mathcal{D}) \), that is, \( l_0(\gamma) = l_E(p(\gamma)) \).
    More generally, \(l_0(\gamma) = l_E(p \circ \gamma)\) for every \(g_1\)-rectifiable curve \(\gamma\) in \(K\), by Theorem~\ref{thm:d_0_is_a_metric}.
    \par Let \(x, y \in \overline S\). Every rectifiable path in \(K\) from \(x\) to \(y\) projects to a path in \(\overline{\mathcal{D}}\) from \(p(x)\) to \(p(y)\) of the same length, so \(d_0(x, y) \geq d_\mathcal{L}(p(x), p(y))\): the map \(p\rvert_{\overline S}\) is non-expanding.
    Conversely, by~(K0) the map \(p\rvert_{\overline S}\) is a homeomorphism onto \(\overline{\mathcal{D}}\) (in fact, it is a diffeomorphism on \(S\), extending continuously and bijectively to the closure, the fold being one-to-one over \(\partial\mathcal{D}\)). A \(d_\mathcal{L}\)-geodesic from \(p(x)\) to \(p(y)\), which exists by compactness of \(\overline{\mathcal{D}}\), lifts to a path in \(\overline S\) from \(x\) to \(y\) of \(l_0\)-length \(d_\mathcal{L}(p(x), p(y))\), so that \(d_0(x,y) \leq d_\mathcal{L}(p(x), p(y))\).
    Hence \(p\rvert_{\overline S}\) is an isometry.
\end{proof}
\begin{thm}\label{thm:K_0_as_gluing_metric_space}
    The space \( K_0 \) is isometric to the metric quotient of the disjoint union \(
        \coprod_{i=1}^2 (\mathcal D,d_{\mathcal L})
    \) of two copies of the table,
    obtained by identifying corresponding boundary points.
\end{thm}
\begin{proof}
    Let \( \widehat d \) be the metric on the disjoint union \( X \coloneqq\coprod_{i=1}^2(\mathcal{D}, d_\mathcal{L}) \).
    It is defined as
    \[
        \widehat d(x,y) \coloneqq \begin{cases}
            d_\mathcal{L}(x,y), \text{ if } x, y \text{ belong to the same copy of } \mathcal{D}; \\
            \infty, \text{ otherwise}.
        \end{cases}
    \]
    In general, for an equivalence relation \( R \) on the disjoint union, the semi-metric \(d_R\) ``identifying'' the \( R \)-related points (by assigning distance zero to them) is given by
    \[
        d_R(x,y) \coloneqq \inf\Bigg\{\sum_{i=0}^q\widehat d(x_i, y_i);\  x_0 = x, y_q = y \text{ and } y_i R x_{i+1} \text{ for } i = 1, \dots, q-1\Bigg\}.
    \]
    In our situation such a semi-metric has the equivalent, simpler description
    \[
        d_\sim(x,y) = \inf\{d_\mathcal{L}(x, m_1) + d_\mathcal{L}(m_2, y); \  m_1 \sim m_2\},
    \]
    since it is possible to change from one copy of the table to the other by ``passing'' through the boundary of \( \mathcal{D} \) only once, and a path switching copies more than once can be shortened to one switching once, the identification being the identity on \(\partial\mathcal{D}\).
    We also note that \(d_\sim (x,y) = 0 \) if and only if \( x \sim y \), so the only pairs of distinct points \( (x, x') \) whose distance is zero are  those where \(x\) and \(x'\) correspond to the same boundary point.
    In particular, since \(X\) is compact the quotients \(K \approx X \big/\!\!\sim\) and \( X\big/d_\sim \) are homeomorphic\footnote{
        For a semi-metric \(d\) on \(X\), we denote by \( X\big/d\) the set of equivalence classes of the relation \(d(x,y) = 0\).
        For a general relation \(R\) on \(X\), it is not always true that the quotients \(X \big/R\) and \(X\big/d_R\) are the same.
        Even when they coincide as sets, in general one can only guarantee that the topology on \(X\big/d_R\) is coarser, i.e., \(\mathrm{id}\colon X\big/R \to X\big/d_R\) is continuous \cite{burago2022course}.
    }.
    \par We need to show that \( \mathrm{id}\colon (K,d_\sim) \to (K, d_0 ) \) is an isometry.
    Due to Lemma \ref{lem:sheet_isometry}, we already know that the metrics coincide for any two points in the same copy of \(\mathcal{D}\), so it is enough to consider \(\widetilde x,\widetilde y\) lying in distinct connected components of \(X\).
    Let \(x, y\) be the corresponding projected points onto \(\mathcal{D}\), and let \([x,y]_\mathcal{L}\) be the geodesic of \( (\mathcal{D}, d_\mathcal{L}) \) joining them.
    There are two cases to analyse.
    \par \emph{Case 1: \([x,y]_\mathcal{L}\) meets the boundary \(\partial\mathcal{D}\).} In this situation we have \(d_\sim(\widetilde x,\widetilde y) = l_E([x,y]_\mathcal{L}) = d_\mathcal{L}(x,y)\).
    Indeed, given any point \( \widetilde m \) in the boundary \(\partial\mathcal{D}\), we have
    \[
        d_\mathcal{L}(\widetilde x, \widetilde m) + d_\mathcal{L}(\widetilde m, \widetilde y) =  d_\mathcal{L}(x, m) + d_\mathcal{L}(m, y) \geq d_\mathcal{L}(x,y),
    \]
    with equality holding exactly when \( m \) is a point in the segment \([x,y]_\mathcal{L}\).
    Thus the infimum \( d_\sim (\widetilde x, \widetilde y) = d_\mathcal{L}(x, y)\) is achieved by the assumption that \([x,y]_\mathcal{L} \cap \partial\mathcal{D} \neq \emptyset\).
    \par On the other hand, given any \( m \in [x,y]_\mathcal{L} \cap \partial\mathcal{D} \), we may lift the geodesic segments \( [x, m]_\mathcal{L}\) and \( [m, y]_\mathcal{L}\) isometrically to geodesics \( [\widetilde x, m]_0\) and \( [m, \widetilde y]_0 \) of \(K_0\).
    Hence the path of minimal length \( [\widetilde x, \widetilde y]_0 \) has length at most \(l_0([\widetilde x, m]_0) + l_0([m, \widetilde y]_0)\).
    But the projection map is non-expanding, that is, \(d_\mathcal{L}(x,y) \leq d_0(\widetilde x, \widetilde y )\), thus \( d_0(\widetilde x, \widetilde y) = d_\mathcal{L}(x, y) = d_\sim (\widetilde x, \widetilde y)\), as we wanted.
    \par \emph{Case 2: \([x,y]_\mathcal{L} \cap \partial\mathcal{D} = \emptyset \).} If the intersection is empty, then the geodesic \([x,y]_\mathcal{L}\) is a straight line contained entirely in \( \mathrm{int}(\mathcal{D})\). 
    Then no \(d_\mathcal{L}\)-geodesic from \(x\) to \(y\) meets \(\partial\mathcal{D}\), and we argue directly.
    For any \(m \in \partial\mathcal{D}\), lifting \([x, m]_\mathcal{L}\) to the sheet of \(\widetilde x\) and \([m, y]_\mathcal{L}\) to the sheet of \(\widetilde y\) (Lemma~\ref{lem:sheet_isometry}) gives a path in \(K\) from \(\widetilde x\) to \(\widetilde y\) through the seam point over \(m\), of \(l_0\)-length \(d_\mathcal{L}(x, m) + d_\mathcal{L}(m, y)\); taking the infimum over \(m\) gives \(d_0(\widetilde x, \widetilde y) \leq d_\sim(\widetilde x, \widetilde y)\).
    Conversely, every rectifiable path \(\gamma\) in \(K\) from \(\widetilde x\) to \(\widetilde y\) crosses the seam at some point over some \(m \in \partial\mathcal{D}\), so that, by Theorem~\ref{thm:d_0_is_a_metric}, \(l_0(\gamma) = l_E(p \circ \gamma) \geq d_\mathcal{L}(x, m) + d_\mathcal{L}(m, y) \geq d_\sim(\widetilde x, \widetilde y)\); hence \(d_0(\widetilde x, \widetilde y) \geq d_\sim(\widetilde x, \widetilde y)\), which concludes 
    (the same two inequalities also prove Case 1. In Case 2 the minimising paths are the billiard paths \([x, m]_\mathcal{L} \ast [m, y]_\mathcal{L}\) through a point \(m\) where the entry and exit angles agree, or paths through a tangency point). 
\end{proof}

\subsubsection{Comparison of the lengths for \(d_0\) and \(d_\epsilon\), \(\epsilon \in [0,1]\).}

The goal of this part is to show certain estimates on the lengths for \(\{d_\epsilon\}_{\epsilon \in [0,1]}\) which will be useful in the following (see e.g. Lemma~\ref{lem:uniform_hyperbolicity}). 
We start by deriving a boundary-layer normal form. 
\begin{lem}\label{lem:normal_form}
    There exist \(0 < \delta_1 < \tau_\mathrm{min}/2\) and \(C_1 > 0\) such that, for every scatterer \(\Gamma_i\), there are functions \(z_i \in C^r(\Gamma_i)\) and \(R_i \in C^{r-2}(\Gamma_i \times (0, \delta_1])\) with the following property.
    Given a connected component \(S\) of \(K^\circ\) and its height function \(\varphi_S \coloneqq z \circ (p\rvert_S)^{-1}\), if \(x\) lies in the collar \(\{d_i(x) \coloneqq d(x, \mathcal{O}_i) \leq \delta_1\}\), then
    \[
        \varphi_S(x) = z_i(u(x)) + s_{S,i}\, R_i(u(x), d_i(x)), \qquad s_{S,i} \in \{-1, +1\},
    \]
    where \(u(x) \in \Gamma_i\) is the point of \(\Gamma_i\) nearest to \(x\), and \(R_i\) satisfies, uniformly in \(u\) and \(0 < d \leq \delta_1\),
    \[
        R_i(u, 0) = 0, \quad 0 < \partial_d R_i(u, d) \leq \frac{C_1}{\sqrt{d}}, \quad |\partial_u R_i(u, d)| \leq C_1\sqrt{d}, \quad R_i(u, d) \leq C_1\sqrt{d}.
    \]
    Moreover, away from these collars the functions \(\varphi_S\) have uniformly bounded gradient.
\end{lem}
\begin{proof}
    Fix a boundary point \(x_0 \in B(K)\), projecting to \(u_0 \in \Gamma_i\).
    Let \(\sigma\) be the normal section obtained by intersecting \(K\) with the affine plane \(x_0 + (\mathbb{R}\partial_z \oplus T^\perp_{x_0}K)\), and set \(\zeta \coloneqq z - z_i\), where \(z_i \coloneqq z \circ (p\rvert_{B(K)})^{-1} \in C^r(\Gamma_i)\), where the inverse branch is chosen according to the sheet \(S\).
    We write \(\sigma = \sigma(n(\zeta), \zeta)\), where \(n\) is the inward normal coordinate to \(\Gamma_i\).
    By the computation of the kernel of \(dp\) along \(B(K)\) (see the discussion after (K2)) \(n'(0) = 0\) and~(K2) gives \(n''(0) \neq 0\); by Taylor, we then have
    \[
        n = \tfrac{a(u_0)}{2}\,\zeta^2 + O(\zeta^3), \qquad a(u_0) > 0,
    \]
    and, by compactness of \(B(K)\), \(a(u_0) \geq a_0 > 0\) uniformly.
    Solving for \(\zeta\), we get
    \[
        \zeta = \pm\sqrt{\tfrac{2n}{a(u_0)}}\,\bigl(1 + O(\sqrt{n})\bigr),
    \]
    and the expansion is uniform along \(B(K)\).
    Shrinking the tubular neighbourhood so that the collars are disjoint and of width \(< \tau_\mathrm{min}/2\), and using that the nearest-point projection \(u\) onto a strictly convex obstacle is well-defined and \(1\)-Lipschitz with \(n = d_i\), we obtain
    \[
        \varphi_S(x) = z_i(u) + s_{S,i}\bigl(\psi_i(u)\sqrt{d_i} + d_i\, E_i(u, d_i)\bigr), \qquad \psi_i \coloneqq \sqrt{\tfrac{2}{a}},
    \]
    with \(E_i\) bounded.
    Setting \(R_i(u,d) \coloneqq \psi_i(u)\sqrt{d} + d\,E_i(u,d)\), all the stated estimates follow (for the derivative bound, note that \(E_i\) is a smooth function of \((u, \sqrt{d})\), so \(\partial_d(d\,E_i)\) is bounded while \(\partial_d(\psi_i\sqrt d) = \psi_i/(2\sqrt d)\)), after possibly shrinking \(\delta_1\) and enlarging \(C_1\).
    Away from the collars, \(K^\circ\) is a \(C^r\) graph over \(\mathcal{D}\) by~(K1), so \(\nabla\varphi_S\) is uniformly bounded.
\end{proof}

Let us now study the variation of \(\varphi_S\) along a chord:
\begin{lem}
\label{lem:chord_variation}
    There is \(C_2 = C_2(K)>0\) such that for every straight segment \(J \subset \mathcal{D}\) of Euclidean length \(\leq \tau_\mathrm{max}\) and every sheet \(S\), we have \(\mathrm{Var}(\varphi_S\rvert_J) \leq C_2\).
\end{lem}
\begin{proof}
    Away from the collars, Lemma~\ref{lem:normal_form} gives a uniform bound \(\mathrm{Var}(\varphi_S\rvert_{J \cap \{d_i > \delta_1\}}) \leq M|J| \leq M\tau_\mathrm{max}\); it remains to bound the contribution of a single collar visit.
    Lift to \(\mathbb{R}^2\) and let \(\widetilde{\mathcal{O}}\) be a \(\mathbb{Z}^2\)-translate of \(\mathcal{O}_i\) whose \(\delta_1\)-collar meets \(J\); since \(|J| \leq \tau_\mathrm{max}\), the number of such translates is bounded by \(N = N(K)\).
    On the collar write \(d(s) = \operatorname{dist}(J(s), \widetilde{\mathcal{O}})\); as \(\widetilde{\mathcal{O}}\) is convex and \(J\) is straight, \(d\) is convex, hence has at most two monotone pieces, so \(\mathrm{Var}(\sqrt{d}\,\rvert_J) \leq 2\sqrt{\delta_1}\).
    Writing \(\varphi_S = z_i(u) + s_{S,i}R_i(u, d_i)\) on the collar, for two points \((u_1, d_1), (u_2, d_2)\) we split
    \begin{align*}
        |\varphi_S(x_2) - \varphi_S(x_1)| & \leq |R_i(u_2, d_2) - R_i(u_2, d_1)| \\ &+ \bigl|z_i(u_2) - z_i(u_1) + s_{S,i}(R_i(u_2, d_1) - R_i(u_1, d_1))\bigr|.
    \end{align*}
    Since \(0 < \partial_d R_i \leq C_1/\sqrt{d}\), the first term is bounded by \(C_1\,|\sqrt{d_2} - \sqrt{d_1}|\) uniformly in \(u_2\); summing over a partition of \(J\) and using that \(\sqrt{d}\) has at most two monotone pieces, the total contribution is at most \(C_1\,\mathrm{Var}(\sqrt{d}\,\rvert_J) \leq 2C_1\sqrt{\delta_1}\).
    For the second, \(|\partial_u R_i| \leq C_1\sqrt{\delta_1}\) and \(z_i \in C^1\), and \(u\) is \(1\)-Lipschitz, so this term is bounded by \((\mathrm{Lip}(z_i) + C_1\sqrt{\delta_1})\tau_\mathrm{max}\).
    We conclude by summing over the at most \(N\) visits and the two  regimes (collar/exterior).
\end{proof}

In the next lemma, we study the vertical variation of \(d_0\)-geodesics. 
\begin{lem}
\label{lem:vertical_variation}
    There is \(C_3 = C_3(K)>0\) such that every unit-speed \(d_0\)-geodesic \(\gamma\colon [0, L] \to K\) satisfies \(\mathrm{Var}(z \circ \gamma) \leq C_3(L + 1)\).
\end{lem}
\begin{proof}
    By Theorem~\ref{thm:K_0_as_gluing_metric_space}, \(\gamma\) decomposes into free arcs (maximal subarcs in a single sheet, projecting to straight chords of \(\mathcal{D}\)) and sliding arcs (maximal subarcs in \(B(K)\)), separated by isolated transversal crossings.
    Each free arc projects to a straight segment of \(\mathcal{D}\); since a tangential contact does not change sheet, it is a chord of Euclidean length at most \(\tau_\mathrm{max}\) or the union of two such chords through a tangency point. In both cases Lemma~\ref{lem:chord_variation} bounds its vertical variation by \(2C_2\).
    Along a sliding arc, \(z\) is \(C^r\) on the boundary curve, so \(z \circ \gamma\) has derivative bounded by some \(C_4\), contributing at most \(C_4 \times (\text{its } d_0\text{-length})\).
    By finite horizon every complete free flight has length at least \(\tau_\mathrm{min}\), so among the free arcs meeting \(\gamma([0,L])\) all but the first and last have \(d_0\)-length at least \(\tau_\mathrm{min}\); their number is less than \(L/\tau_\mathrm{min} + 2\).
    The total \(d_0\)-length of sliding arcs is at most \(L\).
    Hence
    \[
        \mathrm{Var}(z \circ \gamma) \leq 2C_2\Bigl(\tfrac{L}{\tau_\mathrm{min}} + 2\Bigr) + C_4 L \leq C_3(L + 1).\qedhere
    \]
\end{proof}

We can finally compare the lengths for \(d_0\) and \(d_\epsilon\), \(\epsilon \in (0,1]\):
\begin{thm}
\label{thm:length_comparison}
    There is \(C = C(K)>0\) such that, for every \(g_1\)-rectifiable \(\gamma\colon [a,b] \to K\) parametrised by arc length and every \(\epsilon \in (0,1]\),
    \[
        l_0(\gamma) \leq l_\epsilon(\gamma) \leq l_0(\gamma) + C\,\mathrm{Var}(z \circ \gamma)\,\epsilon.
    \]
    Therefore, for all \(\epsilon \in (0,1]\) and all \(x, y \in K\),
    \[
        d_0(x,y) \leq d_\epsilon(x,y) \leq (1 + C\epsilon)\,d_0(x,y) + C\epsilon,
    \]
    and \(d_\epsilon \downarrow d_0\) pointwise and monotonically.
    Moreover, if \(\gamma\) is a unit-speed \(d_\epsilon\)-geodesic on \(\widetilde{K}_\epsilon\), then:
    \begin{enumerate}
        \item[(a)] \(\ \dfrac{|t-s| - C\epsilon}{1 + C\epsilon} \leq d_0(\gamma(t), \gamma(s)) \leq |t-s|\); in particular \(\gamma\) is a \((1+C\epsilon, C\epsilon)\)-quasi-geodesic of \((\widetilde{K}, d_0)\);
        \item[(b)] \(\ |t-s| \leq d_1(\gamma(t), \gamma(s)) \leq (1+C)|t-s| + C\); in particular \(\gamma\) is a \((1+C, C)\)-quasi-geodesic of \((\widetilde{K}, d_1)\).
    \end{enumerate}
\end{thm}
\begin{proof}
    The lower bound follows from \(g_\epsilon \geq g_0\).
    Writing \(v_H(t) = |dp(\gamma'(t))|\), \(v_V(t) = |dz(\gamma'(t))|\), we have \(l_\epsilon(\gamma) = \int_a^b \sqrt{v_H^2 + \epsilon^2 v_V^2}\,dt\).
    Noting that \(\sqrt{v_H^2 + \epsilon^2 v_V^2} \leq v_H + \epsilon v_V\), and since \(z \circ \gamma\) is absolutely continuous, we have
    \[
        l_\epsilon(\gamma) \leq l_0(\gamma) + \epsilon\int_a^b v_V\,dt = l_0(\gamma) + \epsilon\,\mathrm{Var}(z \circ \gamma).
    \]
    Taking \(\gamma = [x,y]_0\), and by Lemma~\ref{lem:vertical_variation}, we obtain the distance bounds.
    On \(\widetilde{K}\) the height \(z_{\widetilde{K}} \coloneqq z \circ \Pi\) (with \(\Pi\colon \widetilde{K} \to K\) the covering) is \(\pi_1(K)\)-invariant and hence bounded, so the same estimates hold on each \((\widetilde{K}_\epsilon, d_\epsilon)\).
    Since \(\widetilde{K}_\epsilon\) has no conjugate points for every \(\epsilon \in (0,1]\) (by (K3) and Klingenberg's theorem, as in the proof of Lemma~\ref{lem:uniform_lower_bound_inj_radius}), geodesics are globally minimising, and evaluating the distance bounds on \((\gamma(t), \gamma(s))\) gives (a) (from \(d_\epsilon(\gamma(t),\gamma(s)) = |t-s|\)) and (b) (from \(d_1 \geq d_\epsilon\) and \(d_1 \leq (1+C)d_0 + C \leq (1+C)d_\epsilon + C\)).
\end{proof}

\subsection{Coarse geometry of \(K_{0}=(K, d_0)\)}
\label{section:coarse_geometry}
Recall that, by~(K0) and Theorem~\ref{thm:K_0_as_gluing_metric_space}, \( K_{0}=(K, d_0) \) consists of two copies of \( (\mathcal{D}, d_\mathcal{L}) \) glued along the boundary. \vspace{3pt}
\par Let \(\widetilde{K}_0\) be the universal covering of \(K\).
For simplicity, when \(\epsilon > 0\), we denote the lift of the metric \(d_\epsilon\) to \(\widetilde{K}_0\) by \(d_\epsilon\) as well, and the same for the lift of the Riemannian tensor \(g_\epsilon\).
We denote by \( \widetilde{K}_\epsilon \) the space \((\widetilde{K}, d_\epsilon)\).
The canonical projection \(\widetilde{K}_\epsilon \to K_\epsilon\) is then a smooth covering and a local isometry.

\par Let us recall the following notion:
\begin{dfn}
A metric space is called~\emph{non-positively curved}, or~\emph{locally CAT(0)}, if every point has a neighbourhood on which every geodesic triangle satisfies the CAT(0) inequality.
\end{dfn}
\begin{pps}\label{cor:k_0.is.hadamard}
    The metric space \(K_{0}= (K, d_0) \) is non-positively curved. 
    The universal covering  \( \widetilde{K}_{0} \)
    is a contractible Hadamard space, that is, a contractible and complete (globally) CAT(0) geodesic space.
\end{pps}
\begin{proof}
    Any point of \(K_0 \setminus B(K)\) has a neighbourhood isometric to an open subset of \(\mathbb{R}^2\), where the CAT(0) inequality holds with equality.
    Let then \(p \in B(K)\).
    A neighbourhood of \(p\) in \(K_0\) is isometric to the double, across its boundary curve, of a plane region
    \[
        U = \{(x, y) \in B_r(0) : y \leq -g(x)\},
    \]
    where \(g\) is convex and \(C^2\) (the local graph of the scatterer); indeed the two sheets are two copies of the table, glued along the wall.
    \par Let us first pretend that \(g\) is piecewise affine and convex.
    The double of \(U\) is then a flat surface with isolated conical singularities, one for each vertex of the graph of \(g\); at a vertex with exterior turning angle \(\tau \geq 0\) (by convexity), the interior angle of \(U\) is \(\pi + \tau\), so the cone angle of the double is \(2\pi + 2\tau \geq 2\pi\).
    A Euclidean cone of angle \(\geq 2\pi\) is CAT(0), and a locally finite flat surface with cone angles \(\geq 2\pi\) is locally CAT(0) (see \cite[Chapter II.5]{bridson2011metric}); the smooth points of the seam are flat in the double, the two adjacent angles being \(\pi\) and \(\pi\) (see Figure~\ref{fig:comparison_triangles_table}).
    \par For a smooth convex function \(g\) describing a piece of boundary, we approximate \(g\) uniformly by convex piecewise affine functions \(g_n\) (e.g.\ the maxima of finitely many tangent lines).
    The doubles of the corresponding regions \(U_n\) converge to the double of \(U\): parametrising both boundaries by arc length and extending sheetwise, one obtains homeomorphisms which distort lengths of curves, hence distances, by a factor \(1 + o(1)\).
    The CAT(0) comparison inequality is a closed condition under such uniform convergence of metrics~\cite[II.3.10]{bridson2011metric}, so the double of \(U\) is CAT(0) for small \(r\), and \(K_0\) is non-positively curved.
    Alternatively, it follows from Reshetnyak's gluing theorem for two-dimensional manifolds of bounded curvature: two flat pieces are glued along a curve whose geodesic curvatures relative to the two sides sum to \(-2\mathcal{K} \leq 0\); see \cite{reshetnyak_two-dimensional_1993}.
    \par The Cartan-Hadamard theorem for complete locally CAT(0) spaces (see e.g.~\cite[Theorem II.4.1]{bridson2011metric}) then gives the assertion about the universal covering.
    \par Finally, we note that each closed sheet of \(K_0\) is locally convex: through a smooth seam point the seam is a local geodesic of \(K_0\) (both adjacent angles equal \(\pi\)), and a \(d_0\)-geodesic between two nearby points of one closed sheet cannot enter the other sheet, since it would have to cross the seam twice, contradicting local uniqueness of geodesics between its crossing points.
    Consequently \((\mathcal{D}, d_\mathcal{L})\), which is isometric to a closed sheet, is itself locally CAT(0), being a locally convex subspace of the locally CAT(0) space \(K_0\).
\end{proof}
\begin{figure}[t]
    \centering
    \includegraphics[width=.8\textwidth]{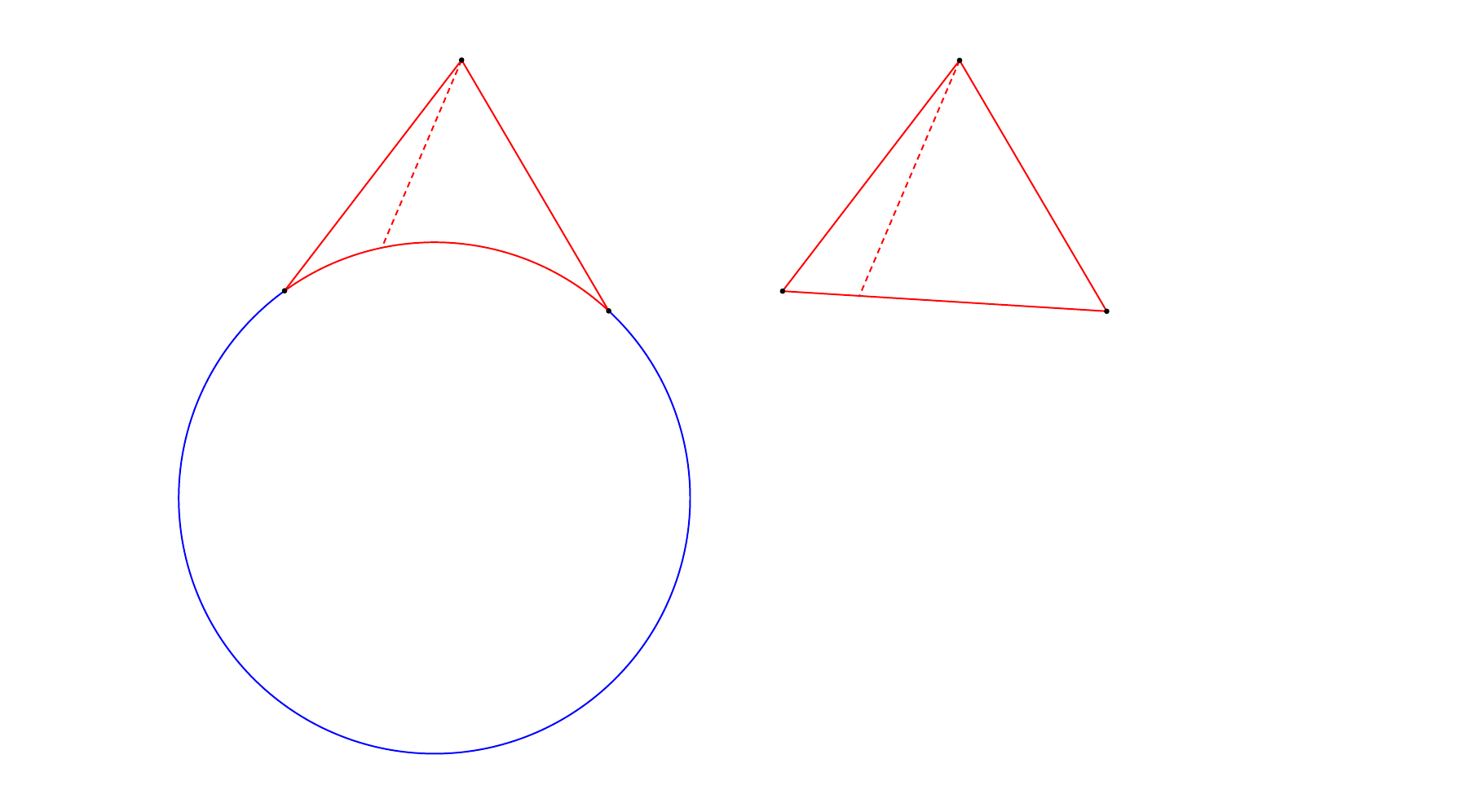}
    \caption{The local model at the seam: the double of the region below a convex graph.
        In the piecewise affine approximation, each vertex contributes a cone point of angle \(2\pi + 2\tau \geq 2\pi\).}
    \label{fig:comparison_triangles_table}
\end{figure}

Let us now recall some standard definitions on Gromov hyperbolicity. 
\begin{dfn}[Gromov hyperbolicity]\label{defi:gromov_hyperbolic}
    Let \((X, d)\) be a geodesic metric space and \(\delta \geq 0\).
    A geodesic triangle in \(X\) is~\emph{\(\delta\)-thin} if each of its sides is contained in the closed \(\delta\)-neighbourhood of the union of the other two.
    The space \(X\) is~\emph{\(\delta\)-hyperbolic} if all its geodesic triangles are \(\delta\)-thin; it is called~\emph{Gromov-hyperbolic} if it is \(\delta\)-hyperbolic for some \(\delta\).
\end{dfn}
Recall that hyperbolicity is invariant under quasi-isometries between geodesic spaces: if \(X\) is \(\delta\)-hyperbolic and \(f\colon X' \to X\) is a \((\lambda, c)\)-quasi-isometry from a geodesic space, then \(X'\) is \(\delta'\)-hyperbolic with \(\delta'\) depending only on \((\lambda, c, \delta)\)~\cite[Theorem III.H.1.9]{bridson2011metric}.
    In the following, we use two sources of hyperbolicity: 
    \begin{enumerate}
    \item[$\bullet$] a proper cocompact \(\mathrm{CAT}(0)\) space containing no isometrically embedded flat plane is Gromov-hyperbolic (Flat Plane Theorem,~\cite[Theorem III.H.1.5]{bridson2011metric}); 
    \item[$\bullet$] the universal cover of a closed Riemannian manifold with Anosov geodesic flow is Gromov-hyperbolic~\cite[Appendix B]{guillarmou_geodesic_2022}.
    \end{enumerate}

\begin{dfn}[Gromov boundary]\label{defi:gromov_boundary}
    Let \((X, d)\) be a proper geodesic \(\delta\)-hyperbolic space.
    Two geodesic rays \(r, r'\colon [0, \infty) \to X\) are \emph{asymptotic} if \(\sup_{t \geq 0} d(r(t), r'(t)) < \infty\).
    The~\emph{boundary at infinity} (or~\emph{Gromov boundary}) \(\partial_\infty X\) is the set of asymptotic classes of geodesic rays; the class of \(r\) is denoted \(r(\infty)\), and for a bi-infinite geodesic \(\gamma\) we write \(\gamma(\pm\infty)\) for the classes of its two half-rays.
    Finally we denote by \(\partial^2_\infty X \coloneqq (\partial_\infty X \times \partial_\infty X) \setminus \Delta\) the space of pairs of distinct boundary points.
\end{dfn}
Let \((X, d)\) be a proper geodesic \(\delta\)-hyperbolic space. 
    For every \(x \in X\) and \(\xi \in \partial_\infty X\) there is a geodesic ray issued from \(x\) in the class \(\xi\)~\cite[Lemma III.H.3.1]{bridson2011metric}; when \(X\) is \(\mathrm{CAT}(0)\) this ray is unique.
    The union \(\overline{X} \coloneqq X \cup \partial_\infty X\) carries a natural topology in which a sequence \(x_n \in X\) converges to \(\xi \in \partial_\infty X\) if and only if, for one (equivalently, every) base point \(x\), the geodesic segments \([x, x_n]\) converge uniformly on compact sets to a ray issued from \(x\) in the class \(\xi\); for \(\mathrm{CAT}(0)\) spaces this is the cone topology of~\cite[II.8]{bridson2011metric}.
    With this topology \(\overline{X}\) is compact, \(X\) is open and dense, and \(\partial_\infty X\) is closed~\cite[Proposition III.H.3.7]{bridson2011metric}; moreover every quasi-isometry between proper geodesic hyperbolic spaces extends to a homeomorphism of their boundaries~\cite[Theorem III.H.3.9]{bridson2011metric}.
    
Finally, we recall the notion of Gromov product and some of its basic properties.
\begin{dfn}[Gromov products]\label{defi:gromov_product}
    Let \((X, d)\) be a metric space.
    The \emph{Gromov product} of \(y, z \in X\) at \(x \in X\) is    \[        (y \mid z)_x \coloneqq \tfrac12\bigl[d(x, y) + d(x, z) - d(y, z)\bigr] \geq 0.    \]    
\end{dfn}
Below, we will use the following standard facts:
    \begin{enumerate}[label=(\roman*)]
        \item \((y \mid z)_x \leq \min\{d(x,y), d(x,z)\}\), with \((y \mid z)_x = 0\) if and only if, in a geodesic space, \(x\) lies on a geodesic from \(y\) to \(z\).
        \item Isometry invariance: \((\psi y \mid \psi z)_{\psi x} = (y \mid z)_x\) for every isometry \(\psi\).
        \item In any geodesic space \((y \mid z)_x \leq d\bigl(x, [y, z]\bigr)\) for every geodesic segment \([y,z]\), and in a \(\delta\)-hyperbolic geodesic space the product computes this distance up to an additive constant:    \[        (y \mid z)_x \;\leq\; d\bigl(x, [y, z]\bigr) \;\leq\; (y \mid z)_x + 4\delta.   \]
        \item Let \(X\) be proper, geodesic and \(\delta\)-hyperbolic, with boundary \(\partial_\infty X\) as in Definition~\ref{defi:gromov_boundary}.
              The product controls convergence at infinity: a sequence \((x_n)\) converges to a point of \(\partial_\infty X\) if and only if \((x_n \mid x_m)_x \to \infty\) as \(n, m \to \infty\), and two such sequences have the same limit iff \((x_n \mid y_n)_x \to \infty\)~\cite[Ch.~7]{coornaert_geometrie_1990}.
              Thus, for \(\xi \neq \eta \in \partial_\infty X\) and  sequences \(y_n \to \xi\), \(z_n \to \eta\) in \(\overline{X}\), the products \((y_n \mid z_n)_x\) remain bounded, and any two limit values differ by at most \(2\delta\).
    \end{enumerate}\quad

\par In the following lemma, we show that the spaces \(\widetilde{K}_\epsilon\), \(\epsilon \in [0, 1]\), are uniformly Gromov hyperbolic. 
\begin{lem}
\label{lem:uniform_hyperbolicity}
    There exists \(\delta = \delta(K) > 0\) such that the spaces \(\widetilde{K}_\epsilon\), \(\epsilon \in [0, 1]\), are \(\delta\)-hyperbolic.
\end{lem}
\begin{proof}
    Since \(g_\epsilon \leq g_{\epsilon'}\) pointwise for \(\epsilon \leq \epsilon'\), the distances are monotone: \(d_0 \leq d_\epsilon \leq d_1\) on \(\widetilde{K}\) for every \(\epsilon \in [0,1]\).
    By Theorem~\ref{thm:length_comparison} at \(\epsilon = 1\), \(d_1 \leq (1+C)\,d_0 + C\), hence
    \[
        d_\epsilon \leq d_1 \leq (1+C)\,d_\epsilon + C,\quad \forall\, \epsilon \in [0,1],
    \]
    so that the identity map \((\widetilde{K}, d_\epsilon) \to (\widetilde{K}, d_1)\) is a \((1+C, C)\)-quasi-isometry, with constants independent of \(\epsilon\).
    All the spaces involved are geodesic: the \((\widetilde{K}, d_\epsilon)\), \(\epsilon \in (0,1]\), are complete Riemannian manifolds, while for \((\widetilde{K}, d_0)\) it follows from Theorem~\ref{thm:d_0_is_a_metric}.
    Finally \((\widetilde{K}, d_1)\) is Gromov-hyperbolic, the geodesic flow of \(g_1\) being Anosov by (K3) (Definition~\ref{defi:gromov_hyperbolic}).
    Since hyperbolicity is stable under quasi-isometries between geodesic spaces, with a constant depending only on the quasi-isometry constants and on \(\delta(g_1)\) (Definition~\ref{defi:gromov_hyperbolic}), the spaces \(\widetilde{K}_\epsilon\) are  \(\delta\)-hyperbolic, for some \(\delta = \delta(K)>0\).
\end{proof}

\par The space of (parametrised) geodesics of \(\widetilde{K}_\epsilon\), \(\epsilon \in [0,1]\), is the subset \( \widehat{\mathcal{G}}_\epsilon \subset C(\mathbb{R}, \widetilde{K})\) consisting of unit speed parametrisations\footnote{
    In a general metric space like \(\widetilde{K}_0\), a parametrisation \(\gamma\) is unit-speed if \(d(\gamma(t), \gamma(s)) = |t-s|\) for every \(t,s \in \mathbb{R}\).
} of geodesics \(\gamma\colon \mathbb{R} \to \widetilde{K}_\epsilon\).
We define the equivalence relation \(\sim\) on \(\widehat{\mathcal{G}}_\epsilon\)  by declaring \(\gamma_1 \sim \gamma_2\) if and only if there is \(T \in \mathbb{R}\) such that \(\gamma_1(\cdot) = \gamma_2(\cdot + T)\), and let \(\mathcal{G}_\epsilon\) be the space of unparametrised geodesics
\[
    \mathcal{G}_\epsilon \coloneqq \widehat{\mathcal{G}}_\epsilon / \sim = \{ [\gamma]; d_\epsilon(\gamma(t), \gamma(s)) = |t-s|,\, t,s \in \mathbb{R} \},
\]
equipped with the compact-open topology inherited from \( C(\mathbb{R}, \widetilde{K}) \).
In the Riemannian cases (i.e., for \(\epsilon > 0\)), there is an identification between \(\widehat{\mathcal{G}}_\epsilon\) and \(T^1\widetilde{K}_\epsilon\) given by
\[
    \gamma \mapsto (\gamma(0), \dot\gamma(0)),
\]
and in this case the space \(\mathcal{G}_\epsilon\) is simply the space of orbits of the geodesic flow acting on \(T^1\widetilde{K}_\epsilon\).
While there is no unit tangent bundle of \(\widetilde{K}_0\) and therefore no such identification, we can still define an abstract geodesic flow on \(\widetilde{K}_0\) as the \(\mathbb{R}\)-action on \(\widehat{\mathcal{G}}_0\) by precomposition with translations:
\begin{align*}
    \mathbf{g}_0\colon \widehat{\mathcal{G}}_0 \times \mathbb{R} & \to \widehat{\mathcal{G}}_0 \\
    (\gamma, t)                                                  & \mapsto \gamma(t + \cdot).
\end{align*}

\par By Lemma~\ref{lem:uniform_hyperbolicity}, every \((\widetilde{K}, d_\epsilon)\), \(\epsilon \in [0,1]\), is a proper geodesic Gromov-hyperbolic space.
Therefore, it has a boundary at infinity \(\partial_\infty \widetilde{K}_\epsilon\) as in Definition~\ref{defi:gromov_boundary}, which for the \(\mathrm{CAT}(0)\) spaces is the boundary of asymptotic geodesic rays with the cone topology, and \(\widetilde{K}_\epsilon \cup \partial_\infty \widetilde{K}_\epsilon\) is compact.
Since \(\widetilde{K}_\epsilon\) is proper, it is a visibility space~\cite[Lemma III.H.3.2]{bridson2011metric}: any two distinct points \(\alpha, \omega \in \partial_\infty \widetilde{K}_\epsilon\) are joined by a geodesic, which is unique for \(\epsilon > 0\) since a Hadamard manifold with Anosov geodesic flow contains no flat strips (Flat Strip Theorem, \cite[II.2.13]{bridson2011metric}); the same argument is carried out for \(\epsilon = 0\) in Lemma~\ref{lem:no_flat_plane} below.
Therefore, for any \(\epsilon \in [0,1]\), we have
\[
    \mathcal{G}_\epsilon \approx \partial^2_\infty\widetilde{K}_\epsilon = (\partial_\infty\widetilde{K}_\epsilon \times \partial_\infty\widetilde{K}_\epsilon) \setminus \Delta.
\]

\subsubsection{The space \(\mathcal{G}_0\) and the abstract geodesic flow.}
\par Based on the relation between geodesics of \(\widetilde{K}_0, \ K_0\) and \((\mathcal{D}, d_\mathcal{L})\), we classify geodesics in terms of their projections on the table \(\mathcal{D}\). Let us recall that in the gluing construction, the seam is the subset \(B(K)\subset K_0\) corresponding to the boundary points where the identification happens. In other words, \(B(K)\) consists of closed curves on \(K\) whose image under the projection \(p\colon E \to \mathbb{T}^2\) are the scatterers on \(\mathcal{D}\).

First, we associate to a geodesic a set of collision points; these are the points at which the corresponding trajectory segment on \(\mathcal{D}\) bounces off a scatterer.
\begin{dfn}
    For a geodesic \( \gamma \) in \( K_0 \), the set of \emph{collision points} of \(\gamma\) is defined as
    \[
        B(\gamma) \coloneqq \partial(\gamma(\mathbb{R}) \cap B(K)).
    \]
\end{dfn}
The intersection of a segment of \(\gamma\) with \(B(K)\) could be a single point (corresponding under \(p\) to a bounce of the associated trajectory on \(\mathcal{D}\)) or an entire segment of a scatterer on the table, in which case the geodesic does not correspond to any billiard trajectory.
This motivates the following terminology.
\begin{dfn}[Geodesic trichotomy]\label{geod_tricho}
    We will say a collision point \( x \in K_0\) is~\emph{degenerate} if the connected component of \( \gamma(\mathbb{R}) \cap B(K) \) to which it belongs has positive length, i.e., is not an isolated point.
    \par A non-degenerate point is said to be~\emph{regular} or~\emph{grazing} if the corresponding collision of \( p(\gamma) \) is regular or grazing, respectively.
    \par Similarly, given a geodesic \(\widetilde\gamma\) on \(\widetilde{K}_0\), let \(\gamma\) be its projection on \(K_0\).
    Then the set \(B(\widetilde\gamma)\) of~\emph{collision points} of \(\widetilde\gamma\) is defined as the pre-image of \(B(\gamma)\) under the canonical projection.
    A collision point of \(\widetilde\gamma\) will be called regular, grazing or degenerate when its projection on \(K_0\) has the corresponding property.
    \par A geodesic is~\emph{degenerate} if it contains a degenerate collision point.
    A non-degenerate geodesic is~\emph{grazing} if it contains a grazing point.
    Otherwise it is~\emph{regular}.
    Let
    \[
        \mathcal{G}^{\mathrm{reg}} \coloneqq \{ [\gamma] \in \mathcal{G}_0; \gamma \text{ is regular} \},
    \]
    and write \( \mathcal{G}^{\mathrm{graz}} \) and \( \mathcal{G}^{\mathrm{deg}} \) for the sets of grazing and of degenerate geodesics, respectively.
    We denote by \(\mathcal{G}(\mathcal{D})\) the closure of \(\mathcal{G}^{\mathrm{reg}}\) in \(\mathcal{G}_0\) with respect to the compact-open topology.
    We call \(\mathcal{G}(\mathcal{D})\) the space of~\emph{admissible} geodesics.
\end{dfn}
The non-degenerate geodesics are those which project to genuine billiard trajectories on \(\mathcal{D}\), while the degenerate geodesics project to curves which are tangent to a scatterer \(\Gamma_i\) for some positive time interval \([t, t+\epsilon]\).
In particular, degenerate geodesics cannot be the limit of regular ones, so that admissible geodesics are non-degenerate, as we make explicit in the following statement.

\begin{lem}\label{lem:topology.geodesics}
    The set \(\mathcal{G}^{\mathrm{deg}}\) of degenerate geodesics is open in \(\mathcal{G}_0\); in particular, every admissible geodesic is non-degenerate.
    Hence
    \[
        \mathcal{G}^{\mathrm{reg}} \subseteq \mathcal{G}(\mathcal{D}) \subseteq \mathcal{G}^{\mathrm{nd}} \coloneqq \mathcal{G}^{\mathrm{reg}} \sqcup \mathcal{G}^{\mathrm{graz}},
    \]
    i.e., admissible geodesics are regular or grazing.
\end{lem}
\begin{proof}
    We start with a quantitative estimate on chords that hug a scatterer.
    Since the scatterers are compact and strictly convex, their boundary curves have curvature bounded below by some \(\mathcal{K}_{\mathrm{min}} > 0\).
    We claim there is a constant \(c_0 > 0\), depending only on \(\mathcal{D}\), such that any straight segment of length \(\ell\) lying entirely within distance \(\eta\) of a curve of curvature at least \(\mathcal{K}_{\mathrm{min}}\) satisfies \(\ell \leq c_0\sqrt{\eta}\).
    Indeed, let \(J\) be such a segment and \(u\) a point of the curve nearest to the midpoint of \(J\).
    If \(J\) is not nearly tangent to the curve at \(u\), it leaves the \(\eta\)-neighbourhood within length \(O(\eta)\) and the estimate is trivial.
    Assume therefore \(J\) is nearly tangent at \(u\).
    In the tangent frame at \(u\) the curve is the graph of a function \(h\) with \(h(0) = h'(0) = 0\) and \(h(s) \geq \tfrac{\mathcal{K}_{\mathrm{min}}}{2}s^2\bigl(1 + O(s)\bigr)\), while \(J\) is the graph of an affine function \(a\).
    Evaluating \(h - a\) at the two endpoints and the midpoint of \(J\), the affine part cancels in the second difference; since each of the three values is at most \(C\eta\), we obtain
    \[
        \tfrac{\mathcal{K}_{\mathrm{min}}}{4}\,\ell^2\bigl(1 + O(\ell)\bigr) \leq 4C\eta,
    \]
    which gives \(\ell \leq c_0\sqrt{\eta}\) for \(\eta\) small, hence for all \(\eta\) after enlarging \(c_0\).

    We use this estimate to keep non-degenerate geodesics away from sliding arcs.
    Fix \(s > 0\) and let \(\eta_0 = \eta_0(s)\) be small enough that \(c_0\sqrt{\eta_0} < \min(\tau_\mathrm{min}, s/4)\).
    We claim that no non-degenerate geodesic can remain within distance \(\eta < \eta_0\) of a sliding arc of length \(s\) throughout a parameter window of length \(s\).
    Indeed, on such a window a non-degenerate geodesic is a concatenation of straight chords, of which at most two meet the endpoints of the window while the remaining ones are complete free flights of length at least \(\tau_\mathrm{min}\) (due to our finite horizon assumption).
    If the window contains at least two collisions, one such complete chord is present; if it contains exactly one collision, the two truncated chords have total length \(s\), so one of them has length at least \(s/2\); and if it contains none, a single chord spans the whole window and has length at least \(s\).
    In every case the window carries a chord of length at least \(\min(\tau_\mathrm{min}, s/2)\).
    Since \(d_0\) dominates the Euclidean distance on each sheet, that chord projects to a straight table segment lying within \(\eta\) of the boundary arc \(\Gamma_i\) carrying the sliding arc.
    But then it would have length at most \(c_0\sqrt{\eta} < \min(\tau_\mathrm{min}, s/4)\), which is impossible.

    Now let \(\{\gamma_n\}_n\) be a sequence of non-degenerate geodesics, and suppose \(\gamma_n \to \gamma\) in the compact-open topology, with \(\gamma\) degenerate, i.e., \(\gamma\) slides along a seam arc over a window of positive length.
    Shrinking it we may take a sliding arc of length \(s\) over a window of length \(s\).
    For \(n\) large, \(\gamma_n\) is within \(\eta < \eta_0(s)\) of \(\gamma\) uniformly on that window, hence within \(\eta\) of the sliding arc throughout it, contradicting the previous paragraph.
    Therefore the non-degenerate geodesics form a closed set, \(\mathcal{G}^{\mathrm{deg}}\) is open, and since \(\mathcal{G}^{\mathrm{reg}}\) is contained in this closed set, \(\mathcal{G}(\mathcal{D}) = \overline{\mathcal{G}^{\mathrm{reg}}}\) consists of non-degenerate geodesics.

The inclusion \(\mathcal{G}^{\mathrm{reg}} \subseteq \mathcal{G}(\mathcal{D})\) holds by definition, and the trichotomy gives the last statement.
\end{proof}
\begin{rmk}\label{rem:branching}
    We note that, in \(K_0\), branching of geodesics occurs only at grazing collisions: at \((x,v) \in T^1K^\circ\) the geodesics \(\gamma_{\epsilon;v}\) converge locally uniformly, on collision-free windows, to a well-defined geodesic \(\gamma_{0;v}\), by Corollary~\ref{unique.directions.away.from.seam.}; moreover, transversal crossings have a unique continuation, by Lemma~\ref{pf:lem_directions}(iii).
\end{rmk}
The following is immediate from the above definitions.
\begin{rmk}\label{rem:sheet_involution}
    The metric space \(K_0\) carries a canonical isometric involution \(\iota\colon K_0 \to K_0\) exchanging the two sheets and fixing the seam \(B(K)\) pointwise.
    A non-degenerate geodesic changes sheet at every transversal collision, and may or may not change sheet at a grazing one; consequently, every billiard trajectory of \(\mathcal{D}\) with only transversal collisions admits exactly two non-degenerate lifts to \(K_0\), interchanged by \(\iota\), while a trajectory with \(k < \infty\) grazing collisions admits \(2^{k+1}\) non-degenerate lifts.
\end{rmk}
\begin{lem}\label{lem:two_to_one_correspondence}
    The projection induces a continuous surjection \(\mathcal{G}^{\mathrm{nd}} \to \Omega/\Psi\) from the space of non-degenerate geodesics onto the orbit space of the billiard flow.
    Over a trajectory with only transversal collisions it is exactly two-to-one, the fibre being an orbit of the involution \(\iota\); in particular it induces a bijection \(\mathcal{G}^{\mathrm{reg}}/\iota \to \Omega^{\mathrm{reg}}/\Psi\). Over a trajectory with \(k < \infty\) grazing collisions the fibre has \(2^{k+1}\) elements.
\end{lem}
\begin{proof}
A non-degenerate geodesic of \(K_0\) is a concatenation of chords, reflected at transversal collisions (Lemma~\ref{pf:lem_directions}(iii)) and continued straight at tangential ones, so it projects to a billiard trajectory; conversely, a billiard trajectory lifts to a non-degenerate geodesic once a starting sheet and, at each grazing collision, a continuation sheet have been chosen.
Continuity follows from the continuity of \(\pi_K\) on \(T^1K^\circ\).
\end{proof}

The following result will be useful later to derive that two Sinai billiards sharing the same enriched length spectrum are isometric:
\begin{pps}\label{cor:isometry-preserves-boundary}
    Let \(K_0 = \coprod_{i=1}^2 (\mathcal{D}, d_\mathcal{L})/\!\!\sim\) and \(K'_0 = \coprod_{i=1}^2 (\mathcal{D}', d'_{\mathcal{L}})/\!\!\sim\) be metric gluings of Sinai billiard tables.
    If \(F\colon (K_0, d_0) \to (K'_0, d'_0)\) is an isometry, then the tables \((\mathcal{D}, d_\mathcal{L})\) and \((\mathcal{D}', d'_\mathcal{L})\) are also isometric; moreover, there is a Euclidean isometry \(\phi\) of \(\mathbb{T}^2\) with \(\phi(\mathcal{D}) = \mathcal{D}'\).
\end{pps}
\begin{proof}
    An isometry preserves distances and consequently must map points where geodesic bifurcation happens to points with the same property\footnote{
        In other words, if for any direction there is a single geodesic emanating from \(x\), then the same must be true for \(F(x)\).}.
    As we have observed above, the points where geodesics bifurcate are exactly those in the seam, hence \(F\) must send the set \(B(K)\) exactly onto \(B(K')\).
    But \(K_0\setminus B(K)\) is the disjoint union of two open sheets, each isometric to \((\mathcal{D},d_\mathcal{L})\).
    Since \(F\) carries sheets to sheets, it induces on each one a distance-preserving bijection with the corresponding sheet in \(K'_0\), i.e., an isometry between \(\mathcal{D}\) and \(\mathcal{D}'\).
    Finally, an isometry \(f\colon (\mathcal{D}, d_\mathcal{L}) \to (\mathcal{D}', d'_\mathcal{L})\) restricts to a local isometry between the flat surfaces \(\mathrm{int}(\mathcal{D})\) and \(\mathrm{int}(\mathcal{D}')\) mapping boundary components to boundary components; thus, it lifts to a local isometry between the \(\mathbb{Z}^2\)-covers \(\mathbb{R}^2 \setminus \bigcup_i \widetilde{\mathcal{O}}_i\) and \(\mathbb{R}^2 \setminus \bigcup_i \widetilde{\mathcal{O}}'_i\), i.e., to a map between connected open subsets of the plane whose differential is everywhere orthogonal and locally constant, hence to the restriction of a Euclidean motion \(\widetilde\phi\) of \(\mathbb{R}^2\). Since \(\widetilde\phi\) is equivariant with respect to the deck translations, its linear part preserves the lattice \(\mathbb{Z}^2\), and \(\widetilde\phi\) descends to an isometry \(\phi\) of \(\mathbb{T}^2\) with \(\phi(\mathcal{D}) = \mathcal{D}'\).
\end{proof}

Let us lift \(\mathcal{G}(\mathcal{D})\) and consider the space \(\widehat{\mathcal{G}}(\mathcal{D})\) of~\emph{parametrised admissible geodesics}.
As noted above, in \(K_0\), branching of geodesics can occur only at grazing collisions.
Therefore, as in the Anosov case, at least at the set \(\mathcal{R} \coloneqq \pi_K^{-1}(\Omega^\mathrm{reg}) \subset T^1K\) there is an identification with a subset of geodesics \(\widehat{\mathcal{G}}_0\) given by
\[
    \mathcal{R} \ni (x,v) \mapsto \gamma_{0;v} \in \widehat{\mathcal{G}}^{\mathrm{reg}} \subset \widehat{\mathcal{G}}_0.
\]
Moreover, using this representation, the restriction of the abstract geodesic flow on \(\widehat{\mathcal{G}}_0\) to the set of regular geodesics has the more familiar form
\begin{equation}\label{eq:geodesic_flow_K_0}
    \begin{aligned}
        \mathbf{g}_0^t\colon \mathcal{R} & \to \mathcal{R}                                                                                                                              \\
        (x,v)                            & \mapsto \big(\gamma_{0;v}(t), \dot\gamma_{0;v}(t)\big) =  \lim_{\epsilon \to 0}\big(\gamma_{\epsilon;v}(t), \dot\gamma_{\epsilon;v}(t)\big),
    \end{aligned}
\end{equation}
which is well defined for \( t\in \mathbb{R} \).
Finally, Kourganoff's theorem~\cite{kourganoff_anosov_2016} implies that such flow maps are uniformly continuous on compact sets.
These flow maps do not extend to a flow on \(T^1\widetilde{K}\); indeed, at a vector tangent to the seam, three germs of geodesics share the same initial data (see e.g. Lemma~\ref{pf:lem_directions}(iv)), and a grazing geodesic may or may not change sheet at a tangential contact. Rather, writing \(\widehat{\mathcal{G}}^{\mathrm{nd}} \supseteq \widehat{\mathcal{G}}(\mathcal{D})\) for the parametrised non-degenerate geodesics, the map \(\widehat{\mathcal{G}}^{\mathrm{nd}} \to T^1\widetilde{K}\), \(\gamma \mapsto (\gamma(0), \dot\gamma(0^+))\), is a continuous surjection which is bijective over \(\mathcal{R}\), with inverse \((x,v) \mapsto \gamma_{0;v}\), and whose fibre over a vector tangent to \(B(\widetilde{K})\) consists of the non-degenerate geodesics with that initial germ.
We may then restate Kourganoff's result as
\begin{thm}[Kourganoff]\label{thm:kourganoff_on_steroids}
    Let \( \mathcal{D} \) be a Sinai billiard table and \( K \in \mathcal{K}(\mathcal{D}) \). 
    Then the abstract geodesic flow on \(\widehat{\mathcal{G}}_0\) preserves the set \(\widehat{\mathcal{G}}^{\mathrm{nd}}\) of non-degenerate geodesics and its closed subset \(\widehat{\mathcal{G}}(\mathcal{D})\) of admissible geodesics, and the map
    \[
        \pi_K\colon \widehat{\mathcal{G}}^{\mathrm{nd}} \to \Omega, \qquad \gamma \mapsto \Bigl(p(\gamma(0)), \frac{dp(\dot\gamma(0^+))}{|dp(\dot\gamma(0^+))|}\Bigr),
    \]
    (which agrees with~\eqref{eq:projection_flow_equivalence} on geodesics through \(T^1K^\circ\)) is a continuous surjection semi-conjugating the abstract flow to the billiard flow \( \Psi \):
    \[
        \pi_K \circ \mathbf{g}^t_0 = \Psi^t\circ\pi_K \quad \text{on } \widehat{\mathcal{G}}^{\mathrm{nd}},\ t \in \mathbb{R}.
    \]
    Over the trajectories with only transversal collisions, \(\pi_K\) is \(2\)-to-\(1\), its fibres being the orbits of the sheet involution of Remark~\ref{rem:sheet_involution}; over a trajectory with \(k < \infty\) grazing collisions its fibre has \(2^{k+1}\) elements.
\end{thm}

\subsubsection{The boundary \(\partial_\infty\widetilde{K}_0\).}

\par Since \( \widetilde{K}_0 \) is a Hadamard space, it has a well-defined circle at infinity \(\partial_\infty\widetilde{K}_0\), defined as the set of equivalence classes of asymptotic rays.
More precisely, \(\gamma, \gamma'\colon [0, \infty) \to \widetilde{K}_0\) are in the same class if \(\sup\{d_0(\gamma(t), \gamma'(t)); t \in [0, \infty)\} < \infty\).
We denote the class of the ray \(\gamma\) by \(\gamma(\infty)\).
More generally, if \(\gamma\colon \mathbb{R} \to \widetilde{K}_0\) is a geodesic, we set its endpoints at infinity to be
\begin{equation*}
    \gamma(\infty) \coloneqq \gamma^+(\infty),\quad
    \gamma(-\infty) \coloneqq \gamma^-(\infty),
\end{equation*}
where \(\gamma^\pm: [0, \infty)\) is the geodesic ray defined by \( \gamma^\pm(t) = \gamma(\pm t). \)
\par The compactification \(\widetilde{K}_0 \cup \partial_\infty \widetilde{K}_0\) of Definition~\ref{defi:gromov_boundary} is homeomorphic to a closed disc, because \(\widetilde{K}_0\) is a Hadamard space homeomorphic to the plane~\cite[II.8.11]{bridson2011metric}; the isometric action of \(\pi_1(K)\) on \(\widetilde{K}_0\) extends continuously to \(\partial_\infty \widetilde{K}_0\)~\cite[II.8.9]{bridson2011metric}.
\par There is a well-defined \(\pi_1\)-equivariant map
\begin{align*}
    \partial\colon \mathcal{G}_0 & \to \partial^2_\infty\widetilde{K}_0,      \\
    [\gamma]                     & \mapsto (\gamma(-\infty), \gamma(\infty)).
\end{align*}

\begin{lem}
\label{lem:no_flat_plane}
    The space \(\widetilde{K}_0\) contains no isometrically embedded flat strip \(\mathbb{R} \times [0, \epsilon]\); a fortiori, no isometrically embedded flat plane.
\end{lem}
\begin{proof}
    Suppose \(\widetilde{K}_0\) contained an isometric copy of \(\mathbb{R} \times [0, \epsilon]\), and consider the family \(\{\mathbb{R} \times \{s\}\}_{s \in [0, \epsilon]}\) of parallel geodesics.
    We first note that no geodesic in the interior of the strip can be degenerate: if some \(\eta_{s_0}\), \(s_0 \in (0, \epsilon)\), contained a sliding arc along (a lift of) a scatterer \(\Gamma\), then, on a sub-window of the sliding arc, the nearby parallels \(\eta_s\) would coincide with equidistant curves of the strictly convex arc \(\Gamma\) inside a flat sheet; such curves have non-vanishing geodesic curvature, contradicting the fact that the \(\eta_s\) are geodesics.
    We may then assume that the interior geodesics of the strip project to genuine billiard trajectories.
    These project, via the composition of the covering projection \(\widetilde{K}_0 \to K_0\) with \(p\colon E \to \mathbb{T}^2\), onto a continuous family of billiard trajectories \(\{\eta_s\}\) on \(\mathcal{D}\); by the finite-horizon hypothesis, each \(\eta_s\) meets a scatterer within time \(\tau_\mathrm{max}\), tangentially or transversally.
    We claim that some interior \(\eta_s\) has a regular (transversal) collision.
    Indeed, assume by contradiction it is not true: each interior \(\eta_s\), having no deflections and no sliding arcs, would project to a complete straight line \(\ell_s\) of \(\mathbb{T}^2\), tangent to every scatterer it meets.
    Pick a tangency point \(q\) of \(\ell_{s_0}\) at a scatterer \(\Gamma\), for some interior \(s_0\).
    The strip contains parallels on both sides of \(\eta_{s_0}\), and those on the side of \(\Gamma\) project, near \(q\), to curves at constant small distance \(\delta\) from \(\ell_{s_0}\) on the side of the obstacle; but the straight line parallel to \(\ell_{s_0}\) at distance \(\delta\) on that side crosses the strictly convex \(\Gamma\), while the projections must stay in \(\overline{\mathcal{D}}\), so these projections are not straight lines,  a contradiction.
    Thus some interior \(\eta_s\) collides regularly with some scatterer \(\Gamma\), and for \(s'\) sufficiently close to \(s\) the trajectory \(\eta_{s'}\) collides transversally with \(\Gamma\) as well.
    Since the trajectories are parallel and \(\Gamma\) is strictly convex, they are dispersed after the collision, contradicting the fact that the transverse distance is constant in a flat strip.
\end{proof}

As a consequence of Lemma~\ref{lem:no_flat_plane}, we can give an intrinsic proof of the Gromov-hyperbolicity of \(\widetilde{K}_0\), independent from that of Lemma~\ref{lem:uniform_hyperbolicity}.
\begin{cor}\label{cor:hyperbolic}
    The space \(\widetilde{K}_0\) is Gromov-hyperbolic: there is \(\delta > 0\) such that every geodesic triangle in \(\widetilde{K}_0\) has each side contained in the \(\delta\)-neighbourhood of the union of the other two.
\end{cor}
\begin{proof}
    The space \(\widetilde{K}_0\) is a proper cocompact CAT(0) space (it is the universal cover of the compact space \(K_0\)), and by Lemma~\ref{lem:no_flat_plane} it contains no isometrically embedded flat plane.
    By the Flat Plane Theorem (see Definition~\ref{defi:gromov_hyperbolic}), it is Gromov-hyperbolic.
\end{proof}

\begin{lem}\label{lem:wK_is_visibility}
    The space \(\widetilde{K}_0\) is a visibility space, i.e., \(\partial\colon \mathcal{G}_0 \to \partial^2_\infty\widetilde{K}_0\) is a homeomorphism.
    Moreover, any two geodesics in \(\mathcal{G}^{\mathrm{reg}}\) intersect at most once, and transversely.
\end{lem}
\begin{proof}
    Let us first show the existence, i.e., the surjectivity of \(\partial\). 
    Since \(\widetilde{K}_0\) is proper, cocompact, CAT(0) and Gromov-hyperbolic (Corollary~\ref{cor:hyperbolic}), any two distinct points \(\xi \neq \eta\) of \(\partial_\infty\widetilde{K}_0\) are joined by a bi-infinite geodesic \cite[Lemma III.H.3.2]{bridson2011metric}.

    Let us now prove uniqueness.
    If two distinct geodesic lines had the same endpoints at infinity, they would remain at bounded Hausdorff distance and, by the Flat Strip Theorem \cite[Theorem II.2.13]{bridson2011metric}, would bound a flat strip, contradicting Lemma~\ref{lem:no_flat_plane}; this shows the injectivity of \(\partial\).
    Continuity of \(\partial\) and of its inverse for the compact-open and cone topologies then follows from Arzel\`a--Ascoli together with uniqueness (see~\cite[Chapter II.8]{bridson2011metric}), hence \(\partial\) is a homeomorphism.
    \par Let us now prove the point on transversality.
    The Hadamard space \(\widetilde{K}_0\) is uniquely geodesic; hence if two geodesic lines meet at two points \(p \neq p'\), they coincide along the segment \([p,p']\) and can only separate at \(p\) or \(p'\) by branching.
    Since branching occurs only at grazing collision points, two distinct geodesics in \(\mathcal{G}^{\mathrm{reg}}\) meet in at most one point; and a tangential intersection at such a point would again be a branching point, which is excluded for regular geodesics.
    Hence the intersection is transverse.
\end{proof}

The geodesic flows \(\{\mathbf{g}_\epsilon\}_{\epsilon\in (0,1]}\) are all Anosov; in particular, by structural stability, they are all orbit equivalent to each other.
In other words, for any \(\epsilon > 0\), there is a H\"older continuous homeomorphism \(H_\epsilon\colon T^1\widetilde{K}_\epsilon \to T^1\widetilde{K}_1\) which is an orbit equivalence between the geodesic flows \(\mathbf{g}_\epsilon\) and \(\mathbf{g}_1\) (see~\cite{ghys_flots_1984, fisher2019hyperbolic}): given a geodesic \(\gamma \in \mathcal{G}_\epsilon\), there is a unique geodesic \(H_\epsilon\gamma \in \mathcal{G}_1\) corresponding to \(\gamma\), and such correspondence is equivariant with respect to the natural action of \(\pi_1(K)\) on the space of geodesics, induced by deck transformations.
\begin{lem}\label{lem:homeomorphic_boundaries}
    For every \(\epsilon \in [0,1]\) there is a homeomorphism
    \[
        \rho_\epsilon\colon \partial_\infty\widetilde{K}_\epsilon \to \partial_\infty\widetilde{K}_1, \qquad \rho_1 = \mathrm{id},
    \]
    which is equivariant with respect to the isometric actions of \(\pi_1(K)\), i.e.,
    \[
        \rho_\epsilon(\alpha\cdot\xi) = \alpha\cdot\rho_\epsilon(\xi) \quad\forall\, \xi \in \partial_\infty\widetilde{K}_\epsilon, \forall\, \alpha \in \pi_1(K).
    \]
    In particular, \(\rho_\epsilon^{-1}\circ\rho_0\colon \partial_\infty\widetilde{K}_0 \to \partial_\infty\widetilde{K}_\epsilon\) identifies, equivariantly, the boundary of the limit space with the boundary of each approximating surface.
\end{lem}
\begin{proof}
    Given a geodesic \(\gamma \in \mathcal{G}_\epsilon\), the geodesic ray \(\gamma^+ \colon [0,\infty) \to K_\epsilon\) defines a unique boundary point \(\gamma(\infty)\).
    For \(\epsilon \in (0,1]\), we define \(\rho_\epsilon\) by mapping the forward point at infinity of \(\gamma\) to the forward point at infinity of \(H_\epsilon\gamma\):
    \begin{align*}
        \rho_\epsilon\colon \partial_\infty\widetilde{K}_\epsilon & \to \partial_\infty\widetilde{K}_1,   \\
        \gamma(\infty)                                                 & \mapsto (H_\epsilon\gamma)^+(\infty).
    \end{align*}
    This extension is well-defined because \(\gamma'\) is another geodesic ray in the class \(\gamma(\infty)\) if and only if \((\gamma',\dot\gamma')\) lies on the weak (centre-)stable manifold \(\mathcal{W}^{cs}_\gamma\) of \(\mathbf{g}_\epsilon\), the boundary point being defined by bounded, not vanishing, distance.
    Since \(H_\epsilon\) is an orbit equivalence, it carries weak-stable leaves to weak-stable leaves, hence \(H_\epsilon\gamma'(\infty) = H_\epsilon\gamma(\infty)\).
    \par As for the continuity of the map \(\rho_\epsilon\) on the boundary, we note that the cone topology on \(\partial_\infty\widetilde{K}_\epsilon\) coincides with the compact-open topology on the space of geodesic rays~\cite[Chapter II.8]{bridson2011metric}.
    Moreover, if \(\sigma\) is another geodesic sufficiently close to \(\gamma\) along a compact subset of \([0,\infty)\), then \(H_\epsilon\sigma\) will be close to \(H_\epsilon\gamma\) along a corresponding compact, implying that such a map is continuous in the compact-open topology of the space of geodesics (and, in particular, geodesic rays) as we wanted.
        \par Reversing the roles of the indices \(\epsilon\) and \(1\) and applying the exact same construction to \(H_\epsilon^{-1}\) provides a continuous inverse, so \(\rho_\epsilon\) is a homeomorphism; it is equivariant because \(H_\epsilon\) is.
        \par For the case \(\epsilon = 0\), Theorem~\ref{thm:length_comparison}(b) shows that every geodesic \(\gamma_\epsilon \in \mathcal{G}_\epsilon\), seen as a curve in \(\widetilde{K}_1\), is a \((1+C, C)\)-quasi-geodesic, with constants independent of \(\epsilon\):
        \[
            |t - s| \leq d_1(\gamma_\epsilon(t), \gamma_\epsilon(s)) \leq (1+C)|t - s| + C, \qquad t,s \in \mathbb{R}.
        \]
        \par We may now apply the Morse Lemma on stability of quasi-geodesics (see e.g.~\cite[Theorem III.H.1.7]{bridson2011metric}): there is \(R = R(g_1, C) > 0\) such that any geodesic in \(\mathcal{G}_\epsilon\), \(\epsilon \in (0, 1]\) (seen as a curve in \(\widetilde{K}_1\)) is in a \(R\)-neighbourhood of a geodesic in \(\mathcal{G}_1\).
    In other words, given \(\gamma_\epsilon \in \mathcal{G}_\epsilon\), there is \(\gamma_1 \in \mathcal{G}_1\) such that
    \[
        \gamma_\epsilon(\mathbb{R}) \subset B_1(\gamma_1; R) \coloneqq \coprod_{t \in \mathbb{R}}\big\{x \in \widetilde{K}_1; d_1(x, \gamma_1(t)) < R\big\}.
    \]
    Moreover, since the geodesic flow of \(g_1\) is Anosov, distinct \(g_1\)-geodesics diverge (there are no conjugate points and \(\widetilde{K}_1\) is a visibility space), so at most one \(g_1\)-geodesic can stay at bounded distance from \(\gamma_\epsilon(\mathbb{R})\).
    \par For \(\epsilon = 0\) we argue directly. Recall that (see e.g. Theorem~\ref{thm:length_comparison})
    \begin{equation}\label{compar_dzero_dun}
    d_0 \leq d_1 \leq (1+C)d_0 + C,
    \end{equation} 
    hence every \(d_0\)-geodesic \(\gamma\) satisfies \(|t-s| \leq d_1(\gamma(t), \gamma(s)) \leq (1+C)|t-s| + C\), so it is a \((1+C, C)\)-quasi-geodesic of \(\widetilde{K}_1\). By the Morse Lemma and the divergence of \(g_1\)-geodesics, it lies within distance \(R\) of a unique \(\gamma_1 \in \mathcal{G}_1\).
    A quasi-geodesic ray \(c\colon [0, \infty) \to \widetilde{K}_1\) defines a unique point at infinity \(\xi(c) \in \partial_\infty\widetilde{K}_1\) \cite[Lemma III.H.3.1]{bridson2011metric}, and we set
    \begin{align*}
        \rho_0\colon \partial_\infty\widetilde{K}_0 & \to \partial_\infty\widetilde{K}_1 \\
        \gamma(\infty)                              & \mapsto \xi(\gamma^+).
    \end{align*}
    This is well defined and injective; indeed, two \(d_0\)-rays define the same point of \(\partial_\infty\widetilde{K}_1\) if and only if they stay at bounded \(d_1\)-distance, if and only if (by~\eqref{compar_dzero_dun}) they stay at bounded \(d_0\)-distance, if and only if they are asymptotic in \(\widetilde{K}_0\).
    It is continuous, because the Morse Lemma is uniform: if \(\gamma_n \to \gamma\) on compact sets, the \(g_1\)-geodesics shadowing \(\gamma_n^+\) converge on compact sets to the one shadowing \(\gamma^+\), hence their ideal points converge.
    Moreover, it is surjective: given \(\xi \in \partial_\infty\widetilde{K}_1\), let \(\gamma\) be a \(g_1\)-geodesic with \(\gamma(\infty) = \xi\); the curves \(H_\epsilon^{-1}\gamma\), \(\epsilon \in (0,1]\), parametrised by \(d_\epsilon\)-arc length, are \(1\)-Lipschitz for \(d_\epsilon\), hence for \(d_0\), and they stay within \(d_1\)-distance \(R\) of \(\gamma\) by the Morse Lemma applied in \(\widetilde{K}_1\). By Arzel\`a--Ascoli, some sequence \(H_{\epsilon_n}^{-1}\gamma\) converges on compact sets to a curve \(\gamma_0\), which is a \(d_0\)-geodesic by Theorem~\ref{thm:d_0_is_a_metric} and stays within distance \(R\) of \(\gamma\), so that \(\rho_0(\gamma_0(\infty)) = \xi\).
    Finally, \(\rho_0\) is a continuous bijection between compact Hausdorff spaces, hence a homeomorphism, and it is \(\pi_1(K)\)-equivariant by construction.
\end{proof}

Let us recall some basic facts about Busemann cocycles which will be useful in the following. 
\begin{dfn}[Busemann cocycles]\label{defi:busemann}
    Let \((X, d)\) be a proper \(\mathrm{CAT}(0)\) space and \(\xi \in \partial_\infty X\).
    The \emph{Busemann cocycle} of \(\xi\) is the function
    \[
        B_\xi\colon X \times X \to \mathbb{R}, \qquad B_\xi(x, y) \coloneqq \lim_{t \to \infty}\bigl[d(x, \gamma(t)) - d(y, \gamma(t))\bigr],
    \]
    where \(\gamma\) is any geodesic ray with \(\gamma(\infty) = \xi\); the limit exists (the bracket is bounded and monotone by the triangle inequality and convexity) and does not depend on the chosen ray \cite[II.8]{bridson2011metric}.
    We will use the following standard properties:
    \begin{enumerate}[label=(\roman*)]
        \item \(B_\xi\) is a cocycle: \(B_\xi(x, z) = B_\xi(x, y) + B_\xi(y, z)\);
        \item \(B_\xi\) is \(1\)-Lipschitz in each variable;
        \item \(B_{\psi\xi}(\psi x, \psi y) = B_\xi(x, y)\) for every isometry \(\psi\) of \(X\);
        \item along a ray towards \(\xi\), we have \(B_\xi(\gamma(t), \gamma(0)) = -t\);
        \item \emph{horofunction continuity}: if \(z_n \to \xi\) in \(\overline{X} = X \cup \partial_\infty X\), then \(d(x, z_n) - d(y, z_n) \to B_\xi(x, y)\), uniformly for \((x, y)\) in compact sets \cite[Ch.~II]{ballmann1995lectures}.
    \end{enumerate}
\end{dfn}

\begin{rmk}\label{rem:anosov_top_equivalences}
    The orbit equivalences \(H_{\epsilon}\) above are essentially unique, due to the Anosov Shadowing Theorem (more precisely, they are~\emph{transversally} unique, see~\cite[Theorem 5.4.1]{fisher2019hyperbolic}).
    The choice of homeomorphism does not change the equivalence between orbits, just the reparametrisation function \(\tau\) appearing in the relation
    \[
        H_\epsilon \mathbf{g}_\epsilon^t(z) = \mathbf{g}_1^{\tau(z, t)}(H_\epsilon(z)).
    \]
    One possibility is to choose \(H_\epsilon\) is the following construction, from \cite[Appendix B, Theorem B.1]{guillarmou_geodesic_2022}.
    Given \(v \in T^1\widetilde{K}_\epsilon\), we write
    \[
        (v^-_\epsilon, v^+_\epsilon) = (\gamma_{\epsilon; v}(-\infty), \gamma_{\epsilon; v}(\infty)) \in \partial^2_\infty\widetilde{K}.
    \]
    The homeomorphism \(H_\epsilon\colon T^1\widetilde{K}_\epsilon \to T^1\widetilde{K}_1\) is then obtained by setting \(H_\epsilon(v) = w\), where \(w \in T^1\widetilde{K}_1\) is the unique vector such that
    \[
        (v^-_\epsilon, v^+_\epsilon) = (w^-_1, w^+_1) \quad \text{and} \quad B^1_{v^+_\epsilon}(\pi(v), \pi(w)) = 0.
    \]
    Here \(\pi\colon T^1\widetilde{K} \to \widetilde{K}\) is the canonical projection and \(B^1_{v^+_\epsilon}\) is the Busemann cocycle (Definition~\ref{defi:busemann}) of the boundary point \(v^+_\epsilon\) with respect to the metric \(g_1\).
    The condition \(B^1_{v^+_\epsilon}(\pi(v), \pi(w)) = 0\) selects, on the \(g_1\)-geodesic with the prescribed endpoints, the unique point lying on the horosphere of \(v^+_\epsilon\) through \(\pi(v)\).
    \par With such choice, the family \(\{H_\epsilon^{-1}\}\) becomes a uniformly equicontinuous deformation of the \(g_1\)-geodesics on \(\widetilde{K}\).
    Indeed, if that were not the case there would be \(\eta> 0\) and sequences \((v_n, u_n) \in T^1K_1, \epsilon_n \to 0\) such that
    \[
        d_1(v_n, u_n) < \frac{1}{n} \quad \text{and} \quad d_\epsilon(H_\epsilon^{-1}(v_n), H_\epsilon^{-1}(u_n)) > \eta.
    \]
    The common limit \(\lim_n v_n = \lim_nu_n = z \in T^1K_1\) defines a unique geodesic \(\gamma \in \widehat{\mathcal{G}}_0\).
    Moreover, by construction, the geodesics defined by the initial conditions \(H_\epsilon^{-1}(v_n)\) and \(  H_\epsilon^{-1}(u_n)\) have the same endpoints at infinity as \(v_n\) and \(u_n\), and thus must both converge to \(\gamma\) on every compact subset as \(\epsilon \to 0\), a contradiction.
\end{rmk}

\subsubsection{Closed geodesics in the space \(K_{0}\).}

We first recall the classification of isometries of \(\mathrm{CAT}(0)\) spaces and the relevant properties of the action of \(\pi_1(K)\) on \(\widetilde{K}_0\); we refer to~\cite[Chapter II.6]{bridson2011metric} for more details.
\begin{dfn}[Isometries of \(\mathrm{CAT}(0)\) spaces]\label{defi:isometries_cat0}
    Let \((X, d)\) be a complete \(\mathrm{CAT}(0)\) space and \(g\) an isometry of \(X\).
    The \emph{displacement function} of \(g\) is \(d_g\colon X \to [0,\infty)\), \(d_g(x) \coloneqq d(x, gx)\), and its \emph{translation length} is
    \[
        \ell(g) \coloneqq \inf_{x \in X} d(x, gx).
    \]
    Write \(\mathrm{Min}(g) \coloneqq \{x \in X : d(x, gx) = \ell(g)\}\).
    The isometry \(g\) is called
    \begin{enumerate}[label=(\roman*)]
        \item \emph{elliptic} if it has a fixed point (equivalently, \(\ell(g) = 0\) and \(\mathrm{Min}(g) \neq \emptyset\));
        \item \emph{hyperbolic} (or \emph{axial}) if \(\mathrm{Min}(g) \neq \emptyset\) and \(\ell(g) > 0\);
        \item \emph{parabolic} if \(\mathrm{Min}(g) = \emptyset\), i.e., the infimum is not attained.
    \end{enumerate}
    Isometries which are elliptic or hyperbolic are called \emph{semi-simple}.
\end{dfn}
We will use the following standard facts about a hyperbolic isometry \(g\) of a proper \(\mathrm{CAT}(0)\) space \(X\)~\cite[II.6.8]{bridson2011metric}:
\begin{enumerate}[label=(\alph*)]
    \item a geodesic line \(A \subset X\) is an \emph{axis} of \(g\) if \(g\) acts on \(A\) as a translation by \(\ell(g)\); the set \(\mathrm{Min}(g)\) is the union of the axes of \(g\), and any two axes are parallel, i.e., at bounded Hausdorff distance (hence they bound a flat strip, by the Flat Strip Theorem);
    \item all the axes of \(g\) have the same endpoints at infinity, denoted \(g^- \coloneqq A(-\infty)\) and \(g^+ \coloneqq A(+\infty)\) and called the \emph{repelling} and \emph{attracting} fixed points of \(g\); these are fixed by the action of \(g\) on \(\partial_\infty X\), and \(g^n x \to g^{\pm}\) as \(n \to \pm\infty\) for every \(x \in X\);
    \item \(\ell(g^n) = |n|\,\ell(g)\) and \(\mathrm{Min}(g^n) = \mathrm{Min}(g)\) for every \(n \neq 0\), and \(\ell(hgh^{-1}) = \ell(g)\) for every isometry \(h\), with \((hgh^{-1})^{\pm} = h\, g^{\pm}\); in particular \(\ell\) is a class function on any group of isometries;
    \item if \(X\) is moreover Gromov-hyperbolic, then \(g^- \neq g^+\) are the only fixed points of \(g\) on \(\partial_\infty X\), and \(g\) has \emph{North--South dynamics}: for every neighbourhood \(U^\pm\) of \(g^\pm\) in \(\overline{X}\), one has \(g^{n}(\overline{X} \setminus U^{\mp}) \subset U^{\pm}\) for all \(n\) large enough (i.e., \(g^{\pm n} \to g^{\pm}\) uniformly on compact subsets of \(\overline{X} \setminus \{g^{\mp}\}\)); in Gromov's terminology, \(g\) is \emph{loxodromic}.
\end{enumerate}
Here, the word ``hyperbolic'' refers to the classification of isometries, and not to any property of the dynamics; in the Gromov-hyperbolic \(\mathrm{CAT}(0)\) spaces considered below, hyperbolic isometries in the \(\mathrm{CAT}(0)\) sense and loxodromic isometries in Gromov's sense coincide.

\begin{lem}\label{lem:pi1_hyperbolic}
    The group \(\pi_1(K)\) acts on the Hadamard space \(\widetilde{K}_0\) by deck transformations, which are isometries; the action is free, properly discontinuous and cocompact.
    Moreover:
    \begin{enumerate}[label=(\roman*)]
        \item \(\pi_1(K)\) is the fundamental group of a closed orientable surface of genus \(m + 1 \geq 2\), where \(m\) is the number of scatterers; in particular it is torsion-free and non-elementary, and \(\partial_\infty\widetilde{K}_0 \approx \mathbb{S}^1\) is its Gromov boundary;
        \item every \(g \in \pi_1(K) \setminus \{1\}\) is a hyperbolic isometry of \(\widetilde{K}_0\);
        \item every \(g \in \pi_1(K) \setminus \{1\}\) has a unique axis \(A_g\), namely the geodesic line \((g^-, g^+)\) of Lemma~\ref{lem:wK_is_visibility}; the projection of \(A_g\) to \(K_0\) is a closed geodesic of length \(\ell(g)\) in the free homotopy class of \(g\), and \(\ell(g) \geq 2R_0 > 0\) with \(R_0\) as in Lemma~\ref{lem:uniform_lower_bound_inj_radius};
        \item two elements \(g, h \in \pi_1(K) \setminus \{1\}\) either have the same pair of fixed points \(\{g^\pm\} = \{h^\pm\}\), in which case they have a common axis and generate a cyclic group up to finite index, or they have four distinct fixed points.
    \end{enumerate}
\end{lem}
\begin{proof}
    The covering \(\widetilde{K}_0 \to K_0\) is a local isometry for the lifted length metric, so deck transformations are isometries; the action is free and properly discontinuous (as for any covering) and cocompact, since \(K_0\) is compact.
    \par (i) The surface \(K\) is the double of the table \(\overline{\mathcal{D}}\), which is a torus with \(m\) discs removed, along its \(m\) boundary circles; hence \(K\) is a closed orientable surface with \(\chi(K) = 2\chi(\overline{\mathcal{D}}) = -2m\), i.e., of genus \(m+1 \geq 2\).
    Its fundamental group is torsion-free and non-elementary (it contains free subgroups of rank two), and, being quasi-isometric to \(\widetilde{K}_0\) by the \v{S}varc--Milnor lemma, its Gromov boundary is \(\partial_\infty\widetilde{K}_0\), which is a circle by~\cite[II.8.11]{bridson2011metric} (see also the beginning of this subsection).
    \par (ii) Since the action is cocompact, every element of \(\pi_1(K)\) is a semi-simple isometry~\cite[II.6.10(2)]{bridson2011metric}; since the action is free, no non-trivial element is elliptic. Hence every \(g \neq 1\) is hyperbolic.
    \par (iii) Two distinct axes of \(g\) would be parallel geodesic lines, bounding a flat strip, which is excluded by Lemma~\ref{lem:no_flat_plane}; so the axis is unique, and it is the unique geodesic line joining \(g^-\) to \(g^+\) (Lemma~\ref{lem:wK_is_visibility}).
    Its projection to \(K_0\) is a closed curve of length \(\ell(g)\) in the free homotopy class of \(g\), which is a local geodesic since \(A_g\) is a geodesic line; its length is at least the length of the shortest non-contractible loop of \(K\) for \(d_0\), which is at least \(2R_0\) by the proof of Lemma~\ref{lem:uniform_lower_bound_inj_radius} (the bound \(s_0\) obtained there is for the Euclidean length of projections, which is dominated by the \(d_0\)-length).
    \par (iv) If \(g\) and \(h\) share one fixed point \(\xi \in \partial_\infty\widetilde{K}_0\), then the stabiliser of \(\xi\) in \(\pi_1(K)\) contains the two loxodromic elements \(g, h\); in a group acting properly on a proper Gromov-hyperbolic space, the stabiliser of a point fixed by a loxodromic element is virtually cyclic and fixes both its fixed points~\cite[Chapter 8]{coornaert_geometrie_1990}, so \(\{g^\pm\} = \{h^\pm\}\), the axes coincide by (iii), and \(\langle g, h\rangle\) is virtually cyclic.
\end{proof}
With this terminology, the length spectrum of \(K_0\) is the function \(\ell\) on the set \(\mathcal{C}(K) \coloneqq \mathrm{Conj}(\pi_1(K))\) of conjugacy classes, and the next lemma identifies it with the lengths of the closed \(d_0\)-geodesics of \(K_0\).

\begin{lem}\label{lem:unique_minimising_geodesic}
    Any element \( \alpha \in \mathcal{C}(K) \coloneqq \mathrm{Conj}(\pi_1(K)) \) contains a single closed geodesic curve \( \gamma^\alpha_0 \), which is the curve of minimal \(d_0\)-length in \( \alpha \); write \( \ell_0(\alpha) \coloneqq l_0(\gamma^\alpha_0) \).
    Moreover, \( \gamma^\alpha_\epsilon \to \gamma^\alpha_0 \) in the \( C^0 \) topology and \( l_\epsilon(\gamma^\alpha_\epsilon) \to \ell_0(\alpha) \) as \( \epsilon \to 0 \).
\end{lem}
\begin{proof}
    \emph{Step 1: existence, minimality and uniqueness of \( \gamma^\alpha_0 \).} Fix a representative \( g \in \pi_1(K) \) of \( \alpha \).
    By Lemma~\ref{lem:pi1_hyperbolic}, \( g \) is a hyperbolic isometry of \( \widetilde{K}_0 \) with a unique axis \( A \), on which it translates by \( \ell(g) = \inf_{x \in \widetilde{K}_0} d_0(x, gx) > 0 \).
    Set \( \gamma^\alpha_0 \coloneqq p(A) \), a closed \(d_0\)-geodesic in \( \alpha \), and \( \ell_0(\alpha) \coloneqq \ell(g) \).
    Any closed curve \( c \in \alpha \) lifts to a path from some \( x \) to \( gx \), so \( l_0(c) \geq d_0(x, gx) \geq \ell(g) = l_0(\gamma^\alpha_0) \); thus \( \gamma^\alpha_0 \) is length-minimising in \( \alpha \).
    If \( \alpha \) contained a second closed \(d_0\)-geodesic, its lift would be a second axis of \( g \), contradicting the uniqueness of the axis (Lemma~\ref{lem:pi1_hyperbolic}(iii)); thus \( \gamma^\alpha_0 \) is the unique closed geodesic in \( \alpha \).

    \emph{Step 2: convergence \( \gamma^\alpha_\epsilon \to \gamma^\alpha_0 \).} Each \( K_\epsilon \) is Anosov, so \( \alpha \) contains a unique closed \( g_\epsilon \)-geodesic \( \gamma^\alpha_\epsilon \) \cite{klingenberg_riemannian_1974}, of length \( \ell_\epsilon(\alpha) \coloneqq \inf_x d_\epsilon(x, gx) \).
    By Theorem~\ref{thm:length_comparison}, \( \ell_0(\alpha) \leq \ell_\epsilon(\alpha) \), and testing against \( \gamma^\alpha_0 \),
    \[
        \ell_\epsilon(\alpha) \leq l_\epsilon(\gamma^\alpha_0) \leq l_0(\gamma^\alpha_0) + C\,\mathrm{Var}(z \circ \gamma^\alpha_0)\,\epsilon = \ell_0(\alpha) + O(\epsilon),
    \]
    so \( \ell_\epsilon(\alpha) \to \ell_0(\alpha) \) as \( \epsilon \to 0 \).
    \hfill(\(*\))
    \par Parametrise each \( \gamma^\alpha_\epsilon \) on \( \mathbb{S}^1 \) proportionally to \( d_\epsilon \)-arclength; since \( d_0 \leq d_\epsilon \), the resulting maps are \( \ell_1(\alpha) \)-Lipschitz into \( (K, d_0) \), so \( \{\gamma^\alpha_\epsilon\} \) is precompact in \( C^0(\mathbb{S}^1, K) \) by Arzel\`a--Ascoli.
    Let \( \gamma_* \) be an accumulation point, \( \gamma^\alpha_{\epsilon_n} \to \gamma_* \) uniformly.
    \par As \( (K, d_0) \) and \( (K, d_1) \) are compact with the same topology, the identity is a uniform homeomorphism, so uniform \( d_0 \)-convergence is uniform \( d_1 \)-convergence; hence for large \( n \), \( (\gamma^\alpha_{\epsilon_n}(s), \gamma_*(s)) \in D_{R_0} \coloneqq \{(x,y) : d_1(x,y) < R_0\} \) for all \( s \), where \( R_0 \) is the uniform injectivity radius (Lemma~\ref{lem:uniform_lower_bound_inj_radius}).
    Since \( D_{R_0} \) deformation-retracts to the diagonal via \( (x,v) \mapsto (x, \exp^1_x v) \), the loops \( \gamma^\alpha_{\epsilon_n} \) and \( \gamma_* \) are freely homotopic, i.e., \( \gamma_* \in \alpha \).
    \par Let us now show that \( l_0(\gamma_*) = \ell_0(\alpha) \).
    By lower semicontinuity of length under uniform convergence \cite{burago2022course} and \( l_0 \leq l_{\epsilon_n} \),
    \[
        l_0(\gamma_*) \leq \liminf_n l_0(\gamma^\alpha_{\epsilon_n}) \leq \liminf_n l_{\epsilon_n}(\gamma^\alpha_{\epsilon_n}) = \liminf_n \ell_{\epsilon_n}(\alpha) \overset{(*)}{=} \ell_0(\alpha),
    \]
    while \( l_0(\gamma_*) \geq \ell_0(\alpha) \) by Step 1.
    Hence \( \gamma_* \) is a minimal closed \(d_0\)-geodesic in \( \alpha \), so \( \gamma_* = \gamma^\alpha_0 \).
    Every accumulation point equals \( \gamma^\alpha_0 \) and the family is precompact, so \( \gamma^\alpha_\epsilon \to \gamma^\alpha_0 \) in \( C^0 \), and \( l_\epsilon(\gamma^\alpha_\epsilon) = \ell_\epsilon(\alpha) \to \ell_0(\alpha) \) by \((*)\).
\end{proof}
\begin{cor}\label{cor:dense_periodic_liftings}
The pairs of endpoints \((g^-, g^+)\) of the axes of the elements \(g \neq 1\) of \(\pi_1(K)\) (i.e., of the lifts of closed geodesics of \(K_0\)) are dense in \(\partial^2_\infty\widetilde{K}_0\). In particular, lifts of closed geodesics are dense in \(\mathcal{G}_0\), and the fixed points of the elements of \(\pi_1(K)\) are dense in \(\partial_\infty\widetilde{K}_0\).
\end{cor}
\begin{proof}
Since the geodesic flow \(\mathbf{g}_1\) is Anosov, hence transitive with dense periodic orbits, the pairs of endpoints of the axes of the elements of \(\pi_1(K)\) are dense in \(\partial^2_\infty\widetilde{K}_1\). The equivariant homeomorphism \(\rho_0\colon \partial_\infty\widetilde{K}_0 \to \partial_\infty\widetilde{K}_1\) of Lemma~\ref{lem:homeomorphic_boundaries} maps the fixed points of each \(g\) to the fixed points of the same \(g\), so \((\rho_0 \times \rho_0)^{-1}\) carries this dense set onto the set of pole pairs in \(\partial^2_\infty\widetilde{K}_0\), which is therefore dense; the remaining assertions follow from Lemma~\ref{lem:wK_is_visibility}.
\end{proof}

\subsubsection{Cross-ratios on the boundary \(\partial_\infty\widetilde{K}_0\).}

Fix the homeomorphisms \(H_\epsilon\colon \mathcal{G}_0 \to \mathcal{G}_\epsilon\) induced by the boundary identifications \(\rho_\epsilon\) of Lemma~\ref{lem:homeomorphic_boundaries}, and identify all the boundaries \(\partial_\infty\widetilde{K}_\epsilon\) with \(\partial_\infty\widetilde{K}_0 \approx \mathbb{S}^1\).
Given \(\gamma_0 \in \mathcal{G}_0\), we write \(\gamma_\epsilon \coloneqq H_\epsilon \gamma_0\), and assume, without loss of generality, that \(\gamma_\epsilon(\infty) = \gamma_0(\infty) \in \mathbb{S}^1\) for all \(\epsilon \in [0,1]\).
In particular, \(\gamma_0\) is a quasi-geodesic on every \(\widetilde{K}_\epsilon\) and thus defines a unique point at infinity in \(\partial_\infty\widetilde{K}_\epsilon\), which, by our previous identifications, is just \(\gamma_0(\infty)\).

\par For \(\epsilon \in [0,1]\) and \(x, y, z, w \in \widetilde{K}_\epsilon\), write
\[
    E_{\epsilon}(x, y, z, w) \coloneqq d_\epsilon(x, y) + d_\epsilon(z, w) - d_\epsilon(x, w) - d_\epsilon(z, y).
\]

As in~\cite{otal_sur_1992, bourdon_sur_1996}, given distinct points \(\{\zeta_i\}_{i=1}^4\in \partial_\infty\widetilde{K}_0^4\), we define their cross-ratio as
\begin{equation}\label{eq:cr_rat_def}
    \operatorname{cr}_{0}(\zeta_1, \zeta_2, \zeta_3, \zeta_4) = \lim_{t \to \infty} e^{\frac{1}{2}E_0(\gamma_1(t), \gamma_2(t), \gamma_3(t), \gamma_4(t))}
\end{equation}
where the \(\gamma_i\) are unit-speed geodesic rays, issued from arbitrary base points, such that \(\gamma_i(\infty) = \zeta_i\), for \(i = 1,2,3,4\).  Lemma~\ref{lem:cr0_exists} shows that the limit exists and depends neither on the base points nor on any other feature of the chosen rays.
\begin{lem}\label{lem:cr0_exists}
    For distinct \(\zeta_1, \dots, \zeta_4 \in \partial_\infty\widetilde{K}_0\) and unit-speed rays \(\gamma_i\) with \(\gamma_i(\infty) = \zeta_i\), the limit defining \(\operatorname{cr}_0\) in \eqref{eq:cr_rat_def} exists and is independent of the chosen rays.
\end{lem}
\begin{proof}
    Let \(f_{ij}(t) \coloneqq d_0(\gamma_i(t), \gamma_j(t)) - 2t\).
    By convexity of \(t \mapsto d_0(\gamma_i(t), \gamma_j(t))\) in the CAT(0) space \(\widetilde{K}_0\) \cite[II.2.2]{bridson2011metric} and its \(2\)-Lipschitz property, \(f_{ij}\) is non-increasing; by \(\delta\)-hyperbolicity (Corollary~\ref{cor:hyperbolic}) and \(\zeta_i \neq \zeta_j\), it is bounded below.
    Hence each \(f_{ij}\) converges, and so does \(E_0(\gamma_1(t), \gamma_2(t), \gamma_3(t), \gamma_4(t)) = f_{12} + f_{34} - f_{14} - f_{32}\).
    Independence of the rays follows from a CAT(0) argument: for two forward-asymptotic rays \(\gamma,\gamma'\) in \(\widetilde{K}_0\), the function \(t \mapsto d_0(\gamma(t), \gamma'(t))\) is convex and bounded, hence non-increasing, and converges to some \(c \geq 0\); if \(c > 0\), the rays would bound a flat half-strip and, translating it by a suitable sequence of deck transformations and passing to a limit, we would obtain a complete flat strip, contradicting Lemma~\ref{lem:no_flat_plane} (see~\cite[Chapter II]{ballmann1995lectures} for more details).
    Hence \(c = 0\), and the limits defining \(\operatorname{cr}_0\) do not depend on the chosen rays.
\end{proof}

\begin{lem}
\label{pf:lem_conv_geodesics}
    Let \(\xi \neq \eta \in \partial_\infty \widetilde{K}_0\) and let \(u_k, v_k \in \overline{\widetilde{K}_0}\) with \(u_k \to \xi\), \(v_k \to \eta\).
    Then the geodesics \([u_k, v_k]\) (segments, rays or lines) converge, after reparametrisation, uniformly on compact sets to the line \((\xi, \eta)\).
    Moreover:
    \begin{enumerate}[label=(\roman*)]
        \item \(d\bigl(p, [u_k, v_k]\bigr) \to 0\) for every \(p \in (\xi, \eta)\);
        \item if \(x_k \in [u_k, v_k]\) and \(d(o, x_k) \to \infty\), then the accumulation points of \((x_k)\) in \(\overline{X}\) are contained in \(\{\xi, \eta\}\).
    \end{enumerate}
\end{lem}
\begin{proof}
    By \(\delta\)-hyperbolicity every point of \([u_k, v_k]\) lies within \(4\delta\) of \([u_k, o] \cup [o, v_k]\); since \(u_k, v_k\) converge to distinct boundary points, the Gromov products \((u_k \mid v_k)_o\) are bounded (Definition~\ref{defi:gromov_product}(iv)), hence so is \(d(o, [u_k, v_k]) \leq (u_k \mid v_k)_o + 4\delta\) (Definition~\ref{defi:gromov_product}(iii)): there is a ball \(B(o, C)\) met by all \([u_k, v_k]\).
    Reparametrise so that \([u_k, v_k](0) \in B(o,C)\); the maps are \(1\)-Lipschitz, so Arzel\`a--Ascoli and properness give subconvergence, uniform on compact sets, to a complete geodesic (the domains exhaust \(\mathbb{R}\) since \(u_k, v_k\) leave every ball).
    Its endpoints are \(\xi, \eta\) because far-away points of \([u_k, v_k]\) lie \(4\delta\)-close to the rays \([o, u_k]\), \([o, v_k]\), whose ideal points converge; this concludes the proof of point (ii). 
    By uniqueness of the geodesic joining \(\xi\) to \(\eta\) (by Lemma~\ref{lem:wK_is_visibility}) every sublimit is \((\xi,\eta)\), hence we have full convergence and (i).
\end{proof}

\begin{cor}\label{cor:cr0_continuous}
    The function \(\operatorname{cr}_0\) is continuous on the set of quadruples of pairwise distinct points of \(\partial_\infty\widetilde{K}_0\).
\end{cor}
\begin{proof}
    Let \((\xi, \eta)\) be one of the pairs \((\zeta_i, \zeta_j)\) appearing in \(E_0\), with rays \(u_t, v_s\) issued from base points \(c, c'\) towards \(\xi, \eta\), and let \(p\) be any point of the line \((\xi, \eta)\) (Lemma~\ref{lem:wK_is_visibility}). We claim that
    \[
        \lim_{t, s \to \infty}\bigl[t + s - d_0(u_t, v_s)\bigr] = -\bigl[B_\xi(p, c) + B_\eta(p, c')\bigr],
    \]
    with the Busemann cocycles of Definition~\ref{defi:busemann}. Indeed \(B_\xi(p, c) = \lim_t [d_0(p, u_t) - t]\) and \(B_\eta(p, c') = \lim_s [d_0(p, v_s) - s]\), so the difference between the two sides is \(\lim [d_0(p, u_t) + d_0(p, v_s) - d_0(u_t, v_s)]\), which is non-negative by the triangle inequality and at most \(2\, d_0(p, [u_t, v_s])\); the latter tends to \(0\) because the segments \([u_t, v_s]\) converge to the line \((\xi, \eta)\) (by Lemma~\ref{pf:lem_conv_geodesics}(i)).
    The right-hand side depends continuously on \((\xi, \eta) \in \partial^2_\infty\widetilde{K}_0\): the line \((\xi, \eta)\) and a point \(p\) on it can be chosen continuously by Lemma~\ref{pf:lem_conv_geodesics}, and Busemann cocycles are jointly continuous in \((\xi, x, y)\) on a proper \(\mathrm{CAT}(0)\) space; indeed, \(\xi \mapsto B_\xi(\cdot, o)\) is the horofunction map, which is a homeomorphism onto its image for the cone topology~\cite[II.8.13]{bridson2011metric}. As \(\lim_t E_0(\gamma_1(t), \dots, \gamma_4(t))\) is a signed sum of four such limits (with \(s = t\)), \(\operatorname{cr}_0\) is continuous.
\end{proof}
\begin{rmk}
\label{rem:cr_eps_limit}
    Since \(d_0 = \lim_{\epsilon \to 0} d_\epsilon\) monotonically, one may ask whether \(\operatorname{cr}_{0}\) is consistent with the \(\epsilon\)-cross-ratios, i.e.,
    \begin{equation}\label{eq:cr_rat_lim}
        \operatorname{cr}_{0}(\zeta_1, \zeta_2, \zeta_3, \zeta_4)
        = \lim_{\epsilon \to 0} \operatorname{cr}_{\epsilon}(\zeta_1, \zeta_2, \zeta_3, \zeta_4),
    \end{equation}
    where \(\operatorname{cr}_\epsilon\) is defined as in \eqref{eq:cr_rat_def} using \(E_\epsilon\), the same geodesics \(\gamma_i \in \mathcal{G}_0\) being viewed as uniform quasi-geodesics of \(\widetilde{K}_\epsilon\) with the same endpoints at infinity (Theorem~\ref{thm:length_comparison}).
    The interchange of the limits in \(t\) and \(\epsilon\) can be justified by Moore--Osgood, using the uniform shadowing of Theorem~\ref{thm:length_comparison} together with a uniform exponential convergence of asymptotic \(g_\epsilon\)-geodesics.
    Since we do not need this fact in the present paper (the constructions of Section~\ref{sec:spectral.rigidity} are intrinsic to \((\widetilde{K}_0, d_0)\)), we do not elaborate on this point. 
\end{rmk}
\par It is easy to check the following properties from the definition:
\begin{pps}\label{pps:crat_properties}
    The mapping \(\operatorname{cr}_{0}\colon (\partial_\infty\widetilde{K}_0)^4 \to \mathbb{R}\) satisfies
    \begin{enumerate}
        \item \( \operatorname{cr}_{0} \) is positive and invariant under the action of \( \pi_1(K) \) on \( (\partial_\infty\widetilde{K}_0)^4 \);
        \item \( \operatorname{cr}_{0}(\zeta_1, \zeta_2, \zeta_3, \zeta_4) = \operatorname{cr}_{0}(\zeta_1, \zeta_4, \zeta_3, \zeta_2)^{-1} \);
        \item \( \operatorname{cr}_{0}(\zeta_1, \zeta_2, \zeta_3, \zeta_4) = \operatorname{cr}_{0}(\zeta_3, \zeta_4, \zeta_1, \zeta_2) \);
        \item \( \operatorname{cr}_{0}(\zeta_1, \zeta_2, \zeta_3, \zeta_4)\,\operatorname{cr}_{0}(\zeta_1, \zeta_4, \zeta_3, \zeta_5) = \operatorname{cr}_{0}(\zeta_1, \zeta_2, \zeta_3, \zeta_5) \);
        \item \( \operatorname{cr}_{0}(\zeta_1, \zeta_2, \zeta_3, \zeta_4)\operatorname{cr}_{0}(\zeta_2, \zeta_3, \zeta_1, \zeta_4)\operatorname{cr}_{0}(\zeta_3, \zeta_1, \zeta_2, \zeta_4) = 1 \).
    \end{enumerate}
\end{pps}

\subsubsection{Geodesic facts on the limit space}
\label{subsec:limit_preliminaries}

Below we abbreviate \(X \coloneqq \widetilde{K}_0\) and \(B \coloneqq B(\widetilde{K})\) for the seam (the union of the lifted walls); \(\overline{X} = X \cup \partial_\infty X\) carries the cone topology.
Let us recall that the minimal free path \(\tau_\mathrm{min} > 0\) of the table, and the minimal distance between distinct lifted walls, are positive.

We first recall the notion of angle.
We recall that, in a CAT(0) space \((X, d)\), the angle \(\angle_x(\gamma, \gamma') \in [0, \pi]\) between two geodesics \(\gamma, \gamma'\colon [0,1] \to X\) intersecting at the point \(x = \gamma(0) = \gamma'(0)\) is given by
\[
    \angle_x(\gamma, \gamma') = \lim_{t\to 0} 2\arcsin{\bigg(\frac{1}{2t}d(\gamma(t), \gamma'(t))\bigg)}.
\]
In particular, for any geodesic triangle \(\triangle(x,y,z)\) the following standard comparison estimate holds:
\begin{equation}\label{eq:GB_estimate_subs}
    \angle_x(y,z) + \angle_y(z,x) + \angle_z(x,y) \leq \pi.
\end{equation}
When \((X,d)\) is a Riemannian manifold with the Riemannian metric, this coincides with the Riemannian angle (see e.g.\cite[Corollary II.1A.7]{bridson2011metric}).
In our situation, at a flat point \(x \in X \setminus B\) a neighbourhood of \(x\) is isometric to a Euclidean disc (see e.g. Lemma~\ref{lem:d_0_is_distance_reestated}), so the angle between two geodesics through \(x\) is the Euclidean angle of their chords; at a seam point the CAT(0) angle is used directly, and it is described in Lemma~\ref{pf:lem_directions} below.
\begin{rmk}
    Note that angles are always well defined on \(\widetilde{K}_0\), even though they do not specify the geodesic, at least not at degenerate points.
    In fact, two geodesics bifurcate at a point \(x\) if and only if their angle at \(x\) is zero, and we may characterise the boundary set \(B(\widetilde{K})\) exactly as the locus of points where geodesic bifurcation happens, that is, points from which emanate a pair of geodesics with angle zero.
\end{rmk}

The following elementary lemmas will be used repeatedly in the following. 
\begin{lem}
    \label{pf:lem_rays}
    For every \(x \in X\) and \(\zeta \in \partial_\infty X\) there is a unique ray \(r_{x,\zeta}\) from \(x\) to \(\zeta\), and \((x, \zeta) \mapsto r_{x,\zeta}\) is continuous (uniformly on compact time intervals).
\end{lem}
\begin{proof}
    If \(r, r'\) are rays from \(x\) to \(\zeta\), the convex function \(t \mapsto d(r(t), r'(t))\) vanishes at \(0\) and is bounded (by asymptoticity); a bounded convex function on \([0,\infty)\) vanishing at the origin is identically zero (since \(d(t)/t\) is non-decreasing and tends to \(0\)).
    Let us now show continuity.
    If \(x_k \to x\) and \(\zeta_k \to \zeta\), the rays \(r_k \coloneqq r_{x_k, \zeta_k}\) satisfy \(d\bigl(r_k(t), r_{o,\zeta_k}(t)\bigr) \leq d(x_k, o)\) for all \(t\) (a convex function bounded by its value at \(0\), by the same argument), so the \(r_k\) stay in a tube around rays from \(o\); Arzel\`a--Ascoli gives sublimits, which are rays from \(x\) asymptotic to \(\zeta\) (their ideal points are \(\lim \zeta_k = \zeta\), by Lemma~\ref{pf:lem_conv_geodesics}(ii) applied to the rays \(r_{o, \zeta_k}\));
    by uniqueness, all sublimits are equal.
\end{proof}
\begin{lem}
\label{pf:lem_directions}
    At every \(x \in X\) the space of directions \(\Sigma_x\), with the Alexandrov angle, is a circle of length \(2\pi\), and:
    \begin{enumerate}[label=(\roman*)]
        \item the two directions of a geodesic through \(x\) cut \(\Sigma_x\) into two arcs of length \(\pi\);
        \item for directions \(u, v, w\) in this cyclic order on an arc of length \(\leq \pi\): \(\angle(u,w) = \angle(u,v) + \angle(v,w)\);
        \item (\emph{reflection law}) a concatenation of two segments crossing \(B\) transversally at \(x\) is geodesic iff the two incidence angles with the seam, measured in the respective sheets, are equal;
        \item (\emph{branching}) two distinct geodesics containing a common ray or segment ending at \(x\) split at \(x\) with angle \(0\); this occurs only if the common direction at \(x\) is tangent to \(B\).
              In particular a concatenation of a segment and a seam arc is geodesic iff the segment arrives tangentially.
    \end{enumerate}
\end{lem}
\begin{proof}
    For \(x\) in the flat interior everything is Euclidean.
    Let us consider \(x \in B\); the tangent cone of \(X\) at \(x\) is the pointed Gromov--Hausdorff limit of the rescalings \((X, x, \lambda d)\), \(\lambda \to \infty\).
    A neighbourhood of \(x\) is the gluing of two flat half-discs along the wall through \(x\), which is a \(C^2\) curve; under rescaling each half-disc converges to a Euclidean half-plane and the wall to its tangent line, so the tangent cone is the gluing of two half-planes along their boundary lines, i.e., the Euclidean plane \(\mathbb{E}^2\).
    For a \(\mathrm{CAT}(0)\) space, the tangent cone in the sense of~\cite[II.3.19]{bridson2011metric} is the Euclidean cone over \(\Sigma_x\) with the angle metric; here it coincides with the Gromov--Hausdorff limit above, since the Alexandrov angle between two germs at \(x\) is the limit of their comparison angles, which are computed in the rescaled spaces and converge to the Euclidean angle of the limit configuration (the two half-discs being \(C^2\)-close to half-planes). Hence \(\Sigma_x\) is the round circle of length \(2\pi\); the two sheets contribute the two semicircles, glued at the two seam directions \(\pm\tau\).
    Angles are then arc lengths, which gives (ii), and a concatenation at \(x\) is geodesic if and only if the two directions make angle \(\pi\)~\cite[II.1.4 and II.9]{bridson2011metric}, i.e., iff both complementary arcs measure exactly \(\pi\); (i) follows.

    For (iii): with incidence angles \(\varphi^-\) (incoming, in sheet 1) and \(\varphi^+\) (outgoing, in sheet 2), the two arcs measure \(\pi - \varphi^- + \varphi^+\) and \(\pi + \varphi^- - \varphi^+\), both \(\geq \pi\) iff \(\varphi^- = \varphi^+\).

    For (iv): two geodesics splitting at \(x\) with a common incoming direction \(u\) must both continue with the direction antipodal to \(u\) in \(\Sigma_x\); if \(u\) is not tangent to \(B\), the antipode is the direction of the unique straight continuation of the incoming chord inside its sheet (or the transversally reflected chord, by (iii)), and the germ of the continuation is unique.
    If \(u = \pm\tau\), the antipode is realised by three germs: the seam and one chord in each sheet, all tangent to \(B\).
    It corresponds to the branching configuration, which gives the last statement (an arc and then a segment junction with incidence \(\varphi^-\) has arcs \(\pi \mp \varphi^-\); it is geodesic iff \(\varphi^- = 0\)).
\end{proof}

Let us conclude this section with a remark on the contacts with the seam:
\begin{lem}
\label{pf:lem_contacts}
    For every geodesic \(\gamma\) of \(X\), consecutive connected components of the contact set \(C_\gamma \coloneqq \{t : \gamma(t) \in B\}\) are separated by open intervals of length \(\geq \tau_\mathrm{min}\).
    Consequently, on every compact window, \(C_\gamma\) is a finite union of closed intervals (seam arcs) and isolated points (transversal crossings or tangential contacts), and \(|C_\gamma \setminus \mathrm{int}\, C_\gamma|\) has zero Lebesgue measure.
\end{lem}
\begin{proof}
    A maximal open interval \((t, t')\) of non-contact times is a straight chord of one flat sheet, i.e. (developing a neighbourhood of the chord isometrically into the plane), a Euclidean segment in a lift of the table, from a wall point to a wall point, meeting no wall in between.
    A segment leaving a point of a strictly convex wall into the table never meets the same obstacle lift again; so the chord joins two distinct obstacle lifts (possibly two lifts of the same wall of \(K\)) and its length is at least the minimal free path \(\tau_\mathrm{min} > 0\) of the table.
    The remaining statements follow: on a compact window there are finitely many components; non-interval components are single times, at which \(\gamma\) either crosses \(B\) transversally or touches it tangentially (see Lemma~\ref{pf:lem_directions}); and the complement of the interior of \(C_\gamma\) in \(C_\gamma\) consists of finitely many points per window.
\end{proof}

%% file: sections/spectral.rigidity.tex
In what follows, we use the notation of Section~\ref{sec:geodesic}, i.e., \(K_0=(K, g_0)\) is the Kourganoff surface obtained by gluing two copies of \( (\mathcal{D}, d_\mathcal{L}) \) at the boundary, with \(K^\circ= K \setminus B(K) = \mathcal{D}^\mathrm{up} \coprod \mathcal{D}^\mathrm{down}\).
As previously, we consider a family of metrics \( g_\epsilon \) on \( K \) converging to the gluing metric when \( \epsilon \) approaches \( 0 \).
In other words, we have metric convergence \( (K, g_\epsilon) \eqqcolon K_\epsilon \xrightarrow{GH} K_0\) in the Gromov--Hausdorff distance, as \( \epsilon \to 0 \).    Below we abbreviate \(X \coloneqq \widetilde{K}_0\), with \(\overline{X} = X \cup \partial_\infty X\), and we denote by \(B \coloneqq B(\widetilde{K})\) the seam (the union of the lifted walls).

The strategy of this section is that of Otal~\cite{otal_spectre_1990}, carried out directly on the limit space \(\widetilde{K}_0\); here, the enriched spectrum is the marked length spectrum of the \(\mathrm{CAT}(0)\) space \(K_0\).
We show that the space carries an intrinsic Liouville geodesic current, defined by the kinematic measure of its flat pieces (\S\ref{subsec:intrinsic_liouville}). The spectrum determines the cross-ratio at infinity (\S\ref{subsec:spectrum_crossratio}), hence the current, and therefore a homeomorphism \(h\) between the spaces of unparametrised geodesics which preserves the currents. An average-angle argument in the spirit of Otal~\cite{otal_spectre_1990} then shows that \(h\) preserves angles (\S\ref{subsec:angle_dissipation}), and the dispersing walls force \(h\) to be induced by an isometry (\S\ref{subsec:isometry}).
The proofs do not use the ergodicity of the billiard flow; in particular, the rigidity theorem does not rely on Sinai's theorem on the ergodicity of finite horizon Sinai billiards.\\

Recall that \( \mathcal{C}(K) \coloneqq \mathrm{Conj}(\pi_1(K)) \) is the set of free homotopy classes on \( K \).

\begin{dfn}[Enriched marked length spectrum]
    Let \( \mathcal{D} \) be a Sinai billiard and \( K \) a surface in \( \mathcal{K}(\mathcal{D}) \).
    The function
    \begin{align*}
        \mathbf{EL}_{\mathcal{D}; K}\colon \mathcal{C}(K) & \to \mathbb{R}             \\
        \alpha                                            & \mapsto l_0(\gamma^\alpha)
    \end{align*}
    is called the enriched marked length spectrum of \( \mathcal{D} \).
\end{dfn}
\begin{rmk}\label{cor:EL_is_translation_length}
    By Lemma~\ref{lem:unique_minimising_geodesic}, this definition agrees with the one given in the Introduction; moreover, it does not depend on the choice of \(K \in \mathcal{K}(\mathcal{D})\), since by Theorem~\ref{thm:char_billiard_cycles_1} and the results of Appendix~\ref{sec:billiards.cycles} the values \(l_0(\gamma^\alpha)\) are computed by the intrinsic length functionals \(\mathrm{EL}^I_\rho\) of the table.
    We therefore drop \(K\) from the notation and write \(\mathbf{EL}_{\mathcal{D}}\). We have \(\mathbf{EL}_\mathcal{D}(\alpha) = \ell_0(\alpha) = \inf_{x} d_0(x, gx)\) for any representative \(g \in \pi_1(K)\) of \(\alpha\): as we noted earlier, the enriched marked length spectrum of \(\mathcal{D}\) is thus the marked length spectrum of the \(\mathrm{CAT}(0)\) space \(K_0\).
\end{rmk}

\subsection{The intrinsic Liouville current of the limit space}
\label{subsec:intrinsic_liouville}

In this part, we define an intrinsic Liouville current on the limit space.
We then prove an exact Crofton formula (Lemma~\ref{lem:intrinsic_crofton}) and show that the values of the current on rectangles of boundary pairs are precisely \(|\ln \operatorname{cr}_0|\) (Proposition~\ref{prop:rectangle_formula}), so that the cross-ratios above determine the current completely.

Let us start by quickly recalling some facts about geodesic currents.
For \( \epsilon > 0 \), the surface \( K_\epsilon \in \mathcal{K}(\mathcal{D}) \) carries a smooth Riemannian metric \( g_\epsilon \).
This structure induces a canonical contact form \( \lambda_\epsilon \) on the unit tangent bundle \( T^1 K_\epsilon \), obtained by pulling back the canonical symplectic form from the cotangent bundle.
The Reeb vector field of \( \lambda_\epsilon \) coincides with the geodesic spray of \( g_\epsilon \), that is, the vector field on \( T^1 K_\epsilon \) generating the geodesic flow.
The associated Liouville measure \( \int\lambda_\epsilon \wedge d\lambda_\epsilon \) on \( T^1 K_\epsilon \) is a finite measure, invariant under the geodesic flow, and consequently under its time reversal.
\par All such structures are lifted to \(T^1\widetilde{K}_\epsilon\) via the local isomorphism defined by the covering map \(\widetilde{K}_\epsilon \to K_\epsilon\), so that the space \(T^1\widetilde{K}_\epsilon\) has the corresponding contact and volume forms \(\widetilde{\lambda}_\epsilon\) and  \( \int\widetilde\lambda_\epsilon \wedge d\widetilde\lambda_\epsilon \).
Given a rectifiable curve \( \gamma \subset \widetilde{K}_\epsilon \), its unit tangent lift \( \dot{\gamma} \subset T^1 \widetilde{K}_\epsilon \) defines the \emph{Liouville current} \( m_\epsilon \) on the space of geodesics \( \mathcal{G}_\epsilon \) by integration against \( d\widetilde\lambda_\epsilon \), assigning mass to Borel sets of geodesics intersecting \( \gamma \) transversely~\cite{wilkinson_anosov}.
\par More generally, a geodesic current for \(K_\epsilon\) is any Radon measure on the space \(\partial^2_\infty\widetilde{K}_\epsilon\) that is invariant under both the diagonal action of \(\pi_1(K_\epsilon)\) and the flip transformation \( (\zeta_1, \zeta_2) \mapsto (\zeta_2, \zeta_1)\).
There is an affine one-to-one correspondence between the cones of finite measures on \(T^1K_\epsilon\) invariant under the geodesic flow and the space of geodesic currents with the weak-* topology (see~\cite{bonahon_bouts_1986},~\cite{furman_coarse-geometric_2002}).
\par In particular, the contact structure \( \lambda_\epsilon \), the cross-ratio \( \operatorname{cr}_{\epsilon} \), and the Liouville current are related in the following way.
For measurable rectangles in the space of geodesics determined by intervals \( [\zeta_1, \zeta_2] \) and \( [\zeta_3, \zeta_4] \) in the boundary \( \partial \widetilde{K}_\epsilon \), we have:
\[
    m_\epsilon\big([\zeta_1, \zeta_2] \times [\zeta_3, \zeta_4]\big)
    = \big| \ln \operatorname{cr}_{\epsilon}(\zeta_1, \zeta_2, \zeta_3, \zeta_4) \big|
    = \int_{\Sigma} d\lambda_\epsilon,
\]
where \( \Sigma \subset T^1 \widetilde{K}_\epsilon \) is a transversal to the geodesic flow intersecting exactly the geodesics joining \( [\zeta_1, \zeta_2] \) to \( [\zeta_3, \zeta_4] \), oriented accordingly.\\

Such a construction cannot be replicated directly on \(\widetilde{K}_0\), since the Riemannian tensors degenerate as \(\epsilon \to 0\) and the metric on \(\widetilde{K}_0\) is not induced by any Riemannian metric.
Below, we will nevertheless construct a Liouville current on \(\mathcal{G}_0\) intrinsically.
For that, we use an atlas consisting  of two kinds of charts.

For a geodesic \(\sigma\) meeting \(B\) only transversally, the~\emph{flat chart} \(\mathcal{U}^{\mathrm{reg}}_\sigma\) is the set of geodesics crossing \(\sigma\) transversally at a flat point from the right, i.e., with oriented crossing angle \(\theta \in (0,\pi)\), an orientation of \(X\) being fixed. Its coordinates are \((t,\theta) \in \mathbb{R}\times(0,\pi)\), the crossing time on \(\sigma\) and the oriented crossing angle. The geodesics crossing \(\sigma\) from the left form the flat chart of the reversed geodesic \(\bar\sigma\), so that every transversal crosser of \(\sigma\) lies in exactly one of \(\mathcal{U}^{\mathrm{reg}}_\sigma\), \(\mathcal{U}^{\mathrm{reg}}_{\bar\sigma}\).

For a lifted wall \(w\), parametrised by arc length \(r\), the two~\emph{wall charts} \(\mathcal{W}_w^{\pm}\) are the sets of geodesics crossing \(w\) transversally from a fixed side, \(\pm\) referring to the two sheets adjacent to \(w\), with coordinates \((r, \varphi) \in \mathbb{R} \times (0, \pi)\), where \(\varphi\) is the incidence angle measured in the incoming sheet; by the reflection law (Lemma~\ref{pf:lem_directions}(iii)) the outgoing sheet gives the same coordinates, so the charts are well defined. A geodesic crossing \(w\) from one side is the reverse of a geodesic crossing it from the other side, so \(\mathcal{W}_w^{-}\) is the image of \(\mathcal{W}_w^{+}\) under the flip, which reads \((r,\varphi) \mapsto (r, \pi - \varphi)\) in these coordinates; as for flat charts, each chart is one-sided. We write \(\mathcal{W}_w\) for either of them when the side is irrelevant.
We can now give the intrinsic definition of the Liouville current.
\begin{dfn}
\label{defi:kinematic_current}
    The \emph{Liouville current} \(\widehat{m}_0\) is the unique Borel measure on \(\widehat{\mathcal{G}}_0/\mathbf{g}_0 \approx \partial^2_\infty \widetilde{K}_0\) whose restriction to every chart is given by
    \begin{equation}\label{eq:kinematic_flux}
        d\widehat{m}_0 = \tfrac14 \sin\theta\, d\theta\, dt \quad \text{on } \mathcal{U}^{\mathrm{reg}}_\sigma,
        \qquad
        d\widehat{m}_0 = \tfrac14 \sin\varphi\, d\varphi\, dr\quad \text{on } \mathcal{W}_w .
    \end{equation}
\end{dfn}
The next results show that this is well defined.

\begin{lem}
\label{pf:lem_planar}
    Let \(c_1, c_2\) be embedded \(C^1\) curves in the Euclidean plane, parametrised by arc length.
    On the set of lines crossing both transversally (crossing parameters separated locally), the correspondence \((t_1, \theta_1) \mapsto (t_2, \theta_2)\) between crossing data is a local \(C^1\)-diffeomorphism and satisfies \(\sin\theta_2\, dt_2\, d\theta_2 = \sin\theta_1\, dt_1\, d\theta_1\).
\end{lem}
\begin{proof}
    Parametrise non-oriented lines by \((p, \psi) \in \mathbb{R} \times [0,\pi)\), \(L_{p,\psi} = \{x : \langle x, n_\psi\rangle = p\}\), \(n_\psi = (-\sin\psi, \cos\psi)\).
    If \(L_{p,\psi}\) crosses \(c_i\) at \(c_i(t_i)\) with angle \(\theta_i\), then \(\psi = \alpha_i(t_i) + \theta_i \bmod \pi\) (\(\alpha_i\) the direction of \(\dot c_i\)) and \(p = \langle c_i(t_i), n_\psi\rangle\), so \((t_i, \theta_i) \mapsto (p, \psi)\) is \(C^1\) with, at fixed \(\psi\), \(\partial p / \partial t_i = \langle \dot c_i, n_\psi\rangle = \sin\theta_i \neq 0\). So \(dp\, d\psi = \sin\theta_i\, dt_i\, d\theta_i\), and chart-to-line maps are local \(C^1\)-diffeomorphisms~\cite{santalo_integral_geometry}; the claim follows.
\end{proof}

\begin{lem}
\label{pf:lem_null_charts}
    Fix a flat or wall chart.
    Each of the following sets of chart parameters is contained in a countable union of \(C^1\) curves, hence is null for the planar Lebesgue measure of the chart:
    \begin{enumerate}[label=(\alph*)]
        \item parameters whose geodesic has a tangential contact with \(B\);
        \item parameters over which the chart fibre contains more than one geodesic;
        \item parameters whose geodesic passes through a fixed point \(p \in X\).
    \end{enumerate}
\end{lem}
\begin{proof}
    (a) By Lemma~\ref{pf:lem_contacts}, along any bounded time window a geodesic has finitely many transversal crossings before a possible tangential contact; so the parameters in (a) are the countable union, over finite itineraries \(I = (w_1, \dots, w_k)\) of lifted walls crossed transversally, over target walls \(w\), over bounded windows, and over the two time directions, of the sets
    \begin{align*}
        T_{I, w} \coloneqq & \{\text{parameters whose continuation crosses } w_1, \dots, w_k \text{ transversally,} \\
                           & \text{then meets } w \text{ tangentially}\}.
    \end{align*}
    Fix \(I\) and let \(O_I\) be the open set of parameters whose continuation realises the itinerary \(I\) transversally and then enters the flat sheet adjacent to \(w\) along a chord.
    On \(O_I\), let \(\Lambda\colon (\text{chart data}) \to (p, \psi)\) assign the line data of that final chord in a fixed local development of its sheet.
    Then \(\Lambda\) is a composition of the chart-to-line map of the initial piece, and, at each \(w_i\), the map (line data) \(\to\) (wall data \((r_i, \varphi_i)\)) followed by (reflected wall data) \(\to\) (line data of the next developed sheet).
    By Lemma~\ref{pf:lem_planar} applied with \(c_1\) a straight developed transversal or the \(C^2\) curve \(w_i\), every factor is a local \(C^1\)-diffeomorphism (the Jacobians are \(\sin\theta\), \(\sin\varphi_i \neq 0\) at transversal data while the reflection itself is the identity in wall data).
    Hence \(\Lambda\) is a local \(C^1\)-diffeomorphism onto open sets of line data.
    The lines tangent to the developed image of \(w\) form an immersed \(C^1\) curve in \((p, \psi)\): parametrised by the tangency foot \(s\), the direction \(\psi(s)\) has \(\psi'(s) = \pm\mathcal{K}(s) \neq 0\) by strict convexity of the walls.
    Its preimage under \(\Lambda\) is a countable union of \(C^1\) curves; and \(T_{I,w} \subset \Lambda^{-1}(\text{tangent lines to } w)\).

    (b) Two geodesics over the same chart datum share the crossing point and direction, hence coincide until they branch; by Lemma~\ref{pf:lem_directions}(iv) branching requires a tangential contact with \(B\), so (b) \(\subset\) (a).

    (c) If \(p\) is a flat point, the lines through the developed image of \(p\) form a \(C^1\) curve in \((p,\psi)\); argue as in (a).
    If \(p = w(r_0)\) is a seam point, the geodesics through \(p\) that cross \(w\) transversally form \(\{r = r_0\}\) in the wall chart of \(w\), and we transfer to the given chart through \(\Lambda\); the remaining geodesics through \(p\) (tangent to \(w\) at \(p\), or running along \(w\)) lie in (a).
\end{proof} 

\begin{pps}
\label{pf:prop_consistency}
    The formulas \eqref{eq:kinematic_flux} are compatible on all chart overlaps and define a unique Radon measure \(\widehat{m}_0\) on \(\partial^2_\infty X\), invariant under \(\pi_1(K)\) and the flip, for which tangential-contact configurations, point pencils, and degenerate geodesics are null.
\end{pps}
\begin{proof}
    Let two charts (flat or wall) overlap at a geodesic \(\gamma\) crossing their transversals at \(\gamma(s) \) and \(\gamma(s')\), \(s < s'\), and remove the null set of geodesics described in Lemma~\ref{pf:lem_null_charts} with a tangential contact on the connecting arc, or belonging to a multiple fibre.
    For the remaining \(\gamma\), the arc \(\gamma([s,s'])\) crosses \(B\) transversally, finitely many times (Lemma~\ref{pf:lem_contacts}), at walls \(w_1, \dots, w_k\), and contains no seam arc.
    On a neighbourhood of \(\gamma\) in the first chart, the transition map factors as
    \[
        (\text{chart}) \to (r_1, \varphi_1) \to (r_1, \varphi_1)^{\mathrm{out}} \to (r_2, \varphi_2) \to \dots \to (\text{chart}'),
    \]
    where consecutive arrows are the planar transitions of Lemma~\ref{pf:lem_planar} inside one developed sheet (with straight or \(C^2\)-curved transversals), or the reflection, which is the identity in wall data.
    Every factor preserves the density \(\tfrac14 \sin(\cdot)\, d(\cdot)\, d(\cdot)\), hence so does the transition, almost everywhere on the overlap, which shows the desired consistency.
    The charts cover the complement of the set \(\mathcal{T}\) of complete wall lifts, which is countable. Indeed, let \(\eta\) be a geodesic which is not a complete wall lift; then \(\eta\) contains a chord \(c\), whose interior points are flat. In any wall chart \(\mathcal{W}_w\), the set of geodesics crossing \(c\) transversally is open and non-empty,
    hence of positive measure, while the geodesics with a tangential contact form a null set by Lemma~\ref{pf:lem_null_charts}(a); so some tangency-free \(\sigma\) crosses \(c\) transversally at a flat point, and \(\eta \in \mathcal{U}^{\mathrm{reg}}_\sigma \cup \mathcal{U}^{\mathrm{reg}}_{\bar\sigma}\).
    By setting \(\widehat{m}_0(\mathcal{T}) = 0\) and gluing over a countable subatlas,
    this defines a Radon measure.
    Isometries of \(X\) permute the charts preserving the densities, which gives \(\pi_1(K)\)-invariance; the flip acts by \((t, \theta) \mapsto (-t, \pi-\theta)\) and \((r, \varphi) \mapsto (r, \pi - \varphi)\) on wall charts, both preserving the densities.
    Nullity of tangential configurations and pencils comes from Lemma~\ref{pf:lem_null_charts}; degenerate geodesics contain a tangential junction (Lemma~\ref{pf:lem_directions}(iv)), hence are null.
\end{proof}
With this normalisation, the geodesics crossing a segment, counted from both sides, have mass equal to its length (Lemma~\ref{lem:intrinsic_crofton}).
In particular, by Lemma~\ref{pf:lem_planar} and the consistency shown above, formula~\eqref{eq:kinematic_flux} holds along every straight segment of a flat sheet and along every wall arc, even if the segment does not extend to a tangency-free geodesic.
Hence the geodesics crossing such a segment \(J\) from a fixed side, in a sector \((\theta_1, \theta_2)\), have mass \(|J|\int_{\theta_1}^{\theta_2}\tfrac14\sin\theta\,d\theta\).

We summarise below what has been proved so far, together with two consequences that we use repeatedly.
\begin{lem}
\label{lem:current_properties}
    Formula~\eqref{eq:kinematic_flux} defines a unique Radon measure \(\widehat{m}_0\) on \(\partial^2_\infty\widetilde{K}_0\).
    It is invariant under \(\pi_1(K)\) and under the flip; moreover,
    \begin{enumerate}
        \item[(i)] two geodesics intersect transversally at a single point if and only if their endpoint pairs are linked in \(\partial_\infty\widetilde{K}_0\) and they do not share a sub-arc; on the other hand, pairs sharing a sub-arc (which are pairs of geodesics with a common tangential contact) form a \(\widehat{m}_0 \otimes \widehat{m}_0\)-null set;
        \item[(ii)] \emph{(Null tangencies)} the set of geodesics tangent to a fixed \(C^2\) curve, and the set of geodesics through a fixed point, are \(\widehat{m}_0\)-null;
        \item[(iii)] \emph{(Null degenerates)} the set of degenerate geodesics is \(\widehat{m}_0\)-null.
    \end{enumerate}
\end{lem}
\begin{proof}
    (i) Linked endpoints force an intersection by a Jordan-curve argument in the closed disc \(\widetilde{K}_0 \cup \partial_\infty \widetilde{K}_0\); if two geodesics had two common points they would share the segment between them (uniqueness of \(\mathrm{CAT}(0)\) segments), whose endpoints are branch points, hence tangential contacts with the seam by Lemma~\ref{pf:lem_directions}(iv); in particular a shared sub-arc forces both geodesics into the tangential-contact family, which is \(\widehat m_0 \otimes \widehat m_0\)-null by Proposition~\ref{pf:prop_consistency}.
    Points (ii) and (iii) are contained in Lemma~\ref{pf:lem_null_charts} and Proposition~\ref{pf:prop_consistency}, because every degenerate geodesic meets the seam tangentially at the endpoints of each of its seam arcs (see Lemma~\ref{pf:lem_directions}(iv)).
\end{proof}

Let \(\widehat{\mathcal{G}}^{\flat}_0 \subset \widehat{\mathcal{G}}_0\) be the set of parametrised geodesics \(\gamma\) such that \(\gamma(0)\) is a flat point. The map
\[
    \widehat{\mathcal{G}}^{\flat}_0 \to T^1\widetilde{K}^\circ, \qquad \gamma \mapsto (\gamma(0), \dot\gamma(0)),
\]
is well defined and continuous. It is injective on the set of regular geodesics (see Remark~\ref{rem:branching}), and \(\widehat{\mathcal{G}}^{\flat}_0\) has full \(\widehat m_0 \otimes dt\)-measure, since the times at which a geodesic lies on the seam form a Lebesgue-null set for regular geodesics (by Lemma~\ref{pf:lem_contacts}) and non-regular geodesics are \(\widehat m_0\)-null (by Lemma~\ref{lem:current_properties}). 

We use this map to identify, up to null sets, the flow-invariant measure \(\widehat m_0 \otimes dt\) on \(\widehat{\mathcal{G}}_0\) with a measure on \(T^1\widetilde{K}^\circ\). In the coordinates \((x,\phi)\) of a developed sheet, \(\phi\) being the direction angle, it is defined as follows.
\begin{cor}\label{pf:cor_dxdphi}
Under the above identification, \(d(\widehat{m}_0 \otimes dt) = \tfrac14\, dx\, d\phi\) on \(T^1\widetilde{K}^\circ\).
Consequently, the quotient \(\mu \coloneqq (\widehat m_0 \otimes dt)/\pi_1(K)\), a measure on \(\widehat{\mathcal{G}}_0/\pi_1(K)\) which we call the \emph{Liouville measure} of \(K_0\), is finite, of total mass \(\pi\operatorname{Area}(\mathcal{D})\), and invariant under the flow, the flip, and each rotation \(R_{\theta_1}\colon (x,\phi) \mapsto (x, \phi + \theta_1)\). We normalise \(\mu\) only in Lemma~\ref{pf:lem_pairs_measure}(iii).
\end{cor}
\begin{proof}
    In a flat chart with the position coordinate \(s\) along the crossing geodesic, \(d(\widehat m_0 \otimes dt) = \tfrac14 \sin\theta\, d\theta\, dt\, ds\); developing the sheet and writing \(x = \sigma(t) + s e_\phi\), \(\phi = \alpha + \theta\), the Jacobian of \((t,s) \mapsto x\) at fixed \(\theta\) is \(|\det(e_\alpha, e_\phi)| = \sin\theta\), so \(dx\, d\phi = \sin\theta\, d\theta\, dt\, ds\). The two flat charts \(\mathcal{U}_{\sigma},\mathcal{U}_{\bar \sigma}\) cover the two half-circles \(\phi \in (\alpha,\alpha+\pi)\) and \((\alpha+\pi,\alpha+2\pi)\), so the formula holds for all \(\phi\).
    The identification makes sense a.e., since times on the seam form a Lebesgue-null set for every geodesic (by Lemma~\ref{pf:lem_contacts} for regular ones, while degenerate ones are \(\widehat m_0\)-null).
    The invariances and the mass computation follow, \(R_{\theta_1}\) clearly preserving \(\tfrac14\,dx\,d\phi\) and mapping the conull regular set to itself.
\end{proof}

\begin{lem}[Crofton formula]\label{lem:intrinsic_crofton}\label{pf:lem_crofton}
    For every geodesic segment \(\gamma|_{[a,b]}\) of \(X\), we have \(\;\widehat{m}_0\bigl(\{\eta : \eta \pitchfork \gamma|_{[a,b]}\}\bigr) = b - a\).
\end{lem}
\begin{proof}
    First, off a null set, a geodesic \(\eta\) crosses \(\gamma\) at most once; indeed, two common points give a common segment (by uniqueness of \(\mathrm{CAT}(0)\) segments), whose endpoints are branch points, hence tangential contacts of \(\eta\) with \(B\) (Lemma~\ref{pf:lem_directions}(iv)); such \(\eta\) are null.
    By Lemma~\ref{pf:lem_contacts}, \([a,b]\) decomposes into finitely many maximal open flat intervals \(I_j\) (chords), finitely many maximal seam intervals \(W_i\) (wall arcs), and a finite set \(Z\) of times.
    Modulo null sets, the crossing set of \(\gamma|_{[a,b]}\) is then the disjoint union of the crossing sets of the pieces; indeed, disjointness follows from the at-most-once property; the geodesics crossing exactly at a time of \(Z\) form finitely many point pencils, which are null by Lemma~\ref{pf:lem_null_charts}(c).
        Each chord piece is a tangency-free transversal segment; its crossers from the right form the box \(\overline{I_j}\times(0,\pi)\) of the flat chart along \(\gamma\) restricted to \(I_j\), of mass \(|I_j|\int_0^\pi\tfrac14\sin = \tfrac12|I_j|\), and its crossers from the left form the corresponding box of the chart along \(\bar\gamma\), of the same mass; together, \(|I_j|\). Similarly, for each wall piece \(\gamma(W_i)\subset w\), the geodesics crossing it from either side form two boxes \(W_i\times(0,\pi)\) of the two wall charts of \(w\), of total mass \(|W_i|\), by~\eqref{eq:kinematic_flux}. 
    By summing, the total mass is then \(\sum |I_j| + \sum |W_i| = (b - a) - |Z| = b - a\).
\end{proof}

\begin{cor}
\label{pf:cor_compact_incidence}
    For every compact \(Q \subset X\), \(\;\widehat{m}_0\bigl(\{\eta : \eta \cap Q \neq \emptyset\}\bigr) \leq \tfrac{\pi}{4}\operatorname{Area}\bigl(N_1(Q)\bigr) < \infty\), where \(N_1(Q)\) is the closed \(1\)-neighbourhood.
\end{cor}
\begin{proof}
    If \(\eta\) meets \(Q\), the sub-segment of \(\eta\) of length \(2\) centred at a point of \(\eta \cap Q\) lies in \(N_1(Q)\), so \(\operatorname{length}(\eta \cap N_1(Q)) \geq 2\).
    Integrating and using Corollary~\ref{pf:cor_dxdphi}, we get
    \begin{align*}
        2\,\widehat{m}_0(\{\eta \cap Q \neq \emptyset\}) & \leq \int \operatorname{length}\bigl(\eta \cap N_1(Q)\bigr)\, d\widehat{m}_0(\eta)  \\
                                                         & = \bigl(\widehat m_0 \otimes dt\bigr)\bigl(\{(\eta, t) : \eta(t) \in N_1(Q)\}\bigr)
        = \tfrac{\pi}{2}\operatorname{Area}(N_1(Q)).\qedhere
    \end{align*}
\end{proof}

\begin{cor}\label{pf:cor_disjoint_arcs}
    If \(A, B \subset \partial_\infty X\) are disjoint closed arcs, the geodesics with one endpoint in \(A\) and the other in \(B\) all meet a common compact set; in particular their \(\widehat{m}_0\)-mass is finite.
\end{cor}
\begin{proof}
    Otherwise there are such geodesics \(\eta_k\) with \(\eta_k \cap B(o, k) = \emptyset\) for all \(k \geq 0\); up to passing to a subsequence, their endpoint pairs converge to \((p, q) \in A \times B\) with \(p \neq q\), so by Lemma~\ref{pf:lem_conv_geodesics} the \(\eta_k\) meet a fixed ball, a contradiction.
    We conclude by Corollary~\ref{pf:cor_compact_incidence}.
\end{proof}

\begin{lem}
    \label{pf:lem_endpoint_null}
    For any \(\zeta \in \partial_\infty X\), the set \(\{\eta : \zeta \in \{\eta(-\infty), \eta(\infty)\}\}\) is \(\widehat{m}_0\)-null.
\end{lem}
\begin{proof}
    We work in a flat chart along \(\sigma\).
    A geodesic with forward endpoint \(\zeta\) crossing \(\sigma\) at \(\sigma(t)\) contains the ray \(r_{\sigma(t), \zeta}\), which is unique by Lemma~\ref{pf:lem_rays}; off the multiple-fibre null set, its angle coordinate is the Alexandrov angle \(\theta_\zeta(t) \coloneqq \angle_{\sigma(t)}(\sigma, r_{\sigma(t), \zeta})\).
    The function \(t \mapsto \theta_\zeta(t)\) is Borel: the comparison angles \(\overline{\angle}\bigl(\sigma(t+s), \sigma(t), r_{\sigma(t),\zeta}(s')\bigr)\) are continuous in \(t\) for fixed rational \(s, s' > 0\) (Lemma~\ref{pf:lem_rays}), and the Alexandrov angle is their infimum over \(s, s'\) \cite[II.1.7]{bridson2011metric}.
    Hence the set \(\{\eta : \zeta \in \{\eta(-\infty), \eta(\infty)\}\}\) meets the chart in a subset of the union of the graphs of two Borel functions and a null set; graphs of Borel functions are Lebesgue-null by Fubini.
    Covering with countably many charts, this concludes the proof.
\end{proof}

\begin{rmk}
\label{rem:current_vs_billiard}
    In the seam chart used in the proof of Lemma~\ref{lem:intrinsic_crofton}, the restriction of \(\widehat{m}_0\) to the geodesics crossing a wall \(\Gamma_i\) from a fixed side is \(\tfrac14 \sin\varphi\, d\varphi\, dr\), i.e., it is the classical invariant collision measure of the billiard map on \(\mathcal{D}\); the current \(\widehat{m}_0\) is thus the geodesic-current analogue of the billiard Liouville measure (for each side; the classical collision measure \(\tfrac12 \sin\varphi\, d\varphi\, dr\) of the billiard map counts each collision once, i.e., both sides of the seam together). 
\end{rmk}

Recall that \(\widehat{\mathcal{G}}(\mathcal{D})\) is the admissible class of Section~\ref{sec:geodesic} (see Definition~\ref{geod_tricho}); in particular, by Lemma~\ref{lem:topology.geodesics} the open set of degenerate geodesics is disjoint from \(\widehat{\mathcal{G}}(\mathcal{D})\).
\begin{pps}\label{pf:prop_support}\label{rem:full_support}
    The support of the current \(\widehat{m}_0\) is equal to
    \[
        \operatorname{supp} \widehat{m}_0 = \overline{\widehat{\mathcal{G}}^{\mathrm{reg}}} = \widehat{\mathcal{G}}(\mathcal{D}).
    \]
    The same holds for \(\mu\) on the quotient.
    In particular \(\widehat m_0\) does not have full support on the space of all geodesics.
\end{pps}
\begin{proof}
    (\(\subseteq\)) Non-regular geodesics, i.e., degenerate or grazing ones, are \(\widehat m_0\)-null, since they have a tangential contact with \(B\) (see Proposition~\ref{pf:prop_consistency}, Lemma~\ref{pf:lem_null_charts}(a)), so every set of positive mass contains regular geodesics; hence every point of the support is a limit of regular geodesics.
    (\(\supseteq\)) Let \(\eta_0\) be regular and \(O \ni \eta_0\) a basic \(C^0\)-neighbourhood over a window \([-R, R]\).
    On that window \(\eta_0\) has finitely many seam contacts (Lemma~\ref{pf:lem_contacts}), all of them transversal since \(\eta_0\) is regular.
    Fix a transversal segment \(\tau\) inside a chord of \(\eta_0\).
    Transversality of finitely many crossings is an open condition on the chart data along \(\tau\): there is an open box \(U \times (\theta_1, \theta_2)\) of data whose geodesics realise the same itinerary on \([-R,R]\) and depend continuously on the data there, hence lie in \(O\) for a small enough box.
    By \eqref{eq:kinematic_flux} the box has positive mass, so \(\widehat m_0(O) > 0\) and \(\eta_0 \in \operatorname{supp}\widehat m_0\); the support being closed, \(\overline{\widehat{\mathcal{G}}^{\mathrm{reg}}} \subset \operatorname{supp}\widehat m_0\).
    Finally, degenerate geodesics are not limits of regular ones (see Lemma~\ref{lem:topology.geodesics}).
\end{proof}

\begin{pps}[Rectangle formula]\label{prop:rectangle_formula}\label{pf:prop_rectangle}
    Let \(\zeta_1, \zeta_2, \zeta_3, \zeta_4 \in \partial_\infty X\) be pairwise distinct, in the cyclic order \(\zeta_1, \zeta_3, \zeta_2, \zeta_4\) (i.e., the pairs \((\zeta_1,\zeta_2)\) and \((\zeta_3, \zeta_4)\) are linked).
    Let
    \[
        R \coloneqq \bigl\{\eta : \text{one endpoint of } \eta \text{ lies in the arc } (\zeta_1\zeta_3),\ \text{the other in the arc } (\zeta_2\zeta_4)\bigr\},
    \]
    the arcs being those not containing the remaining two points.
    Then
    \[
        \widehat{m}_0(R) = \ln \operatorname{cr}_0(\zeta_1, \zeta_2, \zeta_3, \zeta_4) \geq 0 .
    \]
\end{pps}

\begin{proof}
    Let \(L_{41} \coloneqq (\zeta_4, \zeta_1)\) and \(L_{32} \coloneqq (\zeta_3, \zeta_2)\) be the two ``opposite sides'' of the ideal quadrilateral (they exist and are unique by Lemma~\ref{lem:wK_is_visibility}, and are disjoint).
    Choose the rays in the definition \eqref{eq:cr_rat_def} on the sides: \(x_1(T), x_4(T)\) the points of \(L_{41}\) at parameter \(\pm T\) converging to \(\zeta_1, \zeta_4\) respectively, and \(x_3(T), x_2(T)\) the points of \(L_{32}\) converging to \(\zeta_3, \zeta_2\); this is legitimate by the ray-independence in Lemma~\ref{lem:cr0_exists}.

    Let \(G^*\) be the set of geodesics \(\eta\) which
    \begin{enumerate}
        \item[$\bullet$] are not tangent to and share no arc with \(L_{41}, L_{32}\) (so \(\eta\) meets each of them in at most one point, transversally, by the branching criterion of Lemma~\ref{pf:lem_directions}(iv) and uniqueness of segments);
        \item[$\bullet$] have no endpoint in \(\{\zeta_1, \zeta_2, \zeta_3, \zeta_4\}\);
        \item[$\bullet$] pass through none of the countably many points \(x_i(T)\), \(T \in \mathbb{Q}\).
    \end{enumerate}
    By Lemmas~\ref{pf:lem_null_charts} and~\ref{pf:lem_endpoint_null}, \(G^*\) is conull.
    For \(\eta \in G^*\) and any \(T\) with \(x_i(T) \notin \eta\) (all but at most four values of \(T\), namely, the crossing parameters), each \(x_i(T)\) has a well-defined side of \(\eta\) (the components of \(X \setminus \eta\); \(\eta\) is embedded and closed, and \(X\) is a plane).

    Fix \(T\) with \(x_i(T) \notin \eta\) for the four indices, write \(x_i = x_i(T)\), and set \([x_i | x_j] \in \{0, 1\}\) for the indicator of ``opposite sides''.
    Off a null set, \(\eta\) crosses the segment \([x_i, x_j]\) if and only if \([x_i|x_j] = 1\) (indeed, a crossing separates; conversely, separated endpoints force an odd, hence exactly one, transversal crossing, while tangencies and shared arcs with the segments are null configurations).
    By the Crofton formula (Lemma~\ref{pf:lem_crofton}), \(d(x_i, x_j) = \widehat m_0\bigl(\{[x_i|x_j] = 1\}\bigr)\), hence we have
    \[
        E_0(x_1, x_2, x_3, x_4) = \int s_T(\eta)\, d\widehat m_0(\eta), \quad
        s_T \coloneqq [x_1|x_2] + [x_3|x_4] - [x_1|x_4] - [x_3|x_2].
    \]
    Note that the integrand is integrable since each crossing set has finite mass.
    Enumerating the partitions of \(\{x_1, x_2, x_3, x_4\}\) into sides shows that \(s_T = 0\) unless the partition is \(\{x_1, x_4\} \mid \{x_2, x_3\}\), where \(s_T = 2\), or \(\{x_1, x_2\} \mid \{x_3, x_4\}\), where \(s_T = -2\).
    Let us write \(S_+(T)\), \(S_-(T)\) for the corresponding sets of geodesics; we have
    \begin{equation}\label{pf:eq_EO_split}
        E_0\bigl(x_1(T), x_2(T), x_3(T), x_4(T)\bigr) = 2\,\widehat m_0\bigl(S_+(T)\bigr) - 2\,\widehat m_0\bigl(S_-(T)\bigr).
    \end{equation}

    For \(\eta \in G^*\): if two points of \(L_{41}\) lie on the same side of \(\eta\), then the whole segment of \(L_{41}\) between them lies on that side; indeed \(\eta\) crosses \(L_{41}\) at most once, and a crossing of the sub-segment would put its endpoints on opposite sides.
    Thus, for \(T < T'\) (avoiding the finitely many exceptional parameters of \(\eta\)), we have
    \begin{enumerate}[label=(\alph*)]
        \item \(\eta \in S_+(T') \Rightarrow \eta \in S_+(T)\): the pairs \(\{x_1(T'), x_4(T')\}\) and \(\{x_3(T'), x_2(T')\}\) lie on their respective sides, hence so do the sub-segments \([x_4(T'), x_1(T')] \ni x_1(T), x_4(T)\) and \([x_3(T'), x_2(T')] \ni x_3(T), x_2(T)\).
        \item \(\eta \in S_+(T)\), \(T \geq 1\), implies that \(\eta\) separates \(x_1(1)\) from \(x_3(1)\) (same sub-segment argument), i.e., \(\eta\) crosses the fixed compact segment \([x_1(1), x_3(1)]\); by Crofton, the sets \(S_+(T)\), \(T \geq 1\), share the finite-mass dominator \(D_+ \coloneqq \{\eta \pitchfork [x_1(1), x_3(1)]\}\).
        \item \(\eta \in S_-(T)\) implies that \(\eta\) separates \(x_1(T)\) from \(x_4(T)\), both on \(L_{41}\), hence crosses \(L_{41}\); in the same way, it crosses \(L_{32}\).
              A geodesic crossing both \(L_{41}\) and \(L_{32}\) has its endpoints linked with both \((\zeta_4,\zeta_1)\) and \((\zeta_3,\zeta_2)\), which forces (up to the null endpoint fans,  see Lemma~\ref{pf:lem_endpoint_null}) one endpoint in the closed arc \([\zeta_4\zeta_1]\) and the other in \([\zeta_3\zeta_2]\), which are two disjoint closed arcs.
              By Corollary~\ref{pf:cor_disjoint_arcs}, the sets \(S_-(T)\) share a finite-mass dominator \(D_-\).
    \end{enumerate}
    We claim that, for \(\eta \in G^*\),
    as \(T \to \infty\),
    \[
        \mathbf{1}_{S_+(T)}(\eta) \longrightarrow \mathbf{1}_{R}(\eta),
        \qquad
        \mathbf{1}_{S_-(T)}(\eta) \longrightarrow 0 .
    \]
    Let us show the first convergence: if \(\eta \in R\), its endpoints are interior to the arcs \((\zeta_1\zeta_3)\), \((\zeta_2\zeta_4)\); then \(\eta\)'s endpoint pair is unlinked with \((\zeta_4, \zeta_1)\) and with \((\zeta_3, \zeta_2)\), so \(\eta\) crosses neither side; each side, being connected and disjoint from \(\eta\), lies in one component of \(X \setminus \eta\), and the component is the one whose ideal arc contains its endpoints: \(L_{41}\) on the side of the arc containing \(\zeta_4, \zeta_1\), and \(L_{32}\) on the other side (the arcs of \(\eta\) split \(\{\zeta_4, \zeta_1\}\) from \(\{\zeta_3, \zeta_2\}\) since \(\eta \in R\)).
    Hence \(\eta \in S_+(T)\) for every non-exceptional \(T\).
    Conversely, if \(\eta \in S_+(T_k)\) for integers \(T_k \to \infty\): after passing to a subsequence, the assignment of pairs to the two (fixed) sides of \(\eta\) is constant; since \(x_i(T_k) \to \zeta_i\) in \(\overline{X}\) and the closures of the two components of \(X \setminus \eta\) meet \(\partial_\infty X\) in the two closed arcs bounded by the endpoints of \(\eta\), we get \(\zeta_1, \zeta_4\) in one closed arc and \(\zeta_2, \zeta_3\) in the other; as \(\eta \in G^*\) has no endpoint among the \(\zeta_i\), the points lie in the  open arcs, and reading off the cyclic order, we get one endpoint of \(\eta\) in \((\zeta_1\zeta_3)\), the other in \((\zeta_2\zeta_4)\), i.e., \(\eta \in R\).

    For the second one: if \(\eta \in S_-(T_k)\) along a subsequence with constant side assignment, the same argument puts \(\zeta_1, \zeta_2\) in one closed arc of \(\eta\) and \(\zeta_3, \zeta_4\) in the other; but in the cyclic order \(\zeta_1, \zeta_3, \zeta_2, \zeta_4\) any arc containing \(\zeta_1, \zeta_2\) and avoiding \(\zeta_3, \zeta_4\) is impossible unless one of the \(\zeta_i\) is an endpoint of \(\eta\), which is excluded on \(G^*\).

    By dominated convergence, with dominators \(D_\pm\) from (b), (c), we conclude that
    \[
        \widehat m_0(S_+(T)) \to \widehat m_0(R),\quad\text{and}\quad \widehat m_0(S_-(T)) \to 0
    \]
    along integers; by \eqref{pf:eq_EO_split} and the definition \eqref{eq:cr_rat_def},
    \[
        2 \ln \operatorname{cr}_0(\zeta_1,\zeta_2,\zeta_3,\zeta_4) = \lim_{T \to \infty} E_0\bigl(x_1(T), x_2(T), x_3(T), x_4(T)\bigr) = 2\, \widehat m_0(R).\qedhere
    \]
\end{proof}

\subsection{The spectrum determines the cross-ratio}
\label{subsec:spectrum_crossratio}

In this part, we use the notions introduced in Section~\ref{sec:geodesic}, namely, the classification of isometries of \(\mathrm{CAT}(0)\) spaces, Busemann cocycles \(B_\xi\), Gromov products and the cross-ratio \(\operatorname{cr}_0\) of~\eqref{eq:cr_rat_def}, which is well defined (Lemma~\ref{lem:cr0_exists}) and continuous (Corollary~\ref{cor:cr0_continuous}). By Lemma~\ref{lem:pi1_hyperbolic}, every non-trivial \(g \in \pi_1(K)\) is a hyperbolic isometry of \(X\), with a unique axis \(A_g = (g^-, g^+)\) on which it translates by \(\ell(g) = \mathbf{EL}([g])\), and with North--South dynamics on \(\partial_\infty X\). Throughout, \(g, h \in \pi_1(K)\) are two such elements with four pairwise distinct fixed points \(g^\pm, h^\pm\) (equivalently, by Lemma~\ref{lem:pi1_hyperbolic}(iv), with distinct axes), and \(a \in A_g\), \(b \in A_h\) are base points on the axes. The goal is to express \(\operatorname{cr}_0(g^-, g^+, h^-, h^+)\) through the translation lengths \(\ell(g^nh^n)\), \(\ell(g)\), \(\ell(h)\).

\begin{lem}
    \label{pf:lem_busemann_length}
    For any hyperbolic isometry \(w\) and any \(x \in X\), we have \(\ell(w)=B_{w^+}(x, wx)\).
\end{lem}
\begin{proof}
    By naturality of Busemann cocycles (Definition~\ref{defi:busemann}(iii)), with \(\psi = w\), \(\xi = w^+\): \(B_{w^+}(wu, wv) = B_{w^+}(u,v)\).
    For \(a \in A_w\), \(wa\) lies on the ray from \(a\) to \(w^+\) at distance \(\ell(w)\), so \(B_{w^+}(a, wa) = \ell(w)\); for general \(x\), the cocycle relation gives \(B_{w^+}(x, wx) = B_{w^+}(x,a) + \ell(w) + B_{w^+}(wa, wx) = B_{w^+}(x,a) + \ell(w) + B_{w^+}(a, x) = \ell(w)\).
\end{proof}

\begin{lem}
    \label{pf:lem_mixed_displacement}
    For \(y, z \in X\), we have \(d(y, g^n z) = n \ell(g) + B_{g^+}(y, a) + B_{g^-}(z, a) + o(1)\) as \(n \to \infty\), locally uniformly in \((y, z)\).
\end{lem}
\begin{proof}
    Write \(d(y, g^nz) - n\ell(g) = \bigl[d(y, g^nz) - d(a, g^nz)\bigr] + \bigl[d(a, g^nz) - d(a, g^na)\bigr]\).
    Since \(d(g^nz, g^na) = d(z,a)\) is bounded and \(g^na \to g^+\) radially, \(g^nz \to g^+\) in the cone topology; by horofunction continuity (Definition~\ref{defi:busemann}(v)), the first bracket converges to \(B_{g^+}(y, a)\), uniformly for \(y\) in compact sets.
    Moreover, it is \(2\)-Lipschitz in \(z\) with a \(z\)-independent limit, so the convergence is locally uniform in \((y,z)\) (by equi-Lipschitzness).
    The second bracket equals \(d(g^{-n}a, z) - d(g^{-n}a, a) \to B_{g^-}(z, a)\), locally uniformly in \(z\), by the same fact.
\end{proof}

\begin{lem}\label{pf:lem_gap}
    For \(\xi\neq\eta \in \partial_\infty X\) and base points \(c, c' \in X\), we consider the gap
    \[
        G(\xi, \eta; c, c') \coloneqq B_\xi(p, c) + B_\eta(p, c'),
    \]
    where \(p\) is any point on the line \((\xi, \eta)\)).
    Note that \(G(\xi, \eta; c, c') = G(\eta, \xi; c', c)\).
    Then
    \[
        \inf_{y \in X}\ \bigl[B_\xi(y, c) + B_\eta(y, c')\bigr] = G(\xi, \eta; c, c'),
    \]
    the infimum being attained on the line \((\xi, \eta)\).
    Moreover, if \(a\) moves along the axis \(A_g\) and \(b\) along \(A_h\), the sum \(G(g^+, h^-; a, b) + G(g^-, h^+; a, b)\) is unchanged.
\end{lem}
\begin{proof}
    Along \((\xi,\eta)\), parametrised towards \(\xi\), \(B_\xi(\cdot, c)\) decreases and \(B_\eta(\cdot, c')\) increases at unit speed, so the sum is constant on the line; the swap symmetry is clear from the definition.
    For arbitrary \(y\) and \(p = \gamma(0)\) on the line \(\gamma = (\xi, \eta)\), \(\gamma(+\infty) = \xi\):
    \begin{align*}
         & \bigl[B_\xi(y,c) + B_\eta(y,c')\bigr] - G(\xi, \eta; c, c') = B_\xi(y, p) + B_\eta(y, p) \\
         & = \lim_{t,s\to\infty} \bigl[d(y, \gamma(t)) + d(y, \gamma(-s)) - (t+s)\bigr] \geq 0
    \end{align*}
    by the triangle inequality, since \(t + s = d(\gamma(t), \gamma(-s))\).
    Finally, replacing \(a\) by a point at parameter \(t_0\) further on \(A_g\) changes \(B_{g^+}(\cdot, a)\) by \(+t_0\) and \(B_{g^-}(\cdot, a)\) by \(-t_0\), leaving the sum unchanged; we argue similarly for \(b\).
\end{proof}

\begin{lem}\label{pf:lem_localisation}
    Write \(w_n \coloneqq g^nh^n\) and fix \(p^* \in (h^-, g^+)\).
    Then:
    \begin{enumerate}[label=(\roman*)]
        \item \(w_n^+ \to g^+\) and \(w_n^- \to h^-\);
        \item let \(y_n\in A_{w_n}\) be nearest to \(p^*\); then the sequence \((y_n) \) is  bounded;
        \item let \(\zeta_n \in [y_n, w_ny_n] \subset A_{w_n}\) be nearest to \(g^np^*\); the points \(z_n \coloneqq g^{-n}\zeta_n\) form a bounded sequence.
    \end{enumerate}
\end{lem}
\begin{proof}
    (i) Hyperbolic isometries of a proper \(\delta\)-hyperbolic space act on the boundary with North--South dynamics (by fact~(d) following Definition~\ref{defi:isometries_cat0}; see also~\cite[Ch.\ 10]{coornaert_geometrie_1990}).
    Choose disjoint closed neighbourhoods \(U^\pm \ni g^\pm\), \(V^\pm \ni h^\pm\) in \(\overline{X}\), with \(V^+ \cap U^- = \emptyset\).
    For large \(n\), \(h^n(\overline{X}\setminus V^-) \subset V^+\) and \(g^n(\overline{X} \setminus U^-) \subset U^+\), so \(w_n(\overline{X} \setminus V^-) \subset U^+\); iterating, \(w_n^k(\overline X \setminus V^-) \subset U^+\) for all \(k \geq 1\), and since \(w_n^k x \to w_n^+\) as \(k \to \infty\) for every \(x \in X \cap U^+\), while \(w_n^k(U^{+})\subset U^{+}\) with \(U^+\) closed, we have \(w_n^+ \in U^+\).
    Applying this to \(w_n^{-1} = h^{-n}g^{-n}\) gives \(w_n^- \in V^-\).
    By shrinking the neighbourhoods, we obtain (i).

    (ii) The axes \(A_{w_n} = (w_n^-, w_n^+)\) have endpoints converging to \((h^-, g^+)\); by Lemma~\ref{pf:lem_conv_geodesics} they converge to \((h^-, g^+) \ni p^*\) uniformly on compacts, so \(d(p^*, A_{w_n})\) is bounded.

    (iii) In a \(\delta\)-hyperbolic space (recall Corollary~\ref{cor:hyperbolic}), we have \(d\bigl(m, [u,v]\bigr) \leq (u \mid v)_m + 4\delta\) for any geodesic segment \cite[Ch.\ 3]{coornaert_geometrie_1990}.
    With \(m = g^np^*\), \(u = y_n\), \(v = w_ny_n\), and using invariance of Gromov products (Definition~\ref{defi:gromov_product}) under the isometry \(g^{-n}\), we obtain
    \[
        \bigl(y_n \mid w_ny_n\bigr)_{g^np^*} = \bigl(g^{-n}y_n \mid h^ny_n\bigr)_{p^*} .
    \]
    Since the \(y_n\) are bounded, \(g^{-n}y_n \to g^-\) and \(h^ny_n \to h^+\), which are two distinct boundary points (Definition~\ref{defi:gromov_product}(iv)), hence the Gromov products are bounded, and \(d(z_n, p^*) = d(\zeta_n, g^np^*) = d\bigl(g^np^*, [y_n, w_ny_n]\bigr) \leq C\).
\end{proof}

\begin{thm}\label{pf:thm_length_asymptotics}\label{lem:bourdon_asymptotics}
    We have the following length asymptotics:
    \[
        \lim_{n\to\infty}\Bigl[\ell(g^nh^n) - n\ell(g) - n\ell(h)\Bigr]
        = G(g^+, h^-; a, b) + G(g^-, h^+; a, b).
    \]
\end{thm}
\begin{proof}
   Let us start by showing the upper bound. Fix \(y \in (g^+, h^-)\) and \(z \in (g^-, h^+)\).
    For every \(n\), we have 
    \[
        \ell(w_n) \leq d(y, w_ny) \leq d(y, g^nz) + d(g^nz, g^nh^ny) = d(y, g^nz) + d(z, h^ny),
    \]
    and by Lemma~\ref{pf:lem_mixed_displacement} (for \(g\), and for \(h\) with base \(b \in A_h\)),
    \begin{align*}
        \limsup_n \bigl[\ell(w_n) - n\ell(g) - n\ell(h)\bigr]
         & \leq \bigl[B_{g^+}(y, a) + B_{h^-}(y, b)\bigr] + \bigl[B_{g^-}(z, a) + B_{h^+}(z, b)\bigr] \\
         & = G(g^+, h^-; a,b) + G(g^-, h^+; a,b),
    \end{align*}
    where the equalities hold because \(y, z\) lie on the respective lines (Lemma~\ref{pf:lem_gap}).

    Let us now show the lower bound. With \(y_n, \zeta_n, z_n\) as in Lemma~\ref{pf:lem_localisation}: since \(\zeta_n = g^nz_n\) lies on the segment \([y_n, w_ny_n]\), whose length is \(\ell(w_n)\), we have 
    \[
        \ell(w_n) = d(y_n, g^nz_n) + d(g^nz_n, w_ny_n) = d(y_n, g^nz_n) + d(z_n, h^ny_n) .
    \]
    The sequences \((y_n), (z_n)\) are bounded, so Lemma~\ref{pf:lem_mixed_displacement} applies locally uniformly, and each of the two resulting brackets is bounded below by the corresponding gap (by Lemma~\ref{pf:lem_gap} at \(y_n\), \(z_n\)); hence \(\liminf_n \geq G(g^+,h^-;a,b) + G(g^-,h^+;a,b)\).
\end{proof}

Gathering the previous results, we can conclude that the spectrum determines cross-ratios:
\begin{cor}
\label{pf:cor_cr_from_spectrum}
    We have
    \begin{equation}\label{convergence_cr_rat}
        \ln\operatorname{cr}_0\bigl(g^-, g^+, h^-, h^+\bigr) = -\tfrac12 \lim_{n\to\infty}\Bigl[\ell(g^nh^n) - n\ell(g) - n\ell(h)\Bigr].
    \end{equation}
\end{cor}
\begin{proof}
    For \(\xi, \eta \in \partial_\infty X\) distinct and base points \(c, c' \in X\), let \(u, v\) be the rays from \(c\), \(c'\) to \(\xi\), \(\eta\), and set \(P(\xi, \eta; c, c') \coloneqq \lim_{t}\bigl[2t - d(u_t, v_t)\bigr]\); the diagonal limit agrees with the double limit \(\lim_{t,s}[t+s-d(u_t,v_s)]\), both expressions being bounded and monotone in each variable by first variation (as in Lemma~\ref{lem:cr0_exists}).
    By the identity established in the proof of Corollary~\ref{cor:cr0_continuous}, for every \(p\) on the line \((\xi,\eta)\),
    \begin{equation}\label{pf:eq_P_equals_G}
        P(\xi, \eta; c, c') = -\bigl[B_\xi(p,c) + B_\eta(p,c')\bigr] = -\,G(\xi, \eta; c, c').
    \end{equation}
    Let us now evaluate \eqref{eq:cr_rat_def} with the rays \(\gamma_1(t) = \sigma_1(-t)\), \(\gamma_2(t) = \sigma_1(t)\), \(\gamma_3(t) = \sigma_2(-t)\), \(\gamma_4(t) = \sigma_2(t)\), where \(\sigma_1 = A_g\) is parametrised towards \(g^+\) with \(\sigma_1(0) = a\), and \(\sigma_2 = A_h\) towards \(h^+\) with \(\sigma_2(0) = b\) (ray-independence, Lemma~\ref{lem:cr0_exists}).
    With \(f_{ij}(t) = d(\gamma_i(t), \gamma_j(t)) - 2t\): \(f_{12} = f_{34} \equiv 0\), \(\lim f_{14} = -P(g^-, h^+; a, b)\), \(\lim f_{32} = -P(h^-, g^+; b, a)\), so
    \begin{align*}
        \ln\operatorname{cr}_0(g^-,g^+,h^-,h^+) = \tfrac12\lim_t E_0
         & = \tfrac12\bigl[P(g^-,h^+;a,b) + P(h^-,g^+;b,a)\bigr]   \\
         & = -\tfrac12\bigl[G(g^-,h^+;a,b) + G(h^-,g^+;b,a)\bigr].
    \end{align*}
    By Lemma~\ref{pf:lem_gap}, we have \(G(h^-, g^+; b, a) = G(g^+, h^-; a, b)\), and by Theorem~\ref{pf:thm_length_asymptotics}, this concludes the proof.
\end{proof}

\begin{rmk}\label{rem:antecedents_lengths}
    The principle that cross-ratios can be recovered from length data (as in Theorem~\ref{lem:bourdon_asymptotics} and Corollary~\ref{pf:cor_cr_from_spectrum}) goes back to Otal's proof that the marked length spectrum of a negatively curved surface determines its Liouville cross-ratio~\cite[Th.~2.2]{otal_spectre_1990}. It was extended to \(\mathrm{CAT}(-1)\) spaces by Bourdon~\cite{bourdon_structure_1995, bourdon_sur_1996}, whose formulation of the cross-ratio through Busemann functions is the one we use here, to \(\mathrm{CAT}(-1)\) spaces with cone singularities by Hersonsky and Paulin~\cite{hersonsky_rigidity_1997}, and to symmetric spaces by Dal'bo and Kim~\cite{dalbo_marked_2002, kim_marked_2001, kim_cross_2010}. Otal's global scheme was carried over to non-positively curved surfaces with cone singularities by Constantine~\cite{constantine_marked_2018} (see also~\cite{duchin_length_2010, bankovic_marked-length-spectral_2017} for flat cone metrics). Theorem~\ref{lem:bourdon_asymptotics} is the version of this formula for \(\widetilde{K}_0\), a \(\mathrm{CAT}(0)\) space with flat pieces, and its proof uses only Gromov hyperbolicity and the absence of flat strips.
    
    For hyperbolic billiards, and in particular for open dispersing billiards, similar asymptotic expansions were already considered, e.g.\ in~\cite{balint2020marked, de2023marked}.
    In this purely hyperbolic setting, the goal was to recover periodic multipliers from length data through exponentially small error terms in such expansions.
    Here, we do not claim any estimate on the speed of convergence in~\eqref{convergence_cr_rat}; instead, we focus directly on the limit of the expansions.
\end{rmk}

\subsection{Preservation of angles}\label{subsec:angle_dissipation}\label{subsec:synchronised_conjugacy}

\par In the following, we suppose that \( \mathcal{D} \) and \( \mathcal{D}' \) are two distinct billiard tables in \( \mathfrak{D}^m \approx C^r_{\text{emb}}(\mathbb{T}_\mathcal{D}, \mathbb{T}^2) \).
We consider the latter as a metric space with the \( C^r \) metric. 
Remark that the tables \( \mathcal{D} \) and \( \mathcal{D}' \) are diffeomorphic. Indeed, the space of configurations of \(m\) pairwise disjoint closed discs in \(\mathbb{T}^2\) is connected (we can shrink the discs to points, use the connectedness of the configuration space of \(m\) distinct points of \(\mathbb{T}^2\), and enlarge the points back to discs); since the obstacles of \(\mathcal{D}\) and \(\mathcal{D}'\) are isotopic to discs through strictly convex domains, there is an isotopy of embeddings of \(\mathbb{T}_\mathcal{D} = \bigsqcup_{i=1}^m \mathbb{S}^1\) carrying \(\partial\mathcal{D}\) onto \(\partial\mathcal{D}'\), and by the isotopy extension theorem (see e.g.~\cite[Chapter 8, Theorem 1.3]{hirsch2012differential}) it extends to a diffeotopy of \( \mathbb{T}^2 \), whose time-one map sends \(\mathcal{D}\) onto \(\mathcal{D}'\).
\par We will consider surfaces \( K \subset \mathcal{K}(\mathcal{D}), \ K' \subset \mathcal{K}(\mathcal{D}') \), obtained by gluing two copies of \( \mathcal{D} \), respectively \( \mathcal{D}' \), that is
\[
    K \setminus B(K) = \mathcal{D}^\mathrm{up} \coprod \mathcal{D}^\mathrm{down}, \qquad K' \setminus B(K') = (\mathcal{D}')^\mathrm{up} \coprod (\mathcal{D}')^\mathrm{down}.
\]
The surfaces \( K \) and \( K' \) are also diffeomorphic.
Indeed, \( p(\mathcal{D}^\mathrm{up}) \approx \mathcal{D} \approx \mathcal{D}' = p((\mathcal{D}')^\mathrm{up}) \), and similarly for \( \mathcal{D}^\mathrm{down} \).
As \( \partial \mathcal{D} \approx \partial \mathcal{D}' \), we obtain a diffeomorphism \( h\colon K \to K' \), as wanted.
We fix such a diffeomorphism once and for all; it induces an isomorphism \(\varphi \coloneqq h_*\colon \pi_1(K) \to \pi_1(K')\), the~\emph{marking}, well defined up to the action of the mapping class group.

\par In what follows we assume that \(\mathcal{D}\) and \(\mathcal{D}'\) have the same enriched length spectrum:
\[
    \mathbf{EL}_{\mathcal{D}}=\mathbf{EL}=\mathbf{EL}' = \mathbf{EL}_{\mathcal{D}'}.
\]
The hypothesis \(\mathbf{EL} = \mathbf{EL}'\) is understood with respect to the marking \(\varphi\), that is, \(\mathbf{EL}'(\varphi(\alpha)) = \mathbf{EL}(\alpha)\) for every conjugacy class \(\alpha \in \mathcal{C}(K)\).
The conclusion of the main theorem (the existence of an isometry of \(\mathbb{T}^2\)) does not depend on the marking, while the hypothesis does.
\par In what follows, we will fix \( K \) as our reference surface.
As in Section \ref{sec:geodesic}, we will consider families of metrics \( g_\epsilon \), \( g'_\epsilon \) on \( K \) converging to the gluing metric when \( \epsilon \) approaches \( 0 \).
More specifically, \( g = g_1 \) is our reference metric on \( K \), induced by the Euclidean metric of \( E = \mathbb{T}^2 \times \mathbb{R} \).
For each of the tables, the metrics \( g_\epsilon \), \( g'_\epsilon \) are defined as the pullback to \( K \), under the mapping \(\alpha_\epsilon\colon (x,y,z) \mapsto (x,y,\epsilon z) \), of the metric induced on \(\alpha_\epsilon(K) \), \(\alpha_\epsilon(K') \), respectively.
The \( g'_\epsilon \) are seen as metrics in \(K\) by pulling back with the diffeomorphism between the surfaces.
\par We have metric convergence \( (K, g_\epsilon) \eqqcolon K_\epsilon \xrightarrow{GH} K_0 \coloneqq (K, g_0) \), and similarly \( (K, g'_\epsilon) \eqqcolon K'_\epsilon \xrightarrow{GH} K_0' \coloneqq (K, g'_0) \), in the Gromov--Hausdorff distance, as \( \epsilon \to 0 \).
The metric space \( K_0 \) (respectively \( K_0' \)) consists of two copies of \( (\mathcal{D}, d_\mathcal{L}) \) (respectively \( (\mathcal{D}', d_\mathcal{L}) \)) glued at the boundary.
In the following, we abbreviate \(X = \widetilde{K}_0\), \(X' = \widetilde{K}'_0\). We denote by \(\mathcal{G}_0\), \(\mathcal{G}'_0\) the spaces of unparametrised oriented geodesics, and by \(\mathcal{G}(\mathcal{D})\), \(\mathcal{G}(\mathcal{D}')\) the subsets of admissible ones. 

\begin{lem}[Boundary identification]\label{lem:boundary_map}
    The marking isomorphism \(\varphi\colon \pi_1(K) \to \pi_1(K')\) induces a homeomorphism
    \[
        \partial\varphi\colon \partial_\infty X\longrightarrow \partial_\infty X',
    \]
    which is equivariant with respect to the actions of \(\pi_1(K)\) (on the right via \(\varphi\)), and mapping, for every hyperbolic \(g \in \pi_1(K)\), the fixed pair \((g^-, g^+)\) to \((\varphi(g)^-, \varphi(g)^+)\).
\end{lem}
\begin{proof}
    Both actions are free, properly discontinuous and cocompact by isometries on the proper geodesic spaces \(X\), \(X'\), which are Gromov-hyperbolic (Lemma~\ref{lem:uniform_hyperbolicity} and its analogue for \(\mathcal{D}'\)).
    By the \v{S}varc--Milnor lemma~\cite[Proposition I.8.19]{bridson2011metric} the orbit maps are quasi-isometries, so \(\varphi\) induces an equivariant quasi-isometry \(X \to X'\), and every quasi-isometry between hyperbolic proper geodesic spaces extends to a homeomorphism of the boundaries~\cite[Theorem III.H.3.9]{bridson2011metric}.
    Equivariance gives the statement about fixed pairs, since \(g^\pm\) are the unique fixed points of \(g\) on the boundary.
    Finally, \(\partial\varphi\) does not depend on the choice of orbit maps: two equivariant quasi-isometries induced by \(\varphi\) are at bounded distance from each other, hence induce the same boundary map.
\end{proof}

\begin{rmk}\label{remark_identif}
One could also try to identify the boundaries \(\partial_\infty X\) and \(\partial_\infty X'\) through the Anosov approximations.
Indeed, an isotopy \((\mathcal{D}^t)_{t \in [0,1]}\) of tables from \(\mathcal{D}\) to \(\mathcal{D}'\) induces an isotopy \((K^t)_{t\in [0,1]}\) of surfaces, and if, for some fixed \(\epsilon > 0\), the geodesic flows of all the \((K^t, g_\epsilon)\) are Anosov, they are orbit equivalent by structural stability.
Their boundaries at infinity, hence those of the limit spaces (by Lemma~\ref{lem:homeomorphic_boundaries}), are then identified, and by uniqueness of equivariant boundary maps this identification coincides with \(\partial\varphi\).
However, two points would need to be checked: that the isotopy can be chosen through finite-horizon tables with strictly convex obstacles (finite horizon is an open condition, but not a convex one), and that Kourganoff's threshold \(\epsilon_0\) can be taken uniform along the isotopy.
The \v{S}varc--Milnor lemma avoids both, and uses only the marking \(\varphi\).
\end{rmk}

The next proposition says that the enriched spectrum determines the cross-ratio and the current:
\begin{pps}
\label{pf:cor_spectrum_crossratio}
    If \(\mathbf{EL} = \mathbf{EL}'\), then we have \(\operatorname{cr}'_0 \circ (\partial\varphi)^{\times 4} = \operatorname{cr}_0\) on quadruples of pairwise distinct boundary points.
    Therefore, \((\partial\varphi \times \partial\varphi)_* \widehat{m}_0 = \widehat{m}'_0\).
\end{pps}
\begin{proof}
    By Remark~\ref{cor:EL_is_translation_length}, the hypothesis gives \(\ell(w) = \ell'(\varphi(w))\) for every \(w \in \pi_1(K)\).
    Corollary~\ref{pf:cor_cr_from_spectrum}, applied to \((g,h)\) in \(X\) and to \((\varphi(g), \varphi(h))\) in \(X'\), together with \(\partial\varphi(g^\pm) = \varphi(g)^\pm\) (Lemma~\ref{lem:boundary_map}), gives the equality on all quadruples \((g^-, g^+, h^-, h^+)\).
    Given an ordered quadruple \((\zeta_1, \zeta_2, \zeta_3, \zeta_4)\) of pairwise distinct points, choose, by density of pole pairs (Corollary~\ref{cor:dense_periodic_liftings}), hyperbolic elements \(g_k, h_k\) with \((g_k^-, g_k^+) \to (\zeta_1, \zeta_2)\) and \((h_k^-, h_k^+) \to (\zeta_3, \zeta_4)\); both cross-ratios are continuous on pairwise distinct quadruples (Corollary~\ref{cor:cr0_continuous}) and \(\partial\varphi\) is a homeomorphism, so the identity passes to the limit.

    Let us now show the second statement.
    We claim that a flip-invariant Radon measure on \(\partial^2_\infty X\) giving zero mass to endpoint fans is determined by its values on the rectangles \(R\) of Proposition~\ref{pf:prop_rectangle}.
    Indeed, ordered products \(A \times B\) of open arcs with disjoint closures form a \(\pi\)-system generating the Borel \(\sigma\)-algebra of \(\partial^2 \setminus \Delta\), on which our measures are supported and \(\sigma\)-finite (Corollary~\ref{pf:cor_disjoint_arcs}).
    The unordered rectangles determine the ordered products by flip invariance, and arcs with common endpoints are handled by monotone approximation from inside, the endpoint fans being null by Lemma~\ref{pf:lem_endpoint_null}.
    Uniqueness follows from Dynkin's lemma.

    Since \(\partial\varphi\) transports  cross-ratios, by Proposition~\ref{pf:prop_rectangle} applied on both sides, the two measures of the statement agree on all rectangles; and \((\partial\varphi \times \partial\varphi)_* \widehat m_0\) gives zero mass to endpoint fans since \(\partial\varphi\) is a bijection.
\end{proof}

Below, we let \(\varphi\colon \pi_1(K) \to \pi_1(K')\) as in Lemma~\ref{lem:boundary_map}. 

\begin{dfn}[Unparametrised conjugacy]\label{pf:defi_h_unparam}
    For every geodesic \(\gamma \in \mathcal{G}_0\),
    let \(h\gamma \in \mathcal{G}'_0\) be the unique geodesic of \(X'\) with endpoints
    \[
        (h\gamma)(-\infty) = \partial\varphi\bigl(\gamma(-\infty)\bigr), \qquad (h\gamma)(+\infty) = \partial\varphi\bigl(\gamma(+\infty)\bigr),
    \]
    which exists and is unique by Lemma~\ref{lem:wK_is_visibility} applied to \(X'\).
\end{dfn}

\begin{lem}\label{pf:lem_h_unparam}
    The map \(h\colon \mathcal{G}_0 \to \mathcal{G}'_0\) is a homeomorphism of the spaces of unparametrised geodesics; it is \(\pi_1(K)\)-equivariant over \(\varphi\), it commutes with the flip, and preserves the Liouville currents: \(h_* \widehat{m}_0 = \widehat{m}'_0\). In particular,  \(h\bigl(\widehat{\mathcal{G}}(\mathcal{D})\bigr) = \widehat{\mathcal{G}}(\mathcal{D}')\).
\end{lem}
\begin{proof}
    Under the identifications \(\mathcal{G}_0 \approx \partial^2_\infty X\) and \(\mathcal{G}'_0 \approx \partial^2_\infty X'\) (Lemma~\ref{lem:wK_is_visibility} and Lemma~\ref{pf:lem_conv_geodesics}), \(h\) is \(\partial\varphi \times \partial\varphi\), which is a homeomorphism; equivariance and the commutation with the flip \((\xi, \eta) \mapsto (\eta, \xi)\) are inherited from \(\partial\varphi\) (Lemma~\ref{lem:boundary_map}). Moreover, \((\partial\varphi \times \partial\varphi)_* \widehat{m}_0 = \widehat{m}'_0\) follows from Proposition~\ref{pf:cor_spectrum_crossratio}. The homeomorphism \(h\) carries the support of the measure \(\widehat{m}_0\) to the support of its push-forward \(\widehat{m}'_0\); by Proposition~\ref{pf:prop_support}, we deduce that \(h\) maps \(\operatorname{supp}\widehat{m}_0 = \widehat{\mathcal{G}}(\mathcal{D})\) onto \(\operatorname{supp}\widehat{m}'_0 = \widehat{\mathcal{G}}(\mathcal{D}')\). 
\end{proof}

\begin{dfn}\label{pf:defi_clean}
    We call a geodesic \emph{clean} if it meets the seam only transversally.
\end{dfn}
Clean geodesics are neither degenerate nor tangent to a wall, they share no arc with any other geodesic (indeed, a shared compact arc has branch points at its ends, which are tangential contacts by Lemma~\ref{pf:lem_directions}(iv); moreover, a shared ray is excluded for pairs with distinct endpoints), and they cross any geodesic they are linked with in exactly one point (as linked pairs intersect, and two intersection points would give a shared segment).

Next lemma shows the stability of crossing points. 
\begin{lem}\label{pf:lem_crossing_stability}
    We work in \(X\) or \(X'\).
    Let \(\rho_n \to \rho\) uniformly on compacts with \((\rho_n(-\infty), \rho_n(\infty)) \to (\rho(-\infty), \rho(\infty))\), and let \(\tau\) be a geodesic with endpoints disjoint from those of \(\rho\), linked with \(\rho_n\) for all large \(n\), and meeting \(\rho\) in exactly one point \(x^*\).
    If \(x_n \in \tau \cap \rho_n\), then \(x_n \to x^*\).
\end{lem}
\begin{proof}
    Assume \(d(x^*, x_{n_k}) \to \infty\).
    Since \(x_{n_k} \in \tau\), the accumulation points of \((x_{n_k})\) in \(\overline{X}\) lie in \(\{\tau(\pm\infty)\}\); since \(x_{n_k} \in \rho_{n_k}\) with the \(\rho_{n_k}\) as in Lemma~\ref{pf:lem_conv_geodesics}, they also lie in \(\{\rho(\pm\infty)\}\) by part (ii) of that lemma, a contradiction, the two pairs being disjoint.
    Moreover, any sublimit \(\tilde x\) of the bounded sequence lies on \(\tau\) and on \(\rho\) (by uniform convergence on the compact containing the tail), hence \(\tilde x = x^*\); we conclude that \(x_n \to x^*\).
\end{proof}

Recall the \(\mathrm{CAT}(0)\) angle from \S\ref{subsec:limit_preliminaries}, and recall the measure \(\mu\)
of Corollary~\ref{pf:cor_dxdphi}, which is
supported on the admissible class (Proposition~\ref{pf:prop_support}). A point of \(\widehat{\mathcal{G}}_0/\pi_1(K)\) is a \emph{pointed geodesic} \((\gamma, t)\), i.e., a geodesic with a marked point \(\gamma(t)\); for \(\mu\)-a.e.\ \((\gamma,t)\) the point \(\gamma(t)\) is flat and, for a.e.\ \(\theta \in (0,\pi)\), there is a unique geodesic \(\eta_{\gamma,t,\theta}\) crossing \(\gamma\) at \(\gamma(t)\) with oriented angle \(\theta\) (Lemma~\ref{pf:lem_null_charts}); we write \(\eta_{\gamma,\theta}\) when \(t = 0\).

In the following, we say that a geodesic \(\eta\) \emph{crosses \(\gamma\) from the right} if \(\eta\) and \(\gamma\) are linked and the oriented crossing angle from \(\gamma\) to \(\eta\) lies in \((0,\pi)\); equivalently, the endpoints occur in the cyclic order \(\gamma^-,\eta^-,\gamma^+,\eta^+\) (the other cyclic order of a linked pair, \(\gamma^-,\eta^+,\gamma^+,\eta^-\), corresponds to crossing from the left). Let
\[
    \mathcal{P}^+ \coloneqq \bigl\{(\gamma,\eta)\in\widehat{\mathcal{G}}_0\times\widehat{\mathcal{G}}_0 : \eta \text{ crosses } \gamma \text{ from the right}\bigr\},
\]
with the diagonal action of \(\pi_1(K)\), and define \(\mathcal{P}^-\), \(\mathcal{P}'^{\pm}\) likewise. For \(\mu\)-a.e.\ pointed geodesic \((\gamma,t)\) and every \(\theta\in(0,\pi)\) there is a unique \(\eta_{\gamma,t,\theta}\) crossing \(\gamma\) at \(\gamma(t)\) from the right with angle \(\theta\), and every right crosser of \(\gamma\) arises this way for a unique \((t,\theta)\); thus
\begin{equation}\label{eq:pairs_coordinates}
    (\gamma,t,\theta)\longmapsto(\gamma,\eta_{\gamma,t,\theta})
\end{equation}
is an a.e.\ bijection from \(\widehat{\mathcal{G}}_0\times(0,\pi)\) onto \(\mathcal{P}^+\), equivariant for \(\pi_1(K)\), which acts trivially on \(\theta\).

Let us now define a product measure on crossing pairs:
\begin{lem}
\label{pf:lem_pairs_measure}
    \begin{enumerate}[label=(\roman*)]
        \item In the coordinates \eqref{eq:pairs_coordinates}, the restriction of \(\widehat m_0\otimes\widehat m_0\) to \(\mathcal{P}^+\) is \((\widehat m_0\otimes dt)\otimes\tfrac14\sin\theta\,d\theta\); hence the quotient of \((\widehat m_0\otimes\widehat m_0)|_{\mathcal{P}^+}\) by \(\pi_1(K)\) is \(\mu\otimes\tfrac14\sin\theta\,d\theta\). The same holds for \(\mathcal{P}^-\), \(\mathcal{P}'^{+}\) and \(\mathcal{P}'^{-}\), with \(\mu'\).
        \item The map \(h\times h\) sends \(\mathcal{P}^+\) onto \(\mathcal{P}'^{+}\) if \(\partial\varphi\) preserves the cyclic orientation of \(\partial_\infty X\), and onto \(\mathcal{P}'^{-}\) if it reverses it; in both cases it carries \((\widehat m_0\otimes\widehat m_0)|_{\mathcal{P}^+}\) to the restricted product measure of the image family, and descends to the quotients.
        \item Consequently \(\mu(\widehat{\mathcal{G}}_0/\pi_1(K)) = \mu'(\widehat{\mathcal{G}}'_0/\pi_1(K'))\), i.e.\ \(\operatorname{Area}(\mathcal{D})=\operatorname{Area}(\mathcal{D}')\). From now on \(\mu\) and \(\mu'\) are normalised to probability measures.
    \end{enumerate}
\end{lem}
\begin{proof}
    (i) By Fubini, \((\widehat m_0\otimes\widehat m_0)(A) = \int\widehat m_0(\{\eta:(\gamma,\eta)\in A\})\,d\widehat m_0(\gamma)\) for \(A\subset\mathcal{P}^+\). For a.e.\ fixed \(\gamma\), the right crossers of \(\gamma\) are parametrised by \((t,\theta)\), and by the flux formula \eqref{eq:kinematic_flux} along \(\gamma\) (valid along every straight segment and wall arc, by the remark following Proposition~\ref{pf:prop_consistency}) the measure \(\widehat m_0\) restricted to them is \(\tfrac14\sin\theta\,d\theta\,dt\). This shows the first claim; the quotient statement follows since \(\pi_1(K)\) acts on the \((\gamma,t)\) factor only and \((\widehat m_0\otimes dt)/\pi_1(K)=\mu\). The other three families are treated in the same way (for left crossers, we use the chart along the reversed geodesic, and the density is the same).

    (ii) \(\partial\varphi\times\partial\varphi\) preserves linking, and it preserves or reverses all cyclic orders of \(\partial_\infty X\) according to its orientation behaviour; by the characterisation of the side of crossing through cyclic order, \(h\times h\) maps \(\mathcal{P}^+\) onto \(\mathcal{P}'^{+}\) or \(\mathcal{P}'^{-}\). Since \(h_*\widehat m_0=\widehat m'_0\) (Lemma~\ref{pf:lem_h_unparam}), \((h\times h)_*(\widehat m_0\otimes\widehat m_0)=\widehat m'_0\otimes\widehat m'_0\), and restrictions to corresponding sets correspond; equivariance over \(\varphi\) gives the quotient statement.

    (iii) By (i) and (ii), the two quotient measures \(\mu\otimes\tfrac14\sin\theta\,d\theta\) and \(\mu'\otimes\tfrac14\sin\theta'\,d\theta'\) have the same total mass, i.e., \(\tfrac12\mu(\widehat{\mathcal{G}}_0/\pi_1(K)) = \tfrac12\mu'(\widehat{\mathcal{G}}'_0/\pi_1(K'))\). By Corollary~\ref{pf:cor_dxdphi}, these masses are equal to \(\pi\operatorname{Area}(\mathcal{D})\) and \(\pi\operatorname{Area}(\mathcal{D}')\).
\end{proof}

\begin{dfn}[Angle transfer]\label{pf:defi_transfer}
    For \(\mu\)-a.e.\ \((\gamma,t)\) and a.e.\ \(\theta\) with \(\eta_{\gamma,t,\theta}\) crossing \(\gamma\) from the right, the geodesics \(h\gamma\) and \(h\eta_{\gamma,t,\theta}\) are clean, linked, and cross at a unique point \(q\) (by Lemma~\ref{pf:lem_h_unparam} and Proposition~\ref{pf:prop_consistency}). We let
    \[
        \vartheta_{(\gamma,t)}(\theta) \coloneqq \angle_q\bigl(h\gamma,\ h\eta_{\gamma,t,\theta}\bigr) \in (0,\pi)
    \]
    be the angle between the forward directions. The same construction for \(h^{-1}\) defines \(\vartheta'_{(\gamma',s)}\) for \(\mu'\)-a.e.\ pointed geodesic of \(X'\).
\end{dfn}
By flip equivariance of \(h\), we have \(\vartheta_{(\gamma,t)}(\pi-\theta) = \pi - \vartheta_{(\gamma,t)}(\theta)\).

We now show that the angle transfer is monotone and superadditive.
\begin{lem}\label{pf:lem_superadd}
    For \(\mu\)-a.e.\ \((\gamma,t)\) and a.e.\ \(\theta_1, \theta_2 > 0\) with \(\theta_1 + \theta_2 < \pi\), writing \(\eta_1 = \eta_{\gamma,t,\theta_1}\) based at its crossing with \(\gamma\),
    \[
        \vartheta_{(\gamma,t)}(\theta_1 + \theta_2) \geq \vartheta_{(\gamma,t)}(\theta_1) + \vartheta_{(\eta_1, 0)}(\theta_2).
    \]
    In particular \(\theta \mapsto \vartheta_{(\gamma,t)}(\theta)\) is non-decreasing, and \(\vartheta_{(\gamma,t)}(\pi/2) = \pi/2\).
\end{lem}
\begin{proof}
    Let \(\eta_2 = \eta_{\gamma,t,\theta_1+\theta_2}\); at the flat point \(x = \gamma(t)\) it makes oriented angle \(\theta_2\) with \(\eta_1\). The three image geodesics are pairwise linked and cross pairwise exactly once: \(p_1 = h\gamma \cap h\eta_1\), \(p_2 =h\gamma\cap h\eta_2\), \(q = h\eta_1 \cap h\eta_2\). If \(p_1 = p_2\) then \(q = p_1\) and the claim is arc additivity in the circle of directions at that point (Lemma~\ref{pf:lem_directions}(ii)), with equality. Otherwise \(p_1p_2q\) is a geodesic triangle, such that the sum of its interior angles is at most \(\pi\)~\cite[II.1.7]{bridson2011metric}. The cyclic order \(\gamma^+, \eta_1^+, \eta_2^+, \gamma^-, \eta_1^-, \eta_2^-\) of the endpoints is preserved by \(\partial\varphi\), so \(q\) lies on the side \(S\) of \(h\gamma\) whose ideal arc contains \(\eta_1^+, \eta_2^+\), or on the opposite side. 
    
    If \(q \in S\), the forward rays of \(h\eta_i\) from \(q\) cannot recross \(h\gamma\), so \(p_1, p_2\) lie backward from \(q\) along both, and, since \(h\eta_2\) passes from the left of \(h\eta_1\) (at \(\eta_2^+\)) to its right, \(p_2\) lies forward of \(p_1\) on \(h\gamma\). Writing \(u\) for the forward direction of \(h\gamma\) and \(v_i\) for those of \(h\eta_i\), the interior angles are \(\angle_{p_1}(u,v_1) = \vartheta(\theta_1)\), \(\angle_{p_2}(-u,v_2) = \pi - \vartheta(\theta_1+\theta_2)\) (Lemma~\ref{pf:lem_directions}(i)), and \(\angle_q(-v_1,-v_2) = \angle_q(v_1,v_2) = \vartheta_{(\eta_1,0)}(\theta_2)\).
    
    If \(q \notin S\), then \(p_1, p_2\) lie forward from \(q\) along both image geodesics, \(h\eta_2\) passes from the right of \(h\eta_1\) to its left, so \(p_2\) lies backward of \(p_1\) on \(h\gamma\), and the interior angles are \(\angle_{p_1}(-u,-v_1) = \vartheta(\theta_1)\), \(\angle_{p_2}(u,-v_2) = \pi - \vartheta(\theta_1+\theta_2)\) and \(\angle_q(v_1,v_2) = \vartheta_{(\eta_1,0)}(\theta_2)\): the same three values. In both cases
     the claim is proved. Monotonicity follows since angles are non-negative, and \(\vartheta_{(\gamma,t)}(\pi/2) = \pi/2\) from the flip symmetry.
\end{proof}

For \(\mu\)-a.e.\ \((\gamma,t)\), the function \(\theta\mapsto\vartheta_{(\gamma,t)}(\theta)\) is defined for a.e.\ \(\theta\in(0,\pi)\) (Definition~\ref{pf:defi_transfer}), takes values in \((0,\pi)\), is non-decreasing on its domain (Lemma~\ref{pf:lem_superadd}), and satisfies \(\vartheta(\pi-\theta) = \pi-\vartheta(\theta)\). A function which is non-decreasing on a conull subset of \((0,\pi)\) agrees a.e.\ with a non-decreasing function defined on all of \((0,\pi)\) (take \(\theta\mapsto\sup\{\vartheta(\theta'):\theta'\le\theta \text{ in the domain}\}\)); since every quantity below is an integral in \(\theta\), we replace \(\vartheta_{(\gamma,t)}\) by this extension, which is then non-decreasing everywhere. We also set \(\vartheta(\pi/2) \coloneqq \pi/2\), which is compatible with monotonicity since the symmetry gives \(\vartheta(\theta) \le \pi/2 \le \vartheta(\theta')\) for \(\theta < \pi/2 < \theta'\) in the domain. We have \(0\le\vartheta(0^+)\) and \(\vartheta(\pi^-)\le\pi\) (\(\vartheta\) may jump at the endpoints, which does not matter in what follows).   

We define its~\emph{generalised inverse} as
\[
    \vartheta^{-1}_{(\gamma,t)}(b) \coloneqq \sup\bigl\{\theta\in(0,\pi) : \vartheta_{(\gamma,t)}(\theta) < b\bigr\} \in [0,\pi], \qquad b\in(0,\pi),
\]
with \(\sup\emptyset = 0\). On the conull set of clean configurations the crossing point \(h\gamma\cap h\eta_{\gamma,t,\theta}\) depends continuously on \((\gamma,t,\theta)\) (Lemma~\ref{pf:lem_crossing_stability}) and the Alexandrov angle is the infimum of the comparison angles \(\overline{\angle}(\gamma(s), x, \gamma'(s'))\) over rational \(s, s' > 0\), these being continuous in the configuration and non-decreasing in \(s, s'\)~\cite[II.3.1]{bridson2011metric}, so \(\vartheta\) is Borel in \((\gamma,t,\theta)\); hence \(\vartheta^{-1}(b) = \sup\{\theta\in\mathbb{Q}\cap(0,\pi):\vartheta(\theta)<b\}\) is Borel in \((\gamma,t)\); in particular, Fubini can be applied to the integrals below.

\begin{lem}\label{pf:lem_gen_inverse}
    Let \(f\colon(0,\pi)\to(0,\pi)\) be non-decreasing with \(f(\pi/2)=\pi/2\), and let \(f^{-1}\) be its generalised inverse. Then:
    \begin{enumerate}[label=(\roman*)]
        \item for every \(b\in(0,\pi)\), \(\{\theta : f(\theta)<b\}\) is an interval with endpoints \(0\) and \(f^{-1}(b)\), so that \(\int_{\{f<b\}}\tfrac14\sin\theta\,d\theta = \tfrac14\bigl(1-\cos f^{-1}(b)\bigr)\);
        \item \(f^{-1}\) is non-decreasing, and \(f^{-1}(b)\le\pi/2\) for \(b\le\pi/2\);
        \item (Young's identity) \(\displaystyle\int_0^{\pi/2} f(\theta)\,d\theta + \int_0^{\pi/2} f^{-1}(b)\,db = \Bigl(\frac{\pi}{2}\Bigr)^2\).
    \end{enumerate}
\end{lem}
\begin{proof}
    (i) If \(f(\theta)<b\) and \(\theta'<\theta\) then \(f(\theta')\le f(\theta)<b\), so \(\{f<b\}\) is an initial interval, whose supremum is \(f^{-1}(b)\) by definition; whether the endpoint belongs to it is irrelevant for the integral. 
    
    (ii) Monotonicity of \(f^{-1}\) is clear from the definition; for \(b\le\pi/2\) and \(\theta>\pi/2\) we have \(f(\theta)\ge f(\pi/2)=\pi/2\ge b\), so such \(\theta\) do not enter the supremum. 
    
    (iii) For \(\theta,b\in(0,\pi/2)\) one has \(\theta<f^{-1}(b)\Rightarrow f(\theta)<b\Rightarrow\theta\le f^{-1}(b)\) (the first once because some \(\theta'>\theta\) has \(f(\theta')<b\), the second one by definition), so the sets \(\{(\theta,b):b\le f(\theta)\}\) and \(\{(\theta,b):\theta<f^{-1}(b)\}\) partition \((0,\pi/2)^2\) up to the null set \(\{\theta=f^{-1}(b)\}\). Since \(f\le\pi/2\) on \((0,\pi/2]\) by (ii) applied to \(f\) itself (\(f(\theta)\le f(\pi/2)=\pi/2\)), the first set has area \(\int_0^{\pi/2}f\), the second \(\int_0^{\pi/2}f^{-1}\), and their sum is the area of the square.
\end{proof}

\begin{lem}
\label{pf:lem_cosine}
    For every \(b \in (0,\pi)\),
    we have 
    \begin{align*}
        \int \cos\bigl(\vartheta^{-1}_{(\gamma,t)}(b)\bigr)\, d\mu(\gamma,t) &= \cos b,\\
        \int \cos\bigl(\vartheta'^{-1}_{(\gamma',s)}(b)\bigr)\, d\mu'(\gamma',s) &= \cos b .
    \end{align*}
\end{lem}
\begin{proof}
    Consider the set of crossing pairs whose image angle is \(< b\), modulo \(\pi_1\). By Lemma~\ref{pf:lem_pairs_measure} its mass can be computed in image coordinates: \(\int\!\int_0^b \tfrac14\sin\theta'\,d\theta'\,d\mu' = \tfrac14(1-\cos b)\). In source coordinates the same set is \(\{(\gamma,t,\theta): \vartheta_{(\gamma,t)}(\theta) < b\} = \{\theta < \vartheta^{-1}_{(\gamma,t)}(b)\}\) up to an endpoint, by Lemma~\ref{pf:lem_gen_inverse}(i), with mass \(\int \tfrac14\bigl(1-\cos\vartheta^{-1}_{(\gamma,t)}(b)\bigr)\,d\mu\). This gives the first identity. The same argument for \(h^{-1}\) gives the second one.
\end{proof}
 
 Gathering, we can now conclude that \(h\) is angle preserving. 
\begin{thm}\label{lem:angle_preservation}
    For \(\mu\)-a.e.\ \((\gamma,t)\) and every \(\theta \in (0,\pi)\), we have 
    \[
    \vartheta_{(\gamma,t)}(\theta) = \theta.
    \] 
    Therefore, by Lemma~\ref{pf:lem_crossing_stability} and the continuity of adjacent angles (their sum being \(\pi\) at a transversal crossing), for every pair of admissible geodesics crossing at a point \(x\), the images cross at a point \(q\) with \(\angle_q(h\gamma, h\eta) = \angle_x(\gamma,\eta)\).
\end{thm}
\begin{proof}
Following Otal~\cite{otal_spectre_1990}, let us define the averaged angle
\[
\Theta(\theta) \coloneqq \int \vartheta_{(\gamma,t)}(\theta)\,d\mu.
\] 
By Lemma~\ref{pf:lem_superadd} each \(\vartheta_{(\gamma,t)}\) is a monotone map of \((0,\pi/2]\) into itself fixing \(\pi/2\). For a.e.\ \(a \in (0,\pi/2)\),\footnote{By Fubini, the inequality of Lemma~\ref{pf:lem_superadd} at the pair \((a, \pi/2 - a)\) holds for \(\mu\)-a.e.\ \((\gamma,t)\) for a.e.\ such \(a\).} by integrating the superadditivity at \(a + (\pi/2 - a)\),\footnote{The base point of the second term is \(R_{a}(\gamma,t)\) in the phase-space coordinates of Corollary~\ref{pf:cor_dxdphi}, and \(\mu\) is invariant by the rotation \(R_{a}\).} we have \(\Theta(a) + \Theta(\pi/2 - a) \leq \Theta(\pi/2) = \pi/2\) for a.e.\ \(a\), hence
    \begin{equation}\label{eq:theta_upper}
        \int_0^{\pi/2}\bigl(\Theta(a) - a\bigr)\,da\leq 0 .
    \end{equation}
    On the other hand, for \(b \leq \pi/2\), the random variable \(Z_b \coloneqq \vartheta^{-1}_{(\gamma,t)}(b)\) takes values in \([0,\pi/2]\) (Lemma~\ref{pf:lem_gen_inverse}(ii)), where \(\cos\) is strictly concave. Lemma~\ref{pf:lem_cosine} and Jensen's inequality give \(\cos b = \int\cos Z_b\,d\mu \leq \cos\bigl(\int Z_b\,d\mu\bigr)\), i.e., 
    \begin{equation}\label{eq:jensen_cos}
        \int \vartheta^{-1}_{(\gamma,t)}(b)\,d\mu \leq b, \qquad \text{with equality if and only if } Z_b = b \ \mu\text{-a.s.}
    \end{equation}
    For each monotone map \(\vartheta\) of \([0,\pi/2]\) fixing the endpoints,
    Young's identity holds (Lemma~\ref{pf:lem_gen_inverse}(iii)); integrating it against \(\mu\) and using~\eqref{eq:jensen_cos}, we have 
    \[
        \int_0^{\pi/2}\bigl(\Theta(a) - a\bigr)\,da = \int_0^{\pi/2}\Bigl(b - \int \vartheta^{-1}_{(\gamma,t)}(b)\,d\mu\Bigr)db \geq 0 .
    \]
    With~\eqref{eq:theta_upper} both sides vanish, so the non-negative integrand is zero for a.e.\ \(b\), and the equality case of \eqref{eq:jensen_cos} gives \(\vartheta^{-1}_{(\gamma,t)}(b) = b\) for a.e.\ \(b \in (0,\pi/2)\) and \(\mu\)-a.e.\ \((\gamma,t)\).
        It remains to pass from \(\vartheta^{-1}_{(\gamma,t)}(b) = b\) for a.e.\ \(b\in(0,\pi/2)\) to \(\vartheta_{(\gamma,t)}(\theta)=\theta\) for every \(\theta\). Fix such a \((\gamma,t)\), write \(f=\vartheta_{(\gamma,t)}\), and let \(B\subset(0,\pi/2)\) be the conull set on which \(f^{-1}(b)=b\). Let \(\theta\in(0,\pi/2)\). For \(b\in B\) with \(b>\theta\), the definition \(f^{-1}(b)=\sup\{\theta'':f(\theta'')<b\}=b>\theta\) provides some \(\theta''>\theta\) with \(f(\theta'')<b\), hence \(f(\theta)\le f(\theta'')<b\) by monotonicity; letting \(b\downarrow\theta\) within \(B\) gives \(f(\theta)\le\theta\). For \(b\in B\) with \(b<\theta\), we have \(f^{-1}(b)=b<\theta\), so \(\theta\) does not belong to \(\{f<b\}\), i.e., \(f(\theta)\ge b\); letting \(b\uparrow\theta\) within \(B\) gives \(f(\theta)\ge\theta\). Thus \(f(\theta)=\theta\) for every \(\theta\in(0,\pi/2)\), and also at \(\pi/2\); the flip symmetry \(f(\pi-\theta)=\pi-f(\theta)\) then gives \(f=\mathrm{id}\) on \((0,\pi)\).
    
    For the extension, let \((\gamma,\eta)\) be admissible geodesics crossing at \(x\), possibly a seam point; by Proposition~\ref{pf:prop_support} there are pairs \((\gamma_n,\eta_n)\) of the conull set, crossing at flat points \(x_n \to x\), with \(\gamma_n \to \gamma\), \(\eta_n \to \eta\). The images \(h\gamma, h\eta\) are linked, hence meet, and the crossing points \(q_n = h\gamma_n\cap h\eta_n\) stay in a bounded set (the crossing point of two linked geodesics lies within a distance of a base point controlled by the Gromov products of their endpoints, which converge). We first check that \(h\gamma \cap h\eta\) is a single point. Otherwise, by uniqueness of \(\mathrm{CAT}(0)\) segments, \(h\gamma\) and \(h\eta\) share a non-degenerate segment and branch at its endpoints, where both are tangent to the seam; at every point of this segment, including its endpoints, the two geodesics have the same or opposite directions (at a tangential contact, the two straight continuations have the same tangent vector, Lemma~\ref{pf:lem_directions}(iv)), so that the angle between them is \(0\) or \(\pi\) there. But any accumulation point \(q^*\) of \((q_n)\) lies in \(h\gamma \cap h\eta\), and the semicontinuity argument below, applied at \(q^*\), gives \(\angle_{q^*}(h\gamma, h\eta) = \angle_x(\gamma,\eta) \in (0,\pi)\), a contradiction. Hence \(h\gamma \cap h\eta = \{q\}\), and \(q_n \to q\) by Lemma~\ref{pf:lem_crossing_stability}, applied to \(\rho_n = h\gamma_n\) and \(\tau\) a fixed admissible transversal through \(q\), and then to \(h\eta_n\). At a transversal crossing the two adjacent angles sum to \(\pi\) (by Lemma~\ref{pf:lem_directions}(i)), and the Alexandrov angle is upper semicontinuous under such convergence; indeed, it is the infimum of the comparison angles, which are continuous in the configuration and non-decreasing in the parameters~\cite[II.3.1]{bridson2011metric}. If \(\alpha_n, \beta_n\) are the two adjacent angles at \(q_n\) and \(\alpha, \beta\) those at \(q\), then \(\limsup\alpha_n \le \alpha\), \(\limsup\beta_n \le \beta\) and \(\alpha_n + \beta_n = \pi = \alpha + \beta\), so both sequences converge: \(\alpha_n \to \alpha\), \(\beta_n \to \beta\); the same holds at \(x_n \to x\). Hence \(\angle_q(h\gamma,h\eta) = \lim\angle_{q_n}(h\gamma_n,h\eta_n) = \lim\angle_{x_n}(\gamma_n,\eta_n) = \angle_x(\gamma,\eta)\).
\end{proof}

Below, we also verify that \(h\) preserves incidence angles. 
\begin{pps}
\label{pf:prop_wall_angles}
Let \(\tilde w\) be a lifted wall, the axis of a peripheral element \(p\in \pi_{1}(K)\), and let \(\tilde w' = h\tilde w\) be the corresponding wall of \(X'\), the axis of \(\varphi(p)\). For a.e.\ \(r\) and every \(\varphi_0 \in (0,\pi)\), the geodesic crossing \(\tilde w\) at \(\tilde w(r)\) with incidence angle \(\varphi_0\) has image crossing \(\tilde w'\) with incidence angle \(\varphi_0\); by continuity, the same holds for every transversal crossing of \(\tilde w\) by an admissible geodesic.
\end{pps}
\begin{proof}
    Let \(\vartheta_w(r,\varphi)\) be the incidence angle of \(h\eta\) with \(\tilde w'\) at their unique crossing, for \(\eta\) crossing \(\tilde w\) at \(\tilde w(r)\) with incidence \(\varphi\).
    The geodesics crossing \(\tilde w\) from a fixed side modulo \(\langle p\rangle\) carry the measure \(\tfrac14\sin\varphi\,d\varphi\,dr\) on \([0,\ell(p)) \times (0,\pi)\) (wall chart), and \(h\) maps them to the geodesics crossing \(\tilde w'\) from a fixed side modulo \(\langle\varphi(p)\rangle\).
    Moreover, \(\ell(p) = \ell'(\varphi(p))\) since \(\mathbf{EL} = \mathbf{EL}'\) (this also follows from the correspondence between the two families of crossing geodesics, by comparing their total masses).
    Hence the identity of Lemma~\ref{pf:lem_cosine} holds for \(\vartheta_w\), with \(\mu\) replaced by \(dr/\ell(p)\).
    The superadditivity argument of Lemma~\ref{pf:lem_superadd} applies verbatim if we replace \(h\gamma\) by the geodesic \(\tilde w'\) of \(X'\); its second term is now the transfer of the pair \((\eta_1,\eta_2)\) of \emph{admissible} geodesics crossing at the wall point, which equals \(\varphi_2\) by Theorem~\ref{lem:angle_preservation}.
    We obtain \(\vartheta_w(r,\varphi_1+\varphi_2) \geq \vartheta_w(r,\varphi_1) + \varphi_2\).
    Thus, for a.e.\ \(r\), the function \(\varphi \mapsto \vartheta_w(r,\varphi) - \varphi\) agrees a.e.\ with a non-decreasing function (the inequality holds for a.e.\ pair \((\varphi_1,\varphi_2)\)) which vanishes at \(\pi/2\), so \(\vartheta_w(r,\cdot) \leq \mathrm{id}\) on \((0,\pi/2]\), i.e., \(\vartheta_w^{-1}(b) \geq b\). Together with the Jensen inequality \(\int\vartheta_w^{-1}(b)\,dr/\ell(p) \leq b\) we obtain equality, and the strict case gives \(\vartheta_w^{-1}(b) = b\) a.e., hence \(\vartheta_w = \mathrm{id}\).
\end{proof}

\subsection{From angle preservation to isometry}
\label{subsec:isometry}

In this part, we show that, once \(h\) preserves angles, it is induced by an isometry, which completes the proof of Theorem~\ref{thm:main}.
If \(\widetilde{K}_0\) were negatively curved, we could use that the sum of the angles of a geodesic triangle is strictly smaller than \(\pi\), unless the triangle degenerates to a point; this is Otal's original argument~\cite{otal_spectre_1990}.
In our case, however, the space is only CAT(0), and in principle it may contain triangles whose interior angles sum to exactly \(\pi\).
In fact, if \(h\) is ``close to the identity'', three geodesics meeting at an interior point \(x \in \widetilde{K}^\circ\) are mapped to three geodesics of \(\widetilde{K}_0'\) whose intersections determine a geodesic triangle \(\triangle\), also contained in \((\widetilde{K}')^\circ\).
Since the negative curvature of \(\widetilde{K}_0\) is concentrated on \(B(\widetilde{K})\), the triangle \(\triangle\) is flat, and the sum of its interior angles is \(\pi\).
Therefore, in order to obtain an isometry from angle preservation, we look at triangles in a neighbourhood of the boundary set, as we do below.

\begin{thm}[Enriched spectrum rigidity]\label{thm:enriched_rigidity}
    Suppose \( \mathbf{EL} = \mathbf{EL}' \) (with respect to the marking \(\varphi\) fixed in \S\ref{subsec:angle_dissipation}).
    Then there is an isometry \( \phi\colon \mathbb{T}^2 \to \mathbb{T}^2 \) such that \( \mathcal{D}' = \phi \circ \mathcal{D} \).
    In particular, the induced map
    \begin{align*}
        \Phi\colon \Omega & \to \Omega'                      \\
        (x, \omega)       & \mapsto (\phi(x), d\phi(\omega))
    \end{align*}
    is a flow conjugacy between billiard flows.
\end{thm}

Throughout this subsection we assume \(\mathbf{EL} = \mathbf{EL}'\), and \(h\colon \mathcal{G}_0 \to \mathcal{G}'_0\) denotes the map of Definition~\ref{pf:defi_h_unparam}.
Recall that \(h\) is a homeomorphism preserving linking and the currents, and that it maps admissible geodesics to admissible ones (Lemma~\ref{pf:lem_h_unparam}).
It maps the axis of every \(g \in \pi_1(K)\) to the axis of \(\varphi(g)\); in particular, it maps each lifted wall \(\tilde w\) to a lifted wall \(\tilde w' \coloneqq h\tilde w\) of \(\mathcal{D}'\).
Moreover, \(h\) preserves the angles of admissible crossings (Theorem~\ref{lem:angle_preservation}) and the incidence angles of admissible geodesics with the walls (Proposition~\ref{pf:prop_wall_angles}).
All the constructions below are \(\pi_1(K)\)-equivariant over \(\varphi\), and we omit this from the statements.

To prove Theorem~\ref{thm:enriched_rigidity}, we construct a point map \(\widetilde h\colon X \to X'\) inducing \(h\), in three steps.
First, the images under \(h\) of the geodesics through a wall point are concurrent (Lemma~\ref{pf:lem_concurrence_wall}), which defines \(\widetilde h\) on the seam, where it is isometric (Lemma~\ref{pf:lem_seam_map}).
Second, the current controls the crossing positions along every regular geodesic, which gives concurrence at flat points (Lemma~\ref{pf:lem_positions}).
Finally, we show that \(\widetilde h\) is isometric on a dense set by a Crofton argument, and extend it by continuity.

\begin{dfn}[Generic wall points]\label{pf:defi_generic}
    A wall point \(x \in B\) is \emph{generic} if the set of directions at \(x\) (in either adjacent sheet) whose geodesic is regular is conull in the circle of directions at \(x\).
\end{dfn}
\begin{lem}\label{pf:lem_generic_dense}
    All but countably many points of each lifted wall are generic. In particular, generic points are dense in \(B\).
\end{lem}
\begin{proof}
    Let \(\tilde w\) be a lifted wall. By Lemma~\ref{pf:lem_null_charts}(a), the parameters \((r,\varphi)\) of a wall chart of \(\tilde w\) whose geodesic has a tangential contact with \(B\) lie on countably many \(C^1\) curves. A \(C^1\) curve meets all but countably many vertical lines \(\{r = r_0\}\) in a Lebesgue-null set (it can contain a vertical segment for at most countably many \(r_0\), and by Fubini a null set of the plane meets a.e.\ vertical line in a null set); hence, for all but countably many \(r_0\), the directions at \(\tilde w(r_0)\) whose geodesic has a tangential contact form a null set of incidence angles, for each of the two charts. Since a geodesic through a wall point which crosses the wall transversally there and has no tangential contact is regular, \(\tilde w(r_0)\) is generic.
\end{proof}

We start the proof by showing that the images under \(h\) of the regular geodesics through a generic wall point are concurrent; in other words, \(h\) induces a map \(\widetilde h_0\) from the generic points of \(B\) to \(B'\).
\begin{lem}\label{pf:lem_concurrence_wall}
    Let \(x \in B\) be generic and \(\tilde w\) the lifted wall through \(x\). Then the images under \(h\) of all regular geodesics through \(x\) pass through a single point \(\widetilde h_0(x) \in \tilde w'=h\tilde w\).
\end{lem}
\begin{proof}
    Let \(\gamma_1,\gamma_2\) be regular geodesics through \(x\); they cross \(\tilde w\) transversally at \(x\), with incidence angles \(\varphi_1, \varphi_2\), and we label them so that \(\gamma_1\) lies between \(\tilde w\) and \(\gamma_2\), i.e., with \(\psi \coloneqq \angle_x(\gamma_1,\gamma_2)\), \(\varphi_1 + \psi + (\pi-\varphi_2) = \pi\). The images \(h\gamma_1, h\gamma_2\) are admissible, cross \(\tilde w'\) at points \(p_1, p_2\) and cross each other at a point \(q\) (they are linked, and by Theorem~\ref{lem:angle_preservation}, applied to \((\gamma_1,\gamma_2)\) crossing at \(x\), they meet in a single point, at which their angle is \(\psi\)). Suppose that \(p_1, p_2, q\) are not all equal (see Figure~\ref{fig:geodesic_triangles_table}). Then \(p_1p_2q\) is a non-degenerate geodesic triangle of \(X'\) whose side \([p_1,p_2]\) is an arc of \(\tilde w'\) (the unique geodesic between two points of a wall), and, by the proof of Lemma~\ref{pf:lem_superadd}, its interior angles are \(\varphi_1\) and \(\pi - \varphi_2\) at \(p_1, p_2\) (Proposition~\ref{pf:prop_wall_angles}) and \(\psi\) at \(q\) (Theorem~\ref{lem:angle_preservation}); in particular, the sum of these angles is equal to \(\pi\). Since each interior angle is bounded by the corresponding comparison angle and the comparison angles sum to \(\pi\), all three are equal to their comparison angles, and by the Flat Triangle Lemma~\cite[II.2.9]{bridson2011metric} the convex hull of \(p_1p_2q\) is isometric to a Euclidean triangle, in which the side \([p_1,p_2]\) is straight. But \([p_1,p_2]\) is a wall arc of positive length,  whose geodesic curvature with respect to the adjacent sheet is \(-\mathcal{K} \neq 0\), whereas a side of a Euclidean triangle has zero geodesic curvature with respect to the triangle, and the geodesic curvature of a boundary arc of a flat region is intrinsic to the region (by Gauss--Bonnet),
    a contradiction. Hence \(p_1 = p_2 = q\), that is, the images of all regular geodesics through \(x\) meet \(\tilde w'\) at the same point, which we denote by \(\widetilde h_0(x)\). 
\end{proof}

\begin{figure}[t]
        \centering
        \includegraphics[width=0.9\textwidth]{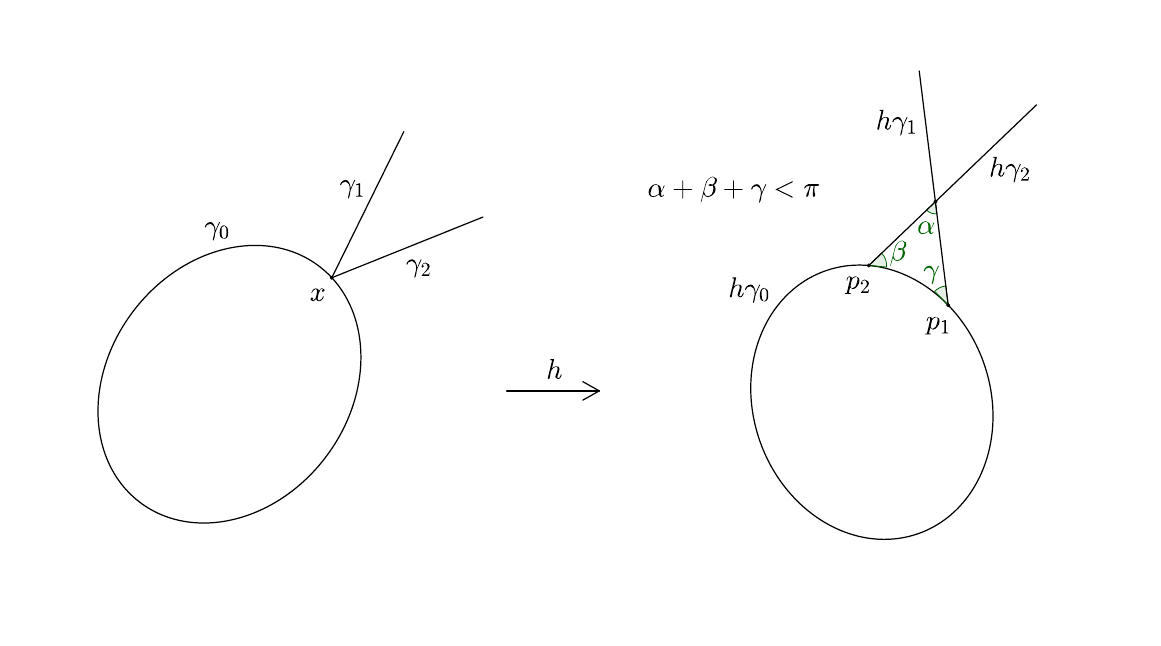}
        \caption{Three geodesics
            intersecting at a boundary point are mapped into a triangle of geodesics whose interior angles sum to \(\pi\),
            which is impossible by the Flat Triangle Lemma since \((p_{1},p_{2})\) is a wall arc.}
        \label{fig:geodesic_triangles_table}
\end{figure}

Next, we show that the map  \(\widetilde h_0\) induces an isometry between the boundaries: 
\begin{lem}\label{pf:lem_seam_map}
    The map \(\widetilde h_0\) extends uniquely to an isometric bijection \(\widetilde h_0\colon (B, d_0) \to (B', d'_0)\), which maps each lifted wall \(\tilde w\) onto \(\tilde w' = h\tilde w\) and satisfies \(\widetilde h_0(x) \in h\gamma\) for every wall point \(x\) and every regular geodesic \(\gamma\) through \(x\).
\end{lem}
\begin{proof}
Let us first assume that the points \(x,y\in B\) are generic. By the Crofton formula (see Lemma~\ref{lem:intrinsic_crofton}), \(d_0(x,y)\) is the \(\widehat m_0\)-mass of the geodesics \(\eta\) separating \(x\) from \(y\). We claim that \(\eta\) separates \(x\) from \(y\) if and only if \(h\eta\) separates \(\widetilde h_0(x)\) from \(\widetilde h_0(y)\), for every \(\eta\) through neither \(x\) nor \(y\). Given such an \(\eta\), let \(S\) be the side of \(\eta\) containing \(x\). The directions at \(x\) whose geodesic is unlinked with \(\eta\) form an arc of length \(\pi - \angle_x(\eta^-,\eta^{+}) > 0\),
and regular directions are conull at \(x\); so there is a regular geodesic \(\gamma_x\) through \(x\), unlinked with \(\eta\), hence contained in \(\overline S\). Then \(h\gamma_x\) is unlinked with \(h\eta\) and contained in the closure of the side \(S'\) of \(h\eta\) corresponding to \(S\) (the sides correspond through \(\partial\varphi\)), and \(\widetilde h_0(x) \in h\gamma_x \subset \overline{S'}\) by Lemma~\ref{pf:lem_concurrence_wall}. The same holds for \(y\), which proves the claim; since \(h_*\widehat m_0 = \widehat m'_0\), we get \(d'_0(\widetilde h_0(x),\widetilde h_0(y)) = d_0(x,y)\).
    \par Now, since generic points are dense in \(B\) (Lemma~\ref{pf:lem_generic_dense}), and \(\widetilde h_0\) is isometric on them, it extends uniquely to an isometric map \(B \to B'\) which sends each lifted wall \(\tilde w\) into \(\tilde w'\). An isometric map from a geodesic line into a geodesic line is onto, so \(\widetilde h_0\) maps \(\tilde w\) onto \(\tilde w'\); since \(h\) maps the set of lifted walls (which are the axes of the peripheral elements) bijectively onto the set of lifted walls of \(X'\), \(\widetilde h_0\colon B \to B'\) is bijective.
    \par Let us now consider an arbitrary point \(x \in B\) on the wall \(\tilde w\), and let \(\gamma\) be regular through \(x\). Choose generic points \(x_n \to x\) on \(\tilde w\), and regular geodesics \(\gamma_n\) through \(x_n\) converging to \(\gamma\) (since \(\gamma\) is regular, its collisions are transversal, so geodesics with nearby initial data converge to \(\gamma\) uniformly on compact sets, with converging endpoints, and the regular ones among them are dense). Then \(\widetilde h_0(x_n) = h\gamma_n \cap \tilde w' \to h\gamma \cap \tilde w'\) by Lemma~\ref{pf:lem_crossing_stability} (with  \(\rho_n = h\gamma_n\) and \(\tau = \tilde w'\)), while \(\widetilde h_0(x_n) \to \widetilde h_0(x)\) by continuity; hence \(\widetilde h_0(x) = h\gamma \cap \tilde w' \in h\gamma\).
\end{proof}

\begin{lem}
\label{pf:lem_positions}
    Let \(\gamma\) be a regular geodesic, with wall points \(x_i = \gamma(t_i)\), \(i \in \mathbb{Z}\), \(t_i < t_{i+1}\). Parametrise \(h\gamma\) by arc length so that \((h\gamma)(t_1) = \widetilde h_0(x_1)\) and \((h\gamma)(t_2) = \widetilde h_0(x_2)\). Then:
    \begin{enumerate}[label=(\roman*)]
        \item \((h\gamma)(t_i) = \widetilde h_0(x_i)\) for every \(i\);
        \item for \(\widehat m_0\)-a.e.\ geodesic \(\eta\) crossing \(\gamma\), at \(\gamma(t)\) say, the image \(h\eta\) crosses \(h\gamma\) at \((h\gamma)(t)\);
        \item for every \(t \in \mathbb{R}\) and every regular geodesic \(\eta\) through \(\gamma(t)\), \(h\eta\) passes through \((h\gamma)(t)\).
    \end{enumerate}
\end{lem}
\begin{proof}
    (i) By Lemma~\ref{pf:lem_seam_map}, the points \(\widetilde h_0(x_i)\) lie on \(h\gamma\), and are at mutual distances \(d'_0(\widetilde h_0(x_i), \widetilde h_0(x_j)) = d_0(x_i,x_j) = |t_i - t_j|\). A point of the line \(h\gamma\) is determined by its distances to the two distinct points \((h\gamma)(t_1)\), \((h\gamma)(t_2)\); since \((h\gamma)(t_i)\) has the right distances, \(\widetilde h_0(x_i) = (h\gamma)(t_i)\).
    \par (ii) It suffices to prove the statement on each chord \([x_1,x_2]\) of \(\gamma\), i.e., for crossings at times \(t \in (t_1,t_2)\), the crossers at the countably many wall points being null. For \(t\in(t_1,t_2)\) and \(\theta\in(0,\pi)\) write \(\eta_{t,\theta} \coloneqq \eta_{\gamma,t,\theta}\) for the geodesic crossing \(\gamma\) at \(\gamma(t)\) from the right with angle \(\theta\) (Lemma~\ref{pf:lem_pairs_measure}), and let \(\sigma(\theta,t)\) be the position on \(h\gamma\) of a crossing point of \(h\gamma\) and \(h\eta_{t,\theta}\), unique when \(h\eta_{t,\theta}\) is clean. By Lemma~\ref{pf:lem_null_charts}(a) and Fubini in the flat chart along \(\gamma\), there is a conull set \(\Theta_0 \subset (0,\pi)\) such that, for \(\theta \in \Theta_0\), the set \(T_\theta\) of those \(t \in (t_1,t_2)\) for which \(\eta_{t,\theta}\) and \(h\eta_{t,\theta}\) are clean is conull in \((t_1,t_2)\). We claim that \(\sigma(\theta,t) = t\) for every \(\theta \in \Theta_0\) and \(t \in T_\theta\); by Lemma~\ref{pf:lem_pairs_measure}(i) this is the statement (ii) for the crossers from the right, and the crossers from the left are the reverses of the crossers from the right of \(\bar\gamma\).
    \par
    Let us now fix \(\theta \in \Theta_0\) and \(t < t'\) in \(T_\theta\). The geodesics \(\eta_{t,\theta}\), \(\eta_{t',\theta}\) are unlinked; indeed, if they were linked, they would cross at a point \(z\), with a positive angle since both are clean, and the geodesic triangle with vertices \(\gamma(t), \gamma(t'), z\) would have interior angles \(\theta\) and \(\pi - \theta\) at \(\gamma(t), \gamma(t')\) (in one order or the other, according to the side of \(z\)) and a positive angle at \(z\), contradicting the angle-sum inequality in the \(\mathrm{CAT}(0)\) space \(X\). Now, for three embedded arcs of the closed disc \(\overline{X}\) with pairwise distinct endpoints, two of which cross the third exactly once and are unlinked with each other, the order in which the third arc meets the other two is determined by the cyclic order of the six endpoints on \(\partial_\infty X\) (the two unlinked arcs are nested, and the one with the inner pair of endpoints is met first or last according to which end of the third arc its endpoints surround). Since \(h\) preserves linking, and \(\partial\varphi\) preserves or reverses all cyclic orders, the order of the crossing points of \(h\eta_{t,\theta}\), \(h\eta_{t',\theta}\) along \(h\gamma\) is that of \(\gamma(t), \gamma(t')\) along \(\gamma\), or globally the opposite; with the orientation of \(h\gamma\) fixed by (i), \(t \mapsto \sigma(\theta,t)\) is non-decreasing on \(T_\theta\). Let \(F_\theta\) be its non-decreasing extension to \((t_1,t_2)\), i.e., \(F_\theta(t) \coloneqq \sup\{\sigma(\theta,t') : t' \in T_\theta,\ t' \le t\}\).
    \par
    As \(t \to t_2\) within \(T_\theta\), the crossers \(\eta_{t,\theta}\) converge, with converging endpoints, to the crosser of \(\gamma\) at \(x_2\) with angle \(\theta\), which for a.e.\ \(\theta\) is admissible and crosses the wall at \(x_2\) transversally; its image passes through \(\widetilde h_0(x_2) = (h\gamma)(t_2)\) (Lemma~\ref{pf:lem_seam_map} and (i)) and crosses \(h\gamma\) only there, so Lemma~\ref{pf:lem_crossing_stability} gives \(F_\theta(t_2^-) = t_2\); likewise \(F_\theta(t_1^+) = t_1\). Shrinking \(\Theta_0\) to a conull set, we assume this for all \(\theta \in \Theta_0\).
    \par
    For \(t_1<a<b<t_2\) and \(0<\theta_1<\theta_2<\pi\), the geodesics crossing \(\gamma|_{[a,b]}\) from the right with angle in \((\theta_1,\theta_2)\) have mass \((b-a)\int_{\theta_1}^{\theta_2}\tfrac14\sin\theta\,d\theta\) by~\eqref{eq:kinematic_flux}; their images are, by Theorem~\ref{lem:angle_preservation}, geodesics crossing \(h\gamma\) with angle in \((\theta_1,\theta_2)\), at the positions \(\sigma(\theta,t)\). The images of the non-clean crossers form an \(\widehat m'_0\)-null set, since \(h_*\widehat m_0 = \widehat m'_0\) and non-clean geodesics are \(\widehat m_0\)-null; hence, by~\eqref{eq:kinematic_flux} along \(h\gamma\), the mass of the image set is \(\int_{\theta_1}^{\theta_2}\tfrac14\sin\theta\,\operatorname{Leb}\bigl(\sigma(\theta, [a,b] \cap T_\theta)\bigr)\,d\theta\). Comparing the two masses for all rational \(a<b\) and \(\theta_1 < \theta_2\), we get \(\operatorname{Leb}\bigl(\sigma(\theta,[a,b]\cap T_\theta)\bigr) = b-a\) for all rational \(a<b\) and all \(\theta\) in a conull subset of \(\Theta_0\), which we still call \(\Theta_0\).
    \par
    Fix \(\theta \in \Theta_0\). For a non-decreasing function \(F\) on an interval and \(a < b\) in it, \(F([a,b]) \subset [F(a), F(b)]\), so \(\operatorname{Leb}(F([a,b])) \le F(b) - F(a)\), with equality if and only if \(F\) is continuous on \([a,b]\) (a jump at a point of \([a,b]\) leaves a gap of positive length in \(F([a,b])\)); moreover \(\sigma(\theta, [a,b]\cap T_\theta) \subset F_\theta([a,b])\). Summing over a partition \(t_1 < a_0 < \dots < a_N < t_2\) by points in \(\mathbb{Q}\) and letting \(a_0 \to t_1\), \(a_N \to t_2\), we get
    \begin{align*}
        t_2 - t_1 = \lim \sum_i \operatorname{Leb}\bigl(\sigma(\theta,[a_i,a_{i+1}]\cap T_\theta)\bigr) &\le \lim \sum_i \bigl(F_\theta(a_{i+1}) - F_\theta(a_i)\bigr)\\
        & = F_\theta(t_2^-) - F_\theta(t_1^+) = t_2 - t_1 .
    \end{align*}
    Hence the inequality is an equality for every partition, so that \(\operatorname{Leb}(F_\theta([a,b])) = F_\theta(b) - F_\theta(a) = b - a\) for all rational \(t_1 < a < b < t_2\): \(F_\theta\) is continuous, and \(F_\theta(b) - F_\theta(a) = b - a\) for all rational \(a < b\). Together with \(F_\theta(t_1^+) = t_1\) this gives \(F_\theta(t) = t\), i.e., \(\sigma(\theta,t) = t\) for every \(t \in T_\theta\), which proves the claim and (ii).
    \par (iii) Let \(t_0 \in (t_1,t_2)\) and let \(\eta\) be a regular geodesic through \(\gamma(t_0)\), crossing \(\gamma\) there from the right (otherwise consider \(\bar\eta\)) with angle \(\theta_0\). Choose \(\theta_n \in \Theta_0\) and \(t_n \in T_{\theta_n}\) with \((t_n,\theta_n) \to (t_0,\theta_0)\). Since \(\eta\) is regular, every collision of \(\eta\) is transversal, so geodesics with nearby initial data converge to \(\eta\) uniformly on compact sets, with converging endpoints: \(\eta_{t_n,\theta_n} \to \eta\). By Lemma~\ref{pf:lem_crossing_stability} and (ii), \(h\gamma \cap h\eta = \lim\, h\gamma \cap h\eta_{t_n,\theta_n} = \lim (h\gamma)(t_n) = (h\gamma)(t_0)\). At a wall point \(\gamma(t_i)\), (iii) is Lemma~\ref{pf:lem_seam_map} together with (i).
\end{proof}

We are now ready to conclude the proof of Theorem~\ref{thm:enriched_rigidity}.
\begin{proof}[Proof of Theorem~\ref{thm:enriched_rigidity}]
    We start by defining \(\widetilde h\) on a dense set. Let \(P \subset X\) be the set of points lying on at least one regular geodesic. For \(x \in P\), choose a regular \(\gamma \ni x\), write \(x = \gamma(t)\), and set \(\widetilde h(x) \coloneqq (h\gamma)(t)\), with the parametrisation of Lemma~\ref{pf:lem_positions}. By Lemma~\ref{pf:lem_positions}(iii), \(\widetilde h(x)\) is the common point of the images of all regular geodesics through \(x\), a description independent of the choice of \(\gamma\); hence \(\widetilde h\) is well defined on \(P\), satisfies \(\widetilde h(\gamma) = h\gamma\) for every regular \(\gamma\) (its restriction to \(\gamma\) being the parametrisation of Lemma~\ref{pf:lem_positions}), and agrees with \(\widetilde h_0\) on \(P \cap B\) by Lemma~\ref{pf:lem_positions}(i). The set \(P\) is dense in \(X\): non-regular geodesics are \(\widehat m_0\)-null (Lemma~\ref{lem:current_properties}), so the phase points lying on them form a \(\widehat m_0 \otimes dt\)-null, hence (Corollary~\ref{pf:cor_dxdphi}) Lebesgue-null, subset of \(T^1\widetilde{K}^\circ\), and \(P\) contains the projection of its complement, which has full measure in \(X\).
    \par Let us now show that \(\widetilde h\) is isometric on \(P\). Let \(x, y \in P\), on regular geodesics \(\gamma_x, \gamma_y\). By the Crofton formula, \(d_0(x,y)\) is the \(\widehat m_0\)-mass of the geodesics separating \(x\) from \(y\), and \(d'_0(\widetilde h(x), \widetilde h(y))\) that of the geodesics separating \(\widetilde h(x)\) from \(\widetilde h(y)\); since \(h_*\widehat m_0 = \widehat m'_0\), it suffices to show that, for \(\widehat m_0\)-a.e.\ \(\eta\), \(\eta\) separates \(x\) from \(y\) if and only if \(h\eta\) separates \(\widetilde h(x)\) from \(\widetilde h(y)\). Let \(\eta\) be a geodesic through neither \(x\) nor \(y\), not sharing a sub-arc with \(\gamma_x\) or \(\gamma_y\) (which is a conull condition by Lemma~\ref{lem:current_properties}(i)). If \(\eta\) is unlinked with \(\gamma_x\), then \(x\) lies on the side of \(\eta\) containing \(\gamma_x\), and \(\widetilde h(x) \in h\gamma_x\) lies on the corresponding side of \(h\eta\), the sides corresponding through \(\partial\varphi\). If \(\eta\) crosses \(\gamma_x\), say at \(\gamma_x(t)\), then for \(\widehat m_0\)-a.e.\ such \(\eta\) the image \(h\eta\) crosses \(h\gamma_x = \widetilde h(\gamma_x)\) at \(\widetilde h(\gamma_x(t))\) (Lemma~\ref{pf:lem_positions}(ii)), and the side of \(h\eta\) containing \(\widetilde h(x)\) is determined by the position of \(\widetilde h(x)\) on \(h\gamma_x\) relative to this crossing point, which is that of \(x\) relative to \(\gamma_x(t)\), the parametrisations of \(\gamma_x\) and \(h\gamma_x\) corresponding; so again \(x\) and \(\widetilde h(x)\) lie on corresponding sides. The same holds for \(y\), and the claim follows.
    \par Since \(\widetilde h\) is isometric on the dense set \(P\), it extends uniquely to an isometric map \(\widetilde h\colon(X,d_0)\to(X',d'_0)\), which is \(\pi_1(K)\)-equivariant over \(\varphi\). It is also surjective, since it is an injective continuous map between the topological planes \(X\) and \(X'\), hence open by invariance of domain, so that its image is open and closed in \(X'\). By continuity, \(\widetilde h = \widetilde h_0\) on \(B\) (the two maps agree on the generic points, which lie on regular geodesics and are dense in \(B\)), so \(\widetilde h\) maps lifted walls onto lifted walls and \(B\) onto \(B'\). Thus \(\widetilde h\) descends to an isometry \(K_0\to K'_0\) mapping \(B(K)\) onto \(B(K')\), and Proposition~\ref{cor:isometry-preserves-boundary} provides a Euclidean isometry \(\phi\) of \(\mathbb{T}^2\) with \(\phi(\mathcal{D}) = \mathcal{D}'\). The last assertion of the theorem is immediate, the billiard flow being determined by the table.
\end{proof}
The proof of Theorem~\ref{thm:main} is complete.

\begin{rmk}
\label{rem:single_trajectory}
    We describe two geometric interpretations of the proof above, which also show where the difficulties lie.

    Let us fix a wall point \(x\) and the sheet \(S\) on one side of it.
    For each direction at \(x\) into \(S\), the geodesic leaves \(x\) along a straight chord until its next wall point \(y\).
    Let \(V_x\) be the union of these chords, which we call the \emph{visibility star} of \(x\) in \(S\); it is a simply connected flat region.
    By Lemma~\ref{pf:lem_seam_map} and Lemma~\ref{pf:lem_positions}(i), the image of such a chord is the chord of \(h\gamma\) between \(\widetilde h_0(x)\) and \(\widetilde h_0(y)\), and it has the same length.
    By Proposition~\ref{pf:prop_wall_angles} and Theorem~\ref{lem:angle_preservation}, the image chords make the same incidence angles with the wall and the same mutual angles.
    Consequently, there is a Euclidean motion \(I_x\) of the developed sheet sending \(V_x\) onto the visibility star \(V'_{\widetilde h_0(x)}\) and agreeing with \(\widetilde h_0\) on \(V_x \cap B\).
    In particular, the wall arcs visible from \(x\) and from \(\widetilde h_0(x)\) have the same shape.

    Similarly, let \(\gamma\) be a regular geodesic with collision points \(x_i = \gamma(t_i)\).
    The image \(h\gamma\) meets the image walls at the points \(\widetilde h_0(x_i)\), with the same mutual distances and the same incidence angles, and since \(\partial\varphi\) preserves or reverses all cyclic orders, the oriented reflection angles coincide as well.
    Hence the billiard trajectories of \(\mathcal{D}\) and \(\mathcal{D}'\) corresponding to \(\gamma\) and \(h\gamma\) have the same flight times and the same reflection angles.
    Their unfoldings to the plane \(\mathbb{R}^2\) (the \(\mathbb{Z}^2\)-cover of the torus) are then polygonal paths with the same chord lengths and turning angles, so they differ by a Euclidean motion \(M_\gamma\) (an anti-motion if \(\partial\varphi\) reverses the orientation).
    This motion maps the collision points of the first path onto those of the second, together with the tangent lines of the obstacles at these points.
    In other words, every regular trajectory of \(\mathcal{D}\) samples the boundary of \(\mathcal{D}'\), with the same local geometry.

    In both descriptions, the geometry between collisions is flat and does not intervene.
    It remains to show that the motions \(I_x\), \(I_{x'}\) (or \(M_\gamma\), \(M_{\gamma'}\)) attached to different stars (or trajectories) are compatible.
    This compatibility is not automatic, and it is hard to obtain it from the data along a single trajectory, even a dense one.
    Indeed, the data of \(\gamma\) (flight times and reflection angles) are invariant under the deck transformations \(g_n \in \pi_1(K)\) lifting the lattice translations, and \(\varphi\) acts trivially on \(\mathbb{Z}^2\).
    Hence the unfolding of \(g_n\gamma\) is the translate \(\Gamma + n\) of the unfolding \(\Gamma\) of \(\gamma\), and that of \(h(g_n\gamma) = \varphi(g_n)h\gamma\) is \(\Gamma' + n\), from where it follows that \(M_{g_n\gamma} = T_{n - L(n)} \circ M_\gamma\), where \(L\) is the linear part of \(M_\gamma\) and \(T_v\) is the translation by \(v\).
    The density of the projection of \(\gamma\) in \(T^1\mathcal{D}\) only gives that the collision points \(x_i\) are dense in \(\partial\mathcal{D}\) modulo \(\mathbb{Z}^2\), that is, that they sample the lifted obstacles \(\widetilde{\mathcal{O}} + n\) for infinitely many \(n\).
    However, the sample taken on \(\widetilde{\mathcal{O}} + n\) is transported by \(M_\gamma\) to a point of \(M_\gamma(\widetilde{\mathcal{O}}) + L(n)\), and unless \(L(n) \in \mathbb{Z}^2\), these samples land on different translates of \(\partial\widetilde{\mathcal{O}}'\) and give no information on the shape of a single obstacle.
    The missing information is \(L(n) = n\), that is, \(M_{g_n\gamma} = M_\gamma\), which means that the motions attached to \(\gamma\) and to \(g_n\gamma\) agree at a common point of these two geodesics; this is the concurrence of Lemma~\ref{pf:lem_positions}(iii).
    Once concurrence is known, the Crofton argument of the proof gives the isometry on a dense set directly; in particular, no dense orbit is used, and the ergodicity of the billiard flow is not needed.
\end{rmk}

%% file: sections/periodic.orbits.tex
\subsection{Length functionals and periodic orbits}
We use the notation introduced in Section~\ref{sec:prelimin}.
We let \( \gamma_{(0,0)}\coloneqq\pi^{-1}\circ \gamma \), and for each integer vector \( (i,j)\in \mathbb{Z}^2 \), we let \( \gamma_{(i,j)}\coloneqq\gamma_{(0,0)}+(i,j) \).
Denote by
\begin{align*}
    \tau_{(i,j)} \colon \mathbb{T}_\mathcal{D}\times \mathbb{T}_\mathcal{D} & \to \mathbb{R}_+                                                            \\
    (s,s')                                                                  & \mapsto \tau_{(i,j)}(s,s')\coloneqq\|\gamma_{(0,0)}(s)-\gamma_{(i,j)}(s')\|
\end{align*}
the Euclidean distance between the points of \( \partial \widetilde{\mathcal{D}} \) with respective parameters \( s,s' \), in the respective cells \( (0,0) \) and \( (i,j) \).
Recall that for any \( (s,s')\in \mathbb{T}_\mathcal{D} \times \mathbb{T}_\mathcal{D} \), we have
\begin{equation}\label{eq:first_gen_fct}
    \partial_1 \tau_{(i,j)}(s,s')=-\sin \varphi_{(i,j)},\quad \partial_2 \tau_{(i,j)}(s,s')=\sin \varphi_{(i,j)}',
\end{equation}
where \( \varphi_{(i,j)} \) and \( \varphi_{(i,j)}' \) are the respective angles that the line segment between \( \gamma_{(0,0)}(s) \) and \( \gamma_{(i,j)}(s') \) makes with respect to the normal vector at these points.
Moreover,
\begin{equation}\label{eq:der_seconde_gen}
    \begin{array}{ll}
         & \partial_{11} \tau_{(i,j)}(s,s')=\mathcal{K}(s)\cos \varphi_{(i,j)}+\frac{\cos^2 \varphi_{(i,j)}}{\tau_{(i,j)}(s,s')},    \\
         &
        \partial_{12} \tau_{(i,j)}(s,s')=-\frac{\cos \varphi_{(i,j)}\cos\varphi_{(i,j)}'}{\tau_{(i,j)}(s,s')},                       \\
         & \partial_{22} \tau_{(i,j)}(s,s')=\mathcal{K}(s')\cos \varphi_{(i,j)}'+\frac{\cos^2 \varphi_{(i,j)}'}{\tau_{(i,j)}(s,s')},
    \end{array}
\end{equation}
where \(\mathcal{K}(s)\), \(\mathcal{K}(s')\) are the respective curvatures at the points with parameters \(s\) and \(s'\).
Let \( \sigma=(\bar s_{0},\dots,\bar s_{q-1}) \) denote the collision points of a periodic orbit of (least) period \( q \geq 2 \), with \( \bar s_k\in \mathbb{T}_{\rho_k} \), \( \rho_k\in \{1,\dots,m\} \), for \( k\in \{0,\dots,q-1\} \).
We let \( \rho(\sigma)\coloneqq(\rho_k)_{k=0,\dots,q-1} \).
Let us consider the lift to \( \widetilde{\mathcal{D}} \) of this orbit starting in the cell \( (0,0) \), and denote by  \( \{(\tilde i_k,\tilde j_k)\}_{k=0,\dots,q} \) the labels of the cells where the first \( (q+1) \) bounces happen, with \( (\tilde i_0,\tilde j_0)=(0,0) \).
For \( k\in \{1,\dots,q\} \), set \( I_k\coloneqq(\tilde i_k,\tilde j_k)-(\tilde i_{k-1},\tilde j_{k-1}) \).
We denote \( I(\sigma)\coloneqq(I_k)_{k=1,\dots,q} \).
As we assume that the table has finite horizon, there is a constant \( M > 0 \) such that \( I_k\in [-M,M]^{2} \), for any \( k\in \{1,\dots,q\} \).

The following result is essentially~\cite[Lemma A.2]{de2023marked}; since the context is slightly different, we reproduce its proof for the reader's convenience.
\begin{lem}\label{lem:quadratic-form}
    Let \( \sigma=(\bar s_{0},\dots,\bar s_{q-1}) \) denote the collision points of a periodic orbit of (least) period \( q \geq 2 \).
    With the above notations, let \( \rho\coloneqq\rho(\sigma)=(\rho_k)_{k=0,\dots,q-1} \), and  \( I=I(\sigma)\coloneqq(I_k)_{k=1,\dots,q} \).
    We define the length functional
    \begin{equation*}
        L_\rho^I\colon
        \left\{
        \begin{array}{rcl}
            \mathbb{T}_{\rho_0}\times \dots \times \mathbb{T}_{\rho_{q-1}} & \to     & \mathbb{R}                                                          \\
            (s_{0},\dots,s_{q-1})                                          & \mapsto & \sum_{k = 1}^{q-1}\tau_{I_k}(s_{k-1},s_{k})+\tau_{I_q}(s_{q-1},s_0)
        \end{array}
        \right..
    \end{equation*}
    Then \( \sigma \) is a critical point of \( L_\rho^I \), and it holds:
    \begin{align*}
        L_{\rho}^I(s_{0},\dots,s_{q-1}) -
        L_{\rho}^I(\bar s_{0},\dots,\bar s_{q}) & =
        Q_{\rho}^I(s_{0}-\bar s_{0},\dots, s_{q}-\bar s_{q}) + R_{\rho}^I,
    \end{align*}
    where \( Q_{\rho}^I \) is a positive definite quadratic form and \( R_{\rho}^I \) is a remainder term that satisfies the estimate:
    \begin{align*}
        R_{\rho}^I = O(\|(s_{0}-\bar s_{0},\dots, s_{q}-\bar s_{q})\|^{3}).
    \end{align*}
\end{lem}
\begin{proof}
    Let us abbreviate \( L=L_{\rho}^I \).
    The fact that \( \sigma=(\bar s_0,\dots,\bar s_{q-1}) \) is a critical point for \( L \) among configurations \( (s_0,\dots,s_{q-1}) \) follows from \eqref{eq:first_gen_fct}, as \( \sigma \) is a closed billiard orbit.
    Hence the lemma follows from Taylor's formula,	provided that we show that the Hessian of \( L \) at \( (\bar s_{0},\dots,\bar s_{q}) \) is positive definite.
    It is immediate from the definition of \( L \) that \( \partial_{ij}L = 0 \) if \( |i-j| > 1 \pmod q \); the Hessian is thus a \emph{cyclic} tridiagonal symmetric matrix, with corner entries \(\partial_{0,q-1}L = \partial_{q-1,0}L \neq 0\) coming from the term \(\tau_{I_q}(s_{q-1}, s_0)\).
    For notational convenience, let \( \tau_{k} = \tau_{I_{k+1}}(\bar{s}_{k},\bar s_{k+1}) \); let \( \bar\varphi_{k}\in[-\frac \pi 2,\frac \pi 2] \) denote the angle formed by the trajectory at the \( k \)-th collision point and the unit normal at \( \bar s_{k} \); and let \( \mathcal{K}_{k} > 0 \) denote the curvature at \( \bar s_{k} \).
    Using~\eqref{eq:der_seconde_gen}, a direct computation gives, for \(u = (u_0, \dots, u_{q-1})\) (indices mod \(q\)),
    \[
        D^2L(\bar s)(u, u) = \sum_{k=0}^{q-1}\frac{1}{\tau_k}\bigl(\cos\bar\varphi_k\, u_k - \cos\bar\varphi_{k+1}\,u_{k+1}\bigr)^2 + 2\sum_{k=0}^{q-1}\mathcal{K}_k \cos\bar\varphi_k\, u_k^2.
    \]
    Since \(\mathcal{K}_k > 0\) and \(\cos\bar\varphi_k > 0\) at a non-grazing orbit, both sums are non-negative and the second is positive unless \(u = 0\); hence \(D^2L(\bar s)\) is positive definite.
\end{proof}

\begin{dfn}
    The spectrum of the table \( \mathcal{D} \) is the metric space \( (\mathrm{Spec}(\mathcal{D}), d_H) \), where \( \mathrm{Spec}(\mathcal{D}) \) is the set of all periodic orbits of the billiard flow of \( \mathcal{D} \), and \( d_H \) is the Hausdorff metric on the phase space \( \Omega \).
\end{dfn}
It follows from Corollary~\ref{cor:fin_card_peri} that \( \mathrm{Spec}(\mathcal{D}) \) is a discrete totally bounded space.

\par Remark that a critical point of some length functional \( L \) as in Lemma~\ref{lem:quadratic-form} is not necessarily a periodic orbit, as it may consist of a non-physical trajectory, that is, one that passes through one of the obstacles.

Let us introduce the following sets of orbits.
\begin{dfn}
    Given a word \(\omega = (\omega_0, \dots, \omega_n) \in \mathcal{A}^{n+1}\), the cylinder \( C(\omega; n) \) is the set of points \( x \in \mathcal{M} \) whose lifted orbits have the first \( n+1 \) bounces in the scatterers \( B_{\omega_0}, \dots, B_{\omega_n} \), in this order.
\end{dfn}

\begin{dfn}
    A \emph{generalised periodic orbit} of \( \mathcal{D} \) is the projection onto \( \mathcal{M} \) of a critical point of one of the functionals \( L \).
    A generalised orbit that is not a periodic orbit of the billiard \( \mathcal{D} \) is called a \emph{ghost} or \emph{non-physical}.
\end{dfn}

\begin{thm}[Theorem 2.2 in \cite{bunimovich_variational_1995}]\label{thm:bunimovich_periodic}
    Let \(  \omega = (\omega_0, \dots, \omega_n)\in \mathcal{A}^{n+1} \) be a word in \( \Sigma_\mathcal{M} \) such that
    \begin{align*}
        \pi(B_{\omega_i}) & \neq \pi(B_{\omega_{i+1}}), \text{ for } i = 0, \dots, n-1; \\
        \pi(B_{\omega_n}) & = \pi(B_{\omega_0}).
    \end{align*}
    Then there is at most one generalised periodic orbit in the cylinder \( C(\omega; n) \).
\end{thm}

\par On the other hand, one periodic orbit in \( \mathrm{Spec}(\mathcal{D}) \) is associated to several different words \( \omega \), corresponding to distinct liftings.
Even when we restrict ourselves to liftings starting at the \( (0,0) \) cell, distinct starting points in the same orbit may lead to distinct words.
We will therefore need a choice of point in each of the periodic orbits of \( \mathrm{Spec}(\mathcal{D}) \).
\par Given \( O \in \mathrm{Spec}(\mathcal{D}) \), we can see it as a union of line segments in the square \( [0,1)^2. \) The finite set \( (\cup_i B_i) \cap O \) corresponds to all the collision points of the orbit \( O \).
It is an ordered set, with the lexicographical ordering inherited from \( [0,1)^2 \times [-\pi/2, \pi/2] \).
Set
\[
    \zeta(O) \coloneqq \min \{O \cap (\cup_i B_i)\} \in \mathcal{M}.
\]
\begin{dfn}
    The \emph{type} of the periodic orbit \( O \in \mathrm{Spec}(\mathcal{D}) \) is the unique sequence \( [O] \coloneqq \imath(\zeta(O)) \in \Sigma_\mathcal{M} \).
\end{dfn}

\subsection{Good tables}\label{sec:good_tables}
Recall that we denote by \( \mathfrak{D}^m \) the space  of all Sinai billiard tables in \( \mathbb{T}^2 \) with finite horizon, with \( m \) scatterers and \(C^r\) boundary, \(r \geq 3\), identified with an open subset of \( C^r_{\mathrm{emb}}(\mathbb{T}_\mathcal{D}, \mathbb{T}^2) \) as in Section~\ref{sec:prelimin}; recall also that a collision is \emph{regular} if it is not grazing, and that a periodic orbit is \emph{singular} if it has a grazing collision, and \emph{regular} otherwise.
\par We want to show that for a generic choice of table in \( \mathfrak{D}^m \), the set of periodic orbits has simple length spectrum and contains no singular orbits.
\newline
\subsubsection{Table perturbations and periodic orbits.}
We begin by analysing how perturbations on the scatterers affect periodic orbits.
\begin{dfn}
    Given \( p = \gamma(s_0) \in \partial\mathcal{D} \), we say a table \( \mathcal{D}' \in \mathfrak{D}^m \) is a localised perturbation of \( \mathcal{D} \) at \( p \) if there is a neighbourhood \( U \) of \( s_0 \in \mathbb{T}_\mathcal{D} \) such that the defining maps of \( \mathcal{D} \) and \( \mathcal{D}' \) coincide outside \( U \).
\end{dfn}
Intuitively, a small perturbation around a single collision point of a periodic orbit should generate a new periodic orbit close to the original one.
In order to see that this is in fact what happens, we investigate how a localised perturbation at a point \( p=\gamma(s_0) \) alters the length functional of the table.
\par We begin by describing the perturbation in terms of a \(C^r\) function \( \lambda \) defined on \( \mathbb{T}_{\mathcal{D}} \).
For \( \epsilon >0 \), consider the \( \epsilon \)-perturbation
\begin{equation}\label{eq:defor_descrip}
    \gamma^\lambda(s) = \gamma(s) + \epsilon\lambda(s)n_s,
\end{equation}
where \( n_s \) is the normal vector to \( \gamma(s) \) at \( s \).
For sufficiently small \( \epsilon \), the perturbation \( \gamma^\lambda \) remains strictly convex and the resulting billiard table is still a Sinai billiard.
The perturbation \( \gamma^\lambda \) is a localised perturbation at \( p = \gamma(s_0) \) exactly when \( \lambda \) is supported in a small neighbourhood around \( s_0 \).
\par As before, we look at liftings of such curves to the Euclidean plane and consider mappings
\begin{equation*}
    \tau^\lambda_{(i,j)}(s,s') \coloneqq \|\gamma^\lambda_{(0,0)}(s)-\gamma^\lambda_{(i,j)}(s')\|,
\end{equation*}
whose first order approximation is
\begin{equation*}
    \tau^\lambda_{(i,j)}(s,s') = \tau_{(i,j)}(s,s') + \epsilon(\lambda(s)\cos\varphi_{(i,j)}(s,s') +\lambda(s')\cos\varphi'_{(i,j)}(s,s')) + O(\epsilon^{2}).
\end{equation*}
Now, we suppose \( \sigma = (s_0, \dots, s_{q-1}) \) corresponds to a sequence of bouncing points of the table \( \mathcal{D} \).
We want to study the relation between critical points for the length functionals of the original and perturbed tables.
For notational convenience, we drop the indices \( (i,j) \).
Then the length functional \( L^\lambda \) can be written in terms of the original length functional \( L \) as
\begin{align*}
     & L^\lambda(\sigma) = \sum_{i=0}^{q-1}\tau^\lambda(s_i, s_{i+1})                                                                                                  \\
     & = \sum_{i=0}^{q-1}h(s_i, s_{i+1}) + \epsilon\sum_{i=0}^{q-1}(\lambda(s_i)\cos\varphi(s_i, s_{i+1}) + \lambda(s_{i+1})\cos\varphi'(s_i,s_{i+1}) + O(\epsilon^2),
\end{align*}
or, more concisely, as
\begin{equation}\label{eq:pert_len_func}
    L^\lambda(\sigma) = L(\sigma) + \epsilon P^\lambda(\sigma) + O(\epsilon^2).
\end{equation}
If \( \sigma \) is a periodic orbit of the unperturbed table \( \mathcal{D} \), then consecutive angles \( \varphi \) and \( \varphi' \) have the same cosine, and hence \( P^\lambda \) becomes
\begin{equation}\label{eq:descrip_Plamb_per_point}
    P^\lambda(\sigma) = 2\sum_{i=0}^{q-1}\lambda(s_i)\cos\varphi(s_i, s_{i+1}),
\end{equation}
with the usual convention that \( s_q = s_0 \).
Writing \( \sigma^\lambda \coloneqq (s_0^\lambda, \dots, s^\lambda_{q-1}) \) to denote the generalised periodic orbit of the perturbed table with same type as \( \sigma \), we have
\begin{equation*}
    s_i^\lambda = s_i + \epsilon\psi_i + O(\epsilon^2).
\end{equation*}
Thus, \( \sigma^\lambda = \sigma + \epsilon\psi \) satisfies
\begin{equation}\label{eq:per_orb_perturb}
    \begin{aligned}
        0 & = DL^\lambda(\sigma^\lambda)                                                           \\
          & = DL(\sigma^\lambda) + \epsilon DP^\lambda(\sigma^\lambda) + O(\epsilon^2)             \\
          & = DL(\sigma) + \epsilon D^2L(\sigma)\psi + \epsilon DP^\lambda(\sigma) + O(\epsilon^2) \\
          & = \epsilon D^2L(\sigma)\psi + \epsilon DP^\lambda(\sigma) + O(\epsilon^2),
    \end{aligned}
\end{equation}
where the last equality comes from the fact that \( \sigma \) is a critical point for \( L \).
Thus, the first order change \( \psi \) we see in the collision points for a critical point \( \sigma^\lambda \) is the solution of the linear system
\begin{equation}\label{eq:1st_order_posi}
    D^2L(\sigma)\psi = -DP^\lambda(\sigma).
\end{equation}
We already know that the Hessian of \( L \) is a non-degenerate positive definite matrix, so for a non-positive solution \( \psi \) to \eqref{eq:1st_order_posi} to exist it is sufficient that the term \( -DP^\lambda(\sigma) \) is non-zero.
We conclude that, unless \( \sigma \) is a critical point of \( P^\lambda \), the perturbed flow will have a new generalised periodic orbit arbitrarily close to the original one.
In particular, for a generic choice of perturbations, such perturbed orbits always exist.
Naturally, if \( \sigma \) is a critical point of \( P^\lambda \), then \( \psi = 0 \) is the unique solution of \eqref{eq:1st_order_posi} and \( \sigma \) remains a generalised periodic orbit after the perturbation.
\begin{dfn}
    Given a point \( p = \gamma(s_0) \) in the obstacle \( \Gamma_i \subset \partial\mathcal{D} \), we say a localised perturbation is a \emph{retraction of \( \mathcal{O}_i \) at \( p \)} if it is of the form
    \begin{equation*}
        \widetilde\gamma(s) = \gamma(s) - \epsilon\Psi(s)n_s,
    \end{equation*}
    where \( \epsilon >0 \) and \( \Psi \) is a positive bump function supported at a small interval around \( s_0 \), with \( \Psi(s_0) = 1. \)
\end{dfn}
\begin{rmk}\label{rem:pertubed_periodic_orbits}
    If a periodic orbit \( \sigma \) has a grazing collision at \( p \), then any retraction of the obstacle tangent to the orbit at \( p \), no matter how small, removes the tangency without changing the original orbit, i.e. the perturbed orbit \( \sigma^\lambda \) has the same trace, but without the grazing collision at \( p \).
    \par On the other hand, a generic perturbation of \( \sigma \) at a regular collision point will generate a non-singular generalised orbit.
    Indeed, the law of reflection means that a change in the angle of one of the collisions will propagate through all the orbit, so that, generically, a change of the angle and position of one the vertices will either kill the periodic orbit \( \sigma \) by pushing it into an obstacle (i.e., the perturbed orbit \( \sigma^\lambda \) is a ghost orbit) or it will move the orbit away from the obstacles it previously touched tangentially (that is, \( \sigma^\lambda \) is regular).
\end{rmk}
\subsubsection{Killing off singular orbits.} As suggested in Remark \ref{rem:pertubed_periodic_orbits}, localised perturbations can ``destroy'' singular periodic orbits by turning them into regular or ghost ones.
We claim this procedure can be applied, in a controlled way, to remove any finite number of singular orbits without producing new ones.
\begin{lem}\label{lem:fin_hor_persists}
    If \( \mathcal{D} \) has finite horizon, so does every sufficiently close localised perturbation of \( \mathcal{D}. \)
\end{lem}
\begin{proof}
    Let \( p \) be a collision point in \( \partial\mathcal{D} \), and let \( L \) be the image of the orbit of \( p \) in the torus.
    If \( p \) is regular, then, for sufficiently small perturbations localised at \( p \), the corresponding perturbed orbit \( L' \) is still regular and bounces at the perturbed obstacle \( \partial\mathcal{D}'_i \).
    Thus, if a perturbation around \( p \) were to create a periodic orbit with no collisions at all, then the collision at \( p \) must be grazing.
    It must also be the only collision on the orbit \( L \); otherwise, a localised perturbation would not eliminate all collisions.
    Hence, the line \( L \) is closed, tangent to the obstacle \( \Gamma_i \) at \( p \), and a positive distance away from all the other curves in \( \partial\mathcal{D} \).
    In particular, this means the obstacle \( \Gamma_i \) lies entirely on one side of \( L \).
    Now, on the side of \( L \) opposite the obstacle, any line parallel to \( L \) and sufficiently close to it would not meet any of the obstacles \( \Gamma_j \), which contradicts our finite horizon hypothesis.
    It follows that localised perturbations of finite horizon tables cannot create tables of unbounded horizon, as we wanted.
\end{proof}
\par Given \( n \in \mathbb{N} \), there exists a positive constant \( M \), depending on \( n \) and on the table \( \mathcal{D} \) (more specifically, on the maximum free-flight time \( \tau_\mathrm{max}) \) with the following property: for every periodic point \( p \in \mathcal{M} \) of order less or equal to \( n \), the fundamental trace of its lift to \( \widetilde{\mathcal{D}} \) is completely contained in \( [-M, M]^2 \).
\begin{figure}[t]
    \centering
    \includegraphics[width=.8\textwidth]{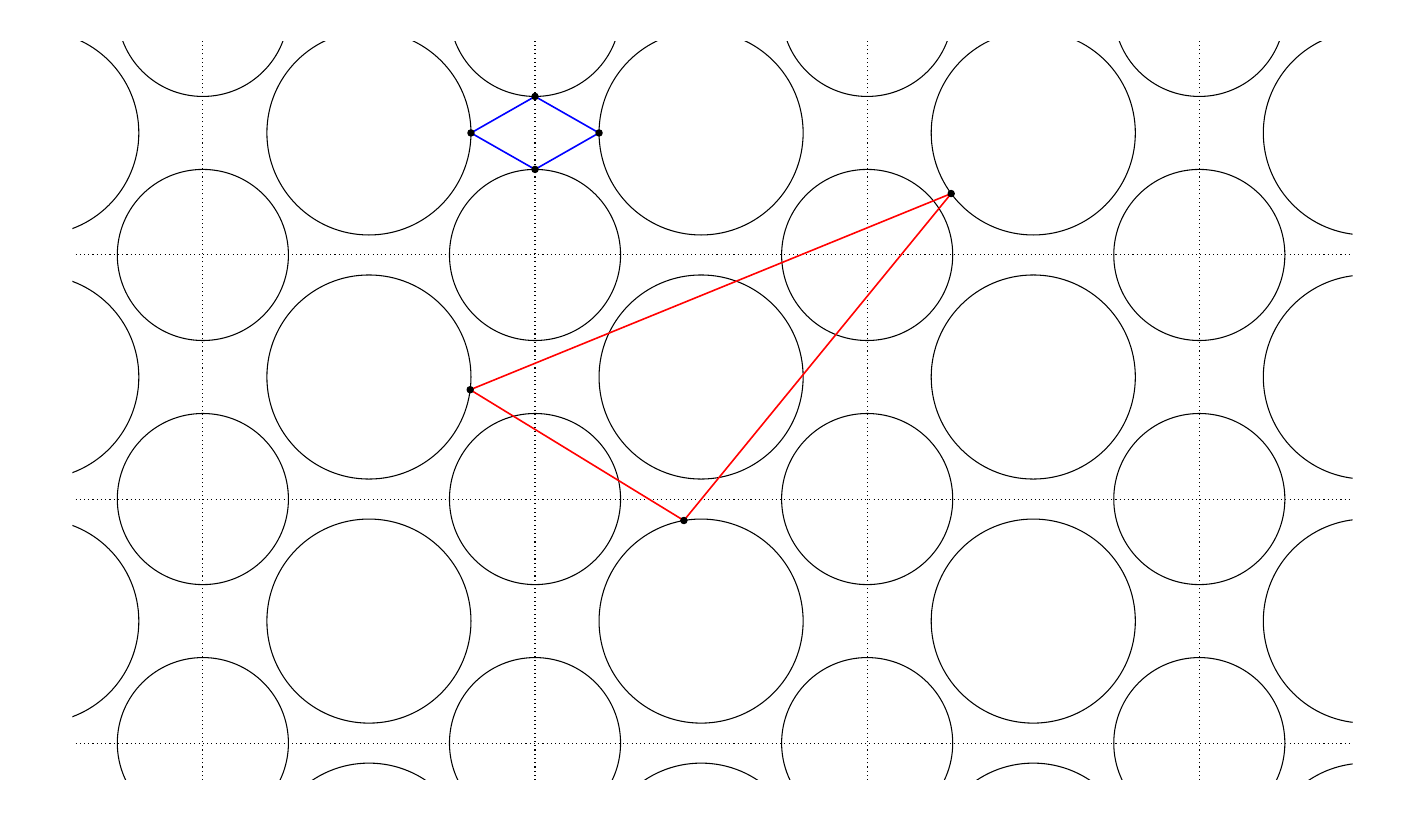}
    \caption{Physical and non-physical trajectories on the lifted table.}
    \label{fig:lifted_table_per}
\end{figure}
\par We consider the set of all possible symbolic sequences of period \( n \) orbits:
\[
    \Sigma_n \coloneqq \left\lbrace (a_0, \dots, a_{n-1}) \;\middle|\;
    \begin{tabular}{@{}l@{}}
        \( a_i \in ([-M, M]^2 \cap \mathbb{Z}^2) \times \{1, \dots, k\} \subset \mathcal{A} \); \\
        \( \pi(B_{a_i}) \neq \pi(B_{a_{i+1}}) \text{ for } i = 0, \dots, n-2 \);                \\
        \( \pi(B_{a_{n-1}}) = \pi(B_{a_0}) \).
    \end{tabular}
    \right\rbrace.
\]
\par Now, according to Theorem~\ref{thm:bunimovich_periodic}, to each element of the finite set \( \Sigma_n \) corresponds at most one periodic orbit, which can be non-physical.
Whether or not a sequence has an associated periodic trajectory depends on the length functional of the table.
The important point here is that localised perturbations may change the length functional slightly, but they do not create or destroy scatterers, and hence do not change the set \( \Sigma_n \).
Moreover, as a consequence of Equation \eqref{eq:1st_order_posi}, ghost and regular trajectories are stable under perturbations, whereas singular orbits can become either regular or non-physical.
Either way, the sequence associated to a generalised orbit will still be associated to the corresponding perturbed orbit.
\begin{lem}\label{lem:remov_tangency}
    Let \( p \in \partial\mathcal{D} \) and \( n \in \mathbb{N} \).
    Then, for any sufficiently small neighbourhood \( U \subset \partial\mathcal{D} \) of \( p \), there exists localised perturbations of \( \mathcal{D} \) at \( p \) with no singular points of order \( n \) or less inside \( U \).
\end{lem}
\begin{proof}
    As there are finitely many periodic orbits of order \( n \), the result is straightforward for any point not lying on a periodic orbit.
    Assume \( p \in \Gamma_a, a \in \mathbb{A} \), is a vertex of a periodic trajectory.
    Let
    \[
        \mathcal{A} \subset \bigcup_{i=2}^n\Sigma_i
    \]
    be the set of all sequences of \( n \) or less elements containing \( a \).
    We further partition \( \mathcal{A} \) into the sets \( \mathcal{A}_{\mathrm{emp}} \), of sequences corresponding to no periodic trajectory; and its complement \( \mathcal{A}_\mathrm{orb}\coloneqq \mathcal{A}\setminus \mathcal{A}_{\mathrm{emp}} \), of sequences corresponding to periodic trajectories (physical or not).
    \par As there are finitely many closed orbits of order less than \( n \), there exists a neighbourhood \( U \) of \( p \) which does not contain the vertices of any periodic orbits of order \( n \) or less other than ones reflecting exactly at \( p \).
    Now, the generalised orbits reflecting at \( p \) will all persist under perturbation localised at \( U \), provided it is sufficiently small.
    In other words, small localised perturbations do not decrease the size of the set \( \mathcal{A}_\mathrm{orb} \).
    They could however, create a new periodic orbit \( O \) associated to one the sequences previously in \( \mathcal{A}_\mathrm{emp} \).
    Since, however, such set is finite, there are only finitely many perturbations which could possibly create such new orbits.
    \par Choose the perturbation to be a retraction \( \widetilde\gamma \) of \( \Gamma_a \) at \( p \) with support \( J \) such that \( \widetilde\gamma(J) \subset U \).
    If any of the orbits containing \( p \) were grazing at \( p \), the grazing collision does not happen any more after the perturbation.
    The regular and ghost orbits reflecting at \( p \) will persist after the perturbation, provided it is small enough.
    Moreover, we can choose our perturbation in such a way that any singular orbit having a regular collision at \( p \) is either regular or non-physical after perturbation.
    Finally, if the perturbation creates any new periodic orbits, that is, if it increases the size of \( \mathcal{A}_\mathrm{orb} \), then we repeat the procedure considering the new orbits and a suitable new neighbourhood \( \widetilde U \subset U \).
    Since \( \#\mathcal{A}_\mathrm{emp} < \infty \), the process terminates after finitely many steps.
\end{proof}
Recall that \( \mathcal{R}_n \) is the subset of \( \mathfrak{D}^m \) consisting of tables without any singular periodic orbits of period less or equal to \( n \).
\begin{lem}\label{lem:dense_nonsingular_n}
    For every \( n\geq 2 \), the set \( \mathcal{R}_n \) is open and dense in \( \mathfrak{D}^m \).
\end{lem}
\begin{proof}
    As there are only finitely many periodic orbits of order less than \( n \), and regular orbits are persistent under sufficiently small perturbations, it follows that any table close enough to a table in \( \mathcal{R}_n \) is also in \( \mathcal{R}_n \).
    We need only show that \( \mathcal{R}_n \) is dense.
    Let \( \mathcal{D} \in \mathfrak{D}^m \).
    Using induction on the order of periodic orbits, we construct an element of \( \mathcal{R}_n \) arbitrarily close to \( \mathcal{D} \).
    \\
    \par\noindent\emph{The base case \( n=2 \)}.
    A singular periodic orbit of period \(2\) has two collision points \(x, y\) (on possibly distinct scatterers), at least one of which, say \(p = x\), is grazing.
    Retract the obstacle at \(p\) by an arbitrarily small localised perturbation tangent to the orbit: by Remark~\ref{rem:pertubed_periodic_orbits} this removes the grazing collision at \(p\) without creating a new period-\(2\) singular orbit, since the other collision (if regular) persists and the reflection law at \(y\) is unchanged.
    Iterating over the finitely many period-\(2\) singular orbits yields a table in \(\mathcal{R}_2\) arbitrarily close to \(\mathcal{D}\).
    \\
    \par\noindent\emph{The general case}.
    Assume, by induction, that the result holds for \( \mathcal{R}_{n-1} \).
    We lift our initial table \( \mathcal{D} \) and consider the set \( \Sigma_n \) of all possible encodings of periodic orbits of order \( n \).
    Since there are only finitely many periodic orbits, the collection of all reflection points of singular orbits is also finite, and we can choose a minimal finite sub-collection \( \{p_1, \dots, p_k\} \) containing at least one reflection point of each singular orbit.
    \par Starting with \( p_1 \), we can use Lemma \ref{lem:remov_tangency} to find an arbitrarily small perturbation \( \mathcal{D}' \) of \( \mathcal{D} \) around \( p_1 \) which removes all the singular orbits of order \( n \) reflecting at \( p_1 \), without creating any new ones.
    Moreover, since the perturbation is localised away from singular orbits not reflecting at \( p_1 \), such orbits comprise the full set of singular orbits of order \( n \) of \( \mathcal{D}' \).
    \par Now, according to Lemma \ref{lem:fin_hor_persists}, \( \mathcal{D}' \) still has finite horizon, so that we can once again repeat the construction done to \( \mathcal{D} \) to encode its periodic orbits using the set \( \Sigma_n \).
    Furthermore, the set \( \{p_2, \dots, p_k\} \) is a minimal sub-collection of reflection points representing all the singular orbits of \( \mathcal{D}' \), and we can repeat the procedure for \( p_2 \).
    \par After finitely many steps we obtain a perturbed table, which we still call \( \mathcal{D}' \), all of whose periodic orbits of order \( n \) are regular.
    In particular, since regular orbits are persistent under perturbation, any table sufficiently close to \( \mathcal{D}' \) also has this property.
    It follows from the induction hypothesis that \( \mathcal{D}' \) can be approximated by a table with no singular orbits of order less than or equal to \( n \).
    As such approximations can be made arbitrarily close, it follows that there are tables in \( \mathcal{R}_n \) arbitrarily close to \( \mathcal{D} \), as we wanted.
\end{proof}

Recall that \( \mathfrak{S}_n \) is the subset of \( \mathfrak{D}^m \) consisting of tables such that no distinct periodic orbits of period less or equal to \( n \) have the same length.

\begin{lem}\label{lem:dense_simple_n}
    For every \( n \geq 2 \), the set \( \mathfrak{S}_n \) is open and dense in \( \mathfrak{D}^m \).
\end{lem}

\begin{proof}
    From the equality in \eqref{eq:per_orb_perturb} we can see that the first order variation of perimeter from a perturbation of the form \( \gamma^\lambda(s) = \gamma(s) + \epsilon\lambda(s)n_s \) is
    \begin{equation}\label{eq:1st_ord_length}
        \Delta_\sigma(\psi) = P^\lambda(\sigma) + DP^\lambda(\sigma)(\psi),
    \end{equation}
    where \( \sigma = (s_1, \dots, s_n) \) are the reflection points of the original periodic orbit and the corresponding periodic orbit for the perturbation reflects at points \( s_i^\lambda = s_i + \epsilon\psi_i \).
    In particular, arbitrarily small perturbations induce arbitrarily small changes in length.
    \par Given a table \( \mathcal{D} \) in \( \mathfrak{S}_n \), since there are only finitely many periodic orbits of order \( n \) or less, we have \( \lvert L(\sigma) - L(\sigma')\rvert > \eta > 0 \) for any two distinct periodic orbits \(\sigma,\sigma'\) of period at most \(n\).
    Any perturbation of \( \mathcal{D} \) small enough not to alter the length of such orbits by more than \( \eta/2 \) would still be an element of \( \mathfrak{S}_n \), so \( \mathfrak{S}_n \) is open.
    \par As for density, let \( U \) be an open set in \( \mathfrak{D}^m \), and \( \mathcal{D} \) a table in \( U \).
    Due to the existence of the minimum free-flight time for any table in \( \mathfrak{D}^m \), it is sufficient to show that arbitrarily close to \( \mathcal{D} \) there is a table where no two periodic orbits of the same order have the same length.
    If \( \mathcal{D} \) itself has such a property there is nothing to show, so assume \( O, O' \) are two distinct closed trajectories in \( \mathcal{D} \) with same length.
    Suppose first that \( O \) has a reflection point \( p \) which is not a collision point of \( O' \).
    We choose the perturbation \( \lambda \) localised around \( p \) so that it changes the tangent vector of the obstacle at \( p \), but does not move \( p \) itself.
    More explicitly, \( \lambda \) satisfies
    \begin{equation}
        \begin{cases}\label{eq:cond_lamb}
            \lambda(s_0) = 0, \\
            \lambda'(s_0) \neq 0.
        \end{cases}
    \end{equation}
    With \eqref{eq:cond_lamb}, every collision parameter of \(\sigma\) lying in the support of \(\lambda\) satisfies \(\lambda(s_i) = 0\), so \(P^\lambda(\sigma) = 0\) by \eqref{eq:descrip_Plamb_per_point} and the first-order variation of length vanishes.
    Expanding to second order and using \eqref{eq:1st_order_posi},
    \begin{align*}
        L^\lambda(\sigma^\lambda) - L(\sigma)
         & = \epsilon^2\Big(\tfrac12 D^2L(\sigma)(\psi,\psi)
        + DP^\lambda(\sigma)(\psi)\Big) + O(\epsilon^3)                       \\
         & = -\tfrac{\epsilon^2}{2}\,D^2L(\sigma)(\psi,\psi) + O(\epsilon^3),
    \end{align*}
    which is strictly negative for \(\epsilon\) small: indeed \(\psi \neq 0\), because the component of \(DP^\lambda(\sigma)\) at \(s_0\) contains the non-vanishing term \(2\lambda'(s_0)\cos\bar\varphi_0\) (recall that the collision at \(p\) is regular, so \(\cos\bar\varphi_0 > 0\)).
    Thus the length of the perturbed orbit \(\widetilde{O}\) differs from that of \(O'\), and the two orbits now have distinct lengths.
    Thus, the increase in length of the perturbed orbit \( \widetilde O \) is non-zero, and now the two orbits have distinct lengths.
    As this process is localised at \( p \), it does not change the perimeter of any other orbit not reflecting at \( p \).
    Thus, the only situation where such a procedure could fail is when \( p \) is a collision point of more than one orbit (in particular, if the orbits \( O, O' \) share all their reflection points).
    If this happens, the orbits cannot have the same reflection angles, so the corresponding vectors \( DP^\lambda(\sigma) \) and \( DP^\lambda(\sigma') \) are distinct.
    In particular, after a preliminary perturbation of the type \eqref{eq:cond_lamb}, the corresponding orbits now have different reflection points, and the procedure above can be applied.
    \par Finally, as all the perturbations involved can be made arbitrarily small, it follows that \( \mathcal{D} \) can be approximated arbitrarily well by a table in \( \mathfrak{S}_n \), and this set is dense.
\end{proof}

%% file: sections/billiard.cycles.tex
\par As seen in Section~\ref{sec:geodesic}, every Kourganoff surface \(K\), when compressed on its third coordinate, converges to a metric space \(K_0\), whose geometry is determined by the table \(\mathcal{D}\).
This suggests that the enriched marked length spectrum, while defined through the limiting surface \(K_0\), is in fact an object intrinsic to the table \( \mathcal{D} \).
\par In what follows, we describe \( \mathbf{EL}_{\mathcal{D}; K} \) in terms of lengths on the table only.
Recall that a geodesic of the length space \((\overline{\mathcal{D}}, d_\mathcal{L})\) is a locally shortest path; by Theorem~\ref{thm:K_0_as_gluing_metric_space} and the structure of the geodesics of \(K_0\), these are the concatenations of straight chords of \(\mathcal{D}\) and of arcs of \(\partial\mathcal{D}\), the junctions between a chord and an arc being tangential.
We begin with the following definition.
\begin{dfn}[Billiard cycle]
    A closed curve \( \sigma\colon I \to \overline{\mathcal{D}} \) is a \emph{billiard cycle} if there is a finite (possibly empty) set \( t_0 < t_1 < \dots < t_{k-1} \) in \( I \) such that
    \begin{enumerate}
        \item[(i)] \( \sigma(t_i) \in \partial\mathcal{D} \) for \( 0 \leq i \leq k-1 \);
        \item[(ii)] the segment \( \sigma\rvert_{[t_i, t_{i+1}]} \) is a geodesic of \( d_\mathcal{L} \) for every \( i = 0, \dots, k-1, \) where \( t_k = t_0 \) (for \(k = 0\), \(\sigma\) itself is a closed geodesic of \(d_\mathcal{L}\));
        \item[(iii)] at each point \( \sigma(t_i) \) the curve undergoes a specular reflection, that is:
              \[
                  \dot\sigma_+(t_i) = \dot\sigma_-(t_i) - 2\langle \dot\sigma_-(t_i), n\rangle n,
              \]
              where
              \( \dot\sigma_s(t_i) = \lim_{t \to t_i^s}\frac{1}{t - t_i}(\sigma(t) - \sigma(t_i))\), for \( s\in \{+ , -\} \),
              are the side limits of \(\sigma\) at \(t_i\), \(n\) is the normal to \( \partial\mathcal{D} \) at \(\sigma(t_i)\) and \( \langle \cdot, \cdot \rangle\) is the usual Euclidean inner product.
    \end{enumerate}
    The points \(\sigma(t_i)\) at which the reflection is not tangential (\(\langle \dot\sigma_-(t_i), n\rangle \neq 0\)) are the \emph{regular collision points} of the cycle; since (iii) is automatic at a tangential contact, where the two adjacent \(d_\mathcal{L}\)-geodesics concatenate into a single one, we may and shall always assume that the \(t_i\) are exactly the regular collision points.
\end{dfn}
It is clear that every closed orbit of the billiard flow is also a billiard cycle, as well as every closed geodesic of the length metric \( d_\mathcal{L} \).

\begin{thm}\label{thm:char_billiard_cycles_1}
    The map \(\alpha \mapsto p \circ \gamma^\alpha_0\) is a surjection from \( \mathcal{C}(K) \) onto the set of billiard cycles of \(\mathcal{D}\), up to parametrisation.
    More precisely, \(p \circ \gamma_0^\alpha\) is a billiard cycle for every \(\alpha\), and every billiard cycle \(\sigma\) is of the form \(p \circ \gamma_0^\alpha\), traversed once if \(\sigma\) has an even number of regular collision points and twice otherwise.
\end{thm}
\begin{proof}
    \par We begin by showing that, for any \( \alpha \in \mathcal{C}(K) \), the geodesic \( \gamma^\alpha_0 \) projects to a billiard cycle.
    We consider an arc-length parametrisation \( \gamma^\alpha_0\colon [0, l_\alpha] \to K_0 \).
    We need to construct a partition \( \{t_0 < \dots < t_{k-1}\} \) of \( [0, l_\alpha] \) satisfying properties (i), (ii) and (iii).
    \par First, suppose \( \gamma^\alpha_0(\mathbb{R}) \subset B(K) \).
    In this case, the curve \( \gamma^\alpha_0 \) coincides with one of the walls, hence projects to one of the \( \Gamma_i \) (possibly traversed several times), which is a closed geodesic of \( d_\mathcal{L} \); we take \(k = 0\).
    \par Otherwise, the set \( \gamma^\alpha_0\setminus B(K) \) is a non-empty open subset of \( \gamma^\alpha_0\), each of whose connected components is a chord of a flat sheet joining two points of \(B(K)\).
    A chord leaving an obstacle tangentially does not meet that obstacle again (the obstacle is convex), so every chord has Euclidean length at least the distance between two distinct obstacles; hence there are finitely many of them, and
    \[
        \gamma^\alpha_0 \setminus B(K) = \coprod_{j=0}^{q}\gamma^\alpha_0(I_j)
    \]
    for finitely many open subintervals \(I_j \subset [0, l_\alpha] \), labelled so that \( \sup I_j < \inf I_{j+1} \).
    Let \( t_j \coloneqq \sup I_j \) and \( s_j \coloneqq \inf I_{j+1} \) (indices modulo \(q+1\)).
    Each \( \gamma^\alpha_0\rvert_{I_j} \) lies in a single sheet, hence projects to a straight segment of \( \mathcal{D} \) (Lemma~\ref{lem:d_0_is_distance_reestated}).
    At the contact \(\gamma_0^\alpha(t_j)\) with the seam, three cases may occur (Lemma~\ref{pf:lem_directions}):
    \begin{enumerate}
        \item[\(\bullet\)] the contact is \emph{transversal}: then \(s_j = t_j\), the geodesic changes sheet, and by Lemma~\ref{pf:lem_directions}(iii) the projections of \(\gamma^\alpha_0\rvert_{I_j}\) and \(\gamma^\alpha_0\rvert_{I_{j+1}}\) form a billiard path, i.e., obey the specular reflection law at \(p(\gamma^\alpha_0(t_j))\);
        \item[\(\bullet\)] the contact is \emph{tangential and isolated} (a grazing contact): then \(s_j = t_j\) and \(p(\gamma^\alpha_0\rvert_{I_j \cup \{t_j\} \cup I_{j+1}})\) is a single straight segment, tangent to \(\partial\mathcal{D}\) at \(p(\gamma^\alpha_0(t_j))\);
        \item[\(\bullet\)] the geodesic \emph{glides}: then \(s_j > t_j\), \(\gamma^\alpha_0\rvert_{[t_j, s_j]}\) is a wall arc, and \(p(\gamma^\alpha_0\rvert_{\overline{I_j} \cup [t_j,s_j] \cup \overline{I_{j+1}}})\) is a chord--arc--chord path of \(\overline{\mathcal{D}}\) with tangential junctions, i.e., a geodesic of \(d_\mathcal{L}\) (the two chords may lie in different sheets of \(K_0\), and this can not be seen on the table).
    \end{enumerate}
    \par So we proceed as follows: we let the partition consist of the parameters \(t_j\) at which the contact is transversal.
    Between two consecutive such parameters, the projection of \( \gamma^\alpha_0 \) is a concatenation of straight segments and boundary arcs with tangential junctions, which is a geodesic of \(d_\mathcal{L}\) by the above; at the parameters of the partition the projection obeys the specular reflection law.
    Hence \(p \circ \gamma^\alpha_0\) is a billiard cycle, whose regular collision points are the transversal contacts of \(\gamma^\alpha_0\) with the seam.
    This finishes the first part of the proof.
    \par Conversely, we show that every billiard cycle of \( \mathcal{D} \) is the image under \( p\) of a geodesic \( \gamma^\alpha_0 \).
    Let \( \sigma \) be a billiard cycle, and \( x_i = \sigma(t_i), \ i = 0, 1, \dots, k-1 \), its regular collision points, with \( t_k = t_0 \).
    Denote by \(\mathcal{D}_\mathrm{up}\) and \(\mathcal{D}_\mathrm{down}\) the two closed sheets of \(K_0\); by Lemma~\ref{lem:sheet_isometry}, \( p \) restricts to isometries \( (\mathcal{D}_\mathrm{up}, d_0) \to (\overline{\mathcal{D}}, d_\mathcal{L}) \) and \( (\mathcal{D}_\mathrm{down}, d_0) \to (\overline{\mathcal{D}}, d_\mathcal{L}) \), and each \( x_i \in \partial\mathcal{D}\) has a single preimage \( \widetilde x_i \in B(K) \).
    We lift \( \sigma \) to a curve \( \sigma_\mathrm{up} \) in \( K_0 \) in the following way.
    We lift the first geodesic segment \( \sigma\rvert_{[t_0, t_1]} \) to the sheet \( \mathcal{D}_\mathrm{up} \), which can be done in a unique way, and obtain a \(d_0\)-geodesic segment from \(\widetilde x_0\) to \(\widetilde x_1\).
    Then we lift the second geodesic segment \( \sigma\rvert_{[t_1, t_2]} \) to the sheet \( \mathcal{D}_\mathrm{down} \), and so on, alternating sheets at each regular collision point.
    By Lemma~\ref{pf:lem_directions}(iii), the concatenation at \(\widetilde x_1\) of two geodesic segments lying in different sheets is a local geodesic of \(K_0\) exactly when the collision is specular; hence the lift of \( \sigma\rvert_{[t_0, t_2]} \) thus defined is a \(d_0\)-geodesic, and so on.
    \par If we proceed this way, we get a curve \( \sigma_\mathrm{up} \) in \(K_0 \) from \(\widetilde x_0\) to \(\widetilde x_k = \widetilde x_0\), which is locally minimising everywhere except maybe at its endpoints.
    If \(k\) is even, the last segment \(\sigma\rvert_{[t_{k-1}, t_k]}\) is lifted to \(\mathcal{D}_\mathrm{down}\), so that \(\sigma_\mathrm{up}\) closes up with a sheet change at \(\widetilde x_0\), and the reflection being specular there, \(\sigma_\mathrm{up}\) is a closed local geodesic of \(K_0\).
    If \(k\) is odd, the lifts of \( \sigma\rvert_{[t_0, t_1]} \) and \( \sigma\rvert_{[t_{k-1}, t_k]} \) lie in the same sheet \(\mathcal{D}_\mathrm{up}\), and \(\sigma_\mathrm{up}\) is not locally minimising at \(\widetilde x_0\): the convexity of the scatterers means that for any two points \( \sigma(t_0-\epsilon) \) and \( \sigma(t_0 + \epsilon) \) the straight segment
    \[
        r(t) = \sigma(t_0-\epsilon) + t\bigl(\sigma(t_0+\epsilon) - \sigma(t_0-\epsilon)\bigr), \qquad t \in [0,1],
    \]
    is entirely contained in \( \mathcal{D} \), lifts to \(\mathcal{D}_\mathrm{up}\), and has Euclidean length smaller than \( l_0(\sigma\rvert_{[t_0-\epsilon, t_0]}) + l_0(\sigma\rvert_{[t_0, t_0+\epsilon]}) \).
    To deal with this, we consider the curve \( \sigma_\mathrm{down} \), the lift of \( \sigma \) constructed in the same way but starting in the sheet \( \mathcal{D}_\mathrm{down} \); it goes from \(\widetilde x_0\) to \(\widetilde x_0\), starting in \(\mathcal{D}_\mathrm{down}\) and ending in \(\mathcal{D}_\mathrm{up}\).
    The concatenation \( \sigma_\mathrm{up}\ast\sigma_\mathrm{down} \) (which traverses \(\sigma\) twice) changes sheet at every passage through \(\widetilde x_0\) and at every other regular collision point, with the specular law, and is therefore a closed local geodesic of \( K_0 \).
    \par Finally, we let \( \alpha \) be the free homotopy class of the closed local geodesic just constructed (\(\sigma_\mathrm{up}\) or \(\sigma_\mathrm{up}\ast\sigma_\mathrm{down}\) according to the parity of \(k\)).
    A closed local geodesic of the locally \(\mathrm{CAT}(0)\) space \(K_0\) lifts to an axis of the corresponding element of \(\pi_1(K)\), hence is the unique closed geodesic in its class (Lemma~\ref{lem:pi1_hyperbolic}(iii), Lemma~\ref{lem:unique_minimising_geodesic}); thus it equals \(\gamma^\alpha_0\).
    In the even-parity case \(p(\gamma^\alpha_0) = \sigma\); in the odd-parity case \(p \circ \gamma^\alpha_0\) traverses \(\sigma\) once in each sheet, so it coincides with \(\sigma\) as a subset of \(\overline{\mathcal{D}}\) and \(l_0(\gamma^\alpha_0) = 2\,l(\sigma)\).
    In either case \(p\) maps \(\gamma^\alpha_0\) onto the billiard cycle \(\sigma\), which shows the surjectivity.
\end{proof}

\subsection{The enriched length functional}
\par In this section we think of \( \mathcal{D} \) as a periodic billiard table on \( \mathbb{R}^2 \), and consequently of \( K \) as a periodic surface in \( \mathbb{R}^3 \).
Our goal is to describe billiard cycles as minima of a length functional, in the spirit of Appendix~\ref{sec:periodic.orbits}.
To that end, we keep the same notation:
a cycle \(\sigma\) visits a periodic sequence of scatterers \( \Gamma_{\rho_0}, \dots , \Gamma_{\rho_{q-1}}, \  \rho_k \in \{ 1,\dots, m \} \), and we let \( \rho = \rho(\sigma)\coloneqq(\rho_k)_{k=0,\dots,q-1} \).
Lifting \(\sigma\) to \( \widetilde{\mathcal{D}} \) starting in the cell \( (0,0) \), we denote by  \( \{(\tilde i_k,\tilde j_k)\}_{k=0,\dots,q} \) the labels of the cells of the first \( (q+1) \) scatterers visited, with \( (\tilde i_0,\tilde j_0)=(0,0) \), and we set \( I_k\coloneqq(\tilde i_k,\tilde j_k)-(\tilde i_{k-1},\tilde j_{k-1}) \) for \( k\in \{1,\dots,q\} \) and \( I = I(\sigma)\coloneqq(I_k)_{k=1,\dots,q} \).
Finally, on each scatterer \(\Gamma_{\rho_k}\) the cycle follows a boundary arc (possibly reduced to a point) from an \emph{entry point} to an \emph{exit point}, in one of the two orientations of \(\Gamma_{\rho_k}\); we let \(\varepsilon_k \in \{+,-\}\) be that orientation and \(\varepsilon = \varepsilon(\sigma) \coloneqq (\varepsilon_k)_{k = 0, \dots, q-1}\) (when the arc is a point, either sign may be chosen).
\par For \(i \in \{1, \dots, m\}\) and \(\varepsilon \in \{+,-\}\), let \( D^\varepsilon_i\colon \mathbb{T}_i \times \mathbb{T}_i \to [0, l_i) \), \( (e,s) \mapsto D^\varepsilon_i(e,s)\), be the length of the arc of \(\Gamma_i\) from \(\gamma_i(e)\) to \(\gamma_i(s)\) traversed in the direction \(\varepsilon\), i.e., \(D^+_i(e,s) = (s - e) \bmod l_i\) and \(D^-_i(e,s) = (e - s) \bmod l_i\); it is smooth away from the diagonal \(\{e = s\}\), where it jumps.
The length \(\min(D^+_i, D^-_i)\) of the shortest arc would not suffice here, for a gliding arc of a closed geodesic of \(K_0\) can be longer than half the perimeter of the scatterer.
\begin{dfn}[Link sets and visibility domain]
    Given a pair of indices \( (i, j) \in \{1, \dots, m\}^2 \) and a cell offset \(J \in \mathbb{Z}^2\), the \emph{link set} \( L_{ij}^J \subset \mathbb{T}_i \times \mathbb{T}_j\) is the set of pairs \((s, e)\) such that the open segment joining \(\gamma_i(s)\) to \(\gamma_j(e) + J\) in \(\mathbb{R}^2\) is contained in the interior of the lifted table \(\widetilde{\mathcal{D}}\).
    It is a closed subset of \(\mathbb{T}_i \times \mathbb{T}_j\), possibly disconnected (a third scatterer may hide part of \(\Gamma_j + J\) from \(\Gamma_i\)).
    Given data \((\rho, I)\) as above, the \emph{visibility domain} is
    \begin{align*}
        \mathbb{L}^I_\rho \coloneqq \bigl\{ S = &(e_{0}, s_0, e_1, s_1, \dots, e_{q-1}, s_{q-1}) \in (\mathbb{T}_{\rho_0} \times \mathbb{T}_{\rho_0}) \times \dots \times (\mathbb{T}_{\rho_{q-1}} \times \mathbb{T}_{\rho_{q-1}}) :\\
         &(s_{k-1}, e_k) \in L^{I_k}_{\rho_{k-1}\rho_k} \text{ for } k = 1, \dots, q \bigr\},
    \end{align*}
    with the convention that \( (e_q, s_q) = (e_0, s_0) \) and \(I_q\) is the offset returning to the initial cell.
\end{dfn}
To a point \(S \in \mathbb{L}^I_\rho\) and orientations \(\varepsilon\) we associate the closed curve \(c_S\) of \(\overline{\mathcal{D}}\) obtained by concatenating, for each \(k\), the chord from \(\gamma_{\rho_{k-1}}(s_{k-1})\) to \(\gamma_{\rho_k}(e_k)\) (unobstructed, by definition of the link sets) with the arc of \(\Gamma_{\rho_k}\) from \(\gamma_{\rho_k}(e_k)\) to \(\gamma_{\rho_k}(s_k)\) traversed in the direction \(\varepsilon_k\).
\begin{dfn}[Enriched length functional]
    The enriched length functional associated to the data \((\rho, I, \varepsilon)\) is
    \[
        \mathrm{EL}^{I,\varepsilon}_\rho \colon \left\{
        \begin{array}{lcl}
            \mathbb{L}^I_\rho \to \mathbb{R} \\
            S \mapsto      {\displaystyle{\sum_{k=1}^{q}}} \tau_{I_k}(s_{k - 1},e_{k}) + D^{\varepsilon_k}_{\rho_k}(e_{k}, s_{k})
            \\ \qquad =
            L^I_\rho(s_0, e_1, \dots, s_{q-1}, e_q) + {\displaystyle{\sum_{k=1}^{q}}} D^{\varepsilon_k}_{\rho_k}(e_{k}, s_{k}),
        \end{array}
        \right.
    \]
    where \(\tau_{I_k}\) and \(L^I_\rho\) are the chord lengths and the length functional of Appendix~\ref{sec:periodic.orbits} (there denoted \(L\)), evaluated on the \(2q\)-tuple of chord endpoints.
    By construction, \(\mathrm{EL}^{I,\varepsilon}_\rho(S)\) is the \(d_\mathcal{L}\)-length of the closed curve \(c_S\).
\end{dfn}
Remark that if \(s =  (s_0, \dots, s_{q-1}) \) is a periodic orbit of the billiard, that is, a critical point of \( L^I_\rho \) (a non-degenerate local minimum, by Lemma~\ref{lem:quadratic-form}), then \(L^I_\rho(s) = \mathrm{EL}^{I,\varepsilon}_\rho(S) \) for \(S =  (s_0, s_0, s_1, s_1, \dots, s_{q-1}, s_{q-1}) \) and any \(\varepsilon\).
More generally, we have the following.
\begin{pps}\label{pps:cycles_are_minima}
    Let \(\sigma\) be a billiard cycle with data \((\rho, I, \varepsilon)\), let \(S_\sigma \in \mathbb{L}^I_\rho\) be the point formed by its entry and exit points, and let \(\mathbb{L}_\sigma\) be the connected component of \(\mathbb{L}^I_\rho\) containing \(S_\sigma\).
    Then \(S_\sigma\) is a global minimum of \(\mathrm{EL}^{I,\varepsilon}_\rho\) on \(\mathbb{L}_\sigma\), and it is the only one: \(\mathrm{EL}^{I,\varepsilon}_\rho(S) > \mathrm{EL}^{I,\varepsilon}_\rho(S_\sigma) = l(\sigma)\) for every \(S \in \mathbb{L}_\sigma \setminus \{S_\sigma\}\).
    In particular, the lengths of the billiard cycles of \(\mathcal{D}\), hence the enriched marked length spectrum, are determined by the functionals \(\mathrm{EL}^{I,\varepsilon}_\rho\), which are defined in terms of the table only.
\end{pps}
\begin{proof}
    Let \(\widetilde\sigma\) be the closed local geodesic of \(K_0\) constructed in the proof of Theorem~\ref{thm:char_billiard_cycles_1}, and \(\alpha\) its free homotopy class; recall that \(\widetilde\sigma = \gamma^\alpha_0\) and that \(p \circ \widetilde\sigma\) traverses \(\sigma\) once or twice according to the parity of the number of regular collision points of \(\sigma\); we treat the even case, the odd one being identical after doubling.
    For \(S \in \mathbb{L}^I_\rho\), lift the curve \(c_S\) to a closed curve \(\widetilde c_S\) of \(K_0\) by lifting its chords to the sheets prescribed by the sheet pattern of \(\widetilde\sigma\) (the \(k\)-th chord of \(c_S\) is lifted to the sheet containing the \(k\)-th chord of \(\widetilde\sigma\)) and its arcs to the seam; the pattern closes up, so \(\widetilde c_S\) is a closed curve, which depends continuously on \(S\), and \(\widetilde c_{S_\sigma} = \widetilde\sigma\).
    Since the lifts of the chords are chords of flat sheets and the lifts of the arcs are wall arcs, the \(d_0\)-length of \(\widetilde c_S\) is the \(d_\mathcal{L}\)-length of \(c_S\), namely \(\mathrm{EL}^{I,\varepsilon}_\rho(S)\) (Theorem~\ref{thm:d_0_is_a_metric}).
    By continuity, all the curves \(\widetilde c_S\), \(S \in \mathbb{L}_\sigma\), are freely homotopic to \(\widetilde\sigma\), i.e., lie in the class \(\alpha\); therefore
    \[
        \mathrm{EL}^{I,\varepsilon}_\rho(S) = l_0(\widetilde c_S) \geq \ell(\alpha) = l_0(\gamma_0^\alpha) = \mathrm{EL}^{I,\varepsilon}_\rho(S_\sigma),
    \]
    with equality only if \(\widetilde c_S\) is a closed curve of minimal length in its class: its lift to \(\widetilde K_0\) then joins a point \(x\) to \(gx\), with \(g \in \alpha\), and has length \(\ell(g) = \min d_0(\cdot, g\,\cdot)\), so that it is an axis segment of \(g\), and \(\widetilde c_S\) is the closed geodesic \(\gamma_0^\alpha = \widetilde\sigma\) (Lemma~\ref{lem:pi1_hyperbolic}(iii)).
    As the entry and exit points of \(\widetilde c_S\) are those of \(S\), this forces \(S = S_\sigma\).
\end{proof}